\documentclass{amsart}
\usepackage{amssymb,amsfonts,amsmath,graphicx,mathtools,tensor}
\usepackage[alphabetic,y2k,lite]{amsrefs}
\usepackage{fullpage, mystyle}
\usepackage{markdown}
\usepackage{MnSymbol}
\usepackage{xcolor}
\usepackage{tikz}
\usetikzlibrary{shapes.geometric, arrows}
\usetikzlibrary{shapes.multipart} 
\usetikzlibrary{calc}
\usetikzlibrary{scopes}
\usetikzlibrary{math}
\usetikzlibrary{decorations.markings,decorations.pathreplacing}
\usetikzlibrary{fadings}
\usepackage{tikz-cd}

\calclayout

\tikzset{
every picture/.style={line width=0.8pt, >=stealth,
                       baseline=-3pt,label distance=-3pt},
emptynode/.style={circle,minimum size=0pt, inner sep=0pt, outer
sep=0},
dotnode/.style={fill=black,circle,minimum size=2.5pt, inner sep=1pt, outer
sep=0},
small_dotnode/.style={fill=black,circle,minimum size=2pt, inner sep=0pt, outer
sep=0},
morphism/.style={fill=white,circle,draw,thin, inner sep=1pt, minimum size=15pt,
                 scale=0.8},
small_morphism/.style={fill=white,circle,draw,thin,inner sep=1pt,
                       minimum size=10pt, scale=0.8},
ellipse_morphism/.style args={#1}{fill=white,circle,draw,thin,inner sep=1pt,
                       minimum size=5pt, scale=0.8,
												ellipse, draw, rotate=#1},
coupon/.style={draw,thin, inner sep=1pt, minimum size=18pt,scale=0.8},
semi_morphism/.style args={#1,#2}{
                  fill=white,semicircle,draw,thin, inner sep=1pt, scale=0.8,
                  shape border rotate=#1,
                  label={#1-90:#2}},
regular/.style={densely dashed}, 
edge/.style={very thick, draw=green, text=black},
overline/.style={preaction={draw,line width=2mm,white,-}},
thin_overline/.style={preaction={draw,line width=#1 mm,white,-}},
thin_overline/.default=2,
thick_overline/.style={preaction={draw,line width=3mm,white,-}},
really_thick/.style={line width=3mm, gray},
boundary/.style={thick,  draw=blue, text=black},
ribbon/.style={line width=1.5mm, postaction={draw,line width=1mm,white}},
ribbon_u/.style args={#1,#2}{line width=#1mm, postaction={draw,line width=#2mm,white}},
cell/.style={fill=black!10},
subgraph/.style={fill=black!30},
midarrow/.style={postaction={decorate},
                 decoration={
                    markings,
                    mark=at position #1 with {\arrow{>}},
                 }},
midarrow/.default=0.5,
midarrow_rev/.style={postaction={decorate},
                 decoration={
                    markings,
                    mark=at position #1 with {\arrow{<}},
                 }},
midarrow_rev/.default=0.5,
block/.style={rectangle, rounded corners, text centered, draw=black, align=center}
}

\tikzstyle{block} = [rectangle, rounded corners, text centered, draw=black, align=center]

\tikzfading[name=fade inside, inner color=transparent!80, outer
color=transparent!10]

\newcommand{\al}{\alpha}
\newcommand{\albar}{{\overline{\alpha}}}
\newcommand{\be}{\beta}
\newcommand{\bebar}{{\overline{\beta}}}
\newcommand{\ga}{\gamma}

\newcommand{\ph}{\varphi}

\newcommand{\eps}{\varepsilon}

\newcommand{\lmb}{\lambda}

\newcommand{\Z}{\mathcal{Z}}      
\newcommand{\ZMu}{\mathcal{Z}^{\text{M\"u}}}
\newcommand{\VV}{\mathbf{V}}
\newcommand{\WW}{\mathbf{W}}
\newcommand{\Graph}{\text{Graph}}
\newcommand{\VGraph}{\text{VGraph}}

\newcommand{\ZTV}{Z_\text{TV}}    
\newcommand{\ZRT}{Z_\text{RT}}    
\newcommand{\ZRTCY}{Z_\text{RT}^\text{CY}}    
\newcommand{\ZCY}{{Z_\text{CY}}}

\newcommand{\ZCYsk}{{Z_\text{CY}^\text{sk}}}

\newcommand{\hatZCY}{\hat{Z}_\text{CY}}
\newcommand{\hatZCYsk}{{\hat{Z}_\text{CY}^\text{sk}}}

\newcommand{\Zel}{{\mathcal{Z}^\text{el}}}
\newcommand{\ZA}{{\cZ(\cA)}}
\newcommand{\ZZA}{{\Zel(\cA)}}
\newcommand{\Iel}{\mathcal{I}^\text{el}}
\newcommand{\torus}{{\mathbf{T}^2}}
\newcommand{\punctorus}{{\mathbf{T}_0^2}}
\newcommand{\disk}{\DD^2}

\newcommand{\PI}{{P\times I'}}
\newcommand{\clI}{{[0,1]}}

\newcommand{\cB}{{\mathcal{B}}}
\newcommand{\cD}{{\mathcal{D}}}

\newcommand{\cN}{{\mathcal{N}}}
\newcommand{\cM}{\mathcal{M}}      
\newcommand{\cW}{\mathcal{W}}
\newcommand{\cH}{\mathcal{H}}
\newcommand{\cS}{{\mathcal{S}}}
\newcommand{\Ann}{\text{Ann}}
\newcommand{\hA}{{\hat{\cA}}}
\newcommand{\hB}{{\hat{\mathcal{B}}}}

\newcommand{\lact}{\vartriangleright}
\newcommand{\ract}{\vartriangleleft}
\newcommand{\ihom}[2]{\Hom_{#1}^{#2}}

\newcommand{\hatbox}[1]{{\hat{\boxtimes}_{#1}}}

\newcommand{\wdtld}{\widetilde}

\DeclareMathOperator{\Kar}{Kar} 
\DeclareMathOperator{\htr}{hTr} 
\DeclareMathOperator{\Skein}{Skein}

\newcommand{\Vect}{\mathcal{V}ec}  
\newcommand{\cc}[1]{\underset{\scriptstyle #1}{\circ}}

\newcommand{\st}{\; | \;}                               

\newcommand{\ldim}[1]{d_{#1}^L}
\newcommand{\rdim}[1]{d_{#1}^R}
\newcommand{\rdual}[1]{\prescript{*}{}#1}
\newcommand{\rvee}[1]{\prescript{\vee}{}#1}

\newcommand{\intr}[1]{{\overset{\circ}{#1}}}
\newcommand{\sk}{{\mathrm{sk}}}

\newcommand{\TT}{{\mathbf{T}}}

\newcommand{\supp}{{\mathrm{supp}}}

\newcommand{\emptystate}[1]{{\widetilde{\emptyset}_{#1}}}
\newcommand{\emptyskein}[1]{{\emptyset_{#1}^{\text{sk}}}}
\newcommand{\emptynorm}[1]{\eps_{#1}}

\newcommand{\iI}{{I'}}

\newcommand{\CYA}{{\ZCY(\text{Ann})}}
\newcommand{\hCYA}{{\hatZCY(\text{Ann})}}

\newcommand{\hP}{{\hat{P}}}
\newcommand{\cAA}{{\cA \boxtimes \cA}}
\newcommand{\cAAbop}{{\cA^\bop \boxtimes \cA}}
\newcommand{\onebar}{{\overline{\one}}}
\DeclareMathOperator{\Mat}{Mat}
\DeclareMathOperator{\Vctsp}{Vec}
\newcommand{\matvec}[1]{\Mat_{#1}(\Vctsp)}

\newcommand{\tnsrbar}{{\, \overline{\tnsr} \,}}
\newcommand{\tnsrproj}[2]{\tensor[_{#1}]\tnsrbar{_{#2}}}

\newcommand{\tnsrz}[1]{{\, \overline{\tnsr}_{#1} \,}}
\newcommand{\htnsr}{{\hat{\tnsr}}}
\newcommand{\tnsrst}{\tnsrz{\text{st}}}
\newcommand{\onest}{{\onebar_{\text{st}}}}
\newcommand{\POP}{{M_\text{POP}}}
\newcommand{\YPOP}{{M_\text{Y}}}

\newcommand{\defend}{\hfill $\triangle$}
\newcommand{\rmkend}{\hfill $\triangle$}

\newcommand{\evbar}{{\ov{\ev}}}
\newcommand{\evsk}{{\ev^{sk}}}
\newcommand{\evdel}{{\ev^{\del}}}
\newcommand{\evdelbar}{{\ov{ev}^{\del}}}

\newcommand{\ups}{\upsilon}
\newcommand{\Ups}{\Upsilon}

\newcommand{\preequiv}{\simeq}

\newcommand{\smallerer}{\footnotesize}

\newcommand{\cev}[1]{\overset{\leftarrow}{#1}}

\newcommand{\EE}{{\mathbf{E}}}

\newcommand{\Fpl}[1]{{\cF_{#1}}}
\newcommand{\Fsk}[1]{{\cF_{#1}^{sk}}}
\newcommand{\Ftld}[1]{{\wdtld{\cF}_{#1}^{sk}}}
\newcommand{\cFx}{{\cF^{\mathbf{x}}}}

\newcommand{\vme}{x} 

\newcommand{\cobx}{{\mathbf{Cob^x}}}

\newcommand{\taucy}{{\tau_{CY}}}

\definecolor{light-gray}{gray}{0.9}
\definecolor{med-gray}{gray}{0.6}

\begin{document}






\newpage
{\centering
\hspace{0pt}
\vfill
\textbf{On the Category of Boundary Values in the extended Crane-Yetter TQFT}
\\
\vspace{20pt}
A Dissertation Presented
\\
\vspace{20pt}
by
\\
\vspace{20pt}
\textbf{Ying Hong Tham}
\\
\vspace{20pt}
to
\\
\vspace{20pt}
The Graduate School
\\
\vspace{20pt}
in Partial Fulfillment of the
\\
\vspace{20pt}
Requirements
\\
\vspace{20pt}
for the Degree of
\\
\vspace{20pt}
\textbf{Doctor of Philosophy}
\\
\vspace{20pt}
in
\\
\vspace{20pt}
\textbf{Mathematics}
\\
\vspace{20pt}
Stony Brook University
\\
\vspace{20pt}
\textbf{August 2021}
\\
}
\vfill
\hspace{0pt}
\newpage
{\centering
\hspace{0pt}
\vfill
\textbf{Stony Brook University}
\\
The Graduate School
\\
\vspace{20pt}
\textbf{Ying Hong Tham}
\\
\vspace{20pt}
We, the dissertation committee for the above candidate for the Doctor of
Philosophy degree, hereby recommend acceptance of this dissertation.
\\
\vspace{40pt}
\textbf{Alexander Kirillov - Dissertation Advisor} \\
\textbf{Professor, Department of Mathematics}
\\
\vspace{40pt}
\textbf{Dennis Sullivan - Chairperson of Defense} \\
\textbf{Professor, Department of Mathematics}
\\
\vspace{40pt}
\textbf{Oleg Viro} \\
\textbf{Professor, Department of Mathematics}
\\
\vspace{40pt}
\textbf{David Ben-Zvi} \\
\textbf{Professor, University of Texas, Austin}
\\
\vspace{40pt}
This dissertation is accepted by the Graduate School.
\\
\vspace{20pt}
Eric Wertheimer \\
Dean of the Graduate School\\
}
\vfill
\hspace{0pt}
\newpage
{\centering
\hspace{0pt}
\vfill
Abstract of the Dissertation
\\
\vspace{20pt}
\textbf{On the Category of Boundary Values in the extended Crane-Yetter TQFT}
\\
\vspace{20pt}
by
\\
\vspace{20pt}
\textbf{Ying Hong Tham}
\\
\vspace{20pt}
\textbf{Doctor of Philosophy}
\\
\vspace{20pt}
in
\\
\vspace{20pt}
\textbf{Mathematics}
\\
\vspace{20pt}
Stony Brook University
\\
\vspace{20pt}
2021
\\
\vspace{20pt}
\begin{minipage}{10cm}
The Crane-Yetter state sum is an invariant of closed 4-manifolds,
defined in terms of a triangulation, based on 15-j symbols associated
to the category $\cA$ of representations over quantum $\mathfrak{s}\mathfrak{l}_2$
(at a root of unity).
In this thesis,
we define the state sum in terms of a ``PLCW decomposition'',
which generalizes triangulations,
and generalize $\cA$ to an arbitrary premodular category.
We extend the state sum to 4-manifolds with corners,
making it an extended TQFT.
We also develop a parallel theory based on skeins,
which are essentially $\cA$-colored graphs,
and we show that the two theories are equivalent.

Focusing on the 2-dimensional part,
we prove several properties of skein categories,
the most important of which is that they satisfy excision.
We provide explicit algebraic descriptions
of the category associated to the once-punctured torus and the annulus,
giving rise to a new tensor product on the Drinfeld center
of a premodular category.

As it is well-known that, when $\cA$ is modular,
the Crane-Yetter state sum computes the signature of a closed 4-manifold,
we connect the Crane-Yetter theory to the signature of a 4-manifold
with boundary and even corners.

Finally, we show that the Reshetikhin-Turaev TQFT
is the boundary theory of the Crane-Yetter theory
(up to a normalization).
\end{minipage}
\vfill
\hspace{0pt}
}
\newpage
{\centering
\hspace{0pt}
\vfill
To mom
\vfill
\hspace{0pt}
}
\newpage

\tableofcontents

\newpage
\section*{Acknowledgements}
\vspace{0.8in}

First, I would like to thank Sasha Kirillov
for his guidance as my advisor.
Much of the work presented here is based off of Sasha's work,
and I am extemely grateful for his encouragement and
his help in navigating the field.

I must also thank Dennis Sullivan for many
stimulating conversations.
His unrelenting enthusiasm and absolute passion for math
was infectious and inspiring.

I also want to thank Oleg Viro for many interesting conversation
and particularly for sharing many concrete topological insights.

I want to thank my parents for going the extra mile,
sometimes literally,
in supporting my mathematical journey,
and my sisters for sharing quality internet content.
I am grateful for their unending love and support.

I also want to thank Alice for her care and kindness.
I look forward to many more collaborations with her.

Finally, I thank my friends at Stony Brook
for keeping life at the office fun and lively,
in particular Tobias, Jae Ho, Yoon Joo, Sasha, Xujia, Dahye;
we thank China Station for their symplectic hospitality,
1089 for sparking amateur number theory speculations,
and Green Tea for reminding us of the taste of China.

\newpage
\section{Introduction}
\vspace{0.8in}
\pagenumbering{arabic}

The study of quantum invariants of knots and low-dimensional manifolds
is a rich and active field of research.
Central to the field is the Topological Quantum Field Theory
(TQFT),
which were introduced in \ocite{atiyah} and \ocite{witten}:

\begin{definition}
An $n$-dimensional TQFT $\tau$ is the following collection of data:
\begin{itemize}
\item to each $(n-1)$-dimensional manifold $M$, an assignment of a
	vector space $\tau(M)$, known as the ``state space'',
\item to each $n$-dimensional manifold $W$,
	an assignment of a vector $\tau(W) \in \tau(\del W)$,
	known as the ``partition function'',
\item to a homeomorphism $\vphi : M \to M'$,
	a natural isomorphism	$\tau(\vphi) : \tau(M) \simeq \tau(M')$,
\item functorial isomorphisms $\tau(\ov{M}) \simeq \tau(M)^*$,
	$\tau(\emptyset) \simeq \kk$,
	$\tau(M \sqcup M') \simeq \tau(M) \tnsr \tau(M')$.
\end{itemize}
These data are required to satisfy a list of axioms
(see e.g. \ocite{BK}*{Chapter 4}).
\label{d:TQFT}
\end{definition}

In 1993, Crane and Yetter \ocite{CY} defined a 4d-TQFT
based on ``15-$j$''-symbols arising out of
a category of representations of a quantum group,
inspired by Ooguri \ocite{ooguri},
who proposed a formal expression for an invariant
of 4-manifolds with an eye towards a theory of
quantum gravity.
It is constructed as a ``state sum'' which
assigns and combines local invariants to each 4-simplex in a
triangulation,
similar to the Turaev-Viro state sum which defines
a 3d-TQFT \ocite{TV}.
The Crane-Yetter TQFT is known to compute the signature
and Euler characteristic of a 4-manifold
\ocite{CKY-eval}, \ocite{roberts-chainmail}.

The Turaev-Viro TQFT is known to be an extended 3-2-1-TQFT
\ocite{balsam-kirillov},
that is, the TQFT also makes an assignment
of a category to each 1-manifold.
The Crane-Yetter TQFT is widely expected to
be an extended TQFT.

Another important object of study is the theory of skeins.
These first arose in Kauffman's work \ocite{kauffman-skein}
on the Jones polynomial \ocite{jones}.
The Kauffman skein relation imposes a relation
on the space of link diagrams
that also reflect the properties of a 
category of representations of a quantum group.

In \ocite{kirillov-stringnet},
Kirillov shows that the state space of the Turaev-Viro TQFT
can be defined as the space of skeins on the surface.
It is widely expected that the state spaces
of the Crane-Yetter TQFT can also be defined as the space
of skeins in the 3-manifold.

\subsection{Main results}
\par \noindent

The main results of this paper are as follows:

\begin{itemize}
\item extend the Crane-Yetter TQFT to an extended TQFT
	based on a ``PLCW decomposition'' of the manifold
	(a generalization of triangulations),
\item construct a parallel theory based on skeins,
	and prove equivalence with the previous constructions,
\item develop the properties of skein categories associated
	to surfaces, in particular the annulus,
\item discuss the signature formula
\item show that Reshetikhin-Turaev TQFT is the boundary theory
	of the extended Crane-Yetter TQFT
\end{itemize}

\newpage
\section{Background}
\vspace{0.8in}

\subsection{Categories}\par \noindent
\label{s:basic-cat}


In this section, we review some definitions in category theory,
leading up to the definition of modular categories.

Most of the categories that we deal with will be abelian
over some algebraically closed field $\kk$
and semisimple; then the structure maps for the various
definitions and constructions in this section
will implicitly be assumed to be appropriately multilinear
(see \rmkref{r:tnsr-linear}).

\begin{definition}
Let $\cC$ be a category.
A \emph{monoidal structure} on $\cC$
is a collection of data as follows:
\begin{itemize}
\item a bifunctor $\tnsr: \cC \times \cC \to \cC$
called the \emph{tensor product};

\item a natural isomorphism
$\alpha: (- \tnsr -) \tnsr - \simeq - \tnsr (- \tnsr -)$,
or more verbosely, for any three objects $U,V,W \in \cC$,
a natural isomorphism
\[
	\alpha_{U,V,W} : (U \tnsr V) \tnsr W \simeq U \tnsr (V \tnsr W)
\]
$\alpha$ is called the \emph{associativity constraint};
\item a \emph{unit object} $\one \in \cC$ and
\emph{unit constraint} $\lmb, \rho$ which are natural isomorphisms
\begin{align*}
\lmb_V : \one \tnsr V \simeq V \\
\rho_V : V \tnsr \one \simeq V
\end{align*}
\end{itemize}
which are subject to the \emph{triangle} and \emph{pentagon} axioms:
\[
\begin{tikzcd}[column sep=small]
(V \tnsr \one) \tnsr W
	\ar[rr, "\alpha"] \ar[rd, "\rho \tnsr \id"']
	& & V \tnsr (\one \tnsr W)
		\ar[ld, "\id \tnsr \lmb"]
	\\
	& V \tnsr W
\end{tikzcd}
\;\;\;
\begin{tikzcd}[column sep=0]
((U \tnsr V) \tnsr W) \tnsr X
	\ar[rr, "\alpha_{U,V,W} \tnsr \id_X"]
	\ar[d, "\alpha_{U \tnsr V, W, X}"]
	& & (U \tnsr (V \tnsr W)) \tnsr X
		\ar[d, "\alpha_{U, V \tnsr W, X}"]
	\\
(U \tnsr V) \tnsr (W \tnsr X)
	\ar[rd, "\alpha_{U,V, W \tnsr X}"']
	& & U \tnsr ((V \tnsr W) \tnsr X)
		\ar[ld, "\id_U \tnsr \alpha_{V,W,X}"]
	\\
& U \tnsr (V \tnsr (W \tnsr X))
\end{tikzcd}
\]
We say that $(\cC, \tnsr, \one, \alpha, \lmb, \rho)$
is a monoidal category.
\end{definition}

Sometimes we simply refer to $(\cC, \tnsr)$,
or even just $\cC$, as a monoidal category
when the other structures are understood.
We often hide the associativity and unit constraints for readability;
this is justified by \thmref{t:maclane-strictness}.

\begin{remark}
\label{r:tnsr-linear}
When $\cC$ is abelian, the tensor product of morphisms
should be a bilinear map: for $V,V',W,W' \in \cC$,
the map $\tnsr: \Hom(V,V') \times \Hom(W,W') \to
\Hom(V \tnsr V', W \tnsr W')$
is bilinear.
\end{remark}

\begin{example}
Here are some common examples of monoidal categories:

\begin{enumerate}
\item The category of vector spaces $\Vect_\kk$ over some field $\kk$,
with $\tnsr$ being the usual tensor product of vector spaces.

\item The category $\Rep_\kk(G)$ of representations of a finite group $G$
over some field $\kk$,
with $\tnsr$ being the usual tensor product of representations.

\end{enumerate}

\end{example}

\begin{definition}
Let $(\cC, \tnsr, \one, \alpha, \lmb, \rho)$
and $(\cC', \tnsr', \one', \alpha', \lmb', \rho')$
be monoidal categories.
A \emph{monoidal functor} from $\cC$ to $\cC'$ is
a pair $(F, J)$, where $F: \cC \to \cC'$ is a functor
such that $F(\one)$ is isomorphic to $\one'$, and
$J: \tnsr' \circ F \times F \to F \circ \tnsr : \cC \times \cC \to \cC'$
is a natural transformation that satisfies the hexagon axiom:
\begin{tikzcd}[column sep=tiny, row sep=small]
& (F(U) \tnsr' F(V)) \tnsr' F(W)
	\ar[rr, "\alpha' F"]
	\ar[dl, "J \tnsr' \id"]
& & F(U) \tnsr' (F(V) \tnsr' F(W))
	\ar[dr, "\id \tnsr' J"]
\\
F(U \tnsr V) \tnsr' F(W)
	\ar[dr, "J"]
& & & & F(U) \tnsr' F(V \tnsr W)
	\ar[dl, "J"]
\\
& F((U \tnsr V) \tnsr W)
	\ar[rr, "F(\alpha)"]
& & F(U \tnsr (V \tnsr W))
\end{tikzcd}

We call $J$ a \emph{monoidal structure} on $F$.
If $F$ is an equivalence of categories, we say that
$(F,J)$ is an \emph{equivalence of monoidal categories}.

Let $(F',J')$ be another monoidal functor from $\cC$ to $\cC'$.
A natural transformation of monoidal functors from $(F,J)$ to $(F',J')$
is a natural transformation $F \to F'$ that respects the
monoidal structures,
\end{definition}

\begin{definition}
A monoidal category $(\cC, \tnsr, \one, \alpha, \lmb, \rho)$ is
\emph{strict} if $(U \tnsr V) \tnsr W = U \tnsr (V \tnsr W)$ and
$\one \tnsr V = V = V \tnsr \one$,
and furthermore $\alpha, \lmb, \rho$ are the identity maps.
\end{definition}

\begin{theorem}[Maclane strictness]
\label{t:maclane-strictness}
Any monoidal category is monoidally equivalent to a strict monoidal category.
\end{theorem}


\begin{definition}
Let $\cC$ be a monoidal category, and $V \in \cC$ an object.
A \emph{left dual} to $V$ is an object $V^*$ together with
morphisms $\ev_V : V^* \tnsr V \to \one$ and
$\coev_V : \one \to V \tnsr V^*$, such that the compositions
\begin{align*}
V \xrightarrow{\coev \tnsr \id} V \tnsr V^* \tnsr V
\xrightarrow{\id \tnsr \ev} V
\\
V^* \xrightarrow{\id \tnsr \coev} V^* \tnsr V \tnsr V^*
\xrightarrow{\ev \tnsr \id} V^*
\end{align*}
are the identity morphisms.
Note that we have suppressed the unit and associativity constraints.

Similarly, a \emph{right dual}\footnotemark to $V$ is an object $\rdual{V}$
together with morphisms $\tilde{\ev}_V : V \tnsr \rdual{V} \to \one$
and $\tilde{\coev}_V : \one \to \rdual{V} \tnsr V$,
such that the compositions
\begin{align*}
V \xrightarrow{\tilde{\coev} \tnsr \id} V \tnsr \rdual{V} \tnsr V
\xrightarrow{\id \tnsr \tilde{\ev}} V
\\
\rdual{V} \xrightarrow{\id \tnsr \tilde{\coev}}
\rdual{V} \tnsr V \tnsr \rdual{V}
\xrightarrow{\tilde{\ev} \tnsr \id} \rdual{V}
\end{align*}
are the identity morphisms.

If $V, W$ have left duals, and $f\in \Hom(V,W)$,
the \emph{left dual} of $f$ is the morphism
$f^* \in \Hom(W^*, V^*)$ given by the composition
\[
f^* : W^* \xrightarrow{\coev_V} W^* \tnsr V \tnsr V^*
\xrightarrow{\id \tnsr f \tnsr \id} W^* \tnsr W \tnsr V^*
\xrightarrow{\ev_W} V^*
\]
Similarly, if $V,W$ have right duals, the right dual to
$f\in \Hom(V,W)$ is the morphism
$\rdual{f} \in \Hom(\rdual{W},\rdual{V})$
given by the composition
\[
\rdual{f} : \rdual{W}
\xrightarrow{\tilde{\coev}_V} \rdual{V} \tnsr V \tnsr \rdual{W}
\xrightarrow{\id \tnsr f \tnsr \id} \rdual{V} \tnsr W \tnsr \rdual{W}
\xrightarrow{\tilde{\ev}_W} \rdual{V}
\]
\end{definition}

\footnotetext{It is always confusing which side of $V$
to put the asterisk on. I remember by thinking,
left dual object evaluates from the left,
and the asterisk should point towards the original object,
thus $V^* \tnsr V \to \one$ is the evaluation map
for the left dual.}

\begin{proposition}
\label{p:dual-unique}
If $V \in \cC$ has a left (respectively right) dual object,
then it is unique up to unique isomorphism.
\end{proposition}
\begin{proof}
See \ocite{EGNO}*{Prop 2.10.5}.
\end{proof}

Observe that $V^*$ is a right dual to $V$,
and $\rdual{V}$ is a left dual to $V$.
Thus by the proposition above,
we can naturally identify
$\rdual{(V^*)} \simeq V \simeq (\rdual{V})^*$.

\begin{proposition}
Let $\cC$ be a monoidal category and $V$ an object in $\cC$.
If $V$ has a left dual, then there are natural adjunction isomorphisms
\begin{align*}
\Hom(U \tnsr V, W) &\simeq \Hom(U, W \tnsr V^*)
\\
\Hom(\rdual{V} \tnsr U, W) &\simeq \Hom(U, V \tnsr W)
\end{align*}
Similarly, if $V$ has a right dual,
then there are natural adjunction isomorphisms
\begin{align*}
\Hom(U \tnsr \rdual{V}, W) &\simeq \Hom(U, W \tnsr V)
\\
\Hom(\rdual{V} \tnsr U, W) &\simeq \Hom(U, \rdual{V} \tnsr W)
\end{align*}
\end{proposition}
\begin{proof}
The first natural isomorphism is given by
$f \mapsto (f \tnsr \id_{V^*}) \circ (\id_U \tnsr \coev_V)$,
with inverse
$g \mapsto (\id_W \tnsr \ev_V) \circ (g \tnsr \id_V)$;
other isomorphisms are similar. See \ocite{EGNO}*{Prop 2.10.8}.
Abstractly, this proposition says that when the left dual exists,
the multiplication functor $V^* \tnsr -$ (respectively $-\tnsr V^*$)
are left (respectively right) adjoints to the multiplication functor
$V \tnsr -$ (respectively $-\tnsr V$). (Similarly for right duals).
\end{proof}

\begin{definition}
We say that a monoidal category is \emph{rigid}
if every object has a left and right dual.
\end{definition}

In general, the left and right duals may not be isomorphic,
let alone naturally isomorphic.

\begin{definition}
Let $\cC$ be a rigid monoidal category.
A \emph{pivotal structure} on $\cC$ is a natural isomorphism
of monoidal functors $\delta : \id \simeq (-)^{**}$.
\end{definition}

By the observation after \prpref{p:dual-unique},
this is equivalent to having a natural isomorphism
$\rdual{V} \simeq V^*$.
Note that by virtue of being an isomorphism of \emph{monoidal}
functors, one has $\delta_{V\tnsr Y} = \delta_V \tnsr \delta_Y$.

\begin{definition}
Let $V$ be an object in a pivotal category $\cC$,
and let $f \in \End_{\cC}(V)$.
The \emph{left trace} of $f$ is the composition
\[
\tr^L(f) : \one \xrightarrow{\coev_V} V \tnsr V^*
\xrightarrow{\delta_V \tnsr f^*} V^{**} \tnsr V^*
\xrightarrow{\ev_{V^*}} \one
\]
and similarly, the \emph{right trace} is the composition
\[
\tr^R(f) : \one \xrightarrow{\tilde{\coev}_V} \rdual{V} \tnsr V
\xrightarrow{\delta_{\rdual{V}} \tnsr f} V^* \tnsr V
\xrightarrow{\ev_V} \one
\]
In particular, we define the \emph{left (respectively right) dimension},
denoted $\dim^L(V)$ (respectively $\dim^R(V)$
of an object $V$ to be the left (respectively right) trace
of $\id_V$.
\end{definition}

\begin{definition}
We say that a pivotal category $\cC$ is \emph{spherical}
if $\dim^L(V) = \dim^R(V) \in \End(\one)$ for every object $V \in \cC$.
\end{definition}

It is easy to check that $\dim^R(V^*) = \dim^L(V) = \dim^L(V^{**})$,
so an equivalent definition of sphericality
is $\dim^L(V) = \dim^L(V^*)$ for all $V \in \cC$.

\begin{definition}
Let $\cC$ be a monoidal category.
A \emph{braiding} on $\cC$ is a natural isomorphism
$c_{X,Y} : X \tnsr Y \simeq Y \tnsr X$
such that
\begin{align*}
c_{X, Y \tnsr Z} = (\id_Y \tnsr c_{X,Z}) \tnsr (c_{X,Y} \tnsr \id_Z)
c_{X \tnsr Y, Z} = (c_{X,Z} \tnsr \id_Y) \tnsr (\id_X \tnsr c_{Y,Z})
\end{align*}
\end{definition}

\begin{definition}
Let $\cC$ be a braided monoidal category with braiding $c$.
A \emph{twist} (or \emph{balancing transformation})
is a natural isomorphism $\theta \in \Aut(\id_\cC)$ such that
\[
\theta_{X \tnsr Y} = (\theta_X \tnsr \theta_Y) \circ c_{Y,X} \circ c_{X,Y}
\]
\end{definition}

\begin{definition}
Let $\cC$ be a monoidal category with a spherical and braided structure.
A \emph{ribbon structure} on $\cC$
is a twist $\theta$ such that $(\theta_X)^* = \theta_{X^*}$.
\end{definition}

\begin{definition}
Let $\cC$ be a $\kk$-linear abelian rigid monoidal category.
We say that $\cC$ is \emph{multifusion} if $\cC$ is finite semisimple
and $\tnsr$ is bilinear on morphisms.
We say that $\cC$ is \emph{fusion}
if $\End(\one) \cong \kk$.
\end{definition}

We denote by $\Irr(\cC)$ the set of isomorphism classes
of simple objects of a multifusion category $\cC$,
and pick representatives $X_i$ for each class $i\in \Irr(\cC)$.
and denote by $\Irr_0(\cC) \subseteq \Irr(\cC)$
the subset of simple objects appearing in the
direct sum decomposition of the unit object $\one$;
it is known (see \ocite{EGNO}*{Corollary 4.3.2})
that $\one$ decomposes into a direct sum of pairwise-distinct simples.
We also assume that the dual functor sends our choice of representatives
to themselves, resulting in an involution on $\Irr(\cC)$,
which we denote $i^*$ for $i \in \Irr(\cC)$.
We will delve into more detail in \secref{s:graphical-multifusion}.

\begin{definition}
A \emph{premodular} category is a ribbon fusion category.
\end{definition}

\begin{definition}
The \emph{$S$-matrix} of a premodular category $\cC$ is
the $\Irr(\cC) \times \Irr(\cC)$ matrix
\[
S := (s_{ij})_{i, j \in \Irr(\cC)}
\;\;\;
\text{where}
\;\;\;
s_{ij} = \tr(c_{X_j, X_i} \circ c_{X_i, X_j})
\]
\end{definition}

\begin{definition}
A premodular category $\cC$ is \emph{modular}
if its $S$-matrix is non-degenerate.
\end{definition}

We summarize the slew of definitions given above
in the flow chart below:
\[
\begin{tikzpicture}
\node[block] (start) at (0,0) {category $\cC$};
\node[block] (monoidal) at (-2,-2)
	{monoidal\\$\tnsr : \cC \times \cC \to \cC$};
\node[block] (rigid) at (-2, -4)
	{rigid \\ $X^*, \rdual{X}$};
\node[block] (pivotal) at (-2, -6)
	{pivotal \\ $\delta_X : X \simeq X^{**}$};
\node[block] (spherical) at (-3, -8)
	{spherical \\ $\dim^L \equiv \dim^R$};
\node[block] (abelian) at (2, -2)
	{abelian over $\kk$};
\node[block] (finitess) at (2,-4)
	{finite \\ semisimple};
\node[block] (multifusion) at (2,-6)
	{multifusion};
\node[block] (fusion) at (4,-8)
	{fusion \\ $\End(\one) \cong \kk$};
\node[block] (pivotalmultifusion) at (0,-8)
	{pivotal \\ multifusion};
\node[block] (sphericalfusion) at (0,-10)
	{spherical \\ fusion};
\node[block] (braided) at (-6,-4)
	{braiding \\ $c_{X,Y} : X \tnsr Y \simeq Y \tnsr X$};
\node[block] (twist) at (-6, -6)
	{twist \\ $\theta_X : X \to X$};
\node[block] (ribbon) at (-6, -8)
	{ribbon \\ $(\theta_X)^* = \theta_{X^*}$};
\node[block] (premodular) at (-4, -10)
	{premodular};
\node[block] (modular) at (-7,-10)
	{modular \\ $\det(S) \neq 0$};
\draw[->] (start) -- (monoidal);
\draw[->] (start) -- (abelian);
\draw[->] (monoidal) -- (rigid);
\draw[->] (monoidal) -- (braided);
\draw[->] (abelian) -- (finitess);
\draw[->] (braided) -- (twist);
\draw[->] (rigid) -- (pivotal);
\draw[->] (rigid) -- (ribbon);
\draw[->] (rigid) -- (multifusion);
\draw[->] (finitess) -- (multifusion);
\draw[->] (twist) -- (ribbon);
\draw[->] (pivotal) -- (spherical);
\draw[->] (pivotal) -- (pivotalmultifusion);
\draw[->] (multifusion) -- (pivotalmultifusion);
\draw[->] (multifusion) -- (fusion);
\draw[->] (ribbon) -- (premodular);
\draw[->] (spherical) -- (sphericalfusion);
\draw[->] (pivotalmultifusion) -- (sphericalfusion);
\draw[->] (fusion) -- (sphericalfusion);
\draw[->] (sphericalfusion) -- (premodular);
\draw[->] (premodular) -- (modular);
\end{tikzpicture}
\]

\subsection{Graphical Calculus and Conventions: Pivotal Multifusion Categories}
\par \noindent
\label{s:graphical-multifusion}

Most of the categories that we encounter in this thesis
will be pivotal multifusion categories.
We therefore dedicate this section to notation and basic results about
the ``internal'' structures of pivotal multifusion categories,
i.e. objects and morphisms;
the next section will be concerned with the ``external structures'',
i.e. how pivotal multifusion categories relate to other categories
via functors, module structures etc.
It is mostly taken from \ocite{KT} which in turn was
adapted from \ocite{kirillov-stringnet},
modified to accommodate for the non-spherical non-fusion case.
We also point the reader to \ocite{EGNO}*{Chapter 4}
and \ocite{BK} for further reference.

We express morphisms in the language of the graphical calculus
\ocite{penrose-string},
a convenient language to express morphisms and equations about them.
As there are many conventions surrounding the use of the
graphical calculus
(see for example \rmkref{r:rt-mirror}),
we explicitly state our conventions as follows:

\begin{convention}
All diagrams are implicitly blackboard-framed.

Morphism run top to bottom:
\begin{equation}
\ph \in \Hom(V_1 \tnsr \cdots \tnsr V_n, W_1 \tnsr \cdots \tnsr W_m)
\longleftrightarrow
\begin{tikzpicture}
\node[coupon] (ph) at (0,0) {$\ph$};
\draw[<-] (ph) -- (-0.5,1) node[above] {$V_1$};
\draw[<-] (ph) -- (0,1) node[above] {$\cdots$};
\draw[<-] (ph) -- (0.5,1) node[above] {$V_n$};
\draw[->] (ph) -- (-0.5,-1) node[below] {$W_1$};
\draw[->] (ph) -- (0,-1) node[below] {$\cdots$};
\draw[->] (ph) -- (0.5,-1) node[below] {$W_m$};
\end{tikzpicture}
\label{e:convention-down}
\end{equation}
\label{cvn:graphical-down}
\end{convention}

\begin{convention}
We draw diagrams in an ambient space
that is oriented by the standard basis
$((1,0,0),(0,1,0),(0,0,1))$ following ``Fleming's right-hand rule''
($(1,0,0)$ points to the right, $(0,1,0)$ points into the page,
$(0,0,1)$ points up).
\label{cvn:ambient-ort}
\end{convention}

\begin{convention}
The braiding $c$ of a category is drawn as
\begin{equation}
c_{X,Y} = 
\begin{tikzpicture}
\draw[->] (1,0.5) to[out=-90,in=90] (0,-0.5);
\draw[->, overline] (0,0.5) to[out=-90,in=90] (1,-0.5);
\node at (0,0.7) {$X$};
\node at (1,0.7) {$Y$};
\end{tikzpicture}
\label{e:braiding-convention}
\end{equation}
If a diagram is said to be drawn in $\RR \times [0,1] \times [-\veps,\veps]$
with the opposite orientation,
then applying the flipping map on the last coordinate
(to change the orientation back to the standard one)
amounts to changing all such crossings to its opposite.

\label{cvn:braiding}
\end{convention}

Let $\cC$ be a $\kk$-linear pivotal multifusion category,
where $\kk$ is an algebraically closed field of characteristic 0.
In all our formulas and computations, we will be suppressing the
associativity and unit morphisms;
we also suppress the pivotal morphism
$V \simeq V^{**}$ when there is little cause for confusion.

We denote by $\Irr(\cC)$ the set of isomorphism classes of simple objects in $\cC$,
and denote by $\Irr_0(\cC) \subseteq \Irr(\cC)$
the subset of simple objects appearing in the direct sum decomposition of the
unit object $\one$;
it is known \ocite{EGNO}*{Corollary 4.3.2}
that $\one$ decomposes into a direct sum of distinct simples,
so $\End(\one) \cong \bigoplus_{l\in \Irr_0(\cC)} \End(\one_l)$.
We fix a representative $X_i$ for each isomorphism class $i\in \Irr(\cC)$;
abusing language, we will frequently use the same letter $i$ for
denoting both a simple object and its isomorphism class.
Rigidity gives us an involution $-^*$ on $\Irr(\cC)$;
it is known that $l^* = l$ for $l\in \Irr_0(\cC)$.
For $l\in \Irr_0(\cC)$, we may use the notation
$\one_l := X_l$ to emphasize that it is part of the unit.

For $k,l\in \Irr_0(\cC)$,
let $\cC_{kl} := \one_k \tnsr \cC \tnsr \one_l$,
so that $\cC = \bigoplus_{k,l\in \Irr_0(\cC)} \cC_{kl}$.
Any simple $X_i$ is contained in exactly one of these $\cC_{kl}$'s,
or in other words, there are unique $k_i,l_i \in \Irr_0(\cC)$
such that $\one_{k_i} \tnsr X_i \tnsr \one_{l_i} \neq 0$.
Since $\cC_{kl}^* = \cC_{lk}$,
we have that $k_{i^*} = l_i$.

When $\cC$ is spherical fusion, the categorical dimension
is a scalar, defined as a trace, but here the non-simplicity of $\one$
and non-sphericality complicates things.
To avoid confusion, denote by $\delta: V \to V^{**}$
the pivotal morphism.
The \emph{left dimension} of an object $V \in \Obj \cC$
is the morphism
\[
  \ldim{V} := (\one
                \xrightarrow{\coev} V \tnsr V^*
                \xrightarrow{\delta \tnsr \id} V^{**} \tnsr V^*
                \xrightarrow{\ev} \one)
                \in \End(\one)
\]
Similarly, the \emph{right dimension} of $V$ is the morphism
\[
  \rdim{V} := (\one
                \xrightarrow{\coev} \rdual{V} \tnsr V
                \xrightarrow{\id \tnsr \delta^\inv} \rdual{V}\tnsr \prescript{**}{}V
                \xrightarrow{\ev} \one)
                \in \End(\one)
\]
Note that these are \emph{vectors} and not scalars,
since $\one$ may not be simple.
It is easy to see that $\rdim{V} = \ldim{V^*} = \ldim{\rdual{V}}$.
When $\cC$ is spherical, we will drop the superscripts.

When $V = X_i$ is simple, we can interpret its left and right dimensions
as scalars as follows. We have $X_i \in \cC_{k_i l_i}$,
so $\Hom(\one, X_i \tnsr X_i^*) = \Hom(\one, \one_{k_i} \tnsr X_i \tnsr X_i^*)
\simeq \Hom(\one_{k_i}, X_i \tnsr X_i^*)$,
and likewise
$\Hom(X_i \tnsr X_i^*,\one) \simeq \Hom(X_i \tnsr X_i^*,\one_{k_i})$,
so $d_{X_i}^L$ factors through $\one_{k_i}$,
and hence we may interpret $d_{X_i}^L$ as an element of
$\End(\one_{k_i}) \cong \kk$.
Similarly, $d_{X_i}^R$ may be interpreted as an element of
$\End(\one_{l_i}) \cong \kk$.
We denote these scalar dimensions by $d_i^L, d_i^R$,
and fix square roots such that $\sqrt{d_i^L} = \sqrt{d_{i^*}^R}$.
The dimensions of simple objects are nonzero.

The \emph{dimension} of $\cC_{kl}$ is the sum
\begin{equation}\label{e:dimC}
\cD := \sum_{i\in \Irr(\cC_{kl})} d_i^R d_i^L
\end{equation}
or simply
\[
\cD = \sum d_i^2
\]
if $\cC$ is tensor.
By \ocite{ENO2005}*{Proposition 2.17},
this is the same for all pairs $k,l \in \Irr_0(\cC)$,
and by \ocite{ENO2005}*{Theorem 2.3},
they are nonzero.

We will fix fourth roots of $d_i^R$
and $\cD$ such that
$(d_i^L)^{1/4} = (d_{i^*}^R)^{1/4}$.

We define functors $\cC^{\boxtimes n}\to \Vect$ by
\begin{align}
\label{e:vev}
\eval{V_1,\dots,V_n} &=
  \Hom_\cC(\one, V_1\otimes\dots\otimes V_n) \\
\label{e:vev_multi}
\eval{V_1,\dots,V_n}_l &=
  \Hom_\cC(\one_l, V_1\otimes\dots\otimes V_n)
  \text{ for }l\in \Irr_0(\cC) \\
  &\simeq \eval{\one_l, V_1,\ldots,V_n} \nonumber
\end{align}
for any collection $V_1,\dots, V_n$ of objects of $\cC$.
Clearly $\eval{V_1,\ldots,V_n} = \bigoplus_l \eval{V_1,\ldots,V_n}_l$.

Note that the pivotal structure gives functorial isomorphisms
\begin{equation}
\label{e:cyclic}
z\colon\<V_1,\dots,V_n\>\simeq \<V_n, V_1,\dots,V_{n-1}\>
\end{equation}
such that $z^n=\id$ (see \ocite{BK}*{Section 5.3}); thus, up to a canonical
isomorphism, the space $\<V_1,\dots,V_n\>$ only depends on the cyclic order
of $V_1,\dots, V_n$.
In general, $z$ does not preserve the direct sum decomposition
of $\eval{V_1,\ldots,V_n}$ above.
For example, for a simple $X_i \in \cC_{k_i l_i}$,
we have $z : \eval{X_i, X_i^*}_{k_i} \simeq \eval{X_i^*, X_i}_{l_i}$.

We will commonly use graphic presentation of morphisms in a category,
representing a morphism
$W_1\otimes \dots \otimes  W_m\to V_1\otimes\dots\otimes V_n$ by a
diagram with $m$ strands at the top, labeled by $W_1, \dots, W_m$ and $n$ strands at the bottom, labeled
$V_1,\dots, V_n$ (Note: this differs from the convention in many other papers!).
We will allow diagrams with  with oriented strands,
using the convention that a strand labeled by $V$ is the same as the strands labeled by $V^*$ with opposite orientation
(suppressing isomorphisms $V\simeq V^{**}$).

We will show a morphism
$\ph\in \eval{V_1,\dots, V_n}$ by a
round circle labeled by $\ph$
with outgoing edges labeled $V_1, \dots, V_n$ in
counter-clockwise order,
as shown in \eqnref{e:morphism-semi-circle}
By \eqnref{e:cyclic} and the fact that $z^n=\id$,
this is unambiguous.
We will show a morphism
$\ph\in \eval{V_1,\ldots,V_n}_l$ by
a semicircle labeled by $\ph$ and $l$ as shown in
\eqnref{e:morphism-semi-circle};
in contrast with a circular node,
a semicircle imposes a strict ordering on the outgoing legs,
not just a cyclic ordering.
\begin{equation}
\label{e:morphism-semi-circle}
\begin{tikzpicture}
\node[morphism] (ph) at (0,0) {$\ph$};
\draw[->] (ph)-- +(240:1cm) node[pos=0.7, left] {$V_n$} ;
\draw[->] (ph)-- +(180:1cm);
\draw[->] (ph)-- +(120:1cm);
\draw[->] (ph)-- +(60:1cm);
\draw[->] (ph)-- +(0:1cm);
\draw[->] (ph)-- +(-60:1cm) node[pos=0.7, right] {$V_1$};
\end{tikzpicture}
\hspace{20pt}
\begin{tikzpicture}
\node[semi_morphism={180,\small $l$}] (ph) at (0,0) {$\ph$};
\draw[->] (ph)-- +(210:1cm) node[pos=0.7, above] {$V_1$} ;
\draw[->] (ph)-- +(240:1cm);
\draw[->] (ph)-- +(270:1cm);
\draw[->] (ph)-- +(300:1cm);
\draw[->] (ph)-- +(-30:1cm) node[pos=0.7, above] {$V_n$};
\end{tikzpicture}
\end{equation}

We have a natural composition map
\begin{equation}\label{e:composition}
\begin{aligned}
 \<V_1,\dots,V_n, X\>\otimes\<X^*, W_1,\dots,
W_m\>&\to\<V_1,\dots,V_n, W_1,\dots, W_m\>\\
\ph\otimes\psi\mapsto \ph\cc{X}\psi= \ev_{X^*}\circ (\ph\otimes\psi)
\end{aligned}
\end{equation}
where $\ev_{X^*}\colon X\otimes  X^*\to \one$ is the evaluation morphism
(the pivotal structure is suppressed).

Repeated applications of the composition map above gives us
a non-degenerate pairing
\begin{equation}
\label{e:pairing}
\ev: \eval{V_1,\dots,V_n}\otimes \eval{V_n^*,\dots,V_1^*}\to \End(\one)
\end{equation}

More precisely, when restricted to the subspaces,
\begin{equation}
\label{e:pairing_sub}
\eval{V_1,\dots,V_n}_k \otimes \eval{V_n^*,\dots,V_1^*}_l \to \End(\one)
\end{equation}
the pairing is 0 if $k\neq l$, and is non-degenerate,
factoring through $\End(\one_k) \simeq \kk$, if $k=l$.
The pairing is illustrated below for
$\ph_1 \in \eval{V_1,\dots,V_n}_k,
\ph_2 \in \eval{V_n^*,\dots,V_1^*}_l$:
\[
(\ph_1, \ph_2) =
\begin{tikzpicture}
\begin{scope}[shift={(0,0.3)}]
\node[semi_morphism={180,\small $k$}] (ph1) at (0,0) {$\ph_1$};
\node[semi_morphism={180,\small $l$}] (ph2) at (1,0) {$\ph_2$};
\draw (ph1) to[out=-40,in=-140] (ph2);
\draw (ph1) ..controls (0,-0.5) and (1,-0.5) .. (ph2);
\draw (ph1) ..controls (-0.5,-0.8) and (1.5,-0.8) .. (ph2);
\draw (ph1) ..controls (-1.5, -1) and (2.5, -1) .. (ph2);
\end{scope}
\end{tikzpicture}
\overset{\text{if }\cC\text{ not spherical}}
  {\not\equiv}
\begin{tikzpicture}
\begin{scope}[shift={(0,0.3)}]
\node[semi_morphism={180,\small $k$}] (ph1) at (0,0) {$\ph_1$};
\node[semi_morphism={180,\small $l$}] (ph2) at (1,0) {$\ph_2$};
\draw (ph1) to[out=-40,in=-140] (ph2);
\draw (ph1) ..controls (0,-0.5) and (1,-0.5) .. (ph2);
\draw (ph1) ..controls (-0.5,-0.8) and (1.5,-0.8) .. (ph2);
\draw (ph1) ..controls (-1, 0) and (-0.3, 0.5) ..
(0.5, 0.5) .. controls (1.3, 0.5) and (2, 0) .. (ph2);
\end{scope}
\end{tikzpicture}
= (z^\inv \cdot \ph_1, z \cdot \ph_2)
\]
Thus, we have functorial isomorphisms
\begin{equation}\label{e:dual}
\<V_1,\dots,V_n\>^*\simeq \<V_n^*,\dots,V_1^*\>
\end{equation}

When $\cC$ is spherical,
this pairing is compatible with the cyclic permutations \eqref{e:cyclic},
in the sense that
$(\ph_1,\ph_2) = (z \cdot \ph_1, z^\inv \cdot \ph_2)$.
Compatibility fails when $\cC$ is not spherical;
for example, it is easy to see that for
$\ph_1 = \ph_2 = \coev_{X_i} \in \eval{X_i,X_i^*}$,
one has $(\ph_1,\ph_2) = d_i^L$,
while for
$z \cdot \ph_1 = z^\inv \cdot \ph_2 = \coev_{X_i^*} \in \eval{X_i^*,X_i}$,
one has instead $(z \cdot \ph_1, z^\inv \ph_2) = d_i^R$.

\begin{lemma}
\label{l:pairing_property}
For $\ph \in \eval{V_1,\dots,V_n}_l,
\ph' \in \eval{V_n^*,\ldots,V_1^*}_l,
\psi \in \eval{W_n^*,\dots,W_1^*}_l$,
and
$f \in \Hom(V_1 \tnsr \cdots \tnsr V_n, W_1 \tnsr \cdots \tnsr W_n)$,
we have
\begin{align}
  (\ph,\ph') &= (\ph',\ph) \\
  (f \circ \ph_1, \ph_2) &= (\ph_1, f^* \circ \ph_2)
\end{align}
\end{lemma}

\begin{proof}
Straightforward from definitions.
\end{proof}

We will make two additional conventions related to the graphic presentation of morphisms.

\begin{notation}\label{n:dashed}
A dashed line in the picture stands for the sum of all colorings of an edge by
simple objects $i$, each taken with coefficient $d_i^R$:
     \begin{equation} \label{e:regular_color}
          \begin{tikzpicture}
        \draw[->, regular] (0, 0.5)--(0, -0.5);
       \end{tikzpicture}
      \;
      =\sum_{i\in \Irr(\cC)} d_i^R \quad
      \begin{tikzpicture}
        \draw[->] (0, 0.5)--(0, -0.5) node[pos=0.8, right] {$i$};
       \end{tikzpicture}
     \end{equation}
\end{notation}
When $\cC$ is spherical, the orientation of such a dashed line is irrelevant.

\begin{notation}\label{n:summation}
Let $\cC$ be spherical.
If a figure contains a pair of circles, one with outgoing edges labeled $V_1,\dots, V_n$ and
the other with edges labeled $V_n^*,\dots, V_1^*$ , and the vertices  are
labeled by the same letter $\al$  (or $\be$, or \dots)
it will stand for summation over the dual bases:
\begin{equation}\label{e:summation_convention}
\begin{tikzpicture}
\node[morphism] (ph) at (0,0) {$\al$};
\draw[->] (ph)-- +(240:1cm) node[pos=0.7, left] {$V_n$};
\draw[->] (ph)-- +(180:1cm);
\draw[->] (ph)-- +(120:1cm);
\draw[->] (ph)-- +(60:1cm);
\draw[->] (ph)-- +(0:1cm);
\draw[->] (ph)-- +(-60:1cm) node[pos=0.7, right] {$V_1$};
\node at (1.5,0) {$\tnsr$};
\node[morphism] (ph') at (3,0) {$\al$};
\draw[->] (ph')-- +(240:1cm) node[pos=0.7, left]  {$V^{*}_1$};
\draw[->] (ph')-- +(180:1cm);
\draw[->] (ph')-- +(120:1cm);
\draw[->] (ph')-- +(60:1cm);
\draw[->] (ph')-- +(0:1cm);
\draw[->] (ph')-- +(-60:1cm) node[pos=0.7, right] {$V^{*}_n$};
\end{tikzpicture}
\quad := \sum_\al\quad
\begin{tikzpicture}
\node[morphism] (ph) at (0,0) {$\ph_\al$};
\draw[->] (ph)-- +(240:1cm) node[pos=0.7, left] {$V_n$};
\draw[->] (ph)-- +(180:1cm);
\draw[->] (ph)-- +(120:1cm);
\draw[->] (ph)-- +(60:1cm);
\draw[->] (ph)-- +(0:1cm);
\draw[->] (ph)-- +(-60:1cm) node[pos=0.7, right] {$V_1$};
\node at (1.5,0) {$\tnsr$};
\node[morphism] (ph') at (3,0) {$\ph {}^\al$};
\draw[->] (ph')-- +(240:1cm) node[pos=0.7, left] {$V^{*}_1$};
\draw[->] (ph')-- +(180:1cm);
\draw[->] (ph')-- +(120:1cm);
\draw[->] (ph')-- +(60:1cm);
\draw[->] (ph')-- +(0:1cm);
\draw[->] (ph')-- +(-60:1cm) node[pos=0.7, right] {$V^{*}_n$};
\end{tikzpicture}
\end{equation}
where $\ph_\al\in \<V_1,\dots, V_n\>$, $\ph^\al\in \<V_n^*,\dots, V_1^*\>$
are dual bases with respect to pairing \eqnref{e:pairing}.
Note that the morphisms can be located far apart.

When $\cC$ is not spherical, the pairing is no longer
compatible with $z$ from \eqnref{e:cyclic},
so such notation can only make sense with semicircles:
\begin{equation}
\begin{tikzpicture}
  \node[semi_morphism={180,}] (al1) at (0,0.3) {$\al$};
\draw[->] (al1) -- +(-150:1cm) node[pos=0.8, left] {$V_1$};
\draw[->] (al1) -- +(-110:1cm);
\draw[->] (al1) -- +(-70:1cm);
\draw[->] (al1) -- +(-30:1cm) node[pos=0.8, right] {$V_n$};
\node at (1.5,0) {$\tnsr$};
\node[semi_morphism={180,}] (al2) at (3,0.3) {$\al$};
\draw[->] (al2) -- +(-150:1cm) node[pos=0.8, left] {$V_n^*$};
\draw[->] (al2) -- +(-110:1cm);
\draw[->] (al2) -- +(-70:1cm);
\draw[->] (al2) -- +(-30:1cm) node[pos=0.8, right] {$V_1^*$};
\end{tikzpicture}
\quad := \sum_{\al, l} \quad
\begin{tikzpicture}
  \node[semi_morphism={180,\small $l$}] (al1) at (0,0.3) {$\ph_\al$};
\draw[->] (al1) -- +(-150:1cm) node[pos=0.8, left] {$V_1$};
\draw[->] (al1) -- +(-110:1cm);
\draw[->] (al1) -- +(-70:1cm);
\draw[->] (al1) -- +(-30:1cm) node[pos=0.8, right] {$V_n$};
\node at (1.5,0) {$\tnsr$};
\node[semi_morphism={180,\small $l$}] (al2) at (3,0.3) {\small $\ph^\al$};
\draw[->] (al2) -- +(-150:1cm) node[pos=0.8, left] {$V_n^*$};
\draw[->] (al2) -- +(-110:1cm);
\draw[->] (al2) -- +(-70:1cm);
\draw[->] (al2) -- +(-30:1cm) node[pos=0.8, right] {$V_1^*$};
\end{tikzpicture}
\end{equation}
where $\ph_\al\in \eval{V_1,\dots, V_n}_l$, $\ph^\al\in \eval{V_n^*,\dots, V_1^*}_l$
are dual bases with respect to the pairing \eqnref{e:pairing}.

\end{notation}

The following lemma illustrates the use of the notation above.

\begin{lemma}
\label{l:summation}
For any $V_1,\ldots,V_n \in \cC$, we have
$$
\begin{tikzpicture}
\draw[->] (-0.4,1.2) node[above] {$V_1$} -- (-0.4,-1.2) node[below] {$V_1$};
\draw[->] (0,1.2) node[above] {$\cdots$} -- (0,-1.2) node[below] {$\cdots$};
\draw[->] (0.4,1.2) node[above] {$V_n$} -- (0.4,-1.2) node[below] {$V_n$};
\end{tikzpicture}
\quad = \quad
\sum_{i\in \Irr(\cC)} d_i^R
\begin{tikzpicture}
\node[semi_morphism={90,}] (T) at (0, 0.5) {$\al$};
\node[semi_morphism={90,}] (B) at (0, -0.5) {$\al$};
\draw (T) -- +(-0.3, 0.7) node[above left] {$V_1$};
\draw (T) -- +(0, 0.7) node[above] {$\cdots$};
\draw (T) -- +(0.3, 0.7) node[above right] {$V_n$};
\draw[->] (B) -- +(-0.3, -0.7) node[below left] {$V_1$};
\draw[->] (B) -- +(0, -0.7) node[below] {$\cdots$};
\draw[->] (B) -- +(0.3, -0.7) node[below right] {$V_n$};
\draw[midarrow={0.6}] (T)--(B) node[pos=0.5, right] {$i$};
\end{tikzpicture}
\quad = \quad
\begin{tikzpicture}
\node[semi_morphism={90,}] (T) at (0, 0.5) {$\al$};
\node[semi_morphism={90,}] (B) at (0, -0.5) {$\al$};
\draw (T) -- +(-0.3, 0.7) node[above left] {$V_1$};
\draw (T) -- +(0, 0.7) node[above] {$\cdots$};
\draw (T) -- +(0.3, 0.7) node[above right] {$V_n$};
\draw[->] (B) -- +(-0.3, -0.7) node[below left] {$V_1$};
\draw[->] (B) -- +(0, -0.7) node[below] {$\cdots$};
\draw[->] (B) -- +(0.3, -0.7) node[below right] {$V_n$};
\draw[->, regular] (T)--(B);
\end{tikzpicture}
\quad = \quad
\sum_{i\in \Irr(\cC)} d_i^L
\begin{tikzpicture}
\node[semi_morphism={270,}] (T) at (0, 0.5) {$\al$};
\node[semi_morphism={270,}] (B) at (0, -0.5) {$\al$};
\draw (T) -- +(-0.3, 0.7) node[above left] {$V_1$};
\draw (T) -- +(0, 0.7) node[above] {$\cdots$};
\draw (T) -- +(0.3, 0.7) node[above right] {$V_n$};
\draw[->] (B) -- +(-0.3, -0.7) node[below left] {$V_1$};
\draw[->] (B) -- +(0, -0.7) node[below] {$\cdots$};
\draw[->] (B) -- +(0.3, -0.7) node[below right] {$V_n$};
\draw[midarrow={0.6}] (T)--(B) node[pos=0.5, right] {$i$};;
\end{tikzpicture}
\quad = \quad
\begin{tikzpicture}
\node[semi_morphism={270,}] (T) at (0, 0.5) {$\al$};
\node[semi_morphism={270,}] (B) at (0, -0.5) {$\al$};
\draw (T) -- +(-0.3, 0.7) node[above left] {$V_1$};
\draw (T) -- +(0, 0.7) node[above] {$\cdots$};
\draw (T) -- +(0.3, 0.7) node[above right] {$V_n$};
\draw[->] (B) -- +(-0.3, -0.7) node[below left] {$V_1$};
\draw[->] (B) -- +(0, -0.7) node[below] {$\cdots$};
\draw[->] (B) -- +(0.3, -0.7) node[below right] {$V_n$};
\draw[<-, regular] (T)--(B);
\end{tikzpicture}
$$
\end{lemma}
Proof of this lemma is straightforward:
first show it for $V_1 \tnsr \cdots \tnsr V_n = \text{simple}$,
then follows for direct sums;
interested reader can find a proof
for spherical $\cC$ in \ocite{kirillov-stringnet}.

\begin{lemma}
\label{l:summation-alpha-tunnel}
Morphisms can ``tunnel through $\al$'':
\begin{equation}
\begin{tikzpicture}
\node[small_morphism] (a1) at (0,0.8) {$\al$};
\node[small_morphism] (a2) at (1.5,0.8) {$\al$};
\node[small_morphism] (f) at (0,0) {\smallerer $f$};
\draw[midarrow={0.6}] (a1) -- (f) node[pos=0.5,left] {\smallerer $V$};
\draw[->] (f) -- (0,-0.8) node[below] {\smallerer $W$};
\node at (0.75,0) {$\tnsr$};
\draw[->] (a2) -- (1.5,-0.8) node[below] {\smallerer $V^*$};
\end{tikzpicture}
=
\begin{tikzpicture}
\node[small_morphism] (a1) at (0,0.8) {$\al$};
\node[small_morphism] (a2) at (1.5,0.8) {$\al$};
\node[small_morphism] (f) at (1.5,0) {\smallerer $f^*$};
\draw[->] (a1) -- (0,-0.8) node[below] {\smallerer $W$};
\draw[midarrow={0.6}] (a2) -- (f) node[pos=0.5,right] {\smallerer $W^*$};
\draw[->] (f) -- (1.5,-0.8) node[below] {\smallerer $V^*$};
\node at (0.75,0) {$\tnsr$};
\end{tikzpicture}
\end{equation}
\end{lemma}

\begin{proof}
By identify $V,W$ with a direct sum of simples,
$f$ becomes a collection of matrices (one for each simple),
then the lemma follows easily from linear algebra.
(See \ocite{kirillov-stringnet}*{Lemma 3.6} for details.)
See \lemref{l:halfbrd}, \eqnref{e:ups-isom-01}
for typical use cases.
\end{proof}

\begin{lemma}
\label{l:summation-separate}
\begin{equation}
\begin{tikzpicture}
\node (c1) at (0,0) {};
\node (c2) at (2,0) {};
\draw[midarrow={0.55}] (c1) to[out=60,in=120] (c2);
\draw[midarrow={0.55}] (c1) to[out=20,in=160] (c2);
\draw[midarrow={0.55}] (c1) to[out=-20,in=-160] (c2);
\draw[midarrow={0.55}] (c1) to[out=-60,in=-120] (c2);
\node at (1,-0.9) {$V_1$};
\node at (1,0.9) {$V_n$};
\draw[fill=white,line width=0pt] (c1) circle(0.5cm);
\draw[fill=white,line width=0pt] (c2) circle(0.5cm);
\draw[fill=gray, opacity=0.6,line width=0pt] (c1) circle(0.5cm);
\draw[fill=gray, opacity=0.6,line width=0pt] (c2) circle(0.5cm);
\end{tikzpicture}
=
\begin{tikzpicture}
\node (c1) at (0,0) {};
\node (c2) at (2,0) {};
\node[small_morphism] (a1) at (0.8,0) {\tiny $\al$};
\node[small_morphism] (a2) at (1.2,0) {\tiny $\al$};
\draw[midarrow={0.8}] (c1) .. controls +(60:0.8cm) and +(90:0.6cm) .. (a1);
\draw[midarrow={0.9}] (c1) .. controls +(20:0.5cm) and +(135:0.3cm) .. (a1);
\draw[midarrow={0.9}] (c1) .. controls +(-20:0.5cm) and +(-135:0.3cm) .. (a1);
\draw[midarrow={0.8}] (c1) .. controls +(-60:0.8cm) and +(-90:0.6cm) .. (a1);
\draw[midarrow_rev={0.8}] (c2) .. controls +(120:0.8cm) and +(90:0.6cm) .. (a2);
\draw[midarrow_rev={0.9}] (c2) .. controls +(160:0.5cm) and +(45:0.3cm) .. (a2);
\draw[midarrow_rev={0.9}] (c2) .. controls +(-160:0.5cm) and +(-45:0.3cm) .. (a2);
\draw[midarrow_rev={0.8}] (c2) .. controls +(-120:0.8cm) and +(-90:0.6cm) .. (a2);
\node at (0.7,-0.7) {\smallerer $V_1$};
\node at (0.7,0.7) {\smallerer $V_n$};
\node at (1.3,-0.7) {\smallerer $V_1$};
\node at (1.3,0.7) {\smallerer $V_n$};
\draw[fill=white,line width=0pt] (c1) circle(0.5cm);
\draw[fill=white,line width=0pt] (c2) circle(0.5cm);
\draw[fill=gray, opacity=0.6,line width=0pt] (c1) circle(0.5cm);
\draw[fill=gray, opacity=0.6,line width=0pt] (c2) circle(0.5cm);
\end{tikzpicture}
\end{equation}
\end{lemma}
\begin{proof}
Apply \lemref{l:summation}, then notice that the ``connecting edge''
must be $i = \one$.
(See \ocite{balsam-kirillov}*{Lemma 2.2.3}).
\end{proof}

\begin{lemma}
\label{lem:sliding}
The following is a generalization of the ``sliding lemma'':
\begin{equation}
\label{e:sliding}
\begin{tikzpicture}
  \draw[regular, midarrow={0.5}] (0,0) circle(0.4cm);
 \path[subgraph] (0,0) circle(0.3cm);
 \draw (0,1.2)..controls +(-90:0.8cm) and +(90:0.4cm) ..
                 (-0.6, 0) ..controls +(-90:0.4cm) and +(90:0.8cm) ..
                 (0,-1.2);
\end{tikzpicture}
\quad=\quad
\begin{tikzpicture}
  \draw[regular, midarrow={0.5}] (0,0) circle(0.4cm);
  \path[subgraph] (0,0) circle(0.3cm);
  \draw (0,1.2)..controls +(-90:0.8cm) and +(90:0.4cm) ..
                 (0.6, 0) ..controls +(-90:0.4cm) and +(90:0.8cm) ..
                 (0,-1.2);
\end{tikzpicture}
\end{equation}
These relations hold regardless of the contents of the shaded region.
\end{lemma}
\begin{proof}
\[
 \begin{tikzpicture}
   \draw[regular, midarrow={0.5}] (0,0) circle(0.4cm);
   \path[subgraph] (0,0) circle(0.3cm);
   \draw (0,1.2)..controls +(-90:0.8cm) and +(90:0.4cm) ..
                   (-0.6, 0) ..controls +(-90:0.4cm) and +(90:0.8cm) ..
                   (0,-1.2);
 \end{tikzpicture}
\quad=\quad
\begin{tikzpicture}
  \draw[regular, midarrow={0.53}] (0,0) circle(0.45cm);
  \draw[regular, midarrow={0.03}] (0,0) circle(0.45cm);
  \node[semi_morphism={0,}] (top) at (0,0.4) {$\al$};
  \node[semi_morphism={180,}] (bot) at (0,-0.4) {$\al$};
  \path[subgraph] circle(0.3cm);
  \draw (0,1.2)--(top) (bot)--(0,-1.2);
\end{tikzpicture}
\quad=\quad
  \begin{tikzpicture}
    \draw[regular, midarrow={0.55}] (0,0) circle(0.4cm);
    \path[subgraph] (0,0) circle(0.3cm);
    \draw (0,1.2)..controls +(-90:0.8cm) and +(90:0.4cm) ..
                   (0.6, 0) ..controls +(-90:0.4cm) and +(90:0.8cm) ..
                   (0,-1.2);
  \end{tikzpicture}
\]
where we use \lemref{l:summation} in the equalities.
See also \ocite{kirillov-stringnet}*{Corollary 3.5}.
Note this trick doesn't work when the circle is oriented the other way
(unless of course if $\cC$ is spherical).
\end{proof}

\begin{lemma}[Killing Lemma, Charge Conservation]
\label{l:killing}
When $\cC$ is modular,
\begin{equation}
\frac{1}{\cD}
\begin{tikzpicture}
\draw[midarrow={0.9}] (0,0) -- (0,-0.7);
\node at (0.15,-0.6) {\smallerer $i$};
\draw[overline={1.5},regular,midarrow={0.5}] (0,0.1) circle (0.35cm);
\draw[overline={1.5}] (0,0) -- (0,0.7);
\end{tikzpicture}
=
\delta_{i,1} \id_{X_i}
\end{equation}
\end{lemma}

\begin{lemma} \label{l:dashed_circle}
\[
\frac{1}{|\Irr_0(\cC)|\cD}
\begin{tikzpicture}
\draw[regular, midarrow] (0,0) circle(0.4cm);
\end{tikzpicture}
=
\id_\one
=
\frac{1}{|\Irr_0(\cC)|\cD}
\sum_{i\in \Irr(\cC)} d_i^L
\begin{tikzpicture}
\draw[midarrow_rev={0.1}] (0,0) circle(0.4cm);
\node at (0.5,0.3) {\small $i$};
\end{tikzpicture}
\]
\end{lemma}
\begin{proof}
Let $\Irr^{kl} = \Irr(\cC_{kl})$,
and let $\Irr^{k*} := \bigcup_l \Irr(\cC_{kl})$,
i.e. the set of simples $X_i$ such that $\one_k \tnsr X_i = X_i$.
Then
\begin{gather*}
\begin{tikzpicture}
\draw[regular, midarrow] (0,0) circle(0.4cm);
\end{tikzpicture}
=
\sum_{k\in \Irr_0(\cC)} \sum_{i \in \Irr^{k*}(\cC)} d_i^R
\hspace{3pt}
\begin{tikzpicture}
\draw[midarrow={0.5}] (0,0) circle(0.4cm);
\node at (0.5,-0.2) {$i$};
\end{tikzpicture}
=
\sum_{k\in \Irr_0(\cC)} \sum_{i \in \Irr^{k*}} d_i^R d_i^L \id_{\one_k} \\
=
\sum_{k\in \Irr_0(\cC)} \sum_{l \in \Irr_0(\cC)}
\sum_{i \in \Irr^{kl}} d_i^R d_i^L \id_{\one_k}
=
\sum_{k\in \Irr_0(\cC)} \sum_{l \in \Irr_0(\cC)} \cD \id_{\one_k}
=
|\Irr_0(\cC)|\cD \id_\one
\end{gather*}
The second equality is proved in a similar manner.
\end{proof}

The following lemma is used to prove that \eqnref{e:I(M)} is a half-braiding
and the functor $G$ in the proof of \thmref{t:htr-center} respects composition:

\begin{lemma}
\label{l:halfbrd}
For $X \in \Obj \cC$, define the following element $\Gamma_X$ of
$\Hom(X \tnsr X_i, X_j) \tnsr \Hom(X_i^*, X_j^* \tnsr X)$:
\begin{equation}
\Gamma_X :=
\begin{tikzpicture}
\node[semi_morphism={90,}] (lf) at (0,0) {$\albar$};
\draw[midarrow_rev] (lf) -- +(90:1cm) node[pos=0.5, right] {$i$};
\draw[midarrow] (lf) -- +(270:1cm) node[pos=0.5, right] {$j$};
\draw[midarrow_rev] (lf) -- +(150:1cm) node[above left] {$X$};
\node at (1,0) {$\tnsr$};
\node[semi_morphism={270,}] (rt) at (2,0) {$\albar$};
\draw[midarrow] (rt) -- +(90:1cm) node[pos=0.5, left] {$i$};
\draw[midarrow_rev] (rt) -- +(270:1cm) node[pos=0.5, left] {$j$};
\draw[midarrow] (rt) -- +(-30:1cm) node[below right] {$X$};
\end{tikzpicture}
\quad := \quad
\sum_{i,j \in \Irr(\cC)} \sqrt{d_i^R}\sqrt{d_j^R}
\begin{tikzpicture}
\node[semi_morphism={90,}] (lf) at (0,0) {$\al$};
\draw[midarrow_rev] (lf) -- +(90:1cm) node[pos=0.5, right] {$i$};
\draw[midarrow] (lf) -- +(270:1cm) node[pos=0.5, right] {$j$};
\draw[midarrow_rev] (lf) -- +(150:1cm) node[above left] {$X$};
\node at (1,0) {$\tnsr$};
\node[semi_morphism={270,}] (rt) at (2,0) {$\al$};
\draw[midarrow] (rt) -- +(90:1cm) node[pos=0.5, left] {$i$};
\draw[midarrow_rev] (rt) -- +(270:1cm) node[pos=0.5, left] {$j$};
\draw[midarrow] (rt) -- +(-30:1cm) node[below right] {$X$};
\end{tikzpicture}
\end{equation}

Note that the strands labeled $i,j$, for which we sum over simples,
must be the two strands at the extremes of the semicircle;
when $\cC$ is spherical, this is not canonical, but usually it is
clear which two strands are chosen.
$\Gamma_X$ satisfies the following properties:

\begin{enumerate}
\item $\Gamma_-$ respects tensor products:
\begin{equation}
\begin{tikzpicture}
\node[semi_morphism={90,}] (lf) at (0,0) {$\albar$};
\draw[midarrow_rev] (lf) -- +(90:1cm) node[pos=0.5, right] {$i$};
\draw[midarrow] (lf) -- +(270:1cm) node[pos=0.5, right] {$j$};
\draw[midarrow_rev] (lf) -- +(150:1cm) node[above left] {$X$};
\draw[midarrow_rev] (lf) -- +(120:1cm) node[above] {$Y$};
\node at (1,0) {$\tnsr$};
\node[semi_morphism={270,}] (rt) at (2,0) {$\albar$};
\draw[midarrow] (rt) -- +(90:1cm) node[pos=0.5, left] {$i$};
\draw[midarrow_rev] (rt) -- +(270:1cm) node[pos=0.5, left] {$j$};
\draw[midarrow] (rt) -- +(-60:1cm) node[below] {$X$};
\draw[midarrow] (rt) -- +(-30:1cm) node[below right] {$Y$};
\end{tikzpicture}
=
\begin{tikzpicture}
\node[semi_morphism={90,}] (lf) at (0,0.4) {\tiny $\albar$};
\node[semi_morphism={90,}] (lf2) at (0,-0.4) {\tiny $\bebar$};
\draw[midarrow_rev] (lf) -- +(90:0.7cm) node[pos=0.5, right] {$i$};
\draw[midarrow={0.7}] (lf) -- (lf2) node[pos=0.5, right] {$k$};
\draw[midarrow] (lf2) -- +(270:0.7cm) node[pos=0.5, right] {$j$};
\draw[midarrow_rev] (lf) -- +(150:1cm) node[left] {$Y$};
\draw[midarrow_rev] (lf2) -- +(150:1cm) node[left] {$X$};
\node at (1,0) {$\tnsr$};
\node[semi_morphism={270,}] (rt) at (2,0.4) {\tiny $\albar$};
\node[semi_morphism={270,}] (rt2) at (2,-0.4) {\tiny $\bebar$};
\draw[midarrow] (rt) -- +(90:0.7cm) node[pos=0.5, left] {$i$};
\draw[midarrow_rev={0.7}] (rt) -- (rt2) node[pos=0.5, left] {$k$};
\draw[midarrow_rev] (rt2) -- +(270:0.7cm) node[pos=0.5, left] {$j$};
\draw[midarrow_rev] (rt) -- +(-30:1cm) node[right] {$Y$};
\draw[midarrow_rev] (rt2) -- +(-30:1cm) node[right] {$X$};
\end{tikzpicture}
\end{equation}

\item $\Gamma_X$ is natural in $X$: for $f: X \to Y$,
\begin{equation}
\begin{tikzpicture}
\node[semi_morphism={90,}] (lf) at (0,0) {\tiny $\albar$};
\draw[midarrow_rev] (lf) -- +(90:1cm) node[pos=0.5, right] {$i$};
\draw[midarrow] (lf) -- +(270:1cm) node[pos=0.5, right] {$j$};
\node[small_morphism] (f) at (-0.5,0.3) {\tiny $f$};
\draw[midarrow_rev={0.7}] (lf) -- (f);
\draw[midarrow_rev] (f) -- (-1,0.6) node[left] {$X$};
\node at (1,0) {$\tnsr$};
\node[semi_morphism={270,}] (rt) at (2,0) {\tiny $\albar$};
\draw[midarrow] (rt) -- +(90:1cm) node[pos=0.5, left] {$i$};
\draw[midarrow_rev] (rt) -- +(270:1cm) node[pos=0.5, left] {$j$};
\draw[midarrow] (rt) -- +(-30:1cm) node[right] {$Y$};
\end{tikzpicture}
\quad = \quad
\begin{tikzpicture}
\node[semi_morphism={90,}] (lf) at (0,0) {\tiny $\bebar$};
\draw[midarrow_rev] (lf) -- +(90:1cm) node[pos=0.5, right] {$i$};
\draw[midarrow] (lf) -- +(270:1cm) node[pos=0.5, right] {$j$};
\draw[midarrow_rev] (lf) -- +(150:1cm) node[left] {$X$};
\node at (1,0) {$\tnsr$};
\node[semi_morphism={270,}] (rt) at (2,0) {\tiny $\bebar$};
\draw[midarrow] (rt) -- +(90:1cm) node[pos=0.5, left] {$i$};
\draw[midarrow_rev] (rt) -- +(270:1cm) node[pos=0.5, left] {$j$};
\node[small_morphism] (f) at (2.5,-0.3) {\tiny $f$};
\draw[midarrow={0.7}] (rt) -- (f);
\draw[midarrow] (f) -- (3,-0.6) node[right] {$Y$};
\end{tikzpicture}
\end{equation}

\end{enumerate}
\end{lemma}

\begin{proof}
The second property follows from \lemref{l:summation-alpha-tunnel}.
The first property follows from using \lemref{l:summation-alpha-tunnel}
to ``pull'' $\bebar$ through $\albar$,
then use \lemref{l:summation} to contract the $k$ strand
(note the $\sqrt{d_k^R}$ coefficients from $\albar$ and $\bebar$
combine to give $d_k^R$,
allowing us to use \lemref{l:summation}).
\end{proof}

\begin{lemma}
\label{l:IM_htr_proj}
Let $M \in \cM$. The morphism
\begin{equation}
P'_M :=
\sum_{i,j,k \in \Irr(\cC)} \frac{\sqrt{d_i^L}\sqrt{d_j^L}d_k^R}{|\Irr_0(\cC)| \cD}
\;
\begin{tikzpicture}
\node[semi_morphism={90,}] (L) at (0.5,0) {$\al$};
\node[semi_morphism={270,}] (R) at (-0.5,0) {$\al$};
\draw (0,1) -- (0,-1) node[below] {$M$} node[above,pos=0] {$M$};
\draw[midarrow_rev] (L)-- +(0,1) node[pos=0.5,right] {$i$};
\draw[midarrow] (L) -- +(0,-1) node[pos=0.5,right] {$j$};
\draw[midarrow] (R) -- +(0,1) node[pos=0.5,left] {$i$};
\draw[midarrow_rev] (R)-- +(0,-1) node[pos=0.5,left] {$j$};
\draw[midarrow={0.6}] (L) to[out=-80,in=180] (1.2,0);
\node at (1, -0.5) {$k$};
\draw[midarrow_rev={0.6}] (R) to[out=-100,in=0] (-1.2,0);
\node at (-1, -0.5) {$k$};
\end{tikzpicture}
=
\sum_{i,j \in \Irr(\cC)} \frac{\sqrt{d_i^L}\sqrt{d_j^L}}{|\Irr_0(\cC)| \cD}
\;
\begin{tikzpicture}
\draw (0,1) -- (0,-1) node[below] {$M$} node[above,pos=0] {$M$};
\draw[midarrow_rev] (0.4,1) to[out=-90,in=180] (1,0.1);
\draw[midarrow] (-0.4,1) to[out=-90,in=0] (-1,0.1);
\node at (-0.7,0.7) {$i$};
\draw[midarrow] (0.4,-1) to[out=90,in=180] (1,-0.1);
\draw[midarrow_rev] (-0.4,-1) to[out=90,in=0] (-1,-0.1);
\node at (-0.7,-0.7) {$j$};
\end{tikzpicture}
\end{equation}
is a projection in $\End_{\htr(\cM)}(\bigoplus X_i \lact M \ract X_i^*)$.
Furthermore, it can be written as a composition
$P'_{(M,\gamma)} = \hat{P}'_{(M,\gamma)} \circ \check{P}'_{(M,\gamma)}$,
where
\begin{equation}
\check{P}'_{(M,\gamma)} :=
\sum_{i \in \Irr(\cC)} \frac{\sqrt{d_i^L}}{\sqrt{|\Irr_0(\cC)| \cD}}
\;
\begin{tikzpicture}
\draw (0,0.6) -- (0,-0.6) node[below] {$M$} node[above,pos=0] {$M$};
\draw[midarrow_rev] (0.3,0.6) to[out=-90,in=180] (0.8,0);
\draw[midarrow] (-0.3,0.6) to[out=-90,in=0] (-0.8,0);
\node at (-0.6,0.4) {$i$};
\end{tikzpicture}
\;\;\;
,
\;\;\;
\hat{P}'_{(M,\gamma)} :=
\sum_{j \in \Irr(\cC)} \frac{\sqrt{d_j^L}}{\sqrt{|\Irr_0(\cC)| \cD}}
\;
\begin{tikzpicture}
\draw (0,0.6) -- (0,-0.6) node[below] {$M$} node[above,pos=0] {$M$};
\draw[midarrow] (0.3,-0.6) to[out=90,in=180] (0.8,0);
\draw[midarrow_rev] (-0.3,-0.6) to[out=90,in=0] (-0.8,0);
\node at (-0.6,-0.4) {$j$};
\end{tikzpicture}
\end{equation}
such that $\check{P}'_{(M,\gamma)} \circ \hat{P}'_{(M,\gamma)} = \id_M$,
thus as objects
in $\Kar(\htr(\cM))$, we have
$M \simeq (\bigoplus X_i \lact M \ract X_i^*, P_{(M,\gamma)}')$.
\end{lemma}

\begin{proof}
Essentially the same as \lemref{l:Mga_proj}.
(Note the use of left dimensions $d_i^L$ instead of right dimensions
$d_i^R$.)
\end{proof}

\subsection{Module categories, balanced tensor product, and center}
\par \noindent
\label{s:bg-htr}

In this section, we review the results about balanced tensor product of module 
categories.
Our main goal is to give two constructions of the center
of an $\cC$-bimodule category $\cM$ -
$\cZ_\cC(\cM)$ (\defref{d:center}) and $\htr_\cC(\cM)$ (\defref{d:htr}),
and show that when $\cC$ is pivotal multifusion,
they are equivalent (\thmref{t:htr-center}).

Recall our convention that all categories considered in this paper
are locally finite  $\kk$-linear. Most of the time, they will be abelian; however, 
in some cases we will need to use $\kk$-linear additive (but not necessarily abelian) 
categories. For such a  category $\cA$, we will denote by $\Kar(\cA)$ the Karoubi envelope (also known as idempotent completion) of $\cA$. By definition, an object of $\Kar(\cA)$ is a pair $(A, p)$, where $A$ is an object of $\cA$ and $p\in \Hom_\cA(A,A)$ is an 
idempotent: $p^2=p$. Morphisms in $\Kar(\cA)$ are defined by 
$$
\Hom_{\Kar(\cA)}((A_1, p_1), (A_2, p_2))=\{f\in \Hom_\cA(A_1, A_2)\st p_2fp_1=f\}
$$

Throughout this section $\cC$ is a pivotal category,
though in the definitions $\cC$ is only required to be monoidal.
When $\cC$ is multifusion,
we use the conventions and notation laid out in the Appendix.
In particular, $\Irr(\cC)$ is the set of isomorphism classes,
$\Irr_0(\cC)$ are those simples appearing as direct summands
of the unit $\one$,
$\{X_i\}$ will be a fixed set of representatives of $\Irr(\cC)$,
$d_i^R$ is the (right) dimension of $X_i$,
and we will be using graphical presentation of morphisms.

We assume that the reader is familiar with the notions of module categories
and module structures on functors between module categories
(refer to \ocite{EGNO}*{Chapter 7});
for a left module category $\cM$ over $\cC$, we will denote the action of 
$A\in\cC$ on $M\in\cM$ by $A\lact M$. Similarly, we use $M\ract A$ for right action. 
In this paper, all module categories are assumed to be semisimple (as abelian categories).


This section is organized as follows:
Subsection~\ref{s:zM} provides the definition
and some properties of $\cZ_\cC(\cM)$,
Subsection~\ref{s:htrM} does so for $\htr_\cC(\cM)$,
and Subsection~\ref{s:equiv} shows that when
$\cC$ is pivotal multifusion,
these definitions are essentially the same.

\subsubsection{$\cZ_\cC(\cM)$}\par \noindent
\label{s:zM}

The following definition is essentially given in  \ocite{GNN}*{Definition 2.1}
(there $\cC$ is assumed to be fusion).

\begin{definition}\label{d:center}
Let $\cC$ be a finite multitensor  category,
and let $\cM$ be a $\cC$-bimodule category.
The center of $\cM$, denoted $\cZ_\cC(\cM)$, is the category with the following objects and morphisms:

Objects: pairs $(M,\ga)$, where $M\in \cM$ and $\ga$ is an isomorphism of functors  
$\ga_A\colon A \lact M\to M \ract A$, $A\in \cC$ (half-braiding) satisfying natural compatibility conditions.

Morphisms: $\Hom ((M,\ga), (M',\ga'))=\{f\in \Hom_\cM(M,M')\st f\ga=\ga' f\}$. 
\end{definition}

In particular, in the special case $\cM=\cC$, this construction gives the Drinfeld center $\cZ(\cC)$. 

\begin{remark}
    Equivalently, the center $\cZ_\cC(\cM)$ can be described as the category of $\cC$-bimodule functors $\cC\to \cM$; see \ocite{GNN} for details. 
\end{remark}    


\begin{theorem}
\label{t:center-adjoint}
\par\noindent
Let $\cC$ be pivotal multifusion,
and $\cM$ a $\cC$-bimodule category.
\item Let $F\colon  \cZ_\cC(\cM)\to \cM$ be the natural forgetful functor $F\colon  (M,\ga) \mapsto M$. 
Then it has a two-sided adjoint functor $I\colon \cM\to \cZ_\cC(\cM)$, given by 
\begin{equation}
\label{e:induction}
I(M)=\bigoplus_{i\in \Irr(\cC)} X_i \lact M \ract X_i^*
\end{equation}
with the half-braiding (recall \lemref{l:halfbrd})
\begin{equation}
\label{e:I(M)}
\begin{tikzpicture}
\draw (0,1)--(0,-1);
\node[above] at (0,1) {$M$};
\node[semi_morphism={90,}] (L) at (-0.5,0) {$\albar$};
\node[semi_morphism={270,}] (R) at (0.5,0) {$\albar$};
\draw (L) -- +(0,1) node[above] {$i$};
\draw (L) -- +(0,-1) node[below] {$j$}; 
\draw (R) -- +(0,1) node[above] {$i^*$};
\draw (R) -- +(0,-1) node[below] {$j^*$}; 
\draw (-1, 1) .. controls +(down:0.5cm) and +(135:0.5cm) .. (L);
\draw (1, -1) .. controls +(up:0.5cm) and +(-45:0.5cm) .. (R);
\end{tikzpicture}
\end{equation}

The adjunction isomorphism for $F\colon \cZ(\cM) \rightleftharpoons \cM \colon  I$,
\[
\Hom_{\cZ(\cM)}((M_1,\ga), I(M_2)) \simeq \Hom_\cM(M_1,M_2)
\]
is given by:
\begin{align}
\label{e:adj_isom}
\sum_{i\in \Irr(\cC)} 
\begin{tikzpicture}
\node[morphism] (ph) at (0,0) {$\ph_i$};
\draw (ph) -- +(90:1cm) node[above] {$M_1$};
\draw (ph) -- +(270:1cm) node[below] {$M_2$};
\draw[midarrow] (ph) to[out=180,in=90] (-1,-1);
\node at (-0.9,-0.1) {$i$};
\draw[midarrow_rev] (ph) to[out=0,in=90] (1,-1);
\node at (0.9,-0.1) {$i$};
\end{tikzpicture}
&\mapsto
\sum_{l \in \Irr_0(\cC)}
\begin{tikzpicture}
\draw[midarrow_rev={0.5}] (0,0) circle(0.5cm);
\node at (-0.7,0.1) {$l$};
\draw (0,1) -- (0,-1) node[pos=0, above] {$M_1$} node[below] {$M_2$};
\node[small_morphism] at (0,0.5) {$\ga$};
\node[small_morphism] at (0,-0.5) {\small $\ph_l$};
\end{tikzpicture}
\\
\label{e:adj_isom_2}
\sum_{j\in \Irr(\cC)} \sqrt{d_j^R}
\begin{tikzpicture}
\node[small_morphism] (ga) at (0,0.3) {$\ga$};
\draw (ga) -- (0,1) node[above] {$M_1$};
\draw (ga) -- (0,-1) node[below] {$M_2$};
\draw[midarrow] (ga) to[out=180,in=90] (-1,-1);
\node at (-0.9,0.1) {$j$};
\draw[midarrow_rev] (ga) to[out=0,in=90] (1,-1);
\node at (0.9,0.1) {$j$};
\node[small_morphism] at (0,-0.3) {\small $f$};
\end{tikzpicture}
&\leftmapsto
\hspace{10pt}
\begin{tikzpicture}
\node[morphism] (f) at (0,0) {$f$};
\draw (f) -- +(90:1cm) node[above] {$M_1$};
\draw (f) -- +(270:1cm) node[below] {$M_2$};
\end{tikzpicture}
\end{align}
(Note the sum on the right in \eqnref{e:adj_isom}
is over $\Irr_0(\cC)$ and not $\Irr(\cC)$.)
The other adjunction isomorphism
for $I\colon \cM \rightleftharpoons \cZ(\cM) :F$,
\[
\Hom_\cM(M_1,M_2) \simeq \Hom_{\cZ(\cM)}(I(M_1), (M_2,\ga)) 
\]
is given by a similar formula,
essentially obtained by rotating all the diagrams above.
\end{theorem}

\begin{proof}
This is essentially the same as when $\cC$ is spherical
as in \ocite{kirillov-stringnet},
but we provide it to assuage any doubts that the non-sphericality,
manifested in requiring semicircular morphisms $\al$,
does not lead to problems.
Note that the isomorphisms here differ slightly from that
of \ocite{kirillov-stringnet}.

Let us check that the morphism on the left side of \eqnref{e:adj_isom_2}
intertwines half-braidings:
\begin{equation*}
\sum_{j\in \Irr(\cC)} \sqrt{d_j^R}
\begin{tikzpicture}
\node[small_morphism] (ga2) at (0,0.9) {$\ga$};
\node[small_morphism] (ga) at (0,0.3) {$\ga$};
\node[small_morphism] (f) at (0,-0.3) {\small $f$};
\draw (ga) -- (ga2);
\draw (ga2) -- (0,1.5) node[above] {$M_1$};
\draw (ga) -- (f);
\draw (f) -- (0,-1.5) node[below] {$M_2$};
\draw (ga2) to[out=0,in=-90] (1.5,1.5) node[above] {$X$};
\draw (ga2) to[out=180,in=90] (-1.5,-1.5) node[below] {$X$};
\draw[midarrow] (ga) to[out=180,in=90] (-1,-1.5);
\draw[midarrow_rev] (ga) to[out=0,in=90] (1,-1.5);
\node at (0.9,0.1) {$j$};
\end{tikzpicture}
=
\sum_{i,j\in \Irr(\cC)} d_i^R \sqrt{d_j^R}
\begin{tikzpicture}
\node[small_morphism] (ga2) at (0,0.9) {$\ga$};
\node[small_morphism] (ga) at (0,0.3) {$\ga$};
\node[small_morphism] (f) at (0,-0.3) {\small $f$};
\draw (ga) -- (ga2);
\draw (ga2) -- (0,1.5) node[above] {$M_1$};
\draw (ga) -- (f);
\draw (f) -- (0,-1.5) node[below] {$M_2$};
\draw (ga2) to[out=0,in=-90] (1.5,1.5) node[above] {$X$};
\draw[midarrow_rev] (ga) to[out=0,in=90] (1,-1.5);
\node at (0.9,0.1) {$j$};
\node[semi_morphism={90,}] (al1) at (-1,-0.3) {\tiny $\al$};
\node[semi_morphism={90,}] (al2) at (-1,-0.9) {\tiny $\al$};
\draw (ga) to[out=180,in=80] (al1);
\draw (ga2) to[out=180,in=110] (al1);
\draw[midarrow={0.7}] (al1) to[out=-110,in=110] (al2);
\node at (-1.3,-0.6) {$i$};
\draw (al2) -- (-1,-1.5);
\draw (al2) to[out=-150,in=90] (-1.5,-1.5) node[below] {$X$};
\end{tikzpicture}
=
\sum_{i\in \Irr(\cC)} \sqrt{d_i^R}
\begin{tikzpicture}
\node[small_morphism] (ga) at (0,0.5) {$\ga$};
\node[small_morphism] (f) at (0,-0.3) {\small $f$};
\draw (ga) -- (0,1.5) node[above] {$M_1$};
\draw (ga) -- (f);
\draw (f) -- (0,-1.5) node[below] {$M_2$};
\node[semi_morphism={90,}] (al1) at (-1,-0.3) {\tiny $\albar$};
\node[semi_morphism={270,}] (al2) at (1,-0.3) {\tiny $\albar$};
\draw (al2) to[out=0,in=-90] (1.5,1.5) node[above] {$X$};
\draw (al1) to[out=180,in=90] (-1.5,-1.5) node[below] {$X$};
\draw[midarrow] (ga) to[out=180,in=90] (al1);
\node at (-0.8, 0.6) {$i$};
\draw (ga) to[out=0,in=90] (al2);
\draw (al1) -- (-1,-1.5);
\draw (al2) -- (1,-1.5);
\end{tikzpicture}
\end{equation*}

In the first equality, we use \lemref{l:summation};
in the second equality, we use the naturality of $\ga$
to pull the top $\al$ to the right,
and absorb the $\sqrt{d_i^R}\sqrt{d_j^R}$ factor
into $\al$ to get $\albar$.

Next we check that applying \eqnref{e:adj_isom} then \eqnref{e:adj_isom_2}
is the identity map.
Let $m_i = 1$ if $i\in \Irr_0(\cC)$, 0 otherwise.
In the following diagrams,
we implicitly sum lowercase latin alphabets
over $\Irr(\cC)$.
Then the composition is the map

\begin{gather*}
\begin{tikzpicture}
\node[morphism] (ph) at (0,0) {\small $\ph_i$};
\draw (ph) -- (0,1.5) node[above] {$M_1$};
\draw (ph) -- (0,-1.5) node[below] {$M_2$};
\draw[midarrow] (ph) to[out=180,in=90] (-1,-1.5);
\node at (-0.9,-0.1) {$i$};
\draw[midarrow_rev] (ph) to[out=0,in=90] (1,-1.5);
\node at (0.9,-0.1) {$i$};
\end{tikzpicture}
\mapsto
m_i\sqrt{d_j^R}
\begin{tikzpicture}
  \draw[midarrow_rev={0.6}] (0,0) circle(0.4cm);
\node at (-0.4,-0.4) {$i$};
\node[small_morphism] (ph) at (0,-0.4) {\small $\ph_i$};
\node[small_morphism] (ga1) at (0,0.9) {$\ga$};
\node[small_morphism] (ga2) at (0,0.4) {$\ga$};
\draw (ga1) -- (0,1.5) node[above] {$M_1$};
\draw (ga1) -- (ga2);
\draw (ga2) -- (ph);
\draw (ph) -- (0,-1.5) node[below] {$M_2$};
\draw[midarrow] (ga1) to[out=180,in=90] (-1,-1.5);
\node at (1.15,0) {$j$};
\draw[midarrow_rev] (ga1) to[out=0,in=90] (1,-1.5);
\end{tikzpicture}
=
m_i \sqrt{d_j^R} d_k^R
\begin{tikzpicture}
\node[semi_morphism={90,}] (al1) at (-1,0) {\tiny $\al$};
\node[semi_morphism={270,}] (al2) at (1,0) {\tiny $\al$};
\node[small_morphism] (ga) at (0,0.9) {$\ga$};
\node[small_morphism] (ph) at (0,0) {\small $\ph_i$};
\draw[midarrow] (ph) to[out=180,in=-80] (al1);
\node at (-0.5,-0.3) {$i$};
\draw (ph) to[out=0,in=-100] (al2);
\draw[midarrow] (ga) to[out=180,in=90] (al1);
\node at (-0.7,1) {$k$};
\draw (ga) to[out=0,in=90] (al2);
\draw (ga) -- (0,1.5) node[above] {$M_1$};
\draw (ga) -- (ph);
\draw (ph) -- (0,-1.5) node[below] {$M_2$};
\draw[midarrow] (al1) to[out=-100,in=90] (-1,-1.5);
\node at (-0.9,-0.8) {$j$};
\draw (al2) to[out=-80,in=90] (1,-1.5);
\end{tikzpicture}
=
m_i \sqrt{d_j^R} d_k^R
\sqrt{d_i^R} \sqrt{d_l^R}
\begin{tikzpicture}
\node[semi_morphism={90,}] (al1) at (-1,-0.5) {\tiny $\al$};
\node[semi_morphism={270,}] (al2) at (1,-0.5) {\tiny $\al$};
\node[small_morphism] (ph) at (0,0.6) {\small $\ph_l$};
\node[semi_morphism={90,}] (bt1) at (-0.5,0) {\tiny $\beta$};
\node[semi_morphism={270,}] (bt2) at (0.5,0) {\tiny $\beta$};
\draw[midarrow] (ph) to[out=180,in=90] (bt1);
\node at (-0.5,0.6) {$l$};
\draw (ph) to[out=0,in=90] (bt2);
\draw[midarrow] (bt1) to[out=180,in=90] (al1);
\node at (-0.8,0.2) {$k$};
\draw (bt2) to[out=0,in=90] (al2);
\draw[midarrow] (bt1) to[out=-90,in=-80] (al1);
\node at (-0.5,-0.6) {$i$};
\draw (bt2) to[out=-90,in=-100] (al2);
\draw (ph) -- (0,1.5) node[above] {$M_1$};
\draw (ph) -- (0,-1.5) node[below] {$M_2$};
\draw[midarrow] (al1) to[out=-100,in=90] (-1,-1.5);
\node at (-0.9,-1.1) {$j$};
\draw (al2) to[out=-80,in=90] (1,-1.5);
\end{tikzpicture}
\\
=
m_i \sqrt{d_j^R} \sqrt{d_i^R} \sqrt{d_l^R}
\begin{tikzpicture}
\node[small_morphism] (ph) at (0,0.6) {\small $\ph_l$};
\node[semi_morphism={90,}] (bt1) at (-0.5,0) {\tiny $\beta$};
\node[semi_morphism={270,}] (bt2) at (0.5,0) {\tiny $\beta$};
\draw[midarrow] (ph) to[out=180,in=90] (bt1);
\node at (-0.5,0.6) {$l$};
\draw (ph) to[out=0,in=90] (bt2);
\draw[midarrow] (bt1)
  ..controls (-0.4,-0.5) and (-0.6,-0.8) .. (-0.8,-0.6)
  ..controls (-1,-0.4) and (-0.7,-0.2) .. (bt1);
\draw[midarrow_rev] (bt2)
  ..controls (0.4,-0.5) and (0.6,-0.8) .. (0.8,-0.6)
  ..controls (1,-0.4) and (0.7,-0.2) .. (bt2);
\node at (-0.5,-0.7) {$i$};
\draw (ph) -- (0,1.5) node[above] {$M_1$};
\draw (ph) -- (0,-1.5) node[below] {$M_2$};
\draw[midarrow] (bt1) to[out=180,in=90] (-1,-1.5);
\node at (-0.9,-1.1) {$j$};
\draw (bt2) to[out=0,in=90] (1,-1.5);
\end{tikzpicture}
=
m_i \sqrt{d_j^R} \sqrt{d_i^R} \sqrt{d_l^R}
\begin{tikzpicture}
\node[small_morphism] (ph) at (0,0.6) {\small $\ph_l$};
\node[semi_morphism={90,}] (bt1) at (-1,0.1) {\tiny $\beta$};
\node[semi_morphism={270,}] (bt2) at (1,0.1) {\tiny $\beta$};
\draw[midarrow] (ph) to[out=180,in=90] (bt1);
\node at (-0.5,0.8) {$l$};
\draw (ph) to[out=0,in=90] (bt2);
\draw[midarrow_rev={0.75}] (-0.5,-0.4) circle(0.3cm);
\node at (-0.5,-0.9) {$i$};
\draw (ph) -- (0,1.5) node[above] {$M_1$};
\draw (ph) -- (0,-1.5) node[below] {$M_2$};
\draw[midarrow] (bt1) -- (-1,-1.5);
\node at (-1.1,-1.1) {$j$};
\draw (bt2) -- (1,-1.5);
\end{tikzpicture}
=
\begin{tikzpicture}
\node[morphism] (ph) at (0,0) {\small $\ph_l$};
\draw (ph) -- (0,1.5) node[above] {$M_1$};
\draw (ph) -- (0,-1.5) node[below] {$M_2$};
\draw[midarrow] (ph) to[out=180,in=90] (-1,-1.5);
\node at (-0.9,-0.1) {$l$};
\draw[midarrow_rev] (ph) to[out=0,in=90] (1,-1.5);
\end{tikzpicture}
\end{gather*}

The first equality is the same as the previous computation.
The second equality uses the fact that $\sum \ph_i$
intertwines half-braidings,
so that we ``pull'' the $k$ strand through $\ph_i$
The third equality comes from ``pulling'' $\al$
through $\beta$.
The fourth equality comes from ``pulling''
the $i$ loop through $\beta$.
Finally, for the last equality,
we observe that (1) only $j=k$ terms in the sum contribute,
and so we have a $d_j^R$ coefficient,
and we may apply \lemref{l:summation};
(2) since
$d_i^R = 1$ for $i\in \Irr_0(\cC)$,
\[
\sum_i m_i \sqrt{d_i^R}
\begin{tikzpicture}
\draw[midarrow_rev={0}] (0,0) circle(0.3cm);
\node at (0.4,-0.2) {$i$};
\end{tikzpicture}
= \sum_{l\in \Irr_0(\cC)} \id_{\one_l}
= \id_\one
\]
\end{proof}

An important special case is when $\cM=\cM_1\boxtimes\cM_2$, where $\cM_1$ is a right module  category over a pivotal multifusion category  $\cC$, and $\cM_2$ is a left module category over $\cC$. In this case, by \ocite{ENO10}*{Proposition~3.8},
one has that $\cZ_\cC(\cM_1\boxtimes \cM_2)$ is naturally equivalent to 
the balanced tensor product of categories:
\begin{equation}\label{e:center-product}
\cZ_\cC(\cM_1\boxtimes \cM_2)\simeq \cM_1\boxtimes_\cC \cM_2
\end{equation}
where the balanced tensor product is defined by the universal property: for any abelian  category $\cA$, we have a natural equivalence
$$
Fun_{bal}(\cM_1\times \cM_2, \cA)=Fun(\cM_1\boxtimes_\cC \cM_2, \cA)
$$
where $Fun$, $Fun_{bal}$ stand for category of  $\kk$-linear additive functors (respectively, category of $\kk$-linear additive  $\cC$- balanced functors); see details in \ocite{ENO10}*{Definition 3.3}.

Under the equivalence \eqref{e:center-product}, the natural functor
$\cM_1\boxtimes \cM_2\to \cM_1\boxtimes_\cC \cM_2$ 
is identified with the functor
$I\colon  \cM_1\boxtimes \cM_2\to \cZ_\cC(\cM_1\boxtimes \cM_2)$ 
constructed in \thmref{t:center-adjoint}.

Recall that a functor $F\colon \cA\to \cB$, where $\cB$ is abelian and $\cA$
additive (not necessarily abelian) is called \emph{dominant} if any object of
$\cB$ appears as a subquotient of $F(X)$ for some $X \in \Obj\cA$.
Similarly, we say that a full  subcategory $\cA\subset \cB$ is dominant if any
object of $\cB$ appears as a subquotient of some $X\in \Obj \cA$.
In the case when $\cA$ is a
full additive subcategory in a semisimple abelian category $\cB$,
this immediately implies that the  Karoubi envelope of $\cA$
is equivalent to $\cB$
(in particular, this implies that $\Kar(\cA)$ is abelian).

\begin{proposition} \label{p:I_dominant}
Under the hypotheses of \thmref{t:center-adjoint},
the functor $I\colon \cM \to \cZ_\cM(\cM)$ is dominant.
\end{proposition}
\begin{proof}
This follows immediately from \lemref{l:Mga_proj} below.
\end{proof}

\begin{lemma} \label{l:Mga_proj}
Let $(M,\ga) \in \cZ_\cC(\cM)$. The morphism
\begin{equation}
\label{e:P_Mga}
P_{(M,\ga)} :=
\frac{1}{|\Irr_0(\cC)| \cD} G(\sum d_i^R \ga_{X_i})
=
\sum_{i,j,k \in \Irr(\cC)} \frac{\sqrt{d_i^R}\sqrt{d_j^R}d_k^R}{|\Irr_0(\cC)| \cD}
\begin{tikzpicture}
\node[small_morphism] (ga) at (0,0) {$\ga$};
\node[semi_morphism={90,}] (L) at (-1,0) {$\al$};
\node[semi_morphism={270,}] (R) at (1,0) {$\al$};
\draw (ga)-- +(0,1.5) node[above] {$M$}; \draw (ga)-- +(0,-1.5) node[below] {$M$};
\draw[midarrow_rev] (L)-- +(0,1.5) node[pos=0.5,left] {$i$};
\draw[midarrow] (L) -- +(0,-1.5) node[pos=0.5,left] {$j$};
\draw[midarrow] (R) -- +(0,1.5) node[pos=0.5,right] {$i$};
\draw[midarrow_rev] (R)-- +(0,-1.5) node[pos=0.5,right] {$j$};
\draw[midarrow={0.6}] (L) to[out=-80,in=180] (ga);
\node at (-0.5, -0.5) {$k$};
\draw[midarrow_rev={0.6}] (R) to[out=-100,in=0] (ga);
\node at (0.5, -0.5) {$k$};
\end{tikzpicture}
=
\sum_{i,j \in \Irr(\cC)} \frac{\sqrt{d_i^R}\sqrt{d_j^R}}{|\Irr_0(\cC)| \cD}
\begin{tikzpicture}
\node[small_morphism] (ga) at (0,0.4) {$\ga$};
\node[small_morphism] (ga2) at (0,-0.4) {$\ga$};
\draw (ga) -- (0,1.5) node[above] {$M$};
\draw (ga) -- (ga2);
\draw (ga2) -- (0,-1.5) node[below] {$M$};
\draw[midarrow_rev] (ga) to[out=180,in=-90] (-1,1.5);
\node at (-0.7,1) {$i$};
\draw[midarrow] (ga) to[out=0,in=-90] (1,1.5);
\draw[midarrow] (ga2) to[out=180,in=90] (-1,-1.5);
\node at (-0.7,-1) {$j$};
\draw[midarrow_rev] (ga2) to[out=0,in=90] (1,-1.5);
\end{tikzpicture}
\end{equation}
is a projection in $\End_{\cZ_\cC(\cM)}(I(M))$.
Furthermore, it can be written as a composition
$P_M = \hat{P}_M \circ \check{P}_M$,
where
\begin{equation}
\check{P}_{(M,\ga)} :=
\sum_{i \in \Irr(\cC)} \frac{\sqrt{d_i^R}}{\sqrt{|\Irr_0(\cC)| \cD}}
\;
\begin{tikzpicture}
\node[small_morphism] (ga) at (0,0) {$\ga$};
\draw (ga) -- (0,1) node[above] {$M$};
\draw (ga) -- (0,-1) node[below] {$M$};
\draw[midarrow_rev] (ga) to[out=180,in=-90] (-1,1);
\node at (-0.7,0.6) {$i$};
\draw[midarrow] (ga) to[out=0,in=-90] (1,1);
\end{tikzpicture}
\;\;\;
,
\;\;\;
\hat{P}_{(M,\ga)} :=
\sum_{j \in \Irr(\cC)} \frac{\sqrt{d_j^R}}{\sqrt{|\Irr_0(\cC)| \cD}}
\;
\begin{tikzpicture}
\node[small_morphism] (ga) at (0,0) {$\ga$};
\draw (ga) -- (0,1) node[above] {$M$};
\draw (ga) -- (0,-1) node[below] {$M$};
\draw[midarrow] (ga) to[out=180,in=90] (-1,-1);
\node at (-0.7,-0.6) {$j$};
\draw[midarrow_rev] (ga) to[out=0,in=90] (1,-1);
\end{tikzpicture}
\end{equation}
such that $\check{P}_{(M,\ga)} \circ \hat{P}_{(M,\ga)} = \id_{(M,\ga)}$,
thus exhibiting $(M,\ga)$ as a direct summand of $I(M)$.
Note that $\check{P}_{(M,\ga)}, \hat{P}_{(M,\ga)}$
are multiples of the morphisms corresponding to $\id_M$
under the adjunctions of \thmref{t:center-adjoint}.
\end{lemma}

\begin{proof}
The second equality in \eqnref{e:P_Mga}
follows from pulling $\al$ through $\ga$
and using \lemref{l:summation}.
$\hat{P}_{(M,\ga)}$ was shown to be a morphism
in $\Hom_{\cZ_\cC(\cM)}((M,\ga), I(M))$
in the proof of \thmref{t:center-adjoint},
and one shows $\check{P}_{(M,\ga)} \in \Hom_{\cZ_\cC(\cM)}(I(M), (M,\ga))$
in a similar fashion.
The following computation shows that
$\check{P}_{(M,\ga)} \circ \hat{P}_{(M,\ga)} = \id_{(M,\ga)}$:
\begin{align*}
\check{P}_{(M,\ga)} \circ \hat{P}_{(M,\ga)} =
\sum_{i\in \Irr(\cC)} \frac{d_i^R}{|\Irr_0(\cC)| \cD}
\begin{tikzpicture}
\draw (0,-0.8) -- (0,0.8) node[above] {$M$} node[pos=0,below] {$M$};
\draw[midarrow] (0,0) circle(0.4cm);
\node at (-0.6,0) {\small $i$};
\node[small_morphism] (ga2) at (0,0.4) {\tiny $\ga$};
\node[small_morphism] (ga3) at (0,-0.4) {\tiny $\ga$};
\end{tikzpicture}
=
\frac{1}{|\Irr_0(\cC)| \cD}
\begin{tikzpicture}
\draw[regular, midarrow] (0,0) circle(0.4cm);
\draw (0.6,-0.8) -- (0.6,0.8) node[above] {$M$} node[pos=0,below] {$M$};
\end{tikzpicture}
=
\id_{(M,\ga)}
\end{align*}
The second equality comes from ``pulling'' the $j$ loop
out to the left,
and the last equality follows from \lemref{l:dashed_circle}.
\end{proof}

\begin{proposition} \label{p:ZM_ss}
Under the hypotheses of \thmref{t:center-adjoint},
if $\cM$ is finite semisimple, then so is $\cZ_\cC(\cM)$.
\end{proposition}
\begin{proof}
Using exactness in $\cM$ of the left and right actions,
abelianness of $\cM$ transfers to $\cZ(\cM)$.
For example, the kernel $K$ of a morphism
$f: M_1 \to M_2$ such that $f\in \Hom_{\cZ(\cM)}((M_1,\ga^1),(M_2,\ga^2))$
would inherit a half-braiding $\ga^1|_K$.
Semisimplicity follows from the semisimplicity of $\cM$
and \lemref{l:Mga_proj}.
Finiteness follows from \prpref{p:I_dominant};
$I$ ensures there can't be too many simples in $\cZ(\cM)$.
\end{proof}

%
%
%

For applications, we will need to consider centers over
a full, dominant, monoidal subcategory $\cC' \subseteq \cC$.
Equivalently,
$\cC'$ is a pivotal category
whose Karoubi envelope is multifusion.

\begin{lemma} \label{l:ZM_sub}
Let $\cC'$ be a pivotal locally finite $\kk$-linear additive category
whose Karoubi envelope $\cC = \Kar(\cC')$ is multifusion.
Let $\cM$ be a $\cC$-bimodule category,
and hence naturally a $\cC'$-bimodule category (as before, we assume that
 $\cM$ is a semisimple abelian category).
Then there is a natural equivalence
\[
  \cZ_\cC(\cM) \simeq \cZ_{\cC'}(\cM)
\]
In particular, for right, left $\cC$-module categories
$\cM_1,\cM_2$, there is a natural equivalence
\[
  \cM_1 \boxtimes_\cC \cM_2 \simeq \cM_1 \boxtimes_{\cC'} \cM_2.
\]
\end{lemma}

\begin{proof}
The equivalence is given as follows:
objects $(M,\gamma)$ in $\cZ_\cC(\cM)$ are naturally objects in $\cZ_{\cC'}(\cM)$
by forgetting some of the half-braiding, i.e. $(M, \gamma|_{\cC'})$;
morphisms $f: (M,\gamma) \to (M',\gamma')$
are naturally morphisms $f: (M,\gamma|_{\cC'}) \to (M',\gamma|_{\cC'})$.
We need to check that this is an equivalence.

The functor is essentially surjective: any half-braiding over $\cC'$
can be completed to a half-braiding over $\cC$.
To see this, let $\gamma$ be a half-braiding over $\cC'$.
Let $X \in \Obj \cC \backslash \Obj \cC'$,
and let it be a direct summand of some $Y \in \Obj \cC'$,
$X \overset{\iota}{\underset{p}{\rightleftharpoons}} Y$.
Then we define the extension of $\gamma$ to $X$ by
$\gamma_X = (\id_{M_2} \ract p) \circ \gamma_Y \circ (\iota \lact \id_{M_1})$.
It is easy to check, using the semisimplicity of $\cC$,
that $\gamma_X$ is independent on the choice
of $Y$ and $p,\iota$.
It is also easy to check that the resulting extension
is indeed natural in $X$.

For morphisms, it is clear that this functor is faithful.
To show fullness, consider
$f \in \Hom_{\cZ_{\cC'}(\cM)}((M_1,\gamma^1),(M_2,\gamma^2))$.
We need to check that it also intertwines half-braiding
with $X \in \cC$, but this follows easily from the
definition of the extension of half-braiding given above.

Note since $\gamma$ has a unique extension to all of $\cC$,
this proof actually shows that the equivalence is an isomorphism.\\
\end{proof}

Note that in the proof above, we do not use the rigidity of $\cC$,
but we need it to conclude the second statement concerning
balanced tensor products.

\subsubsection{$\htr_\cC(\cM)$}\par \noindent
\label{s:htrM}

Next we define the other notion of center.

\begin{definition}\label{d:htr}
Let $\cC$ be monoidal, and
$\cM$ a $\cC$-bimodule category.
Define the {\em horizontal trace} 
$\htr_\cC(\cM)$ as the  category with the following objects and morphisms: 

Objects: same as in $\cM$

Morphisms: $\Hom_{\htr_\cC(\cM)}(M_1,M_2)=
\bigoplus_X \ihom{\cM}{X}(M_1,M_2)/\sim$,
where $\ihom{\cM}{X}(M_1,M_2) := \Hom_\cM(X \lact M_1, M_2 \ract X)$,
the sum is over all (not necessarily simple) objects $X\in \cC$,
and $\sim$ is the equivalence relation generated by
the following:

For any $\psi\in \ihom{\cM}{Y,X}(M_1,M_2) := \Hom_\cM(Y \lact M_1, M_2 \ract X)$
  and $f\in \Hom_\cC(X,Y)$, we have 

\[
\begin{tikzpicture}
\node[morphism] (psi) at (0,0) {$\psi$};
\draw (psi) -- +(0,1) node[above] {$M_1$};
\draw (psi) -- +(0,-1) node[below] {$M_2$};
\coordinate (L) at (-1,1);
\coordinate (R) at (1,-1);
\node[above] at (L) {$X$};
\node[below] at (R) {$X$};
\draw (L) to[out=-90,in=170] (psi);
\draw (R) to[out=90,in=-10] (psi);
\node[morphism] at (-0.8,0.4) {$f$};
\node at (-0.5,-0.15) {$Y$};
\end{tikzpicture}
\quad \sim \quad
\begin{tikzpicture}
\node[morphism] (psi) at (0,0) {$\psi$};
\draw (psi) -- +(0,1) node[above] {$M_1$};
\draw (psi) -- +(0,-1) node[below] {$M_2$};
\coordinate (L) at (-1,1);
\coordinate (R) at (1,-1);
\node[above] at (L) {$Y$};
\node[below] at (R) {$Y$};
\draw (L) to[out=-90,in=170] (psi);
\draw (R) to[out=90,in=-10] (psi);
\node[morphism] at (0.8,-0.4) {$f$};
\node at (0.5,0.15) {$X$};
\end{tikzpicture}
\]
In other words,
$\Hom_{\htr_\cC(\cM)}(M_1,M_2) = \int^X \ihom{\cM}{X,X}(M_1, M_2)$
is the coend of the functor
$\ihom{\cM}{-,-}(M_1, M_2) : \cC^{\text{op}} \times \cC \to \Vect$
(see e.g. \ocite{maclane}).

Composition is given by
\[
\ihom{\cM}{Y}(M_2,M_3)) \tnsr \ihom{\cM}{X}(M_1,M_2)
\to \ihom{\cM}{Y\tnsr X}(M_1,M_3)
\]
which sends $\psi \tnsr \vphi$ to
\[
  Y \lact (X \lact M_1)
  \xxto{\id_Y \lact \psi} Y \lact (M_2 \ract X)
  \simeq (Y \lact M_2) \ract X
  \xxto{\vphi \ract \id_X} (M_3 \ract Y) \ract X
\]

For right, left $\cC$-module categories $\cM_1,\cM_2$,
we denote $\cM_1 \hatbox{\cC} \cM_2 = \htr_\cC(\cM_1 \boxtimes \cM_2)$.
\end{definition}

When the context is clear, we will drop the subscript $\htr = \htr_\cC$.
We will write $[\vphi] \in \Hom_{\htr(\cM)}(M_1,M_2)$ for the morphism
represented by $\vphi \in \ihom{\cM}{X}(M_1,M_2)$ for some $X$.\\

%

It can be shown that in a certain sense this definition is dual to the definition of center 
given above and is closely related to the notion of co-center as described in \ocite{DSSP}*{Section 3.2.2}. However, we will not be discussing the exact relation here. 

It is easy to see that the category $\htr(\cM)$ is 
additive but not necessarily abelian.

There is a natural inclusion functor $\htr \colon  \cM \to \htr(\cM)$
which is identity on objects,
and on morphisms it is the natural map
$\Hom_\cM(M_1,M_2) = \ihom{\cM}{\one}(M_1,M_2)
\to \Hom_{\htr(\cM)}(M_1,M_2)$.

The horizontal trace construction is functorial
with respect to bimodule functors:

\begin{lemma}
\label{l:htr-functorial}
Given a functor of $\cC$-bimodule categories
$(F,J): \cM \to \cM'$,
there is a natural functor $\htr(F,J): \htr(\cM) \to \htr(\cM')$
that is the same as $F$ on objects.
\end{lemma}
\begin{proof}
We define the functor $\htr(F,J)$ to act the same as $F$ on objects,
and sends a morphism $\vphi \in \ihom{\cM}{X}(M_1,M_2)$
to the composition
\[
X \lact F(M_1) \simeq F(X \lact M_1)
\xrightarrow{F(\vphi)}
F(M_2 \ract X) \simeq F(M_2) \ract X
\]
where the equivalences are from the bimodule structure $J$.
It is easy to check that composition is respected.
\end{proof}

We also consider $\cC' \subseteq \cC$
as in \lemref{l:ZM_sub},
but here we do not need rigidity nor semisimplicity on $\cC$:

\begin{lemma} \label{l:htr_sub}
Let $\cC'$ be monoidal,
and let $\cC = \Kar(\cC')$ be its Karoubi envelope.
Let $\cM$ be a $\cC$-bimodule category,
and hence naturally a $\cC'$-bimodule category.
Then there is a natural equivalence
\[
  \htr_{\cC'}(\cM) \simeq \htr_\cC(\cM)
\]
In particular, for right, left $\cC$-module categories
$\cM_1,\cM_2$, there is a natural equivalence
\[
  \cM_1 \hatbox{\cC} \cM_2 \simeq \cM_1 \hatbox{\cC'} \cM_2
\]

\end{lemma}

\begin{proof}
The equivalence is given by the identity map on objects,
and for two objects $M_1,M_2 \in \Obj \cM$,
the map on morphisms is given by completing the bottom arrow:
\[
\begin{tikzcd}
  \bigoplus_{X \in \cC'} \ihom{\cM}{X}(M_1,M_2)
    \ar[r] \ar[d]
  & \bigoplus_{X \in \cC} \ihom{\cM}{X}(M_1,M_2)
    \ar[d]
  \\
  \Hom_{\htr_{\cC'}(\cM)}(M_1, M_2)
    \ar[r]
  & \Hom_{\htr_\cC(\cM)}(M_1, M_2)
\end{tikzcd}
\]
It remains to prove that the bottom arrow is an isomorphism.

Let us first observe the following.
Let $X, Y\in \Obj \cC'$, and suppose $X$ is a
direct summand of $Y$, with
$X \overset{\iota}{\underset{p}{\rightleftharpoons}} Y$.
Let $\vphi \in \ihom{\cM}{X}(M_1,M_2)$.
Then $\vphi = \vphi \circ p \circ \iota
\sim \iota \circ \vphi \circ p \in \ihom{\cM}{Y}(M_1,M_2)$,
where we write $p,\iota$ instead of
$p \lact \id_{M_1}, \id_{M_2} \ract \iota$
for simplicity. This works for $\cC$ too.
Thus one can identify $\ihom{\cM}{X}(M_1,M_2)$
with a subspace of $\ihom{\cM}{Y}(M_1,M_2)$.

Surjectivity: Essentially, we need to show that any
morphism in $\htr_\cC(\cM)$ can be ``absorbed'' into $\htr_{\cC'}(\cM)$.
Let $[\vphi] \in \Hom_{\htr_\cC(\cM)}(M_1,M_2)$
be represented by some $\vphi \in \ihom{\cM}{X}(M_1,M_2)$.
By the above observation,
we can choose $Y \in \Obj \cC'$
with $X$ a direct summand of $Y$,
then $\vphi \in \ihom{\cM}{X}(M_1,M_2)$
is identified with some morphism in
$\ihom{\cM}{Y}(M_1,M_2)$,
so $[\vphi]$ is in the image.

Injectivity: Essentially, we need to show that
relations can also be ``absorbed'' into $\htr_{\cC'}(\cM)$.
Let $[\vphi] \in \Hom_{\htr_{\cC'}(\cM)}(M_1,M_2)$
that is sent to 0.
By the observation above, we may represent it by some
$\vphi \in \ihom{\cM}{Y}(M_1,M_2)$ for some $Y\in \Obj \cC'$.
Since it is 0 in $\Hom_{\htr_\cC(\cM)}(M_1,M_2)$,
there exists
\begin{itemize}
\item a finite collection of objects $J = \{A_j\} \subset \Obj \cC$
  so that $A_0 = Y$.
\item $\Phi_i \in \ihom{\cM}{A_{m_i},A_{n_i}}(M_1,M_2)$,
\item $f_i : A_{n_i} \to A_{m_i}$,
\end{itemize}
such that $\vphi = \sum_i f_i \circ \Phi_i - \Phi_i \circ f_i
\in \bigoplus_{A_j \in J} \ihom{\cM}{A_j}(M_1,M_2)$.

We want to be able to replace the $A_j$'s with objects in $\cC'$.
For each $j \neq 0$, choose some $B_j \in \Obj \cC'$
such that $A_j$ is a direct summand of $B_j$:
$A_j \overset{\iota_j}{\underset{p_j}{\rightleftharpoons}} B_j$.
For $j=0$, we take $B_0 = A_0 = Y$ and $\iota_0 = p_0 = \id_Y$.
This gives us maps
$\Theta_j: \psi \mapsto \iota_j \circ \psi \circ p_j:
  \ihom{\cM}{A_j}(M_1,M_2) \to \ihom{\cM}{B_j}(M_1,M_2)$.
Denote $\Theta = \sum \Theta_j$.

Now consider
\begin{itemize}
\item $L = \{B_j\} \subset \Obj \cC'$,
\item $\Psi_i = \iota_{n_i} \circ \Phi_i \circ p_{m_i}
  \in \ihom{\cM}{B_{m_i},B_{n_i}}(M_1,M_2)$,
\item $g_i = \iota_{m_i} \circ f_i \circ p_{n_i} :
  B_{n_i} \to B_{m_i}$.
\end{itemize}
It is a simple matter to verify that
$g_i \circ \Psi_i - \Psi \circ g_i =
  \Theta(f_i \circ \Phi_i - \Phi_i \circ f_i)$.
Hence $\vphi = \Theta(\vphi)
  = \Theta(\sum f_i \circ \Phi_i - \Phi_i \circ f_i)
  = \sum g_i \circ \Psi_i - \Psi \circ g_i$
is 0 in $\Hom_{\htr_{\cC'}}(M_1,M_2)$.

\end{proof}

The following lemma is similar in spirit to \lemref{l:Mga_proj}.

\begin{lemma}
\label{l:IM_htr_proj}
Let $M \in \cM$. The morphism
\begin{equation}
P'_M :=
\sum_{i,j,k \in \Irr(\cC)} \frac{\sqrt{d_i^L}\sqrt{d_j^L}d_k^R}{|\Irr_0(\cC)| \cD}
\;
\begin{tikzpicture}
\node[semi_morphism={90,}] (L) at (0.5,0) {$\al$};
\node[semi_morphism={270,}] (R) at (-0.5,0) {$\al$};
\draw (0,1) -- (0,-1) node[below] {$M$} node[above,pos=0] {$M$};
\draw[midarrow_rev] (L)-- +(0,1) node[pos=0.5,right] {$i$};
\draw[midarrow] (L) -- +(0,-1) node[pos=0.5,right] {$j$};
\draw[midarrow] (R) -- +(0,1) node[pos=0.5,left] {$i$};
\draw[midarrow_rev] (R)-- +(0,-1) node[pos=0.5,left] {$j$};
\draw[midarrow={0.6}] (L) to[out=-80,in=180] (1.2,0);
\node at (1, -0.5) {$k$};
\draw[midarrow_rev={0.6}] (R) to[out=-100,in=0] (-1.2,0);
\node at (-1, -0.5) {$k$};
\end{tikzpicture}
=
\sum_{i,j \in \Irr(\cC)} \frac{\sqrt{d_i^L}\sqrt{d_j^L}}{|\Irr_0(\cC)| \cD}
\;
\begin{tikzpicture}
\draw (0,1) -- (0,-1) node[below] {$M$} node[above,pos=0] {$M$};
\draw[midarrow_rev] (0.4,1) to[out=-90,in=180] (1,0.1);
\draw[midarrow] (-0.4,1) to[out=-90,in=0] (-1,0.1);
\node at (-0.7,0.7) {$i$};
\draw[midarrow] (0.4,-1) to[out=90,in=180] (1,-0.1);
\draw[midarrow_rev] (-0.4,-1) to[out=90,in=0] (-1,-0.1);
\node at (-0.7,-0.7) {$j$};
\end{tikzpicture}
\end{equation}
is a projection in $\End_{\htr(\cM)}(\bigoplus X_i \lact M \ract X_i^*)$.
Furthermore, it can be written as a composition
$P'_M = \hat{P}'_M \circ \check{P}'_M$,
where
\begin{equation}
\check{P}'_M :=
\sum_{i \in \Irr(\cC)} \frac{\sqrt{d_i^L}}{\sqrt{|\Irr_0(\cC)| \cD}}
\;
\begin{tikzpicture}
\draw (0,0.6) -- (0,-0.6) node[below] {$M$} node[above,pos=0] {$M$};
\draw[midarrow_rev] (0.3,0.6) to[out=-90,in=180] (0.8,0);
\draw[midarrow] (-0.3,0.6) to[out=-90,in=0] (-0.8,0);
\node at (-0.6,0.4) {$i$};
\end{tikzpicture}
\;\;\;
,
\;\;\;
\hat{P}'_M :=
\sum_{j \in \Irr(\cC)} \frac{\sqrt{d_j^L}}{\sqrt{|\Irr_0(\cC)| \cD}}
\;
\begin{tikzpicture}
\draw (0,0.6) -- (0,-0.6) node[below] {$M$} node[above,pos=0] {$M$};
\draw[midarrow] (0.3,-0.6) to[out=90,in=180] (0.8,0);
\draw[midarrow_rev] (-0.3,-0.6) to[out=90,in=0] (-0.8,0);
\node at (-0.6,-0.4) {$j$};
\end{tikzpicture}
\end{equation}
such that $\check{P}'_M \circ \hat{P}'_M = \id_M$,
thus as objects
in $\Kar(\htr(\cM))$, we have
$M \simeq (\bigoplus X_i \lact M \ract X_i^*, P_M')$.
\end{lemma}

\begin{proof}
Essentially the same as \lemref{l:Mga_proj}.
(Note the use of both left and right dimensions.)
\end{proof}

\subsubsection{Equivalence}\par \noindent
\label{s:equiv}

\begin{theorem}
\label{t:htr-center}
Let $\cC$ be pivotal multifusion,
and $\cM$ a $\cC$-bimodule category.
One has a natural equivalence 
\[
  \Kar(\htr(\cM)) \simeq \cZ_\cC(\cM)
\]
Under this equivalence, the inclusion functor $\htr: \cM\to \htr(\cM)$ is 
identified with the functor $I : \cM\to \Z_\cC(\cM)$.

In particular, for right, left $\cC$-modules $\cM_1,\cM_2$,
we have
\[
  \cM_1 \boxtimes_\cC \cM_2 \simeq \cZ_\cC(\cM_1 \boxtimes \cM_2)
    \simeq \Kar(\cM_1 \hatbox{\cC} \cM_2)
\]

\end{theorem}
Before proving the theorem, we will need the following lemma.
\begin{lemma}\label{l:isom1}
The natural linear map 
\begin{equation}\label{e:isom1}
  \bigoplus_{i \in \Irr(\cC)} \ihom{\cM}{X_i}(M_1,M_2) \to \Hom_{\htr(\cM)}(M_1, M_2)
\end{equation}
is an isomorphism.
Moreover, composition of morphisms
$\Hom_{\htr(\cM)}(M_2,M_3) \tnsr \Hom_{\htr(\cM)}(M_1,M_2)
\to \Hom_{\htr(\cM)}(M_1,M_3)$
carries over to the composition rule
\begin{align}
\label{e:isom3}
\ihom{\cM}{X_i}(M_2,M_3) \tnsr \ihom{\cM}{X_j}(M_1,M_2) 
&\to
\bigoplus_{k \in \Irr(\cC)} \ihom{\cM}{X_k}(M_1,M_3) 
\\
\vphi_2 \tnsr \vphi_1
&\mapsto
\sum_{k \in \Irr(\cC)} d_k^R
\begin{tikzpicture}
\node[small_morphism] (ph1) at (0,0.4) {\tiny $\vphi_1$};
\node[small_morphism] (ph2) at (0,-0.4) {\tiny $\vphi_2$};
\node[semi_morphism={180,}] (al1) at (-0.8,0) {\tiny $\al$};
\node[semi_morphism={180,}] (al2) at (0.8,0) {\tiny $\al$};
\draw (ph1) -- (0,1) node[above] {$M_1$};
\draw (ph1) -- (ph2);
\node at (0.2,0) {\tiny $M_2$};
\draw (ph2) -- (0,-1) node[below] {$M_3$};
\draw[midarrow={0.7}] (al1) to[out=10,in=180] (ph1);
\node at (-0.4,0.5) {\tiny $i$};
\draw (ph1) to[out=0,in=170] (al2);
\draw[midarrow={0.7}] (al1) to[out=-10,in=180] (ph2);
\node at (-0.4,-0.5) {\tiny $j$};
\draw (ph2) to[out=0,in=190] (al2);
\draw[midarrow_rev] (al1) to[out=180,in=-90] (-1.5,1)
	node[above] {$X_k$};
\draw (al2) to[out=0,in=90] (1.5,-1)
	node[below] {$X_k$};
\end{tikzpicture}
\end{align}
\end{lemma}

\begin{proof}
Define a linear map  
$$
\Hom_{\htr(\cM)}(M_1, M_2) \to 
\bigoplus_{i \in \Irr(\cC)} \ihom{\cM}{X_i}(M_1,M_2)
$$
by 

\begin{equation}\label{e:isom2}
\psi\mapsto 
\sum_{i\in \Irr(\cC)} d_i^R \quad
\begin{tikzpicture}
\node[morphism] (psi) at (0,0) {$\psi$};
\draw (psi) -- +(0,1.5) node[above] {$M_1$};
\draw (psi) -- +(0,-1.5) node[below] {$M_2$};
\node[semi_morphism={90,}] (al1) at (-1,0.6) {$\al$};
\node[semi_morphism={90,}] (al2) at (1,-0.6) {$\al$};
\draw[midarrow={0.6}] (al1) to[out=-90,in=180] (psi);
\draw[midarrow={0.6}] (psi) to[out=0,in=90] (al2);
\node at (-0.7,-0.3) {$X$};
\node at (0.7,0.3) {$X$};
\draw[midarrow_rev] (al1) -- (-1,1.5) node[above] {$X_i$};
\draw[midarrow] (al2) -- (1,-1.5) node[below] {$X_i$};
\end{tikzpicture}
\end{equation}
for $\psi \in \ihom{\cM}{X}(M_1,M_2)$;
$\al$ is a sum over dual bases - see Notation~\ref{n:summation}.
\eqnref{e:isom2} is well-defined by \lemref{l:pairing_property}.
Using \lemref{l:summation}, it is easy to see that
\eqref{e:isom1} and \eqref{e:isom2} are mutually inverse,
and it is clear that they respect composition.
\end{proof}

\begin{proof}[Proof of \thmref{t:htr-center}]
Define the functor $G : \htr(\cM)\to \cZ_\cC(\cM)$ on objects by $G(M)=I(M)$, and on morphisms by 
\begin{equation}
G(\psi) = 
\sum_{i,j \Irr(\cC)} \sqrt{d_i^R}\sqrt{d_j^R}
\begin{tikzpicture}
\node[morphism] (psi) at (0,0) {$\psi$}; 
\node[semi_morphism={90,}] (L) at (-1,0) {$\al$};
\node[semi_morphism={270,}] (R) at (1,0) {$\al$};
\draw (psi)-- +(0,1.5) node[above] {$M_1$}; \draw (psi)-- +(0,-1.5) node[below] {$M_2$}; 
\draw[midarrow_rev] (L)-- +(0,1.5) node[pos=0.5,left] {$i$};
\draw[midarrow] (L) -- +(0,-1.5) node[pos=0.5,left] {$j$}; 
\draw[midarrow] (R) -- +(0,1.5) node[pos=0.5,right] {$i$};
\draw[midarrow_rev] (R)-- +(0,-1.5) node[pos=0.5,right] {$j$};
\draw[midarrow={0.6}] (L) to[out=-80,in=180] (psi);
\node at (-0.5, -0.5) {$X$};
\draw[midarrow_rev={0.6}] (R) to[out=-100,in=0] (psi);
\node at (0.5, -0.5) {$X$};
\end{tikzpicture}
\end{equation}
for $\psi \in \ihom{\cM}{X}(M_1,M_2)$;
once again see Notation~\ref{n:summation}
for definition of $\al$.

It is easy to check the following properties: 
\begin{enumerate}
\item $G$ is well-defined on morphisms (i.e. it preserves the equivalence relation):
this follows from \lemref{l:halfbrd}.

\item $G$ is dominant: any $Y\in \cZ_\cC(\cM)$ appears as a direct summand 
of  $G(M)$ for some $M\in \cM$. Namely, if $Y=(M,\ga)$,
then it appears as a direct summand of $G(M)$;
the projection to $Y$ is, up to a factor, $G(\sum d_i^R \ga_{X_i})$
(see \lemref{l:Mga_proj} in Appendix for proof;
compare \prpref{p:I_dominant}).

\item $G$ is bijective on morphisms:
by adjointness property (\thmref{t:center-adjoint}), we have 
$$
\Hom_{\cZ_\cC(\cM)}(I(M), I(M')) \cong \Hom_{\cM}(I(M), M')
  =\bigoplus_{i}\Hom_\cM(X_i \lact M \ract X_i^*,  M')
$$
and by \lemref{l:isom1}, the right hand side coincides with $\Hom_{\htr(\cM)}(M,M')$.

\end{enumerate}

This immediately implies the statement of the theorem
by the universal properties of Karoubi envelopes.

\end{proof}

By \lemref{l:ZM_sub} and \lemref{l:htr_sub},
we extend the above theorem to $\cC' \subseteq \cC$:
\begin{corollary} \label{cor:center}
Let $\cC'$ be a pivotal category
whose Karoubi envelope $\cC = \Kar(\cC')$ is multifusion.
Let $\cM$ be a $\cC$-bimodule category,
and hence naturally a $\cC'$-bimodule category.
Then we have
\[
  \Kar(\htr_{\cC'}(\cM)) \simeq \Kar(\htr_\cC(\cM)) \simeq \cZ_\cC(\cM) \simeq \cZ_{\cC'}(\cM).
\]
\end{corollary}

Note $\Kar(\cM)$ inherits a $\cC'$-bimodule structure from $\cM$.
For example, $A \lact (M,p) = (A \lact M, \id_A \lact p)$.
We compare these constructions for $\cM$ and its Karoubi envelope:
\begin{lemma}
\label{lem:M_Karoubi}
Under the same hypotheses as \corref{cor:center},
\[
  \Kar(\htr_{\cC'}(\cM)) \simeq \Kar(\htr_{\cC'}(\Kar(\cM)))
\]
In particular, if $\cM'$ is a dominant submodule category of $\cM$,
then
\[
  \Kar(\htr_{\cC'}(\cM')) \simeq \Kar(\htr_{\cC'}(\cM))
\]
\end{lemma}
\begin{proof}
The natural inclusion $\cM \to \Kar(\cM)$ 
is a full, dominant functor of $\cC'$-bimodules,
and it is easy to see that the corresponding functor
$\htr(\cM) \to \htr(\Kar(\cM))$
is also full and dominant.
It follows that the induced functor
on their Karoubi envelopes is an equivalence.

The second statement follows because $\Kar(\cM') \simeq \Kar(\cM)$.
\end{proof}

\subsection{Piecewise-Linear Topology}
\label{s:bg-PLtop}
\par \noindent

In this section, we provide some background on piecewise-linear
(henceforth abbreviated as PL) topology.
Most of the background material in this section is obtained from
\ocite{hudson};
we provide a lightning recap of basic definitions
and quote results without proof.

As most of the PL topological arguments are in dimensions 4
and below, the reader who is unfamiliar with PL topology
may quite safely replace all things PL with their smooth counterparts,
as it is known that 
every PL manifold of dimension less than or equal to 6 admits
a unique smooth structure up to diffeomorphism
(see \ocite{hirschmazur}, \ocite{kervairemilnor}).

Most objects and maps discussed in this section are PL,
and are implicitly assumed to be such when no qualifiers are added;
when we want to emphasize that a map is merely continuous,
we say it is $C^0$.
Technically the term \emph{affine linear} should be used instead of linear,
but we simply say linear for brevity.

%

A \emph{convex linear cell}, or simply a \emph{cell},
in $\RR^n$ is the convex hull of a finite set of points in $\RR^n$;
the \emph{dimension} of the cell is the dimension of the affine
subspace spanned by the cell.
A \emph{face} of a $k$-dimensional cell is the intersection of
the cell with an affine plane of dimension $k-1$ which does not
meet the interior of the cell.

A \emph{polyhedron} in $\RR^n$ is the union of a finite collection
of cells in $\RR^n$.

A continuous map $f : P \to Q$ from one polyhedron $P\subset \RR^p$
to another polyhedron $Q \subset \RR^q$
is \emph{PL} if the graph of $f$,
$\Gamma_f = \{(x,f(x)|x\in P\} \subset \RR^{p+q}$,
is a polyhedron.
The cells that make up $\Gamma_f$ project to cells in $\RR^p$
to make up $P$, such that the restriction of $f$ to these cells
is linear, so an equivalent definition of $f$ being PL is if
$P$ can be presented as a union of cells such that $f$
is linear on each cell.
The composition of PL maps is PL.

The \emph{cone} over a polyhedron $P \subset \RR^p$,
denoted $\text{cone}(P)$,
is the polyhedron $P * e_{p+1} \subset \RR^{p+1}$,
where $e_{p+1} = (0,\ldots,0,1)$.
For a PL map $f: P \to Q$, the $\text{cone}(f)$ is naturally defined
and is also PL.

A \emph{coordinate map} of a topological space $X$
is a $C^0$-embedding $f: P \to X$ of a polyhedron $P$ in $X$;
we write $(f,P)$ to denote such a map.
Two maps are \emph{compatible} if their overlap is a coordinate map;
more precisely, $(f,P)$ and $(g,Q)$ are compatible if either
their images don't intersect or there exists a third coordinate map
$(h,R)$ such that $h(R) = f(P) \cap g(Q)$ and
$f^\inv h$, $g^\inv h$ are PL.

A \emph{PL structure} on $X$ is a family $\mathcal{F}$ of coordinate maps satisfying:
\begin{itemize}
\item Any two elements are compatible
\item For every $x \in X$ has a coordinate neighborhood,
	i.e. there exists $(f,P) \in \mathcal{F}$ such that $f(P)$ is a neighborhood
	of $x$.
\item $\mathcal{F}$ is maximal.
\end{itemize}
If $\mathcal{F}$ is not necessarily maximal, it is a \emph{basis for a PL
structure} on $X$.

A \emph{convex linear cell complex} $K$ in $\RR^n$ is a finite collection
of cells that is closed under taking faces of cells and taking intersections
of pairs of cells that meet.
We write $\sigma \leq \tau$ for $\sigma, \tau \in K$
when there is a sequence
$\sigma_0 = \sigma,\sigma_1,\ldots,\sigma_k = \tau$, $k\geq 0$,
such that $\sigma_i$ is a face of $\sigma_{i+1}$;
note by the conditions of being a complex, if $\sigma \subset \tau$
as sets, $\sigma = \sigma \cap \tau$ is a face of $\tau$,
so the relations `$\leq$' and `$\subseteq$' are equivalent for complexes.
We denote by $|K| \subset \RR^n$ the union of cells.

An $n$-simplex in $\RR^N$ is the convex hull of $(n+1)$ linearly independent
points; a \emph{simplicial complex} is simply a cell complex whose cells
are simplices.
A \emph{triangulation} of a topological space $X$ is a simplicial complex $K$
together with an identification of $X$ with $|K|$ as $C^0$-spaces.
A triangulated space is naturally a PL space.

By \ocite{hudson}*{Chapter 3.2},
any PL space $X$ (i.e. a space with a PL structure)
admits a triangulation $|K| \to X$
that is a PL map,
i.e. the triangulation gives the same PL structure on $X$.

A cell complex $K$ is a subdivision of another $L$ if $|K|=|L|$ and
every cell of $K$ is a subset of some cell of $L$.



Here $I = [0,1]$.

An \emph{isotopy} $F$ between embeddings $f,g:M \to Q$
is a map $F: M \times I \to Q$ such that,
writing $F_t = F(-,t)$,
$F_0 = f, F_1 = g$ and $F_t$ is an embedding for each $t$.
(Warning: the linear maps $\id_\RR$ and $x \mapsto 2x$
are indeed PL-isotopic, but $F_t(x) = tx$ is \emph{not} a PL-isotopy,
as $F$ is not linear map as a map $\RR \times I \to \RR$.)
An equivalent formulation of $F$ is as a \emph{level-preserving} embedding
$\overline{F} : M \times I \to Q \times I$;
we will write $F$ instead of $\overline{F}$ when it does not lead to confusion.
We say $f$ and $g$ are \emph{ambient isotopic}
if there exists an \emph{ambient isotopy} $H : Q \times I \to Q$
such that $g = H_1 \circ f$;
we say $H$ \emph{throws} $f(A)$ onto $g(A)$,
where $A \subseteq M$ is some subset of $M$.
If $F$ is an isotopy from $f$ to $g$,
and $H|_M = F$, then we say that $H$ \emph{carries} $F$.

The \emph{support} of a self-homeomorphism $h:Q \to Q$
is $\supp(h) := \{x \in Q| h(x) \neq x\}$;
we say $h$ is \emph{supported by} $X \subseteq Q$
if $\supp(h) \subseteq X$.
Similarly, an ambient isotopy is \emph{supported by}
$X \subseteq Q$ if the isotopy is fixed away from $X$.

An \emph{$m$-ball} is a polyhedron which is PL-homeomorphic
to an $m$-simplex.
A \emph{PL manifold} of dimension $m$ is a polyhedron in which
every point has a (closed) neighborhood which is
a $m$-ball.
We denote an $m$-ball by $B^m$ or $\DD^m$.

A \emph{move} on a $q$-manifold $Q$ is a homeomorphism $h$ that is
supported by a $q$-ball $B$ in $Q$; we call $B$ a \emph{supporting ball}.
A move $h$ is a \emph{proper move} if either
$h$ is fixed on $\del Q$ or $B$ meets $\del Q$
in at most one face of $B$.
By (the proof of) \ocite{hudson}*{Lemma 6.1},
a proper move supported by a ball $B$
is isotopic to the identity via an isotopy supported by $B$.

Two embeddings $f,g : M \to Q$ of manifolds are
\emph{isotopic by moves} if they are related by a finite sequence of moves,
i.e. $g = h_k \circ \cdots \circ h_1 \circ f$.
By \ocite{hudson}*{Theorem 6.2},
if $H$ is an ambient isotopy of a compact manifold $Q$,
then $H_1$ is isotopic by moves to the identity.
This result can be refined in several ways:

\begin{theorem}[\ocite{hudson}*{Theorem 6.2.3}]
Given an open cover $\{U_\al\}$,
if $f,g : M \to Q$ are ambient isotopic embeddings,
where $M$ is compact,
then $f,g$ are isotopic by moves
such that each move is supported on a ball contained in some $U_\al$.
Moreover, if the ambient isotopy carrying $f$ to $g$
is fixed on the boundary,
then each move can be chosen to be fixed on the boundary.
\label{t:isotopy-open-cover-boundary}
\end{theorem}

A \emph{manifold pair} $(Q,M)$ is a pair of manifolds
with $M$ embedded in $Q$.
We say $(Q,M)$ is \emph{proper} if $M \cap \del Q = \del M$,
and is \emph{locally unknotted} if for any $x\in M$,
there exists a $Q$-neighborhood $V$ of $x$ such that
$(V, V\cap M)$ is an unknotted ball pair
(i.e. homeomorphic to the standard pair $(I^q, I^m \times 0)$;
note locally unknotted implies proper).
Furthermore, an isotopy $F : M \times I \to Q \times I$
is \emph{locally unknotted} if, for all $0\leq s \leq t \leq 1$,
the proper manifold pair
$(Q \times [s,t], F(M \times [s,t]))$
is locally unknotted.

\begin{theorem}[\ocite{hudson}*{Theorem 6.12},
Isotopy Extension Theorem]
Let $F : M \times I \to Q \times I$, $M$ compact,
be a proper locally unknotted isotopy.
Then there exists an ambient isotopy $H$ of $Q$
that carries $F$, i.e.
\[
F_t = H_t \circ F_0 \times \id_I
\]
Furthermore, if $F$ is fixed on $\del M$,
then we may choose $H$ to be fixed on $\del Q$.
\label{t:isotopy-ext}
\end{theorem}

\begin{remark}
We will be using this result particularly for extending an
isotopy of a ribbon graph to the ambient manifold
(see \secref{s:cy-ribbongraphs}, \secref{s:skein_modules}).
Strictly speaking, \thmref{t:isotopy-ext} does not apply there,
as the graphs, as surfaces, have boundary in the interior
(thus violating local-unknottedness).
However, one can make a double of the ribbon graph,
and glue it to the original graph,
making it a closed surface except at the boundary.
Then \thmref{t:isotopy-ext} applies to this double,
and thus to the original graph.
\label{r:isotopy-ext-ribbongraph}
\end{remark}

\begin{theorem}[\ocite{hudson}*{Theorem 6.11}]
A manifold with boundary admits a collar neighborhood,
unique up to ambient isotopy fixing the boundary.
\label{t:collar-exist-unique}
\end{theorem}

\begin{corollary}[\ocite{hudson}*{Corollary 6.12}]
Let $(Q,M)$ be a locally unknotted compact proper manifold pair.
Any boundary collar on $M$ extends to a compatible boundary collar
on $Q$.
\label{c:collar-extend}
\end{corollary}


The following is adapted from \ocite{PL-az}*{Definition 2}:

\begin{definition}
Let $P,Q,M$ be manifolds with $P \subset Q$.
Let $f: M \to Q$ be a map, and $x \in \Int(M)$
such that $f(x) \in P$.
We say $f$ is \emph{transversal to $P$ at $x$}
if there is a commutative diagram
\[
\begin{tikzcd}
D^{q-p} \times D^{m+p-q}, 0 \times 0
\ar[r,"\id \times k"]
\ar[d,"\phi"]
& D^{q-p} \times D^p, 0 \times 0
\ar[d,"\psi"]
& 0 \times D^p, 0 \times 0
\ar[l,"0 \times \id"]
\ar[d,"\psi"]
\\
M,x \ar[r,"f"]
& Q, f(x)
& P, f(x) \ar[l]
\end{tikzcd}
\]
where $\phi,\psi$ are embeddings onto neighborhoods of $x,f(x)$,
respectively,
and $k$ is some map $D^{m+p-q} \to D^p$.
We say $f$ is \emph{transversal to $P$} if it is transversal
to $P$ at all such points $x$.
\label{d:transversal-map}
\end{definition}

Note that this definition was originally meant for
closed manifolds, but we will be using it for
$M = $surface with boundary
(more specifically, ribbon graphs in \secref{s:cy-ribbongraphs}).
An obvious deficiency in this definition is that
$M$ may intersect $P$ only at its boundary,
but the definition says nothing about this intersection.
We will not go into any detail on this,
as our use case is very simple,
and will be more interested in a stronger notion of transversality
for ribbon graphs (see \defref{d:ribbon-transverse}).


\begin{proposition}
Let $f : M \to \RR$ be a proper map.
We say $s \in \RR$ is a \emph{regular value} if
$f$ is transversal to $s$ (as a 0-dim submanifold of $\RR$).
We say $s$ is a \emph{critical value} if it is not a regular value.

The set of regular values is dense in $\RR$.
\label{p:regular-level-set}
\end{proposition}

\begin{proof}
By definition of PL maps,
$M$ can be presented as a union of cells
such that the restriction of $f$ to each cell is linear.
Let $P$ be the collection of cells on which $f$ is constant
(in particular, $P$ contains all 0-cells).
Then it is easy to see that any $s \nin f(P)$ is a regular value.

By propreness of $f$, the preimage $f^\inv([a,b])$ of any
compact interval only meets finitely many cells in $M$,
thus $f^\inv(P) \cap [a,b]$ is a finite set,
and we are done.
\end{proof}

\begin{remark}
\label{r:orientations}
\textbf{Orientations and vectors.}
It is somewhat tricky to define the tangent space of a point
in a PL manifold, as the linear structure is broken,
thus it is not immediately clear how to adapt the usual
definition of orientation using an ordered tuple of vectors.
Strictly speaking, one should use local homology groups to define
orientations, but this is too cumbersome for our setup.
Instead, we first pick a triangulation of the manifold,
and since each simplex has a well-defined linear structure,
the orientation at a point can be defined in the usual way;
if a point belongs to multiple simplices,
an orientation in one simplex naturally defines an orientation
in another simplex.
If we present the PL structure of a manifold $M$ as a collection
of charts $f: U \to M$, such that the transition maps are PL,
then an orientation at a point can again be defined the usual way
in $\RR^N$.
If the point belongs to another chart,
it should define a corresponding orientation in the other chart
as follows: propogate the orientation to the rest of the first chart,
then, since the transition map must be linear on some open subset,
the transition map transfers the propogated orientation to the other chart.

In short, we will define orientations at points
by describing an ordered tuple of vectors at that point,
implicitly assuming that some linear structure is chosen as reference.

We will also be discussing co-orientations of surfaces in a solid.
As we only deal with locally unknotted manifold pairs,
a co-orientation at an interior point of the surface
can be defined as a choice of one component of $B^3 \backslash B^2$.

Finally, the convention on outward orientation of the boundary
of an oriented manifold is: in some coordinate chart,
at a point on the boundary,
$o = (v_1,\ldots,v_{k-1})$ defines the outward orientation if
$(\vec{n}, v_1,\ldots,v_{k-1})$ defines the orientation
of the manifold, where $\vec{n}$ is an outward-pointing vector.
\end{remark}

Next we discuss cobordisms and handle decompositions.
We will only apply these results to 4-dimensional cobordisms,
but we present them in general for clarity;
the reader may assume $n=4$, $W$ is a 4-manifold,
$M$ is a 3-manifold, and $N$ is a 2-manifold.

\begin{definition}
A \emph{cobordism from $M$ to $M'$} is an oriented manifold $W$
with a partition of its boundary into disjoint manifolds
$\del W = \ov{M} \sqcup M'$
(so $M'$ has the outward orientation, and $M$ the inward orientation,
with respect to $W$).
We will denote a cobordism by $W: M \to M'$.
\label{d:cobordism}
\end{definition}

We will be considering 4-manifolds with corners
as part of our extended TQFT.
In particular, we will have to consider cobordisms
between 3-manifolds with boundary.
The following is a variation of a ``cobordism with boundary''
(as in \ocite{PL-rourke-sanderson}*{Chapter 6});
Yetter calls them ``cobordism with corners''
\ocite{yetter-hopf}.

\begin{definition}
A \emph{cornered cobordism over $N$ from $M$ to $M'$}
is an oriented manifold $W$
with $\del W = \ov{M} \cup_N M'$,
$\ov{M} \cap M' = N = \del M = \del M'$.

We may regard a cobordism as a cornered cobordism
over the empty manifold.
\label{d:cornered-cobordism}
\end{definition}

\begin{definition}
Let $M$ be a manifold (with boundary $N$).
The \emph{identity cobordisms on $M$},
denoted $\id_M$,
is the cornered cobordism
$W = M \times I : M \to_{N \times \{0\}} M$.
\label{d:cornered-cobordism-identity}
\end{definition}

\begin{definition}
Given cornered cobordisms $W : M \to_N M'$,
$W : M' \to_N M''$,
their \emph{composition},
denoted $W' \circ W$, is the cornered cobordism
\[
W' \circ W = W' \cup_{M'} W
	: M \to_N M''
\]
\label{d:cornered-cobordism-composition}
\end{definition}

We also consider a slightly variation of composition,
but first:

\begin{definition}
Let $W: M_0 \to_N M_1$ be a cornered cobordism,
and let $M : N \to N'$ be a cobordism.
The cornered cobordism
\emph{obtained from $W$ by extending along $M$},
denoted $W_{M: N \to N'}$,
is a cornered cobordism $M_0 \cup_N M \to_{N'} M_1 \cup_N M$
whose underlying manifold is obtained by gluing $W$
to the identity cobordism on $M_1 \cup_N M$
(or equivalently, gluing $W$ to the identity cobordism on $M_0 \cup_N M$).
\label{d:cornered-cobordism-extend}
\end{definition}

\begin{definition}
Let $W,W'$ be cornered cobordisms
\begin{align*}
W &: M_0 \to_N M_1
\\
W' &: M_0' \to_{N'} M_1'
\end{align*}
Suppose $M_1 \subset M_0'$ is a submanifold
that does not meet $N'$.
Then $M_0' \backslash M_1$ is a cobordism
$N \to N'$.
The \emph{extended composition}
of $W$ and $W'$ is
\[
W' \wdtld{\circ} W := W' \circ W_{M_0' \backslash M_1 : N \to N'}
	: M_0 \to_{N'} M_1' \cup (M_0' \backslash M_1)
\]

Similarly, if $M_0' \subset M_1$ is a submanifold
that does not meet $N$,
then $M_1 \backslash M_0'$ is a cobordism
$N' \to N$,
and we define
\[
W' \wdtld{\circ} W := W'_{M_1 \backslash M_0' : N' \to N} \circ W
	: M_0  \cup (M_1 \backslash M_0') \to_N M_1'
\]
\label{d:composition-extended}
\end{definition}

The simplest cobordisms, besides the
\emph{identity cobordisms},
arise from attaching balls to $M$
in a particular manner, known as \emph{handles}.
Here we present handles as cornered cobordisms:

\begin{definition}
Fix a dimension $n$,
and let $0 \leq k \leq n$.
The \emph{$n$-dimensional $k$-handle}, denoted $\cH_k^n$,
or simply $\cH_k$,
is the cornered cobordism
\[
\cH_k = B^k \times B^{n-k} : \del B^k \times B^{n-k}
	\to_{\del B^k \times \del B^{n-k}} B^k \times \del B^{n-k}
\]
Note that $\cH_k$ and $\cH_{n-k}$ are the same as manifolds with corners.

Some useful terminology associated to handles are as follows:
\begin{itemize}
\item \emph{core}: $B^k \times \{0\}$; the handle $\cH_k$ should be
	regarded as a thickening of the core,
\item \emph{co-core}: $\{0\} \times B^{n-k}$,
\item \emph{attaching sphere}: $\del B^k \times \{0\}$;
	this will make more sense in the context of attaching handles,
\item \emph{attaching region}: $\del B^k \times B^{n-k}$;
	again, this will make more sense in the context of attaching handles,
\item \emph{belt sphere}: $\{0\} \times \del B^{n-k}$.
\end{itemize}
\label{d:handle}
\end{definition}

\begin{definition}
Let $W : M_0 \to_N M_1$ be a cornered cobordism.
We say a cornered cobordisms $W' : M_0 \to_N M_1'$
is \emph{obtained from $W$ by attaching a $k$-handle}
if $W' = W \cup_f \cH_k$,
where $f : \ov{\del B^k \times B^{n-k}} \hookrightarrow M_1'$
is an embedding, called the \emph{attaching map},
such that its image is disjoint from $N$.
We call the image of the attaching sphere/region
the \emph{attaching sphere/region of $f$}.
\label{d:handle-attachment}
\end{definition}

\begin{definition}
An \emph{elementary cornered cobordism of index $k$}
is a cobordism obtained from an identity cobordism
by attaching a $k$-handle.
In other words, it is an extension of a $k$-handle
by some cobordism.
\label{d:elementary-cobordism}
\end{definition}

\begin{lemma}
Given an elementary cornered cobordism
$W : M \to_N M'$ of index $k$,
the \emph{dual elementary cornered cobordism}
is obtained by reversing the direction of $W$ by
reversing the orientations of $M,M'$,
so that we have
\[
W : \ov{M'} \to_{\ov{N}} \ov{M}
\]
This is an elementary cornered cobordism of index $n-k$.
\label{l:dual-elementary-cobord}
\end{lemma}
\begin{proof}
It is easy to check that if
$f$ is the attaching map of the $k$-handle to $M$,
and $f'$ is the inclusion map of the belt sphere
into $M'$, then
\[
W = \cH_k \cup_f \id_M = \id_{M'} \cup_{M'} \cH_k \cup_f \id_M
= \id_{M'} \cup_{f'} \cH_k
\]
\end{proof}

\begin{definition}
A \emph{handle decomposition} of a cornered cobordism
$W : M \to_N M'$
is an identification of $W$
with a composition of elementary cornered cobordisms.
\label{d:handle-decomposition}
\end{definition}

\begin{definition}
Given a handle decomposition
$W = W_l \circ \cdots \circ W_1 : M \to_N M'$,
the \emph{dual handle decomposition}
is the handle decomposition
$W = W_1^* \circ \cdots \circ W_l^* : \ov{M'} \to_{\ov{N}} \ov{M}$,
where $W_i^*$ is the dual elementary cobordism
from \lemref{l:dual-elementary-cobord}.
\label{d:dual-handle-decomp}
\end{definition}

\begin{proposition}
Handle decompositions exist for any cornered cobordism.
\label{p:handle-decomposition-exist}
\end{proposition}
\begin{proof}
In broad strokes,
for a cornered cobordism $W : M \to_N M'$,
take a triangulation of $W$,
then each $k$-simplex not in $M$
defines a $k$-handle
(with the $k$-simplex as its core).
Details can be found in \ocite{bryant-PL}*{Section 6}.
\end{proof}

\begin{definition}
\label{d:handle-decomp-modification}
Let $W_i, W_{i+1}$ be two successive elementary cornered cobordisms
in a handle decomposition.
Let $W_i = \cH_k \wdtld{\circ} \id_{M_i} : M_i \to_{N_i} M_{i+1}$,
and $W_{i+1} = \id_{M_{i+1}} \wdtld{\circ} \cH_{k'}
	: M_{i+1} \to_{N_{i+1}} M_{i+2}$.
Let $f$ be the attaching map for $\cH_{k'}$.
There are several ways to modify a handle decomposition:
\begin{itemize}
\item reordering: if $f$ misses $\cH_k$, we may swap $W_i$ and $W_{i+1}$.
	This can be arranged if $k' \leq k$ by general position and transversality
	arguments
	(isotope $f$ to be transversal to the co-core of $\cH_k$,
	so by dimensional considerations, the attaching sphere for $f$
	is disjoint from it;
	then apply an isotopy to ``push away'' from the co-core).
\item handle slide, or handle addition/subtraction:
	when $k=k'$, arrange so $f$ misses $\cH_k$ as above,
	then isotope $f$ by dragging a part of the attaching sphere
	through the core of $\cH_k$,
	so that, if $S,S'$ are the attaching spheres
	of $\cH_k$ and $\cH_{k'}$,
	then the new attaching sphere for $\cH_{k'}$
	is $S \# S'$.
\item handle pair cancellation/creation: if $k' = k+1$
	and the attaching sphere of $f$ intersects the co-core of $\cH_k$
	transversally,
	then $W_{i+1} \circ W_i$ is simply an identity cobordism,
	and we may eliminate this pair from the handle decomposition.
	(Pair creation is just the inverse operation of adding such a
	canceling pair).
\end{itemize}
\end{definition}

Once again, we refer the reader to the reference texts for details.

\begin{proposition}
Any two handle decompositions of a given cornered cobordism
of dimension $\leq 6$
are related by a sequence of modifications
as in \defref{d:handle-decomp-modification}.
\label{p:handle-decomp-relate}
\end{proposition}

\begin{proof}
While we do rely on this proposition for a proof of
\prpref{p:zcysk-corner-indep},
that result is also proved by a different method in
\thmref{t:sk-equiv-4mfld}.
Thus, we merely outline a sketch of the proof of this proposition here,
leaving the reader to details.

This is a well-known result in the smooth category,
and is one of the main results of Cerf theory.
The general outline of the proof goes as follows.
A handle decomposition determines a Morse function on the cobordism
$W \to \RR$, i.e. a smooth function whose critical points
are all non-degenerate.
The two handle decompositions yield Morse functions $f_0,f_1$.
For a generic smooth family $\{f_t\}$ interpolating them,
$f_t$ is either a Morse function or has ``birth-death'' degenerate
critical points.
Between $f_t$ of the birth-death type, the $f_t$'s lead to the same
handle decomposition,
and each passing of a birth-death type,
the handle decomposition undergoes a modification
of the type discussed in \defref{d:handle-decomp-modification}.
See for example \ocite{milnor-cobord} for details.


Now we transfer the smooth result to the PL category.
Apply the unique smoothing of $W$,
and embed it smoothly in some Euclidean space $\iota : W \to E^k$.
Add an extra coordinate $z$ to the Euclidean space,
and modify the embedding to an isotopy
$\iota_t = (\iota,f_t) : W \times I \to E^{k+1}$
such that the $z$ coordinate is $f_t$.
Then apply a PL approximation by choosing a dense (in the colloquial sense)
set of points on $\iota_t(W \times I)$ as vertices
and triangulate accordingly (see \ocite{cairns}).
\end{proof}



\newpage
\section{Crane-Yetter Invariant for PLCW Decompositions}
\label{s:cy-plcw}
\vspace{0.8in}

Recall that the Crane-Yetter invariant for 4-manifolds,
as defined in \ocite{CKY}, is an invariant constructed
using a triangulation on the 4-manifold.
We give a definition using a more general cellular decomposition,
here called a ``PLCW decomposition'',
and prove that they agree.
The methods and exposition presented here closely mirror \ocite{balsam-kirillov},
where they define the Turaev-Viro invariant
using a polytope decomposition of a 3-manifold.
We also draw heavily from \ocite{kirillov-PL}, where PLCW decomposition
is defined.

In what follows, the word ``manifold'' denotes a compact,
oriented, piecewise-linear (PL) manifold;
unless otherwise noted, we assume that it has no boundary.
Similarly, all maps are assumed to be piecewise-linear
unless otherwise specified.


\subsection{PLCW Decompositions}
\label{s:cy-top-plcw}
\par \noindent


Most of the definitions and results here are taken from
\ocite{kirillov-PL};
the main new additions are on ``restricted elementary moves'',
``collapsible pairs'', and 
\thmref{t:elementary-boundary}
on relating relative PLCW decompositions via elementary moves.

Recall that a cellular decomposition of a topological manifold $M$ is a collection of
inclusion maps $\DD^d \to M$, where $\DD^d$ is the (open) $d$-dimensional ball,
satisfying certain conditions.
Equivalently, we can replace d-dimensional balls with $d$-dimensional cubes
$(0, 1)^d$.


%
%

We denote the $n$-dimensional (PL) ball by
$B^n = [-1,1]^n$ for $n>0$, and $B^0 = \{*\}$.


\begin{definition}[\ocite{kirillov-PL}*{Definition 3.3}]
Let $M$ be a PL manifold possibly with boundary.
A \emph{generalized $n$-cell} in $M$ is a subset $C \subseteq M$
together with a map $\vphi : B^n \to C \subseteq M$,
such that the restriction of $\vphi$ to the interior
is a homeomorphism onto its image;
$\vphi$ is called the \emph{characteristic map}.
The \emph{interior} $\intr{C}$ of the generalized $n$-cell $(C,\vphi)$
is the image $\vphi((-1,1)^n) \subseteq M$.
\end{definition}

\begin{figure}
\begin{tikzpicture}
\draw (0,0) circle (0.5cm);
\draw (0,0) circle (1.5cm);
\node[dotnode] at (0,0.5) {};
\node[dotnode] at (0,1.5) {};
\draw (0,0.5) to (0,1.5);
\node at (0,-1) {$\intr{C}$};
\end{tikzpicture}
\caption{Example of generalized 2-cell \ocite{kirillov-PL}*{Figure 1 (a)}}
\label{f:generalized-cell}
\end{figure}

By abuse of notation, we sometimes refer to $C$ as the
generalized $n$-cell,
which can be justified by the following lemma:

\begin{proposition}[\ocite{kirillov-PL}*{Theorem 3.4}]
The characteristic map of a generalized $n$-cell is
unique up to a homeomorphism of $B^n$;
more precisely, given two generalized $n$-cells
$(C, \vphi)$, $(C, \vphi')$
with the same image,
there exists a PL homeomorphism $\psi : B^n \to B^n$
such that $\vphi' = \vphi \circ \psi$.
\end{proposition}

Thus, we will consider generalized cells up to such equivalences,
although we will often pick a characteristic map arbitrarily
for concreteness.

\begin{definition}[\ocite{kirillov-PL}*{Definition 4.1}]
A \emph{generalized cell complex structure}
on a PL manifold $M$ is a finite collection $\{(C, \vphi_C)\}$ of
generalized cells in $M$ such that
\begin{itemize}
\item $\intr{C} \cap \intr{C'} = \emptyset$ for distinct cells,
\item $\bigcup C = M$.
\end{itemize}

The $k$-skeleton $\sk^k(M)$ is the union of all $k$-cells.
\end{definition}

\begin{definition}[\ocite{kirillov-PL}*{Definition 4.3}]
Let $M,M'$ be manifolds with generalized cell complex structures
$\cM, \cM'$.
A map $f : M \to M'$ is a \emph{cellular map}
if for each cell $(C,\vphi_C) \in \cM$,
the composition $f \circ \vphi_C$ is a characteristic map
of some cell $C'$ of $\cM'$.
\end{definition}

\begin{definition}[\ocite{kirillov-PL}*{Definition 5.1}]
A generalized cell complex structure on an $n$-manifold $M$
is a \emph{PLCW decomposition} or a \emph{PLCW structure} if
for each $k$-cell $(C,\vphi)$ of $M$ with $k>0$,
there exists a generalized cell complex structure $\cB_\vphi$ on $\del B^k$
such that the restriction $\vphi|_{\del B^k}$ is a
cellular map.

We call a manifold $M$ endowed with a PLCW decomposition $\cM$
a \emph{PLCW combinatorial manifold},
or simply \emph{PLCW manifold};
we often refer to the PLCW manifold simply by $\cM$.

We say that two PLCW decompositions on $M$ are equivalent
if there is a homeomorphism from $M$ to itself
that identifies them.
\end{definition}

Note that by induction on $k$,
the generalized cell complex structure $\cB_\vphi$ on $\del B^k$
is itself a PLCW decomposition,
hence this definition coincides with
\ocite{kirillov-PL}*{Definition 5.1};
we call it the \emph{induced PLCW decomposition}.
We also call $\cB_\vphi \cup B^k$
the induced PLCW decomposition on the $k$-cell $\vphi(B^k)$.

\begin{definition}
A PLCW decomposition $\cB$ of $B^k$ with exactly one $k$-cell
is called \emph{cell-like}.
\end{definition}
In particular, the induced PLCW decomposition by a characteristic map
is cell-like.

One convenient feature of PLCW decompositions is that products are
easy to define:

\begin{definition}
Let $\cM,\cM'$ be two PLCW manifolds.
For cells $(C,\vphi : C \to M) , (C', \vphi' : C' \to M')$,
we define their product cell as
$(C \times C', \vphi \times \vphi' : C \times C' \to M \times M')$.
It is easy to check that the product cells define
a PLCW structure on $M \times M'$;
we denote this structure by $\cM \times \cM'$.
\label{d:plcw-product}
\end{definition}

A PLCW decomposition determines the PL structure
of the underlying manifold.
A natural question to ask is whether there is a simple
set of modifications that can turn any PLCW decomposition
into another (like the Pachner (a.k.a. bistellar) moves
for triangulations).
The answer is in the affirmative, given by the ``elementary
subdivisions'' described later.

\begin{definition}[\ocite{kirillov-PL}*{Definition 6.1}]
Let $\cM, \cM'$ be PLCW decompositions of a manifold $M$.
We say that $\cM$ is \emph{finer} than $\cM'$,
or is a \emph{subdivision} of $\cM'$, if every cell
$C \in \cM$ is contained in a cell $C' \in \cM'$;
we say $\cM'$ is \emph{coarser} than $\cM$.
\end{definition}

A subdivision of the induced PLCW decomposition
on a cell yields a subdivision of the PLCW decomposition
of the entire manifold.
In particular, we consider the simplest possible subdivision:

\begin{definition}
Let $\cB$ be a cell-like PLCW decomposition on $B^k$.
An \emph{elementary subdivision} of $\cB$ is a PLCW decomposition
on $B^k$ that is the same as $\cB$ on the boundary,
but has exactly two $k$-cells, separated by exactly one $(k-1)$-cell;
if $F$ is the new $(k-1)$-cell, we say the elementary subdivision
is \emph{along $F$}.
\label{d:elem-subdiv-cell}
\end{definition}

Note that our definition differs slightly from
\ocite{kirillov-PL}*{Section 7}.

\begin{definition}
An elementary subdivision of a cell-like ball along $F$ is \emph{restricted}
if $F$ is embedded.
\label{d:elem-subdiv-cell-restr}
\end{definition}

Clearly, such a subdivision on a cell in some PLCW decomposition
induces a subdivision of the PLCW decomposition itself:

\begin{definition}[\ocite{kirillov-PL}*{Lemma 7.1, Definition 7.2}]
Let $\cM$ be a PLCW decomposition of an $n$-manifold $M$.
Let $C \in \cM$ be an $k$-cell, $k>0$.
Consider a
elementary subdivision of $C$ along some $(k-1)$-cell $F$,
say, $C \to \del C \cup \{C_1,C_2,F\}$.
Then this defines an \emph{induced elementary subdivision
of $\cM$ along $F$},
\[
\cM \to_e \cM' := (\cM \backslash C) \cup \{C_1,C_2,F\}
\]

We say that $\cM'$ is obtained from $\cM$ by performing a
\emph{elementary move} if one is a
elementary subdivision of the other, and denote it by
\[
\cM \leftrightarrow_e \cM'
\]

\label{d:elem-subdiv}
\end{definition}

\begin{lemma}
Bistellar (Pachner) moves can be obtained by sequence of elementary moves.
(see \ocite{kirillov-PL}*{Section 6, 8}).
\label{l:bistellar-elem}
\end{lemma}

It is not so easy to ``undo'' an elementary subdivision,
e.g. after performing an elementary subdivisions along some $k$-cell,
one might perform another subdivision that attaches a $(k+1)$-cell
along the new $k$-cell, then one can no longer ``undo''
the first elementary subdivision.
The following operation is in some sense a way to ``undo''
an elementary subdivision:

\begin{definition}
Let $C$ be a $k$-cell in some PLCW manifold $\cM$,
and consider a $(k-1)$-cell $F$ in its boundary,
$k\geq 1$.
Suppose $F$ is not identified with any other boundary $(k-1)$-cell of $C$;
that is, $C$ only meets the interior of $F$ once.
Moreover, suppose the characteristic map of $F$ is an embedding
(not just on its interior).
Then $\del F$ separates $\del C$ into two disks,
$\del C = F \cup_{\del F} D$.
We say the pair $(C,F)$ is \emph{collapsible}.

Arrange the cells that have $C$ in its boundary in increasing dimension,
$C_1,C_2,\ldots,C_m$ (beginning with dimension $k+1$);
cells that have $C$ in its boundary multiple times
should appear once for each time.
Let $X = \{C\}$; we iteratively add elements to $X$,
each element will be a subspace of some $C_i$,
with one such element for each $C_i$.
From $i=1$ to $m$, consider the union $Y$
of elements $y \in X$ that are in the boundary of $C_i$,
then consider a point $p \in C_i$ that is ``close'' to $Y$
(say in some collar neighborhood of $\del C_i$),
and add the cone of $p$ over $Y$, denoted $x_i$, to $X$.
Thus we have $|X| = m+1$ elements.

Let $f: F \to D$ be some homeomorphism that fixes $\del F$.
Remove the interior of $F$ from $F$,
and remove the interior of $x_i$ from $C_i$.
We patch them up by apply $f$ to glue $F$ to $D$,
then the cone over $f$ in $C_1$ will glue the cone over $F$
to the cone over $D$, and so on,
iteratively performing cones over maps in the boundary.
Now we have a new PLCW decomposition with all the same cells but
$F$ and $C$.

We call this operation \emph{collapsing $(C,F)$}.
(Diagram below depicts $k=1$.)

\begin{equation}
\begin{tikzpicture}
\node[dotnode] (a) at (0,0) {};
\node[dotnode] (b) at (0,2) {};
\node[dotnode] (c) at (0,1) {};
\draw (a) -- (b);
\draw (a)
	.. controls +(170:2cm) and +(-90:0.2cm) .. (-2,0.9)
	.. controls +(90:0.2cm) and +(180:2cm) .. (0,1.4)
	.. controls +(0:2cm) and +(90:0.2cm) .. (2,0.9)
	.. controls +(-90:0.2cm) and +(10:2cm) .. (a);
\draw (b)
	.. controls +(170:2cm) and +(-90:0.2cm) .. (-2,2.9)
	.. controls +(90:0.2cm) and +(180:2cm) .. (0,3.4)
	.. controls +(0:2cm) and +(90:0.2cm) .. (2,2.9)
	.. controls +(-90:0.2cm) and +(10:2cm) .. (b);
\draw (a) -- +(-150:1.6cm);
\draw (b) -- +(-150:1.6cm);
\node[dotnode] (d) at (0.5,2.1) {};
\draw (c) -- (d);
\node at (0.2,0.5) {\smallerer $C$};
\node at (0.2,1) {\smallerer $F$};
\begin{scope}[shift={(5,0)}]
\node[dotnode] (a) at (0,0) {};
\node[dotnode] (b) at (0,2) {};
\node[dotnode] (c) at (0,1) {};
\draw (a) -- (b);
\draw (a)
	.. controls +(170:2cm) and +(-90:0.2cm) .. (-2,0.9)
	.. controls +(90:0.2cm) and +(180:2cm) .. (0,1.4)
	.. controls +(0:2cm) and +(90:0.2cm) .. (2,0.9)
	.. controls +(-90:0.2cm) and +(10:2cm) .. (a);
\draw (b)
	.. controls +(170:2cm) and +(-90:0.2cm) .. (-2,2.9)
	.. controls +(90:0.2cm) and +(180:2cm) .. (0,3.4)
	.. controls +(0:2cm) and +(90:0.2cm) .. (2,2.9)
	.. controls +(-90:0.2cm) and +(10:2cm) .. (b);
\draw (a) -- +(-150:1.6cm);
\draw (b) -- +(-150:1.6cm);
\node[dotnode] (d) at (0.5,2.1) {};
\draw (c) -- (d);
\draw[line width=0.4pt] (a) -- (-0.3,0.6) -- (c);
\draw[line width=0.4pt] (a) -- (0.3,0.6) -- (c);
\draw[line width=0.4pt] (a) -- (-0.25,0.3) -- (c);
\node at (-0.8,0.4) {\smallerer $C_1$};
\node at (-0.3,0.8) {\tiny $x_1$};
\node at (0.8,0.4) {\smallerer $C_2$};
\node at (0.3,0.8) {\tiny $x_2$};
\node at (-0.8,-0.2) {\smallerer $C_3$};
\node at (-0.3,0.18) {\tiny $x_3$};
\end{scope}
\begin{scope}[shift={(10,0)}]
\node[dotnode] (a) at (0,0) {};
\node[dotnode] (b) at (0,2) {};
\node[dotnode] (c) at (0,1) {};
\draw (a) -- (b);
\draw (a)
	.. controls +(170:2cm) and +(-90:0.2cm) .. (-2,0.9)
	.. controls +(90:0.2cm) and +(180:2cm) .. (0,1.4)
	.. controls +(0:2cm) and +(90:0.2cm) .. (2,0.9)
	.. controls +(-90:0.2cm) and +(10:2cm) .. (a);
\draw (b)
	.. controls +(170:2cm) and +(-90:0.2cm) .. (-2,2.9)
	.. controls +(90:0.2cm) and +(180:2cm) .. (0,3.4)
	.. controls +(0:2cm) and +(90:0.2cm) .. (2,2.9)
	.. controls +(-90:0.2cm) and +(10:2cm) .. (b);
\draw (a) -- +(-150:1.6cm);
\draw (b) -- +(-150:1.6cm);
\node[dotnode] (d) at (0.5,2.1) {};
\draw (c) -- (d);
\draw[line width=0.4pt] (a) -- (-0.3,0.6) -- (c);
\draw[line width=0.4pt] (a) -- (0.3,0.6) -- (c);
\draw[line width=0.4pt] (a) -- (-0.25,0.3) -- (c);
\draw[line width=0.4pt] (-0.3,0.6) -- (-0.25,0.3) -- (0.3,0.6) -- (-0.3,0.6);
\node at (-1.3,-0.1) {\smallerer $C_4$};
\node at (0.6,-0.6) {\smallerer $C_5$};
\node at (-0.3,1.2) {\smallerer $C_6$};
\node at (-0.5,0.4) {\tiny $x_4$};
\node[inner sep=0] (l5) at (0.5,-0.1) {\tiny $x_5$};
\draw[line width=0.4pt] (l5) -- (0.1,0.5);
\node[inner sep=0] (l6) at (-0.5,0.8) {\tiny $x_6$};
\draw[line width=0.4pt] (l6) -- (-0.08,0.6);
\end{scope}
\begin{scope}[shift={(6,-5)}]
\draw[->] (-5,1.5) -- (-3,1.5);
\node at (-4,1.8) {collapse};
\node[dotnode] (a) at (0,0) {};
\node[dotnode] (b) at (0,1) {};
\draw (a) -- (b);
\draw (a)
	.. controls +(170:2cm) and +(-90:0.2cm) .. (-2,0.9)
	.. controls +(90:0.2cm) and +(180:2cm) .. (0,1.4)
	.. controls +(0:2cm) and +(90:0.2cm) .. (2,0.9)
	.. controls +(-90:0.2cm) and +(10:2cm) .. (a);
\draw (b)
	.. controls +(170:2cm) and +(-90:0.2cm) .. (-2,2.4)
	.. controls +(90:0.2cm) and +(180:2cm) .. (0,3.4)
	.. controls +(0:2cm) and +(90:0.2cm) .. (2,2.4)
	.. controls +(-90:0.2cm) and +(10:2cm) .. (b);
\draw (a) -- +(-150:1.6cm);
\draw (b) -- (-1.4,1.2);
\node[dotnode] (d) at (0.5,1.1) {};
\draw (a) -- (0.2,0.7) -- (d);
\draw[line width=0.4pt] (a) -- (-0.3,0.5);
\draw[line width=0.4pt] (a) -- (0.3,0.5);
\draw[line width=0.4pt] (a) -- (-0.25,0.2);
\draw[line width=0.4pt] (-0.3,0.5) -- (-0.25,0.2) -- (0.3,0.5) -- (-0.3,0.5);
\end{scope}
\end{tikzpicture}
\end{equation}

\label{d:collapse}
\end{definition}

In particular, if $F, C_1$ are such that one could have obtained them
by performing a \emph{restricted} elementary subdivision
of a cell $C$ along $F$, splitting $C$ into $C_1,C_2$,
then we may collapse $C_1$ (or $C_2$) by $F$.

\begin{lemma}
Let $F$ restricted-elementarily subdivide $C$ of $\cM$
into $C_1, C_2$ of $\cM'$.
Then $(C_1,F)$, $(C_2,F)$ are collapsible pairs,
and performing the collapsing operation on either one
will return $\cM'$ to $\cM$.
\label{l:collapse-inverse}
\end{lemma}
\begin{proof}
Obvious.
\end{proof}

\begin{lemma}
Let $(C,F)$ be a collapsible pair in $\cM$.
Consider an elementary move that does not affect $C$ nor $F$
(that is, neither $C$ nor $F$ is removed from $\cM$
when performing this elementary move).
Then the operation of collapsing $(C,F)$ and the elementary move
commute.
\label{l:collapse-commute}
\end{lemma}
\begin{proof}
Left as an exercise to the reader.
\end{proof}

\begin{theorem}[\ocite{kirillov-PL}*{Theorem 8.1}]
Let $\cM, \cM'$ be two PLCW decompositions of an $n$-manifold $M$.
Then there exists a finite sequence of elementary moves that take
$\cM$ to $\cM'$.
\label{t:elementary}
\end{theorem}

\begin{proof}
We simply sketch a proof, referring the reader to \ocite{kirillov-PL}
for details.
It suffices to show that there is a sequence of elementary moves
that take any PLCW decomposition
to a triangulation; the result then follows from Pachner's theorem
\ocite{Pachner-kons}, \ocite{Pachner-PL},
which says that any two triangulations are related by bistellar moves,
and \lemref{l:bistellar-elem}.

One shows that one can perform, on any cell,
a ``radial subdivision'' by a sequence of elementary moves;
a radial subdivision of a cell-like ball is the cone
over the boundary.
Apply radial subdivisions on cells in order of increasing dimension,
and after enough steps, we arrive at a triangulation.
\end{proof}

In order to prove a relative version of the result above,
we need to strengthen the result about elementary moves to
triangulations:

\begin{lemma}
For $n \leq 3$, $\cM$ a PLCW $n$-manifold,
there exists a sequence of restricted elementary subdivisions
that takes $\cM$ to a triangulation;
more precisely, there exists PLCW decompositions
$\cM_0 = \cM, \cM_1, \ldots, \cM_k$ such that
$\cM_{i+1}$ is an elementary subdivision of $\cM_i$,
and $\cM_k$ is a triangulation.
\label{l:elementary-oneway}
\end{lemma}

\begin{proof}
For $n=0,1$ there is nothing to prove, and for $n=2$
it is easy, simply subdivide each 1-cell many times,
then perform a radial subdivision on every 2-cell.

Let $n=3$. We can easily perform elementary subdivisions
to make the 2-skeleton of $\cM$ a simplicial complex.
In particular, each 3-cell is now a cell-like ball
whose boundary is a triangulated 2-sphere.

Claim: a cell-like 3-ball whose boundary is a triangulated
2-sphere can be converted to a triangulation by
elementary subdivisions along interior cells.

We proceed by induction on the number $k$ of triangles
in the boundary of such a 3-cell $C$.
The base case is $k=4$, where $C$ must be a tetrahedron
and there is nothing to be done.
So assume $k>4$,
and select a vertex $v_0$ arbitrarily;
$v_0$ meets $l$ triangles
$T_1,\ldots, T_l$
(note that $3 < l < k$).

Apply an elementary subdivision to split $C$,
separating all the $T_i$ from the triangles not meeting $v_0$,
then apply $l-3$ elementary subdivisions
(i.e. adding diagonals) on that new 2-cell $F$
so it is made up of $l-2$ triangles
Now we have two 3-cells like $C$,
one with $2l-2$ triangles (the one meeting $v_0$),
the other with $k-2$ triangles,
in their boundaries.

The former is easy to triangulate - simply add a 2-cell
for each diagonal.
To apply the induction hypothesis to the latter,
we need to ensure that the resulting boundary is
still a triangulation.
It is a simple exercise to show that there exists
a choice of triangulation of $F$ that does this.
Thus we are done.

(Note that all the elementary subdivisions performed here are restricted.)
\end{proof}

\begin{theorem}
Let the boundary $N=\del M$ of an $n$-manifold $M$ have a PLCW decomposition
$\cN$, with $n \leq 4$.
Then $\cN$ can be extended into a PLCW decomposition for $M$,
and this extension is unique up to elementary subdivisions
of interior cells;
more precisely, if $\cM,\cM'$ are two PLCW decompositions of $M$
that agree on $\del M$, then
$\cM, \cM'$ are related by a sequence of elementary subdivisions
as in \defref{d:elem-subdiv}, where only interior cells
are used for subdivision.
\label{t:elementary-boundary}
\end{theorem}

\begin{proof}
The existence result is known for when $\cN$ is a triangulation
\ocite{casali-relative-pachner}.
To prove the existence result for general PLCW $\cN$,
by \lemref{l:elementary-oneway},
we may ``induct'' on PLCW decompositions, with ``base case''
$\cN$ being triangulations, and ``inductive steps''
being applying inverse elementary subdivisions.

Suppose $\cN'$ is a PLCW decomposition that can be extended
to a PLCW decomposition on $M$.
Consider $\cN$ such that $\cN'$ is an elementary subdivision
of $\cN$ along a cell $C$.
By \ocite{bryant-PL}*{Corollary 2.5},
$N$ has a collar neighborhood,
i.e. there is a PL-embedding $N \times [0,1] \hookrightarrow M$
such that $N \times 0$ is sent to $\del M$.
Thus we may extend $\cN$ to the collar neighborhood
by using the product PLCW decomposition $\cN \times [0,1]$.
Apply the elementary subdivision to the copy of $C$ in $\cN \times 1$
so that $N \times 1$ now has the PLCW decomposition $\cN'$.
Now $M \backslash N \times [0,1] \simeq M$ as PL-manifolds,
so by our ``inductive hypothesis'',
we may extend the PLCW decomposition to the rest of $M$.

Now we show uniqueness.
Once again we employ a similar ``inductive'' strategry,
beginning with the ``base case'' of triangulations.
For $\cN$ a triangulation,
uniqueness follows from \ocite{casali-relative-pachner},
where it is shown that two triangulations
that agree on the boundary are related by Pachner moves
(which can be converted into a sequence of elementary moves
involving interior cells only).

Now suppose, as before, a PLCW decomposition $\cN'$
that has unique extensions to $M$ up to interior elementary moves,
and consider $\cN$ such that $\cN'$ is a \emph{restricted}
elementary subdivision of $\cN$ along a cell $F$,
splitting the cell $C$ into $C_1, C_2$.
Consider two extensions $\cM,\cM'$ of $\cN$ to the interior of $M$.
By the ``inductive hypothesis'',
if we perform the restricted elementary subdivision to $\cN$
and turn it into $\cN'$,
then there is a sequence of elementary moves that takes
$\widetilde{\cM}$ to $\widetilde{\cM'}$
(obtained from $\cM,\cM'$ by the same subdivision),
say $\widetilde{\cM} = \cM_0 \leftrightarrow_e \cM_1 \leftrightarrow_e
\ldots \leftrightarrow_e \cM_l = \widetilde{\cM'}$.
Now collapse $(C_1,F)$ in each of $\cM_i$;
by \lemref{l:collapse-commute}, we have a sequence of
elementary moves from $\cM$ to $\cM'$
that does not affect the boundary.
\end{proof}

\begin{proposition}
Let $\cM$ be a PLCW manifold, possibly with boundary,
with underlying manifold $M$ of dimension $n \leq 4$,
and let $C$ be a top dimensional cell.
A \emph{single-cell subdivision at $C$}
is a subdivison $\cM'$ of $\cM$ that keeps all cells except $C$ fixed;
as with \defref{d:elem-subdiv},
we say that $\cM$ and $\cM'$ are related by a
\emph{single-cell move}.
(In particular, it leaves the boundary fixed.)

Then any two PLCW decompositions of $M$ are related by
a sequence of single-cell moves.
\label{p:single-cell-subdivision}
\end{proposition}

\begin{proof}
It suffices to check that an elementary subdivision of a $k$-cell can be
achieved as a sequence of single-cell moves.
It is clear that for $k=n$, an elementary subdivision is a
special type of single-cell subdivision.
Now, we induct on $k$ from $n$ down to 1.
Let $C$ be a $k$-cell that we wish to subdivide.
For every $l$-cell containing $C$ in its boundary,
apply elementary subdivision to ``cordon'' off $C$.
Let $\cM''$ be this new PLCW decomposition.
Then the union of the new cells that meet $C$ is an $n$-ball,
and applying an inverse single-cell subdivision,
and then applying a single-cell subdivision that basically
returns to $\cM''$, except that $C$ is subdivision.
(This is essentially the same trick as in \figref{f:M1}.
\end{proof}

\subsection{Colored Ribbon Graphs}
\label{s:cy-ribbongraphs}
\par \noindent

Next, we discuss (co-)ribbon graphs embedded in 3-manifolds,
which is adapted from \ocite{rt-ribbon}, \ocite{rt-3mfld}.
Note that they work in the smooth category
while we are in PL,
but since these are equivalent in low dimensions,
this should not be a problem.

\begin{definition}[\ocite{rt-ribbon}*{4.1}]
Let $M$ be a PL 3-manifold.
A \emph{band} is the image of $B^2$ under a
locally unknotted embedding into $M$;
the (images of the) segments $[-1,1] \times \{-1\}$
and $[-1,1] \times \{1\}$ are its \emph{bases},
and $\{0\} \times [-1,1]$ is its \emph{core}.

An \emph{annulus} is the image of $B^1 \times S^1$
under a locally unknotted embedding into $M$;
the (image of) $\{0\} \times S^1$ is its \emph{core}.

A band or annulus is \emph{directed} if its core is oriented;
we name the bases \emph{incoming} and \emph{outgoing}
so that the orientation of the core points from the former
to the latter. (\ocite{rt-ribbon} calls them initial and final.)
\end{definition}

The following is essentially \ocite{rt-ribbon}*{4.2},
where they have $M = \RR^2 \times [0,1]$,
and ribbon graphs must meet the top and bottom boundary
in particular positions.

\begin{definition}
\label{d:coribbon-graph}
A \emph{ribbon graph (resp. co-ribbon graph)} in a PL 3-manifold $M$
is an embedded oriented (resp. co-oriented) surface $S$
that is presented as a union of
a finite collection of bands and annuli,
and each band is assigned a type of either
\emph{ribbon} or \emph{coupon}
(annuli are considered ribbons),
such that
\begin{itemize}
\item bands of the same type are pairwise disjoint;
\item a ribbon and coupon may only meet along their bases,
	in which case the base of a ribbon is entirely contained
	in the base of the coupon;
\item coupons are disjoint from the boundary $\del M$;
\item a ribbon may only meet $\del M$ along its bases,
	in which case the base is entirely contained
	in $\del M$.
\end{itemize}

A \emph{directed} (co-)ribbon graph is one in which
the core of each band is oriented.

The \emph{boundary} of a ribbon graph $\Gamma$,
denoted $\del \Gamma$,
is the intersection of $\Gamma$ with the boundary $\del M$,
$\del \Gamma = \Gamma \cap \del M$
(not the boundary of the surface $S$);
it is a collection of (co-)oriented arcs in $\del M$.
\end{definition}

\begin{remark}
%
%
Note also that we also consider \emph{co-oriented} graphs
in addition to the usual oriented ribbon graph
(as is done in \ocite{rt-ribbon});
it is merely a convenient way of assigning
an orientation to the surface which
changes when the orientation of the ambient 3-manifold changes.
Of course, when the ambient 3-manifold is oriented,
orientation and co-orientation of a surface are equivalent notions;
we use the convention where, if $\hat{n}$ defines a co-orientation,
and $(\hat{x},\hat{y})$ defines an orientation,
then $(\hat{n},\hat{x},\hat{y})$ defines the ambient orientation.

We will \emph{never} color a co-oriented ribbon graph,
unless the ambient manifold is oriented
(so that, by the discussion above, the graph is automatically oriented).
\label{r:co-ribbon}
\end{remark}

Later, we will see that essentially only a circular ordering,
rather than a linear ordering, on the ribbons
attached to a coupon is needed (see \rmkref{r:circular-coupons}),
so coupons may be drawn as a disk,
without notion of a core or bases.

\begin{definition}
An \emph{isotopy} of (co-)ribbon graphs is a family
$\{\Gamma_t\}_{t \in [0,1]}$ of (co-)ribbon graphs that can be presented
as a union of isotopies of bands.
More precisely, there is a collection of isotopies of bands
$\{\psi_t^j : B^2 \to M\}_{t\in [0,1]}$, $j=1,2,\ldots,n$,
such that $\bigcup_j \psi_t^j(B^2) = \Gamma_t$
for each $t$.
An isotopy does not necessarily fix the boundary.
\label{d:graph-isotopy}
\end{definition}

For a (co-)ribbon graph in a PLCW 3-manifold $\cM$ with boundary,
we typically assume that the graph meets the faces of $\cM$ 
in ``general position''; more precisely:

\begin{definition}
We say that a (co-)ribbon graph $\Gamma$ intersects a
(PLCW) surface $N$
\emph{transversally}
if
\begin{itemize}
\item the coupons of $\Gamma$ do not meet $N$,
\item $\Gamma$ only meets $N$ in the interior (of 2-cells),
\item the ribbons of $\Gamma$ transversally intersect $N$ in arcs,
and each arc crosses every para-core of the ribbon exactly once,
where a para-core is the image of $\{*\} \times [-1,1]$ parallel to
the core.
\end{itemize}
If the third condition is relaxed to
the core being transverse to $N$,
along with a small neighborhood of the core
being transverse to $N$ also,
then we say $\Gamma$ is \emph{weakly transverse} to $N$.
\label{d:ribbon-transverse}
\end{definition}

Note that being transversal as a graph is stronger than
having the surface of the graph $S$ be transversal to the other surface $N$
(in the sense of \defref{d:transversal-map}).
For example, $S$ may intersect $N$ in ``unnecessary'' circle components.
However, the lemma below shows that one simply needs to
``trim the fat''.

\begin{definition}
A \emph{narrowing} of a (co-)ribbon graph $\Gamma^0$ in $M$
is an isotopy $\Gamma^t = \Phi^t(\Gamma^0)$ such that
\begin{itemize}
\item $\Gamma^t \subset \Gamma^s$ for $t > s$,
\item coupons are sent into themselves,
\item $\Phi^t$ ``respects the para-core foliation'' of ribbons
away from the boundary $\del \Gamma^0$,
i.e., for a ribbon $e : [-1,1] \times [-1,1] \to M$ of $\Gamma^0$,
the only ribbon of $\Gamma^0$ that $e^t = \Phi^t(e)$ meets
is $e$ itself, and the para-cores match up except near $\del M$
(see \eqnref{e:narrowing}).
\end{itemize}
We also say that $\Gamma^1$ is a narrowing of $\Gamma^0$.

If the narrowing preserves the core of ribbons,
then we say it is a \emph{strict narrowing}.
\begin{equation}
\label{e:narrowing}
\begin{tikzpicture}
\draw (-1,1) -- (1,1);
\node at (1.3,1) {$\del M$};
\draw (-0.5,0) -- (0.5,0) -- (0.5,-0.5) -- (-0.5,-0.5) -- (-0.5,0);
\foreach \x in {-2,...,2} {
	\draw (\x * 0.125, 1) -- (\x * 0.125,0);
}
\draw[->] (1.5,0.2) -- (2,0.2);
\begin{scope}[shift={(3,0)}]
\draw (-1,1) -- (1,1);
\node at (1.3,1) {$\del M$};
\draw (-0.3,-0.2) -- (0.3,-0.2) -- (0.3,-0.4) -- (-0.3,-0.4) -- (-0.3,-0.2);
\foreach \x in {-2,...,2} {
	\draw (\x * 0.125, 1) -- (\x * 0.0625,0.8);
}
\foreach \x in {-2,...,2} {
	\draw (\x * 0.0625,0.8) -- (\x * 0.0625,-0.2);
}
\begin{scope}[opacity=0.5]
\draw (-1,1) -- (1,1);
\node at (1.3,1) {$\del M$};
\draw (-0.5,0) -- (0.5,0) -- (0.5,-0.5) -- (-0.5,-0.5) -- (-0.5,0);
\foreach \x in {-2,...,2} {
	\draw (\x * 0.125, 1) -- (\x * 0.125,0);
}
\end{scope}
\end{scope}
\end{tikzpicture}
\end{equation}
\label{d:graph-narrowing}
\end{definition}

\begin{lemma}
Suppose the surface $S$ of a (co-)ribbon graph $\Gamma$
intersects another surface $N$ in $M$ transversally
(see \defref{d:transversal-map}),
and $N$ does not meet $\del \Gamma$.
Then there exists a narrowing $\Gamma'$ of $\Gamma$
that intersects $N$ transversally (as a ribbon graph).

Moreover, if $\Gamma$ is \emph{weakly} transverse to $N$,
then there exists a \emph{strict} narrowing $\Gamma'$ of $\Gamma$
that intersects $N$ transversally (as a ribbon graph).
\label{l:narrow-transverse}
\end{lemma}

Note that we can always apply a small isotopy to make
the surfaces transversal, so a simple corollary
is that we can isotope a ribbon graph to be transverse
to a surface.

\begin{proof}
By isotoping a coupon into itself, making it very small
(with respect to some $C^0$ metric),
and dragging attached ribbons along with it,
we can make coupons avoid $N$.

Now consider a ribbon
$e : [-1,1] \times [-1,1] \to M$ of the graph.
The projection to the first factor
(whose fibers are parallel to the core)
defines a function $e \cap N \subset e \to [-1,1]$.
Take a regular value $x \neq -1,1$,
and perform a narrowing so that $e$ is squeezed
into a sufficiently small neighborhood of $\{x\} \times [-1,1]$
(except at the bases of $e$,
which may need to be fixed say if the base is on the boundary of $M$).
Then it is clear that the graph is now transversal to $N$.

The weakly transverse condition essentially allows
us to take the regular value to be $x = 0$,
so the argument above applies, where we may use strict narrowings instead.
\end{proof}

It is also clear that narrowings preserve transversality
(i.e. a narrowing of a graph that is
already transverse to $N$ is also transverse to $N$)
if $N$ is sufficiently far from the boundary.
Furthermore, strict narrowings are essentially closed under intersection
(with perhaps some inconsequential differences near the boundary).

\begin{remark}
From the discussions and lemmas above,
we may assume that coupons are small enough
and ribbons are always narrow enough.
Indeed, it is possible to rewrite the definition of ribbon graphs
in terms of ``infinitesimal ribbon graphs'',
where we have a typical graph $\Gamma$ of vertices and edges,
but it comes with the germ of a surface $\mathcal{S}(\Gamma)$,
i.e. a maximal collection of surfaces $\{S_\al\}$
such that each $S_\al$ includes $\Gamma$ in its interior,
the collection is closed under finite intersections.
Similarly, markings (as discussed in the next section)
can be taken to be a point with a germ of an arc.

Thus, in effect, we can just focus on the core of a ribbon
when dealing with transversality.
\label{r:narrowing-infinitesimal}
\end{remark}

\subsubsection{Co-Ribbon Graphs in PLCW Decompositions}

To define the Crane-Yetter invariant for PLCW decompositions,
we will need to consider co-ribbon graphs in 3-cells
of a particularly simple kind,
which we call an \emph{anchor}.

\begin{definition}
A \emph{marking} of a 2-cell $F$ in a PLCW decomposition
is a co-oriented embedded line segment in the interior of $F$,
$[0,1] \hookrightarrow \intr{F}$.
A \emph{marked PLCW decomposition} is a PLCW decomposition
with a choice of marking on each 2-cell.
\label{d:marking}
\end{definition}

\begin{definition}
An \emph{anchor} of a cell-like PLCW 3-ball $(B^3, \cB)$
is a co-ribbon graph $\psi$ that meets each 2-face of $\cB$
at a marking,
and is isotopic (not fixing the boundary) to the following co-ribbon graph,
which we refer to as the \emph{standard anchor}:
\begin{itemize}
\item exactly one coupon $[0,0.5] \times [0, 0.5] \times 0$,
\item a ribbon $[\frac{2i}{4k}, \frac{2i+1}{4k}] \times [-1,0] \times 0$
	for each $i = 0,1,\dots,k-1$, where $k$ is the number of 2-cells
	of $\cB$.
\end{itemize}
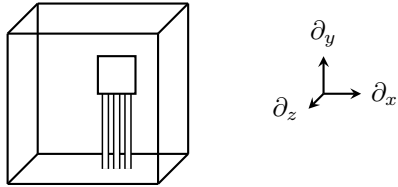
\begin{figure}[ht]
\begin{tikzpicture}
\node (000) at (-0.8,-0.8) {};
\node (100) at (1.2,-0.8) {};
\node (010) at (-0.8,1.2) {};
\node (110) at (1.2,1.2) {};
\node (001) at (-1.2,-1.2) {};
\node (101) at (0.8,-1.2) {};
\node (011) at (-1.2,0.8) {};
\node (111) at (0.8,0.8) {};
\draw (000.center) -- (100.center);
\draw (001.center) -- (101.center);
\draw (010.center) -- (110.center);
\draw (011.center) -- (111.center);
\draw (000.center) -- (010.center);
\draw (001.center) -- (011.center);
\draw (100.center) -- (110.center);
\draw (101.center) -- (111.center);
\draw (000.center) -- (001.center);
\draw (010.center) -- (011.center);
\draw (100.center) -- (101.center);
\draw (110.center) -- (111.center);
\draw[ribbon_u={1,0.6}] (0.1,0) to (0.1,-1);
\draw[ribbon_u={1,0.6}] (0.25,0) to (0.25,-1);
\draw[ribbon_u={1,0.6}] (0.4,0) to (0.4,-1);
\draw (0,0) rectangle (0.5,0.5);
\begin{scope}[shift={(3,0)}]
\draw[->] (0,0) -- (0.5,0) node[right] {$\del_x$};
\draw[->] (0,0) -- (0,0.5) node[above] {$\del_y$};
\draw[->] (0,0) -- (-0.2,-0.2) node[left] {$\del_z$};
\end{scope}
\end{tikzpicture}
\caption{Standard anchor}
\label{f:standard-anchor}
\end{figure}
An anchor endows the ribbons with a natural ``monotonic'' structure,
in that there is a well-defined notion of one ribbon being
in between two ribbon, but there is no canonical choice of ordering.\\

When $B^3$ is \emph{oriented}, an anchor also defines an \emph{ordering}
on the ribbons, and hence the faces, as follows:
let $p$ be the point of intersection of the core of the coupon
and the base meeting the ribbons; if $u,v,w$ are vectors at $p$
with
\begin{itemize}
\item $u$ tangent to the base of the coupon, pointing in increasing
	ordering of the ribbons,
\item $v$ is tangent to the core of the coupon, pointing inwards,
\item $w$ is a vector giving the co-orientation,
\end{itemize}
then $(u,v,w)$ gives the ambient orientation of $B^3$.
(See \rmkref{r:orientations} on orientations in PL manifolds).
In other words, if $B^3 \subset \RR^3$ is given the standard orientation
$(\del_x, \del_y, \del_z)$,
and we have isotoped the co-ribbon graph into the standard anchor,
we have $u=\del_x$, $v=\del_y$, $w=\del_z$,
so the ordering of ribbons is from low to high $x$-value.

Two anchors are \emph{equivalent} if they are isotopic fixing
the boundary.
\hfill $\triangle$
\label{d:anchor}
\end{definition}

The main point of an anchor is to make the identification of
the boundary sphere $\del B^3$ with some standard $S^2$
explicit and concrete.
Such a graph is essentially ``radial'', having no knotting or linking.
More precisely, if we thicken the coupon into a ball,
the ribbons of an anchor should form a $k$-strand braid
in $S^2 \times [0,1]$; this connection is pursued later.

Note that changing the orientation of $B^3$ flips the ordering
on the faces.

It is clear that an anchor $\Gamma$ is determined up to equivalence
by the movement of its boundary;
that is, if $\Phi_t$ is an isotopy from the standard anchor
to the anchor $\Gamma$,
and $s_i = [\frac{2i}{4k}, \frac{2i+1}{4k}] \times -1 \times 0$
is the $i$-th boundary segment of the standard anchor,
then $\Phi_t|_{\{s_i\}}$, the isotopy restricted to $s_i$'s,
determines $\Gamma$ up to equivalence.
Moreover, if $\Phi_t'$ is another isotopy such that
$\Phi_t'|_{\{s_i\}}$ is isotopic to $\Phi_t|_{\{s_i\}}$,
then $\Phi_t'$ also determines an equivalent anchor.

Intuitively, if we imagine each face has a special marked point,
then an anchor is essentially equivalent to a choice
of how to move the special marked points into a line;
this choice is not just an ordering on the faces,
but also involves braid groups, as we discuss below.

Let $\mathfrak{F}_n$ be the \emph{framed braid group}
as defined in \ocite{ko-smolinsky},
namely, $\mathfrak{F}_n$ is a semi-direct product
$\ZZ^n \rtimes \mathfrak{B}_n$,
where $\mathfrak{B}_n$ is the braid group on $n$ points.
The braid group naturally projects onto the symmetric group on
$n$ elements, $S_n = \Aut(\{1,\ldots,n\})$,
which gives a right action of $\mathfrak{B}_n$
on $\ZZ^n = \ZZ^{\{1,\ldots,n\}}$, given by precomposition:
$(r_1,\ldots,r_n) \cdot \sigma = (r_{\sigma(1)},\ldots,r_{\sigma(n)})$.
$\mathfrak{F}_n$ is generated by $\sigma_i$, $i= \sigma_i, i=1,2,\ldots,n-1$,
which generate $\mathfrak{B}_n$,
and $t_j, j = 1,2,\ldots,n$, which generate $\ZZ^n$,
with the following relations:
\begin{align}
\begin{split}
\sigma_i \sigma_j &= \sigma_j \sigma_i \;\;,\;\; |i-j|>1 \\
\sigma_i \sigma_{i+1} \sigma_i &= \sigma_{i+1} \sigma_i \sigma_{i+1} \\
t_j t_k &= t_k t_j\\
\sigma_i t_j &= t_{\sigma_i(j)} \sigma_i
\end{split}
\end{align}
For $\mathbf{r} = (r_1,\ldots,r_n)$,
we write $t^\mathbf{r} = t_1^{r_1} \cdots t_n^{r_n}$,
so we have
\begin{equation}
\sigma t^\mathbf{r}
= t_{\sigma(1)}^{r_1} \cdots t_{\sigma(n)}^{r_n} \sigma
= t_1^{r_{\sigma^\inv(1)}} \cdots t_n^{r_{\sigma^\inv(n)}} \sigma
= t^ {\mathbf{r} \cdot \sigma^\inv} \sigma.
\end{equation}
Geometrically, an element $t^\mathbf{r} \sigma$
is represented by the geometric braid $\sigma$
with the label $r_i$ on the $i$-th strand from the left according
to the top endpoint.
Then the product $t^\mathbf{s} \tau \cdot t^\mathbf{r} \sigma
= t^{\mathbf{s} + \mathbf{r}\cdot \tau^\inv} \tau \sigma$
amounts to stacking $\tau$ on top of $\sigma$,
and labeling each strand with the sum of the two labelings
(see \figref{f:framed-braids-product}).
\begin{figure}[ht]
\includegraphics[width=12cm]{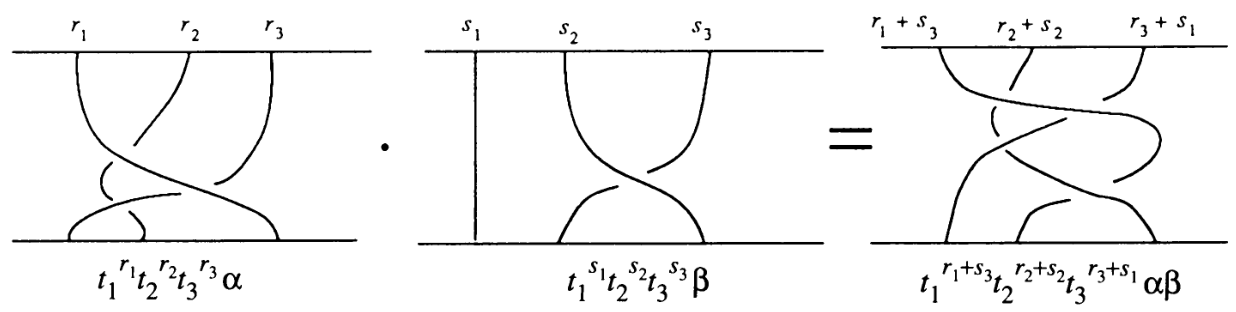}
\caption{Product of framed braids; this was taken from \ocite{ko-smolinsky}}
\label{f:framed-braids-product}
\end{figure}

As the name of $\mathfrak{F}_n$ suggests,
these labels can be interpreted as framings for the strands of the braid,
or equivalently and more fittingly for this discussion,
a strand labeled with $r$ can be interpreted
as a co-oriented ribbon, making $r$ positive twists about its core
relative to the blackboard framing.

Given an oriented cell-like PLCW 3-ball $(B^3, \cB)$,
$\mathfrak{F}_n$ naturally acts (from the left) on the set of anchors
up to equivalence, by inserting the framed braid near the coupon:

\begin{equation}
\sigma \;\; \cdot \;\;
\begin{tikzpicture}
\begin{scope}[shift={(0,0.3)}]
\draw[ribbon_u={1.4,0.9}] (-0.3,0) to[out=-90,in=60] (-0.6,-1);
\draw[ribbon_u={1.4,0.9}] (-0.1,0)
	.. controls +(down:0.1cm) and +(60:0.3cm) .. (-0.2,-0.6)
	.. controls +(-120:0.3cm) and +(0:0.3cm) .. (-1,-0.5);
\draw[ribbon_u={1.4,0.9}] (0.1,0) to[out=-90,in=100] (0.3,-1);
\draw[ribbon_u={1.4,0.9}] (0.3,0)
	.. controls +(down:0.4cm) and +(-150:0.2cm) .. (0.8,-0.6);
\draw (-0.4, 0) rectangle (0.4,0.3);
\end{scope}
\end{tikzpicture}
\;\; = \;\;
\begin{tikzpicture}
\draw[ribbon_u={1.4,0.9}] (-0.3,0) to[out=-90,in=60] (-0.6,-1);
\draw[ribbon_u={1.4,0.9}] (-0.1,0)
	.. controls +(down:0.1cm) and +(60:0.3cm) .. (-0.2,-0.6)
	.. controls +(-120:0.3cm) and +(0:0.3cm) .. (-1,-0.5);
\draw[ribbon_u={1.4,0.9}] (0.1,0) to[out=-90,in=100] (0.3,-1);
\draw[ribbon_u={1.4,0.9}] (0.3,0)
	.. controls +(down:0.4cm) and +(-150:0.2cm) .. (0.8,-0.6);
\draw[ribbon_u={1.4,0.9}] (-0.3,0) -- (-0.3,0.6);
\draw[ribbon_u={1.4,0.9}] (-0.1,0) -- (-0.1,0.6);
\draw[ribbon_u={1.4,0.9}] (0.1,0) -- (0.1,0.6);
\draw[ribbon_u={1.4,0.9}] (0.3,0) -- (0.3,0.6);
\draw[fill=white] (-0.4, 0) rectangle (0.4,0.3) node[pos=0.5] {$\sigma$};
\draw[fill=white] (-0.4, 0.5) rectangle (0.4,0.8);
\end{tikzpicture}
\end{equation}

\begin{lemma}
Let $X$ be the boundary of some anchor
in an oriented cell-like PLCW 3-ball $(B^3, \cB)$
(thus $X$ is a collection of co-oriented arcs in $\del B^3$,
one for each 2-face of $\cB$).
Let $\mathfrak{A}_X$ be the set of anchors $\Gamma$,
considered up to equivalence,
such that $\del \Gamma = X$ (as co-oriented 1-manifolds).
Then $\mathfrak{F}_n$ acts transitively on $\mathfrak{A}_X$,
and the stabilizer group for any anchor is
the normalizer of $x = t_1^2 \sigma_1 \sigma_2 \cdots \sigma_{n-1}^2 \sigma_{n-2} \cdots \sigma_1$.
\label{l:framed-action-anchor}
\end{lemma}

\begin{proof}
For simplicity, we may assume, without loss of generality,
that $X$ is the boundary of the standard anchor.
It is clear that any isotopy of anchors can be isotoped to one
that fixes the coupon pointwise,
thus we may treat $\mathfrak{A}_X$ as the set of anchors
with the same fixed coupon,
up to isotopy of anchors that fixes the coupon and boundary
pointwise.

As pointed out after \defref{d:anchor},
upon thickening the coupon, we may identify
an anchor with a framed braid in $S^2 \times [0,1]$,
thus identifying $\mathfrak{A}_n$ with
$\mathfrak{F}_n(S^2)$, the framed braid group
on the sphere
(see \ocite{BG-framed});
clearly the actions of $\mathfrak{F}_n$ on $\mathfrak{A}_n$
and $\mathfrak{F}_n(S^2)$ agree.
It suffices to show that 
$\mathfrak{F}_n(S^2) \cong \mathfrak{F}_n / \eval{x}$.

Let $\mathfrak{B}_n(S^2)$ be the braid group on the sphere.
From \ocite{sphere-braid}, $\mathfrak{B}_n(S^2)$
has generators $\sigma_i$, $i=1,\ldots,n-1$,
with the relations from $\mathfrak{B}_n$
and the additional
$\bar{x} := \sigma_1\cdots\sigma_{n-1}\sigma_{n-1}\cdots\sigma_1=1$.
Thus we have the exact sequence
\[
1 \to N(\bar{x}) \to \mathfrak{B}_n \to \mathfrak{B}_n(S^2) \to 1
\]
where $N(\bar{x})$ is the normalizer of $\bar{x}$ in $\mathfrak{B}_n$.
It is not hard to see that the natural forgetful map
$\mathfrak{F}_n \to \mathfrak{B}_n$ that kills all $t_i$
sends $N(x)$ onto $N(\bar{x})$.
The exact sequence above fits in the following commutative diagram:
\[
\begin{tikzcd}
1 \ar[r]
	& N(x) \ar[r, hook] \ar[d, two heads]
	& \mathfrak{F}_n \ar[r] \ar[d, two heads]
	& \mathfrak{F}_n(S^2) \ar[d]
\\
1 \ar[r]
	& N(\bar{x}) \ar[r]
	& \mathfrak{B}_n \ar[r]
	& \mathfrak{B}_n(S^2) \ar[r]
	& 1
\end{tikzcd}
\]
Exactness of the bottom row and injectivity in top left
implies the exactness of top row, thus
$\mathfrak{F}_n(S^2) \cong \mathfrak{F}_n / \eval{x}$.
\end{proof}

\subsubsection{Colorings of Ribbon Graphs and Their Evaluations}

Now we consider colorings of ribbon graphs
in an oriented 3-manifold
(we will not need it for co-ribbon graphs).

\begin{definition}[\ocite{rt-ribbon}*{4.4}]
Let $\cA$ be a premodular category.
An \emph{$\cA$-coloring} $\Psi$ of a ribbon graph $\Gamma$
in an oriented 3-manifold $M$
is the following data:
\begin{itemize}
\item Choice of an object $\Psi(\vec{e}) \in \Obj \cA$
  for each oriented directed ribbon $\vec{e}$,
    so that $\Psi(\cev{e}) = \Psi(\vec{e})^*$,
		where $\cev{e}$ is the oppositely directed ribbon.
\item Choice of a vector
  $\Psi(\mathbf{c}) \in
		\Hom(\Psi(\vec{e}) \tnsr \cdots \tnsr \Psi(\vec{e}_k),
		\Psi(\vec{e}_1') \tnsr \cdots \tnsr \Psi(\vec{e}_l'))$,
  for each directed coupon $\mathbf{c}$,
	where $\vec{e}_i,\vec{e}_j'$ are the incoming and outgoing edges,
	such that $\Psi(\overline{\mathbf{c}}) = \Psi(\mathbf{c})^*$.
\end{itemize}
We say $(\Gamma,\Psi)$ is a $\cA$-colored ribbon graph.
\end{definition}

When the premodular category $\cA$ is clear,
we simply say coloring.
We often denote $(\Gamma,\Psi)$ by $\Gamma$ for simplicity.

A coloring $\Psi$ of a ribbon graph $\Gamma$
defines a coloring of its boundary $\del \Gamma$
in the following manner:
for a base $b \subset \vec{e} \cap \del M$,
we assign it the object $\Psi(b) = \Psi(\vec{e})$,
where $\vec{e}$ is taken with outward direction at $b$.
(If $\vec{e}$ meets the boundary again at $b'$,
then $b'$ is assigned $\Psi(b') = \Psi(b)^*$.)
Such $\del \Gamma$ with coloring is called
a \emph{boundary value} on $\del M$;
we will discuss this in more detail in \secref{s:sk-defn}.

\begin{definition}
An isomorphism $\Psi \simeq \Psi'$ between colorings
is an isomorphism $\Psi(\vec{e}) \simeq \Psi'(\vec{e})$
for each directed ribbon
which respects the dualities and
intertwines the assignments to coupons.
\end{definition}

We give two closely related notions of evaluating a
colored ribbon graph, one is a ``global'' evaluation
while the other is a ``local'' evaluation:

\begin{proposition}
Given a colored ribbon graph $(\Gamma,\Psi)$
in an oriented $S^3$,
there is a well-defined \emph{Reshetikhin-Turaev evaluation}
$\ZRT(\Gamma,\Psi) \in \kk$
that is invariant with respect to isotopy and isomorphism
of colorings.
\label{p:rt-evaluation-S3}
\end{proposition}

\begin{proposition}
Given a colored ribbon graph $(\Gamma,\Psi)$
in an oriented ball $D$,
and an identification of $D$ with $B^3$
so that $\Gamma \cap \del D$ is sent to $\Gamma' \cap \del B^3$,
where $\Gamma'$ is the standard anchor \figref{f:standard-anchor},
there is a well-defined \emph{Reshetikhin-Turaev evaluation}
$\ZRT((\Gamma,\Psi)) \in \Hom(\one,V_1 \tnsr \cdots \tnsr V_k)$,
that is invariant with respect to isotopy rel boundary and isomorphism
of colorings
(strictly speaking, covariant with respect to isomorphism of colorings
of the ribbons meeting the boundary).

Furthermore, for $\Gamma = \Gamma', D = B^3$,
with the coupon labeled by $\vphi \in \Hom(\one,V_1 \tnsr \cdots \tnsr V_k)$,
the Reshetikhin-Turaev evaluation is simply the labeling of the coupon,
i.e. $\ZRT(\Gamma,\Psi) = \vphi$.

If $\Gamma$ does not meet the boundary,
then $\ZRT((\Gamma,\Psi)) \in \kk$,
with the empty graph identified with 1,
and the evaluation is independent of the identification of $D$ with $B^3$.
(In general, the evaluation does depend on this identification.)
\label{p:rt-evaluation-ball}
\end{proposition}

Thus, given a colored ribbon graph $\Gamma$ in some oriented 3-manifold $M$,
and some embedded ball $D$,
we may speak of a ``local evaluation'' of $\Gamma$,
and we can use this to cut down the space of colored ribbon graphs,
saying two graphs are equivalent if they agree everywhere outside $D$,
and have the same evaluation in $D$;
this is the idea behind skeins, and we will discuss this in detail
in \secref{s:sk-defn}.

\begin{remark}
\label{r:ribbongraph-rt-evaluation}
\prpref{p:rt-evaluation-ball} is often stated in terms of
graphs in a thickened plane $\RR^2 \times [0,1]$,
and the graphs should have endpoints in particular positions.
More precisely,
treat $\RR^2 \times [0,1]$ as a ball (missing an $S^1$ at its equator,
which does not affect the discussion),
and suppose the graph meets the boundary planes $\RR^2 \times 0$,
$\RR^2 \times 1$ along their $x$-axes
$\RR \times 0 \times 0$, $\RR \times 0 \times 1$,
and the graph is always co-oriented towards the $y$-axis.
Then under the graphical calculus,
the graph defines a morphism
$V_1 \tnsr \cdots \tnsr V_n \to W_1 \tnsr \cdots \tnsr W_m$,
where the $V$'s are the colors of the legs meeting $\RR^2 \times 0$,
directed inwardly,
and the $W$'s are the colors of the legs meeting $\RR^2 \times 1$,
directed outwardly.

Then the results of \ocite{rt-ribbon} state that
this morphism is invariant under isotopy of the graph.
(See \ocite{BK}*{Theorem 2.3.8}.)
\end{remark}


\subsection{The Crane-Yetter Invariant}
\label{s:cy-defn}
\par \noindent


The original Crane-Yetter invariant, as given in
\ocite{CKY} (first announced in \ocite{CY})
is for a triangulated 4-manifold.
We give a definition of the Crane-Yetter invariant
for a 4-manifold with a PLCW decomposition;
we prove it is equivalent to the original definition
in the next section.

The constructions presented here closely mirror those
of \ocite{balsam-kirillov}, which recasts the definition
of the Turaev-Viro invariant in terms of a 3-manifold
with PLCW decomposition
(there they call it a ``polytope decomposition'').
The following additional structure is meant
to impose an ordering (and more) on the boundary 2-faces
of a 3-cell;
in one dimension lower, such a structure is not needed in
the Turaev-Viro case as the orientation of a 2-cell
furnishes a (cyclic) ordering on the boundary edges.


\begin{definition}
An \emph{anchoring} of a combinatorial 3- or 4-manifold
is the choice of an anchor for each 3-cell
that agree on every 2-cell, i.e. there is a marking
of each 2-cell such that every anchor meets a 2-cell
in exactly that marking.
We say that a combinatorial 4-manifold equipped with an
anchoring is \emph{anchored}.
\label{d:anchoring}
\end{definition}

Let $\cA$ be a premodular category as in \secref{s:basic-cat},
and $\cW$ a combinatorial 4-manifold, possibly with boundary.
We denote by $F^\textrm{or}$ the set of oriented 2-cells of $\cW$.
For an oriented 2-cell $f$, denote by $\ov{f}$ the same 2-cell
with the opposite orientation.

\begin{definition}
A \emph{labeling} of a combinatorial 2-, 3-, or 4-manifold $\cM$
(possibly with boundary)
is an assignment $l : F^\textrm{or} \to \Obj \cC$ such that
$l(\ov{f}) \simeq l(f)^*$.
A labeling is called \emph{simple} if for every 2-cell,
$l(f)$ is simple.

Two labelings $l_1,l_2$ are \emph{equivalent} if $l_1(f) \simeq l_2(f)$
for every 2-cell $f$.

The \emph{dual labeling} $l^*$ assigns the dual objects:
$l^*(f) = l(f)^*$.
\end{definition}

If a labeling $l$ is simple, we will assume that $l(f) = X_i$
(the chosen representative for isomorphism class $i$)
for some $i$, so in particular we may write
$l(\ov{\ov{f}}) = l(\ov{f})^* = l(f)^{**} = l(f)$.

We will often denote $d_l = \prod_{f} d_{l(f)}$
where the product is over all unoriented faces in the labeling $l$.

\begin{definition}
Let $\cW$ be an anchored combinatorial 4-manifold with a labeling $l$,
and $C \in F^{\textrm{or}}$ be an oriented 3-cell.
The \emph{local state space} is the vector space
\[
	H(C,l,\psi_C) := \eval{l(f_1),\ldots,l(f_k)}
\]
where $f_i$ are the faces of $\del C$, with the outward orientation,
taken in the order imposed by the anchor,
and $\psi_C$ is the anchor for $C$.
We refer its elements as \emph{local states}.
\label{d:local-state-space}
\end{definition}

The definition of local state space depends on a choice of anchor,
but different choices of anchors give local state spaces
that are clearly isomorphic, and in fact canonically so:

\begin{definition}
Let $C$ be as in \defref{d:local-state-space},
and let $\psi, \psi'$ be two anchors for $C$;
by \lemref{l:framed-action-anchor},
$\psi' = \sigma \cdot \psi$ for some framed braid $\sigma \in \mathfrak{F}_n$.
Define
\begin{align*}
	f_\sigma : H(C,l,\psi) &\simeq H(C,l,\psi')
	\\
	\vphi &\mapsto \sigma^\inv \circ \vphi
\end{align*}
to be the map obtained from interpreting
the braid $\sigma^\inv$ in the graphical calculus.
(Note it is ``contravariant'' in the sense that
$f_{\tau \sigma} = f_\sigma f_\tau$.)
\label{d:anchor-morphism}
\end{definition}

Observe that $f_{\tau \sigma} = f_\tau \circ f_\sigma$,
and if $\sigma \in \mathfrak{F}_n$ fixes an anchor $\psi$,
i.e. $\sigma \cdot \psi = \psi$,
then clearly $f_\sigma = \id$.
Thus, the isomorphisms
$H(C,l,\psi) \simeq H(C,l,\psi')$ are canonical
in the sense that it does not depend on the choice of $\sigma$
that relates $\psi$ and $\psi'$.

\begin{definition}
Let $\cM$ be an anchored oriented combinatorial 3-manifold with a labeling $l$,
with anchor $\psi_C$ for 3-cell $C$.
We define the \emph{pre-state space of $(\cM,l,\psi_C)$} as the vector space
\[
	H(\cM, l, \{\psi_C\}) := \bigotimes_C H(C,l,\psi_C)
\]
where the tensor product is over all 3-cells $C$ of $\cM$;
$H(\cM,l,\{\psi_C\})$ transforms functorially with change of anchors,
as in \defref{d:anchor-morphism},
so we write $H(\cM,l)$ for short.

The \emph{pre-state space of $\cM$} is
\[
	H(\cM) := \bigoplus_{\textrm{simple } l} H(\cM, l)
\]
where the sum is over all simple labelings up to equivalence.
We refer to elements of $H(\cM,l)$ and $H(\cM)$
as \emph{pre-states}.
\label{d:pre-state-space}
\end{definition}

The corresponding state space will be defined in
\prpref{p:projectors-compose}.

\begin{definition}
Let $\cM, l, \psi_C$ be as in \defref{d:pre-state-space},
and suppose the underlying 3-manifold of $\cM$ is $S^3$.
Suppose we are given $\vphi_C \in H(C,l,\psi_C)$ for each
3-cell $C$ (oriented the same as $\cM$),
defining an element $\bigotimes_C \vphi_C \in H(\cM,l,\{\psi_C\})$.
We define its \emph{Reshetikhin-Turaev evaluation} as follows:

Consider the graph $\Gamma$ obtained as the union of the anchors,
$\Gamma = \cup \psi_C$.
We make $\Gamma$ a directed co-ribbon graph
such that (1) the ribbons are attached to the outgoing base of the coupon
and (2) ribbons are directed arbitrarily.
Furthermore, the ambient orientation of $C$ picks out an
orientation on the graph, making it a ribbon graph.
Label the coupon in $C$ by $\vphi_C$,
and label the ribbon intersecting the 2-face $f$ by $l(f)$,
where, if the core of the ribbon is directed outwardly from $C$,
then $f$ is taken with the outward orientation with respect to $C$.
Then we define the Reshetikhin-Turaev evaluation
of $\bigotimes_C \vphi_C$ to be
$\ZRT(\Gamma, l, \bigotimes_C \vphi_C)$
(see \prpref{p:rt-evaluation-S3}).

This is multilinear in $\vphi_C$'s, so defines a linear map
\begin{equation}
	\ZRT(\Gamma, l, -) : H(\cM,l,\{\psi_C\}) \to \kk
\label{e:rt-ev}
\end{equation}

$\ZRT$ is also functorial with respect to anchor change:
if $\{\psi_C'\}$ is another collection of anchors,
related to the old by $\psi_C' = \sigma_C \cdot \psi_C$,
with union $\Gamma' = \cup \psi_C'$,
then clearly
\[
	\ZRT(\Gamma', l, \bigotimes_C f_{\sigma_C}(\vphi_C))
	= \ZRT(\Gamma, l, \bigotimes_C \vphi_C)
\]
\label{d:zrt-eval}
\end{definition}


Since $H(\ov{C},l) = \eval{l(f_k)^*,\ldots,f(f_1)^*}$
(recall that flipping orientation of $C$ also flips
the ordering on the faces),
$H(\ov{C},l)$ and $H(C,l)$ are naturally dual
via \eqnref{e:pairing}.
The following lemma states that the pairing
is natural with respect to the choice of anchor:

\begin{lemma}
Let $C$ be an oriented 3-cell with an anchor $\psi$,
as in \defref{d:local-state-space}.
Consider the orientation-preserving PL homeomorphism
$C \cup_{\del} \ov{C} \simeq S^3$, where $\del C$ and $\del \ov{C}$
are glued via the identity map;
from \defref{d:zrt-eval}, we have the following map:
\[
	\ZRT: H(C,l,\psi) \tnsr H(\ov{C},l,\psi) \to \kk
\]
We call this the \emph{local state space pairing}.
Then this pairing agrees with $\ev$ (see \eqnref{e:pairing}).
Furthurmore, this pairing is natural with respect to choice of anchors:
if $\psi' = \sigma \cdot \psi$,
\[
\begin{tikzcd}
H(C,l,\psi) \tnsr H(\ov{C},l,\psi)
	\ar[d, "f_{\sigma} \tnsr f_\sigma"]
	\ar[r]
	& \kk \ar[d, equal]
\\
H(C,l,\psi') \tnsr H(\ov{C},l,\psi')
	\ar[r] & \kk
\end{tikzcd}
\]
\label{l:local-state-space-dual}
\end{lemma}

It is useful to keep in mind that when the ambient orientation
is flipped, we can draw/visualize things by flipping them
left/right, up/down, or front/back.
For example, a braid that is mirrored up/down is
taken to its inverse.

\begin{remark}
Note that for the braid $\sigma$ defined below,
$f_\sigma = z$ from \eqnref{e:cyclic},
thus, we may use circular coupons instead of rectangular ones
in (co-)ribbon graphs.
\[
\begin{tikzpicture}
\node[small_morphism] (a) at (0,1.4) {};
\draw[ribbon] (a) to[out=-150,in=90] (-0.3,1);
\draw[ribbon] (a) to[out=-90,in=90] (0,1);
\draw[ribbon] (a) to[out=-30,in=90] (0.3,1);
\draw[ribbon] (-0.3,1) to[out=-90,in=90] (0,0);
\draw[ribbon] (0,1) to[out=-90,in=90] (0.3,0);
\draw[ribbon,overline] (0.3,1) to[out=-90,in=90] (-0.3,0);
\draw (-0.365,0) to[out=-90,in=90] (-0.236,-0.2);
\draw[thin_overline={0.7}] (-0.236,0) to[out=-90,in=90] (-0.365,-0.2);
\draw (-0.365,-0.2) to[out=-90,in=90] (-0.236,-0.4);
\draw[thin_overline={0.7}] (-0.236,-0.2) to[out=-90,in=90] (-0.365,-0.4);
\draw[ribbon] (0,0) -- (0,-0.4);
\draw[ribbon] (0.3,0) -- (0.3,-0.4);
\draw[decorate,decoration={brace,amplitude=8pt,mirror,raise=3pt}]
	(-0.5,1) -- (-0.5,-0.4);
\node at (-1.1,0.3) {\smallerer $\sigma$};
\node at (1,0.5) {=};
\begin{scope}[shift={(2,0)}]
\node[small_morphism] (a) at (0,1.4) {};
\draw[ribbon] (a) to[out=-150,in=90] (-0.3,1);
\draw[ribbon] (a) to[out=-90,in=90] (0,1);
\draw[ribbon] (a)
	.. controls +(-30:0.3cm) and +(-30:0.3cm) .. (0.2,1.7)
	.. controls +(150:0.4cm) and +(90:0.8cm) .. (-0.6,1);
\draw[ribbon] (-0.3,1) to[out=-90,in=90] (0,-0.4);
\draw[ribbon] (0,1) to[out=-90,in=90] (0.3,-0.4);
\draw[ribbon] (-0.6,1) to[out=-90,in=90] (-0.3,-0.4);
\end{scope}
\end{tikzpicture}
\]
\label{r:circular-coupons}
\end{remark}

We may simply write $H(C,l)$ for the local state space
when it does not lead to confusion;
thus $H(C,l)$ is naturally dual to $H(\ov{C},l)$.
When we construct elements of $H(C,l)$,
typically a choice of anchor is needed,
and the resulting element should vary functorially
with change of anchor;
when we say two elements of $H(C,l)$ are equal,
we implicitly assumed that a choice of anchor $\psi$ is made,
and they are equal in $H(C,l,\psi)$.

\begin{definition}
Let $\cM$ be an oriented combinatorial 3-manifold,
and $l$ a labeling of $\cM$.
Applying the local state space pairing of \lemref{l:local-state-space-dual}
to each 3-cell, we have a natural pairing
\begin{equation}
\label{e:state-space-pairing-labeling}
ev: H(\cM, l) \tnsr H(\overline{\cM}, l) \to \kk
\end{equation}
Summing up over simple labeling, we thus get the
\emph{state space pairing}:
\begin{equation}
ev: H(\cM) \tnsr H(\overline{\cM}) \to \kk
\label{e:state-space-pairing}
\end{equation}
\label{d:state-space-pairing}
\end{definition}

\begin{definition}
Let $T$ be a 4-cell in a combinatorial 4-manifold $\cW$
with labeling $l$. We define the \emph{local invariant}
\begin{equation}
\label{e:z-4cell}
	Z(T, l) \in H(\del T,l)
	\simeq H(\overline{\del T}, l)^*
\end{equation}
as the functional that defines the Reshetikhin-Turaev evaluation
from \defref{d:zrt-eval}; that is, upon choosing anchors
$\{\psi_C\}_{C\in \del T}$,
\[
	Z(T, l) : \bigotimes_C \vphi_{\ov{C}} \mapsto \ZRT(\Gamma, l, \{\vphi_{\ov{C}}\})
\]
where $\vphi_{\ov{C}} \in H(\ov{C}, l, \psi_C)$
for each inwardly-oriented 3-cell $\ov{C} \in \overline{\del T}$,
and $\Gamma = \cup \psi_C$ is the graph obtained from
some choice of anchoring of $\del T$.

Equivalently, upon choosing a basis $\{\vphi_{C,\al}\}_\al$ of $H(C,l)$
for each 3-cell $C$,
\[
Z(T,l) = \sum_\al \big( \ZRT(\Gamma, l, \{\ov{\vphi}_{C,\al}\})
	\bigotimes_C \vphi_{C,\al} \big)
\]
where the sum is over all assignments of a basis vector $\vphi_{C,\al}$
to each $C$.
\label{d:ZCY3cell}
\end{definition}

\begin{definition}
Let $\cW$ be an oriented combinatorial 4-manifold,
possibly with boundary.
For a labeling $l$, we define
\begin{equation}
Z(\cW,l) =
\ev (\bigotimes_T Z(T,l)) \in H(\del \cW, l)
\label{e:zcy-label}
\end{equation}
where $\ev$ applies the dual pairing of \lemref{l:local-state-space-dual}
to every interior 3-cell,
and $T$ runs over every 4-cell.

The \emph{Crane-Yetter invariant} $\ZCY(\cW)$ is
\begin{equation}
	\ZCY(\cW) = \cD^{\vme(\cW \backslash \del \cW) + \frac{1}{2}\vme(\del \cW)}
		\sum_l \prod_f d_{l(f)}^{n_f} Z(\cW,l)
	\in H(\del \cW)
\label{e:cy}
\end{equation}
where
\begin{itemize}
\item the sum is taken over all equivalence classes of simple labelings
\item $f$ runs over the set of unoriented 2-cells of $\cW$
\item $n_f =
	\begin{cases}
		1 &\textrm{if } f \textrm{ is an internal 2-face} \\
		\frac{1}{2} & f \in \del \cW
	\end{cases}
	$
\item $\cD$ is the dimension of the category (see \eqnref{e:dimC})
\item $\vme(X) = v(X) - e(X) =
	\textrm{number of 0-cells} - \textrm{number of 1-cells}$
\item $d_{l(f)}$ is the categorical dimension of $l(f)$
\end{itemize}
\label{d:zcy}
\end{definition}

Note that the evaluation giving $Z(\cW,l)$ is well-defined
by \lemref{l:local-state-space-dual},
in that it varies functorially with a change of anchors
in boundary 3-cells, and is invariant under change of anchors
in interior 3-cells.

The factor of $1/2$ are meant to account for the fact that,
when a boundary is glued to another,
that boundary becomes part of the interior,
so the coefficients combine to a whole.

It is helpful to write $\ZCY(\cW)$ in terms of bases.
Choose some anchoring of $\cW$.
For each labeling $l$, choose a basis $\{\vphi_{C,\al}\}$
of $H(C,l)$ for each oriented 3-cell $C$, so that $C$ and $\ov{C}$
have dual bases.
Then
\begin{align}
\ZCY(\cW,l) &=
	\sum_{\al}
	\prod_{T \in \cW} \ZRT(\Gamma_T, l, \{\ov{\vphi}_{C,\al}\})
	\cdot
	\ev\big(\bigotimes_{ \substack{C \in \del T \\ T \in \cW} } \vphi_{C,\al} \big)
\\
&=
	\sum_\al\nolimits'
	\prod_{T \in \cW} \ZRT(\Gamma_T, l, \{\ov{\vphi}_{C,\al}\})
	\cdot
	\bigotimes_{C \in \del \cW} \vphi_{C,\al}
\label{e:cy-basis-2}
\\
&=
	\sum_\al\nolimits'
	\ZRT(\coprod \Gamma_T, l, \{\ov{\vphi}_{C,\al}\})
	\cdot
	\bigotimes_{C \in \del \cW} \vphi_{C,\al}
\end{align}

\begin{equation}
\ZCY(\cW)
= \cD^{\vme(\cW \backslash \del \cW) + \frac{1}{2}\vme(\del \cW)}
	\sum_l \prod_f d_{l(f)}^{n_f}
	\sum_\al\nolimits'
	\ZRT(\coprod \Gamma_T, l, \{\ov{\vphi}_{C,\al}\})
	\cdot
	\bigotimes_{C \in \del \cW} \vphi_{C,\al}
\label{e:cy-basis}
\end{equation}
where the sum $\sum_\al$ is over all assignments of a basis vector
$\vphi_{C,\al}$ to each oriented 3-cell $C$,
while $\sum_\al'$ is restricted to those assignments where
$\vphi_{C,\al}$ and $\vphi_{\ov{C},\al'}$ are dual
(others are eliminated by the evaluation);
$\Gamma_T$ is graph obtained from the union of anchors
in $\del T$,
and $\coprod \Gamma_T$ is the union of those graphs placed in a single $S^3$,
each in a ball disjoint from the others.
Each coupon of $\coprod \Gamma_T$ corresponds to an oriented 3-cell;
internal 3-cells appear twice (once for each orientation),
and boundary 3-cells only once (with the outward orientation).

\begin{theorem}
For an oriented combinatorial 4-manifold $\cW$, possibly with boundary,
$\ZCY(\cW)$ is independent of the choice of PLCW decomposition
in the interior;
that is, if $\cW'$ is another PLCW decomposition that agrees with $\cW$
on the boundary, then
$\ZCY(\cW') = \ZCY(\cW) \in H(\del \cW)$.
In particular, $\ZCY$ is a well-defined invariant of closed PL 4-manifolds.
\label{t:invariance-zcy}
\end{theorem}

\begin{proof}
By \corref{t:elementary-boundary},
we just need to show that elementary subdivision of an interior $k$-cell
$C$ of $\cW$ preserves $\ZCY(\cW)$. Let $\cW'$ be the new
PLCW decomposition after subdivision. We consider different $k$
in separate cases:

\textbf{Case k = 1:} The only thing that changes in $\ZCY(\cW)$ are
the number of internal 0- and 1-cells,
but $\vme(\cW \backslash \del \cW)$ remains constant.

The rest of the cases follow from arguments very similar to
the proof of invariance of the TV state sum in
\ocite{balsam-kirillov}*{Section 5}.

\textbf{Case k = 3:}
This follows from the fact that if $\{\vphi_\al^i\}, \{\phi_\beta^i\}$
are bases for
$\eval{A_1,\ldots,A_k,X_i}, \eval{X_i^*,B_1,\ldots,B_l}$,
and $\{\ov{\vphi}_\al^i\}, \{\ov{\phi}_\beta^i\}$
are their dual bases for
$\eval{X_i^*,A_k^*,\ldots,A_1^*}, \eval{B_l^*,\ldots,B_1^*,X_i}$,
then
$\{\sqrt{d_i}\ev_{X_i}(\vphi_\al^i \tnsr \phi_\beta^i)\},
\{\sqrt{d_i}\ev_{X_i}(\ov{\phi}_\beta^i \tnsr \ov{\vphi}_\al^i)\}$
are dual bases for
$\eval{A_1,\ldots,A_k,B_1,\ldots,B_l},
\eval{B_l^*,\ldots,B_1^*,A_k^*,\ldots,A_1^*}$.

More precisely,
let the new 2-face added be $F$,
splitting an old 3-cell $C$ of $\cW$ into $C_1$ and $C_2$.
There are exactly two 4-cells $T_1,T_2$ that sandwich $C$ in $\cW$,
thus also sandwich $C_1$ and $C_2$ in $\cW'$.
To fix notation, we take $C_1, C_2$ to have the outward orientation
with respect to $T_1$, and take $F$ to have the outward orientation
with respect to $C_1$.
For a labeling with $l(F) = X_i$,
$H(C_1,l) = \eval{\ldots,X_i}$ and $H(C_2,l) = \eval{X_i^*,\ldots}$;
choose bases $\{\vphi_\al^i\}, \{\phi_\beta^i\}$ for them.
Then in the $\ZRT(\cdots)$ coefficient in \eqnref{e:cy-basis},
for $\ZCY(\cW')$, the 3-cells $C_1,C_2$ contribute the subgraph on the left
(with the extra $d_i$ coming from the $d_{l(f)}^{n_f}$ term,
while for $\ZCY(\cW)$, the 3-cell $C$ contributes the subgraph on the right:
\[
\sum_{i,\al}
d_i
\begin{tikzpicture}
\node[morphism] (a) at (0,0) {$\al$};
\node[morphism] (b) at (1,0) {$\be$};
\draw (a) -- (b) node[pos=0.5, above] {$i$};
\draw (a) -- +(150:1cm);
\draw (a) -- +(180:1cm);
\draw (a) -- +(210:1cm);
\draw (b) -- +(30:1cm);
\draw (b) -- +(0:1cm);
\draw (b) -- +(-30:1cm);
\node at (0.5,-1) {$\del T_1$};
\node[morphism] (a) at (3.5,0) {$\be$};
\node[morphism] (b) at (4.5,0) {$\al$};
\draw (a) -- (b) node[pos=0.5, above] {$i$};
\draw (a) -- +(150:1cm);
\draw (a) -- +(180:1cm);
\draw (a) -- +(210:1cm);
\draw (b) -- +(30:1cm);
\draw (b) -- +(0:1cm);
\draw (b) -- +(-30:1cm);
\node at (4,-1) {$\del T_2$};
\end{tikzpicture}
\;\;\;=\;\;\;
\begin{tikzpicture}
\node[morphism] (a) at (0,0) {$\ga$};
\draw (a) -- +(150:1cm);
\draw (a) -- +(180:1cm);
\draw (a) -- +(210:1cm);
\draw (a) -- +(30:1cm);
\draw (a) -- +(0:1cm);
\draw (a) -- +(-30:1cm);
\node at (0,-1) {$\del T_1$};
\node[morphism] (a) at (2.5,0) {$\ga$};
\draw (a) -- +(150:1cm);
\draw (a) -- +(180:1cm);
\draw (a) -- +(210:1cm);
\draw (a) -- +(30:1cm);
\draw (a) -- +(0:1cm);
\draw (a) -- +(-30:1cm);
\node at (2.5,-1) {$\del T_2$};
\end{tikzpicture}
\]

\textbf{Case k = 4:}
This follows easily from \lemref{l:summation-separate}.

\textbf{Case k = 2:}
Let $F$ be the 2-cell to be elementarily subdivided.
We first use elementary subdivisions of 3-cells and 4-cells
to modify $\cW$ so that $F$ meets exactly one 3-cell.
Since $F$ is of codimension 2, the (germ of neighborhood
of $F$ in the) 3-cells
adjacent to $F$ are naturally cyclically ordered
(up to a choice of reversing the cyclic order).
Let $C_1,\ldots,C_m$ be the 3-cells in such an ordering;
note that $C_i$ are not necessarily distinct.
Let $T_i$ be 4-cell between $C_i$ and $C_{i+1}$ (indices taken
modulo $m$); as with $C_i$, the $T_i$'s are not necessarily distinct.

For each $i$, perform an elementary subdivision on $C_i$
that separates $F$ from the rest of the boundary;
let $C_i'$ be the new 3-cell that is adjacent to $F$.
Now for each $i$, perform an elementary subdivision on $T_i$
that separates $C_i'$ and $C_{i+1}'$ from the rest of the boundary;
let $T_i'$ be the new 4-cell that meets $C_i'$ and $C_{i+1}'$.
Note that $C_i'$, $T_i'$ are pairwise distinct.
Then we sequentially remove $C_2',C_3',\ldots,C_m'$
using inverse elementary subdivisions (of 4-cells),
leaving $C_1'$ as the only 3-cell meeting $F$.

We essentially performed the following operations,
but in one dimension higher:
\begin{figure}[ht]
\includegraphics[width=8cm]{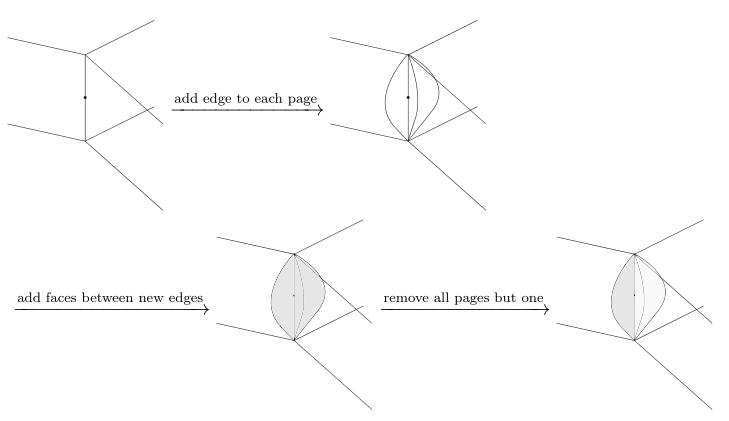}
\caption{Preparing for elementary subdivision
along 2-cell, here showing the same move but
in one dimension lower.
(Figure taken from \ocite{balsam-kirillov})}
\label{f:M1}
\end{figure}

Thus we may assume $F$ appears in the boundary of exactly
one 3-cell $C$,
and $C$ appears in the boundary of exactly one 4-cell $T$
(appearing twice, once for each orientation).
Now the invariance follows from the following identity
(the two nodes represent the oriented 3-cells $C, \ov{C}$,
and edges represent 2-cells;
on the right, the single edge between nodes is $F$,
and on the left, the two edges are from the two 2-cells
into which $F$ is subdivided):
\[
\frac{1}{\cD} \sum_{i,j} d_i d_j
\begin{tikzpicture}
\node[morphism] (a) at (0,0) {$\al$};
\node[morphism] (b) at (1,0) {$\al$};
\draw (a) to[out=30,in=150] (b);
\node at (0.5,0.4) {\small $i$};
\draw (a) to[out=-30,in=-150] (b);
\node at (0.5,-0.4) {\small $j$};
\draw (a) -- +(150:1cm);
\draw (a) -- +(180:1cm);
\draw (a) -- +(210:1cm);
\draw (b) -- +(30:1cm);
\draw (b) -- +(0:1cm);
\draw (b) -- +(-30:1cm);
\end{tikzpicture}
\;\;\;=\;\;\;
\sum_i d_i
\begin{tikzpicture}
\node[morphism] (a) at (0,0) {$\al$};
\node[morphism] (b) at (1,0) {$\al$};
\draw (a) -- (b) node[pos=0.5,above] {\small $i$};
\draw (a) -- +(150:1cm);
\draw (a) -- +(180:1cm);
\draw (a) -- +(210:1cm);
\draw (b) -- +(30:1cm);
\draw (b) -- +(0:1cm);
\draw (b) -- +(-30:1cm);
\end{tikzpicture}
\]

\end{proof}

\begin{corollary}
$\ZCY$ satisfies the gluing axiom:
if $\cW$ is a combinatorial 4-manifold with boundary
$\del \cW = \cM_0 \sqcup \cM \sqcup \overline{\cM}$,
and $\cW'$ is the manifold obtained by identifying
boundary components $\cM, \overline{\cM}$,
then
\begin{equation}
\ZCY(\cW') = \ev (\ZCY(\cW))
\end{equation}
where the $\ev$ applies the pairing of \eqnref{e:state-space-pairing}.
\end{corollary}

Given a combinatorial 4-manifold $\cW$,
we may arbitrarily partition its boundary components into two sets,
and call them \emph{incoming} and \emph{outgoing} boundary components,
respectively.
Denote by $\cM_{in}$ the disjoint union of incoming boundary components,
taken with inward orientation, and
$\cM_{out}$ the disjoint union of outgoing boundary components,
taken with outgoing orientation.
In other words, we treat $\cW$ as a cobordism
$\cW : \cM_{in} \to \cM_{out}$.
Then we may interpret
\[
\ZCY(W) \in H(\overline{\cM_{in}}) \tnsr H(\cM_{out})
\]
as a linear operator
\begin{align}
\label{e:zcy-linear-operator}
\ZCY(W) : H(\cM_{in}) &\to H(\cM_{out})
\\
\vphi &\mapsto \ev (\ZCY(W), \vphi)
\end{align}
obtained by applying the pairing of \eqnref{e:state-space-pairing}
to $\cM_{in}$.
Here we may use $W$, the underlying PL manifold of $\cW$,
because of \thmref{t:invariance-zcy}.

\begin{lemma}
Let $\cW$ be a combinatorial 4-manifold
with $\del \cW = \overline{\cM_{in}} \sqcup \cM_{out}$,
interpreted as a linear operator as in \eqnref{e:zcy-linear-operator} above.
Then its adjoint map is also given by $\ZCY(W)$:
\begin{equation*}
\begin{tikzcd}
H(\cM_{out})^* \ar[r, "\ZCY(W)^*"] \ar[d, "\simeq"]
	& H(\cM_{in})^* \ar[d, "\simeq"]
\\
H(\overline{\cM_{out}}) \ar[r, "\ZCY(W)"]
	& H(\overline{\cM_{in}})
\end{tikzcd}
\end{equation*}
\label{l:zcy-adjoint}
\end{lemma}

\begin{proposition}
\label{p:projectors-compose}
For a combinatorial closed 3-manifold $\cM$
with underlying PL manifold $M$,
the linear map
\begin{equation}
A_\cM := \ZCY(M \times I) : H(\cM) \to H(\cM)
\end{equation}
is a projector.  Moreover, the spaces
\begin{equation}
\ZCY(\cM) := \im ( A_\cM )
\end{equation}
are canonically identified by the system of isomorphisms
\[
A_{\cM', \cM''} := \ZCY(M \times I) : H(\cM') \to H(\cM'')
\]
where $\cM', \cM''$ are PLCW structures on $M$.
Furthermore, $A_\cM$ is self-adjoint,
and the pairing of \eqnref{e:state-space-pairing}
restricts to a non-degenerate pairing
\[
\ZCY(\cM) \tnsr \ZCY(\overline{\cM}) \to \kk
\]
which transforms functorially with $A_{\cdot,\cdot}$.
Thus, we may define $\ZCY(M) := \ZCY(\cM)$ for some PLCW structure
$\cM$ on $M$,
which we call the \emph{state space of $M$},

We will often simple denote $A_{\cM,\cM'}$ by $A$
when the context is clear.
\end{proposition}
\begin{proof}
Simple exercise left to the reader.
\end{proof}

By the existence of collar neighborhoods of boundaries
and invariance under choice of PLCW decomposition,
$\ZCY(W) \in \ZCY(\del W)$ and is well-defined.
We also have a canonical pairing
\begin{equation}
\label{e:zcy-pairing}
\ZCY(M) \tnsr \ZCY(\overline{M}) \to \kk
\end{equation}
inherited from \eqnref{e:state-space-pairing}.

\begin{theorem}
The invariants $\ZCY$ above define a TQFT.
\label{t:zcy-tqft}
\end{theorem}

It is often more convenient to consider elements of $H(\cM)$
instead of the subspace $\ZCY(\cM) \subset H(\cM)$:

\begin{definition}
A \emph{lift} of a state $\vphi \in \ZCY(\cM)$ is a pre-state
$\wdtld{\vphi} \in H(\cM)$ that projects to $\vphi$
under $A_\cM$.
\label{d:H-lift}
\end{definition}

\begin{definition}
Let $\cM,\cM'$ be PLCW structures on a closed 3-manifold $M$.
We say two pre-states $\vphi \in H(\cM), \vphi' \in H(\cM')$
are \emph{equivalent as states}, or simple \emph{equivalent},
denoted $\vphi \preequiv \vphi'$,
if for some (and hence any) PLCW decomposition $\cM''$ of $M$,
$A_{\cM,\cM''} (\vphi) = A_{\cM',\cM''}(\vphi')$.
\label{d:pre-state-equiv}
\end{definition}

\begin{lemma}
For a combinatorial 4-manifold $\cW$,
let $\wdtld{\vphi}$ be a lift of $\vphi \in H(\ov{\del \cW})$.
Then
\[
\ev (\ZCY(\cW), \vphi) = \ev (\ZCY(\cW), \wdtld{\vphi})
\]
\label{l:H-lift}
\end{lemma}
\begin{proof}
By the invariance of $\ZCY$ with respect to PLCW decomposition,
$A_{\del \cW} (\ZCY(\cW)) = \ZCY(\cW)$, thus
\[
\ev(\ZCY(\cW), \wdtld{\vphi}) = \ev(A(\ZCY(\cW)), \wdtld{\vphi})
= \ev(\ZCY(\cW), A(\wdtld{\vphi}))
= \ev(\ZCY(\cW), \vphi)
\]
\end{proof}

There is a canonical element of $H(\cM)$
in the subspace $H(\cM,l \equiv \one)$
where $l$ labels all 2-cells with $\one$:
\begin{equation}
\label{e:emptystate}
\emptystate{\cM} := \bigotimes \id_\one \in
\bigotimes \Hom(\one, \one \tnsr \cdots \tnsr \one)
= H(\cM, l \equiv \one) \subset H(\cM)
\end{equation}
which we call the \emph{empty pre-state}.
We claim that these are, up to factors of $\cD$,
equivalent as states;
these will be discussed further later
(see \eqnref{e:empty-state-relate} and \secref{s:sk-equiv}).


However, we warn the reader that for lifts,
$\wdtld{\vphi} \in H(\cM),\wdtld{\vphi'} \in H(\overline{\cM})$
of $\vphi \in \ZCY(\cM), \vphi' \in \ZCY(\overline{\cM})$,
the pairing \eqnref{e:state-space-pairing}
and \eqnref{e:zcy-pairing} disagree;
they are related by
\begin{equation}
\ev (\vphi, \vphi')
= \ev (\wdtld{\vphi}, A(\wdtld{\vphi}))
\neq 
(\wdtld{\vphi}, \wdtld{\vphi'})
\label{e:pairing-diff}
\end{equation}

Thus, we make the following definition:

\begin{definition}
\label{d:state-space-pairing-reduced}
For $\vphi \in H(\cM), \vphi' \in H(\ov{\cM})$,
their \emph{reduced pre-state space pairing} is defined as
\begin{equation}
\evbar (\vphi, \vphi')
:= \ev (\vphi, A(\vphi'))
\label{e:pairing-A}
\end{equation}

Equivalently, treating $M \times I$ as a cobordism
$M \sqcup \ov{M} \to \empty$,
$\ZCY(M \times I)$ defines the same pairing
$H(\cM) \tnsr H(\ov{\cM}) \to \kk$.
\end{definition}

Note that if either of $\vphi$ or $\vphi'$ is in the
image of $A$, i.e. in $\ZCY$,
then $\evbar(\vphi, \vphi') = \ev(\vphi, \vphi')$.



\begin{remark}
Warning.
We point out another potentially confusing thing.
A 4-ball ball $\cW \simeq B^4$
defines a Crane-Yetter invariant
$\ZCY(\cW) \in \ZCY(\del \cW)$.
On the other hand, if we ignore the PLCW decomposition
in the interior of $\cW$ and
treat it as a cell-like 4-ball,
then we have $Z(\cW, l) \in H(\del \cW, l)$.
The $H(\del \cW,l)$ component of $\ZCY(\cW)$
does not equal $Z(\cW,l)$
(however the difference is simply by a factor
coming from the $d_{l(f)}^{n_f}$ term).

Personally, I make the mistake of conflating the two
most often when I am using them as functionals
$H(\ov{\del \cW},l) \to \kk$,
and particularly when using the skein definition
of state spaces.
This will be more apparent in computations
as we will see in \secref{s:signature}.
(see \eqnref{e:z-4cell}).
\label{r:z-zcy-confusion}
\end{remark}

To summarize the section,
we first have local state spaces $H(C,l)$
for each 3-cell $C$
(\defref{d:local-state-space}).
The pre-state space of $\cM$ is the tensor product
of local state spaces, summed over all labelings
(\defref{d:pre-state-space}).
A 4-cell in a combinatorial 4-manifold $\cW$
is given its local invariant
$Z(T,l) \in H(\del T,l)$,
defined as the (dual to) evaluating colored ribbon graphs
(\defref{d:ZCY3cell}).
The Crane-Yetter invariant $\ZCY(\cW) \in H(\del \cW)$
is obtained by taking the tensor product of local invariants,
applying the pairing \eqnref{e:pairing} to matching 3-cells,
and summing over all labelings with certain coefficients
(\defref{d:zcy}).
Finally, the state space $\ZCY(M)$
is the ``common'' subspace of $H(\cM)$
that is the image of the cylinder map $\ZCY(\cM \times I)$.


\subsection{Equivalence of $\ZCY$ with Crane-Yetter}
\par \noindent

Let us recall the original construction of the Crane-Yetter
state sum, as laid out in \ocite{CY} and generalized in \ocite{CKY}.

As before, we fix a premodular category $\cA$
(``semisimple tortile categories'' in \ocite{CKY}).
We will present their construction using our notation and conventions,
but let us point out one potentially confusing difference.  \footnote{
In \ocite{CKY},
\begin{itemize}
\item $\cA[X,Y] = \Hom_\cA(X,Y)$.
\item Composition is written left to right:
	$\circ : \cA[X,Y] \tnsr \cA[Y,Z] \to \cA[X,Z]$.
\item Morphisms in diagrams go from top to bottom
\end{itemize}
}

Let $S$ be a simple object, and $\cB = \{b_\al\}$ a basis of $\Hom(S,X)$.
Then $\overline{\cB} = \{\overline{b_\beta}\}$ is the basis of $\Hom(X,S)$
such that $\overline{b_\beta} \circ b_\al = \delta_{\al,\beta}\id_S$.
In other words, $\overline{\cB}$ is the dual basis to $\cB$
with respect to the pairing
\begin{align}
\Hom(X,S) \tnsr \Hom(S,X) &\to \kk \\
\psi \tnsr \vphi &\mapsto a, \;\;\text{where } \psi \circ \vphi = a \cdot \id_S
\label{e:pairing-cky}
\end{align}
which is $1/d_S$ times the pairing defined by \eqnref{e:pairing}.
Thus, $\overline{b_\al} = d_S \cdot b^\al$,
where $\{b^\al\}$ is the dual basis to $\{b_\al\}$
as seen in \eqnref{e:summation_convention}.

For each triple of simple objects $i,j,k \in \Irr(\cA)$,
choose a basis $\cB_k^{ij}$ of $\Hom(X_i X_j, X_k)$,
and write $\cB = \bigsqcup \cB_k^{ij}$.

Let $\TT$ be an ordered triangulation on a 4-manifold $M$
(that is, the set of vertices are ordered).
Denote by $\TT_{(i)}$ the set of $i$-simplices of $\TT$.
The ordering defines, on each simplex, an orientation as follows:
if $\{v_0, v_1,\ldots,v_i\}$ are the vertices of an $i$-simplex ($i>0$)
in increasing order, then the ordered set of directed line segments
$(\overrightarrow{v_0v_1},\ldots,\overrightarrow{v_0v_i})$
defines an orientation at $v_0$, which is extended to the rest of the simplex.
We call this the \emph{ordering-orientation} on $\xi$.

\begin{definition}[\ocite{CKY}, Definition 3.1]
A CSB-coloring of $\TT$ is a triple of maps
$\lmb = (\lmb_0, \lmb^+, \lmb^-)$, where
\begin{align*}
\lmb_0 &: \TT_{(2)} \sqcup \TT_{(3)}  \to \Irr(\cA)
\\
\lmb^+ &: \TT_{(3)} \to \cB
\\
\lmb^- &: \TT_{(3)} \to \cB
\end{align*}
such that for 3-simplex $\tau$, letting $\sigma_k$ denote the boundary 2-simplex
of $\tau$ opposite the $k$-th vertex (w.r.t. the ordering),
\begin{align*}
\lmb^+(\tau) &\in
\cB_{\lmb_0(\tau)}^{\lmb_0(\sigma_0), \lmb_0(\sigma_2)}
\\
\lmb^-(\tau) &\in
\cB_{\lmb_0(\tau)}^{\lmb_0(\sigma_1), \lmb_0(\sigma_3)}
\end{align*}
The set of CSB-colorings of $\TT$ is denoted by
$\Lambda_{CBS}(\TT)$.
\end{definition}

Observe that the ordering-orientation on $\sigma_0,\sigma_2$
agrees with the outward orientation on $\del \tau$
(when $\tau$ is given the ordering-orientation),
while the orientations disagree for $\sigma_1,\sigma_3$.

We sometimes denote $\lmb_0$ simply by $\lmb$.

\begin{definition}
Let $\lmb$ be a CSB-coloring of an ordered triangulation $\TT$
of an oriented 4-manifold $M$,
and let $\xi$ be a 4-simplex in $\TT$.
We define the quantity $||\lmb,\xi|| \in \kk$ as follows.
Let vertices of $\xi$ be $v_0,\ldots,v_4$ in increasing order.
Suppose the ordering-orientation on $\xi$ agrees with the orientation of $M$.
Then $||\lmb,\xi||$ is defined as the evaluation of the
graph in \figref{f:network-evaluation}.
\footnote{When $\cA = \Rep U_q sl_2$ (with $q$ a root of unity),
with appropriate choice of bases,
$||\lmb,\xi||$ is the quantum 15-j symbol
as described in \ocite{CY}.}
\begin{figure}[ht]
\begin{tikzpicture}
\node[morphism] (v2) at (-1,2) {\Large $\widehat{v_2}$};
\node[morphism] (v0) at (1,2) {\Large $\widehat{v_0}$};
\node[morphism] (v1) at (-2,0) {\Large $\widehat{v_1}$};
\node[morphism] (v4) at (0,0) {\Large $\widehat{v_4}$};
\node[morphism] (v3) at (2,0) {\Large $\widehat{v_3}$};
\draw (v2) to[out=-60, in=60] (v4);
\draw (v2) .. controls +(60:0.8cm) and +(120:0.3cm) .. (0.7,1)
	.. controls +(-60:0.3cm) and +(120:0.5cm) .. (v3);
\draw (v2) .. controls +(120:0.8cm) and +(120:1cm) .. (0.4,2)
	.. controls +(-60:0.2cm) and +(-120:0.4cm) .. (v0);
\draw (v2) to[out=-120, in=-60] (v1);
\draw[overline] (v0) to[out=-60, in=120] (v4);
\draw (v0) to[out=60, in=60] (v3);
\draw (v0) to[out=120, in=90] (-2,2) to[out=-90, in=60] (v1);
\draw (v1) .. controls +(120:1cm) and +(180:4cm) .. (0,-1)
	.. controls +(0:2cm) and +(-60:0.5cm) .. (v3);
\draw (v1) .. controls +(-120:0.8cm) and +(-120:1cm) .. (v4);
\draw (v4) to[out=-60, in=-120] (v3);
\end{tikzpicture}
\;\;
\begin{tikzpicture}
\node at (0,2.5) {Type I};
\node[morphism] (+) at (0,1) {$\lmb^+$};
\node[morphism] (-) at (0,0) {$\overline{\lmb^-}$};
\draw[midarrow] (+) -- (-);
\draw[midarrow_rev] (+) -- +(135:1cm);
\draw[midarrow_rev] (+) -- +(45:1cm);
\draw[midarrow] (-) -- +(-135:1cm);
\draw[midarrow] (-) -- +(-45:1cm);
\node at (0.5,0.5) {$\lmb(\widehat{v_j})$};
\node at (-0.7,2) {$\lmb(f_{j,0})$};
\node at (0.7,2) {$\lmb(f_{j,2})$};
\node at (-0.7,-1) {$\lmb(f_{j,1})$};
\node at (0.7,-1) {$\lmb(f_{j,3})$};
\end{tikzpicture}
\;\;
\begin{tikzpicture}
\node at (0,2.5) {Type II};
\node[morphism] (+) at (0,1) {$\overline{\lmb^+}^*$};
\node[morphism] (-) at (0,0) {${\lmb^-}^*$};
\draw[midarrow=0.6] (-) -- (+);
\draw[midarrow] (+) -- +(135:1cm);
\draw[midarrow] (+) -- +(45:1cm);
\draw[midarrow_rev] (-) -- +(-135:1cm);
\draw[midarrow_rev] (-) -- +(-45:1cm);
\node at (0.5,0.5) {$\lmb(\widehat{v_j})$};
\node at (-0.7,2) {$\lmb(f_{j,2})$};
\node at (0.7,2) {$\lmb(f_{j,0})$};
\node at (-0.7,-1) {$\lmb(f_{j,3})$};
\node at (0.7,-1) {$\lmb(f_{j,1})$};
\end{tikzpicture}
\caption{Evaluation of the left graph yields $||\lmb,\xi||$;
the circles labeled with $\widehat{v_j}$ contain subgraphs,
defined on the right;
for $j=0,2,4$, it contains the Type I subgraph,
and for $j=1,3$, it contains the Type II subgraph.
Here $f_{j,k}$ is the 2-simplex obtained from $\widehat{v_j}$ by
omitting the $k$-th vertex
(for example, $\widehat{v_2} = \{v_0,v_1,v_3,v_4\}$,
so $f_{2,2} = \{v_0,v_1,v_4\}$).}
\label{f:network-evaluation}
\end{figure}
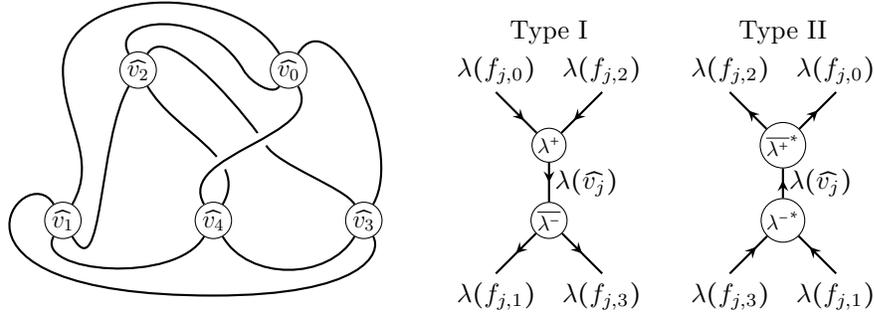

If the ordering-orientation on $\xi$ differs from the orientation of $M$,
then we use the mirror image of the graph in \figref{f:network-evaluation},
and the choice of subgraphs for each $\widehat{v_j}$ is flipped
(i.e. for $j=0,2,4$, Type II subgraph is used,
while for $j=1,3$, Type I subgraph is used).

\end{definition}

\begin{definition}
For a CSB-coloring $\lmb$ of an ordered triangulation $\TT$,
we define
\[
\eval{\eval{\lmb}} =
\cD^{n_0-n_1}
\prod_{\sigma : 2-splx.} d_{\lmb(\sigma)}
\prod_{\tau : 3-splx.} d_{\lmb(\tau)}^\inv
\prod_{\xi : 4-splx.} ||\lmb,\xi||
\]
\label{d:csb-eval}
where $n_0,n_1$ are the number of vertices and edges, respectively.
\end{definition}

\begin{theorem}[\ocite{CKY}, Theorem 3.2]
The state-sum
\[
CY(M) = \sum_{\lmb \in \Lambda_{CSB}(\TT)} \eval{\eval{\lmb}}
\]
is an invariant of the PL manifold $M$.
\label{t:cky-original}
\end{theorem}

The main result of this section is to show that
this definition is equivalent to ours:

\begin{theorem}
For a closed PL manifold $M$,
\[
CY(M) = \ZCY(M)
\]
\label{t:cy-equiv}
\end{theorem}

\begin{proof}
Fix an ordered triangulation $\TT$ on $M$, so that we have
\[
CY(M) = \cD^{n_0-n_1} \sum_{\lmb \in \Lambda_{CSB}(\TT)}
\prod_{\sigma : 2-splx.} d_{\lmb(\sigma)}
\prod_{\tau : 3-splx.} d_{\lmb(\tau)}^\inv
\prod_{\xi : 4-splx.} ||\lmb,\xi||
\]

Consider the PLCW structure $\cM = \cM_\TT$ on $M$ derived from $\TT$,
where we apply, to each 3-simplex $\tau = \{a,b,c,d\}$ (with $a<b<c<d$),
the elementary subdivision move that adds the 2-cell $f_\tau$ with boundary
$\overline{abcda}$, i.e. it separates $\tau$ into a top half
$\tau^+$ (bounded by 2-cells $\{a,b,d\}, \{b,c,d\}, f_\tau$)
and bottom half $\tau^-$ (bounded by 2-cells $\{a,b,c\}, \{a,c,d\}, f_\tau$).
They are called top and bottom by virtue of their ordering-orientations:
the ordering-orientations of $\{a,b,d\}$ and $\{b,c,d\}$
coincide with the outward-orientation with respect to
the ordering orientation of $\tau$,
while it is the opposite for $\{a,b,c\}$ and $\{a,c,d\}$.
Thus, the 2-cells of $\cM$ come in two varieties,
the \emph{triangular} type (coming from $\TT$),
and \emph{quadrilateral} type (the added 2-cells $f_\tau$).

\[
\begin{tikzpicture}
\tikzmath{
	\ax = -1;
	\ay = 0;
	\bx = 2;
	\by = 2.5;
	\cx = 2;
	\cy = 1;
	\dx = -2;
	\dy = 3;
}
\node[emptynode, label=-170:{$a$}] (a) at (\ax,\ay) {};
\node[emptynode, label=0:{$b$}] (b) at (\bx,\by) {};
\node[emptynode, label=-90:{$c$}] (c) at (\cx,\cy) {};
\node[emptynode, label=180:{$d$}] (d) at (\dx,\dy) {};
\draw (a) -- (b);
\draw (a) -- (d);
\draw (b) -- (c);
\draw (b) -- (d);
\draw[opacity=0.4] (c) -- (d);
\draw (a) -- (c);
\node[emptynode] (abcd) at
	(\ax/4+\bx/4+\cx/4+\dx/4,\ay/4+\by/4+\cy/4+\dy/4) {};
\draw[opacity=0.2] (abcd) -- (a);
\draw[opacity=0.2] (abcd) -- (b);
\draw[opacity=0.2] (abcd) -- (c);
\draw[opacity=0.2] (abcd) -- (d);
\fill[opacity=0.3] (abcd) -- (a.center) -- (b) -- cycle;
\fill[opacity=0.3] (abcd) -- (c.center) -- (b) -- cycle;
\fill[opacity=0.3] (abcd) -- (c.center) -- (d) -- cycle;
\fill[opacity=0.3] (abcd) -- (a.center) -- (d) -- cycle;
\draw[line width=0.4pt] (\ax*0.8+\dx*0.2+0.2,\ay*0.8+\dy*0.2+0.2)
	-- +(-150:1.2cm) node[left] {$f_\tau$};
\draw[line width=0.4pt] (\bx*0.7+\dx*0.3,\by*0.7+\dy*0.3-0.2)
	-- +(60:0.6cm) node[above] {$\tau_+$};
\draw[line width=0.4pt] (\ax*0.3+\cx*0.7,\ay*0.3+\cy*0.7+0.2)
	-- +(-60:0.6cm) node[below] {$\tau_-$};
\end{tikzpicture}
\]

Since $M$ is closed, $v(\cM) = n_0, e(\cM) = n_1$,
so the exponents of $\cD$ agree.
Every 2-cell in $\cM$ corresponds to either an original 2-simplex of $\TT$
or to a 3-simplex of $\TT$ (as its separating 2-cell $f_\cdot$).
Thus, we have
\begin{align*}
\ZCY(\cM) &= \cD^{x(\cW)} \sum_l \prod_f d_{l(f)}^{n_f} Z(\cM,l)
\\
&= \cD^{n_0 - n_1} \sum_l \prod_{\sigma : 2-splx.} d_{l(\sigma)}
	\prod_{\tau : 3-splx.} d_{l(f_\tau)}
	Z(\cM,l)
\end{align*}
which is intentionally written to resemble $CY(M)$.
There are two main differences to note: the summing index
(labelings vs. CSB-colorings) and the $d_-^\pm$ term
in the product over 3-simplices.

The second difference arise simply because of the different
convention of dual pairing used.
Namely, recall that $\overline{b_\al} = d_S b^\al$
for a basis $\{b_\al\}$ of $\Hom(S,X)$.
Thus, if we define $|\lmb,\xi|$ to be the evaluation of the same
graph as in $||\lmb,\xi||$, except that we use
$b^\al$ instead of $\overline{\lmb^\pm} = \overline{b_\al}$
in the Type I,II subgraphs.
Then
\[
||\lmb,\xi|| = |\lmb,\xi| \prod_{\tau \subset \xi} d_{\lmb(\tau)}
\]
where the product is over 3-simplices of $\xi$.
Since each 3-simplex appears in exactly two 4-simplices
(and hence exactly two terms $||\lmb,\xi||$),
we see that
\[
CY(M) = \cD^{n_0-n_1} \sum_{\lmb \in \Lambda_{CSB}(\TT)}
\prod_{\sigma : 2-splx.} d_{\lmb(\sigma)}
\prod_{\tau : 3-splx.} d_{\lmb(\tau)}
\prod_{\xi : 4-splx.} |\lmb,\xi|
\]

Thus, we move on to addressing the first difference,
that of summing over labelings vs. CSB-colorings,
for which it suffices to show that
\[
Z(\cM,l) = \sum_{\lmb | \lmb_0 = l} \prod_\xi |\lmb,\xi|
\]
From \eqnref{e:cy-basis-2},
\[
Z(\cM,l) = \sum_\al\nolimits' \prod_\xi \ZRT(\Gamma_\xi, l,
\{\vphi_{\overline{C},\al}\})
\]
where recall that $\Gamma_\xi$ is the union of anchors
in $\del \xi$, a basis $\{\vphi_{C,\al}\}$ of $H(C,l)$
is chosen for each oriented 3-cell $C$ such that $C$ and $\bar{C}$
have dual bases (w.r.t. \eqnref{e:pairing}),
and the sum $\sum_\al\nolimits'$ is summed over all choices of assignments
$C \mapsto \vphi_{C,\al}$ such that the assignments of $C$ and $\bar{C}$ are dual.
We may choose the basis of $H(C,l)$ to be the appropriate $\cB_k^{ij}$;
then it is apparent that the sum $\sum_\al\nolimits'$
is the same as $\sum_{\lmb | \lmb_0 = l}$.
With this choice of basis, it remains to show that
there is a choice of anchoring of $\cM$ such that for any 4-simplex $\xi$,
\begin{equation}
|\lmb,\xi| = \ZRT(\Gamma_\xi, \lmb_0, \{\overline{\lmb^\pm}\})
\end{equation}
where $\overline{\lmb^\pm}$ is the dual to $\lmb^\pm$ with respect to the
pairing of \eqnref{e:pairing} (not from \eqnref{e:pairing-cky}).

We claim that the anchors as in \figref{f:anchor-tetrahedron}
fits the bill.
\figref{f:anchors-xi} show how the anchors fit together
when the ordering orientation of $\xi$ agrees with the orientation of $M$,
and \figref{f:anchors-xi-simplified} show a more simplified version,
which can be easily seen to match \figref{f:network-evaluation},
both in terms of the graph and the labels.

Let us point out a sticky orientation matter that may be confusing.
Recall that our convention for drawing figures/graphical calculus
is that we follow the right-hand rule (see Convention \ref{cvn:ambient-ort}).
When $\xi = \{v_0,\ldots,v_4\}$ has the ordering orientation,
the face $\hat{v}_4 = \{v_0,\ldots,v_3\}$ has the outward-orientation
given by
$(\overrightarrow{v_0v_1},\overrightarrow{v_0v_2},\overrightarrow{v_0v_3})$.
This means that \figref{f:anchors-xi} is drawing $\overline{\del \xi}$.
This is exactly what we want, since the Reshetikhin-Turaev evaluation
$\ZRT(\Gamma_\xi, \lmb_0, \{\overline{\lmb^\pm}\})$
takes place in $\overline{\del \xi}$, not $\del \xi$.

\begin{figure}[ht]
\begin{tikzpicture}
\tikzmath{
	\ax = 0;
	\ay = 0;
	\bx = 4;
	\by = -2;
	\cx = 6;
	\cy = 1;
	\dx = 3;
	\dy = 4;
}
\node at (1,-1.4) {$\tau = \{a,b,c,d\}$};
\node at (1,-1.8) {$a < b < c < d$};
\node[emptynode, label=-170:{$a$}] (a) at (\ax,\ay) {};
\node[emptynode, label=-80:{$b$}] (b) at (\bx,\by) {};
\node[emptynode, label=20:{$c$}] (c) at (\cx,\cy) {};
\node[emptynode, label=90:{$d$}] (d) at (\dx,\dy) {};
\draw (a) -- (b);
\draw (a) -- (d);
\draw (b) -- (c);
\draw (b) -- (d);
\draw (c) -- (d);
\draw[opacity=0.4] (a) -- (c);
\node[dotnode, label=120:{$m_{ad}$}] (ad) at (\ax/2 +\dx/2, \ay/2 + \dy/2) {};
\node[dotnode, label=-30:{$m_{bd}$}] (bd) at (\bx/2 +\dx/2, \by/2 + \dy/2) {};
\node[dotnode, label=-30:{$m_{bc}$}] (bc) at (\bx/2 +\cx/2, \by/2 + \cy/2) {};
\node[dotnode, label=-90:{$m_{ac}$}] (ac) at (\ax/2 +\cx/2, \ay/2 + \cy/2) {};
\node[dotnode, label=-100:{$m_{abd}$}] (abd)
	at (\ax/3 + \bx/3 +\dx/3, \ay/3 + \by/3 +\dy/3) {};
\node[dotnode, label=70:{$m_{bcd}$}] (bcd)
	at (\bx/3 + \cx/3 +\dx/3, \by/3 + \cy/3 +\dy/3) {};
\node[dotnode, label=-100:{$m_{abc}$}] (abc)
	at (\ax/3 + \bx/3 +\cx/3, \ay/3 + \by/3 +\cy/3) {};
\node[dotnode, label=70:{$m_{acd}$}] (acd)
	at (\ax/3 + \cx/3 +\dx/3, \ay/3 + \cy/3 +\dy/3) {};
\draw (ad) -- (abd) -- (bd) -- (bcd) -- (bc);
\draw[opacity=0.3] (bc) -- (abc) -- (ac) -- (acd) -- (ad);
\tikzmath{
	\madx = 0;
	\mady = 0;
	\mbdx = 2;
	\mbdy = -2;
	\mbcx = 4;
	\mbcy = 0;
	\macx = 2;
	\macy = 2;
}
\begin{scope}[shift={(8,1)}]
\node[emptynode, label=180:{$m_{ad}$}] (mad) at (\madx,\mady) {};
\node[emptynode, label=-90:{$m_{bd}$}] (mbd) at (\mbdx,\mbdy) {};
\node[emptynode, label=0:{$m_{bc}$}] (mbc) at (\mbcx,\mbcy) {};
\node[emptynode, label=90:{$m_{ac}$}] (mac) at (\macx,\macy) {};
\node[emptynode, label=-135:{$m_{abd}$}] (mabd)
	at (\madx/2+\mbdx/2,\mady/2+\mbdy/2) {};
\node[emptynode, label=-45:{$m_{bcd}$}] (mbcd)
	at (\mbcx/2+\mbdx/2,\mbcy/2+\mbdy/2) {};
\node[emptynode, label=45:{$m_{abc}$}] (mabc)
	at (\macx/2+\mbcx/2,\macy/2+\mbcy/2) {};
\node[emptynode, label=135:{$m_{acd}$}] (macd)
	at (\madx/2+\macx/2,\mady/2+\macy/2) {};
\node[small_morphism] (mbot) at
	(\madx/3+\mbdx/3+\mbcx/3,\mady/3+\mbdy/3+\mbcy/3) {};
\node[small_morphism] (mtop) at
	(\madx/3+\macx/3+\mbcx/3,\mady/3+\macy/3+\mbcy/3-0.1) {};
\node[emptynode] (mbot') at
	(\madx/3+\mbdx/3+\mbcx/3,\mady/3+\mbdy/3+\mbcy/3-0.4) {};
\draw[ribbon] (mbot) to[out=-30,in=30] (mbot')
	to[out=-150,in=45] (\madx/2+\mbdx/2,\mady/2+\mbdy/2);
\draw[ribbon,overline] (mbot) to[out=-150,in=150] (mbot')
	to[out=-30,in=135] (\mbcx/2+\mbdx/2,\mbcy/2+\mbdy/2);
\draw[ribbon] (mtop) to[out=10,in=-135] (\macx/2+\mbcx/2,\macy/2+\mbcy/2);
\draw[ribbon] (mtop) to[out=170,in=-45] (\macx/2+\madx/2,\macy/2+\mady/2);
\draw[ribbon] (mbot) -- (mtop);
\draw (mad) -- (mbd) -- (mbc) -- (mac) -- (mad) -- (mbc);
\draw[line width=0.4pt] (\madx*0.2+\mbcx*0.8,\mady*0.2+\mbcy*0.8)
	-- +(-30:1.5cm) node[right] {$f_\tau$};
\draw[line width=0.4pt] (\mbdx*0.8+\mbcx*0.2-0.2,\mbdy*0.8+\mbcy*0.2+0.2)
	-- +(-30:1.5cm) node[right] {$\tau^+$};
\draw[line width=0.4pt] (\macx*0.8+\mbcx*0.2-0.2,\macy*0.8+\mbcy*0.2-0.2)
	-- +(30:1.5cm) node[right] {$\tau^-$};
\end{scope}
\end{tikzpicture}
\caption{Choice of anchor for $\tau^\pm$.
$m_{ab}, m_{abc}$ stand for the midpoint and centroids of
$ab, abc$ respectively.
The anchors lie in a disk (except at the half braiding),
a ``normal quadrilateral'' in the sense of normal surface theory,
as seen on the right;
this disk is not part of the PLCW structure,
it is only there to guide the definition of the anchors.
The line connecting $m_{ad}$ to $m_{bc}$ is the intersection
of $f_\tau$ with this disk.
Note that the intersection of the anchor with the 2-simplex
$efg$, with $e < f < g$, is along a segment lying on
the line segment $\overline{m_{eg} m_{egh}}$
(is irrespective of which 3-simplex the anchor belongs to),
so the anchoring condition of agreeing at 2-cells is satisfied.
}
\label{f:anchor-tetrahedron}
\end{figure}
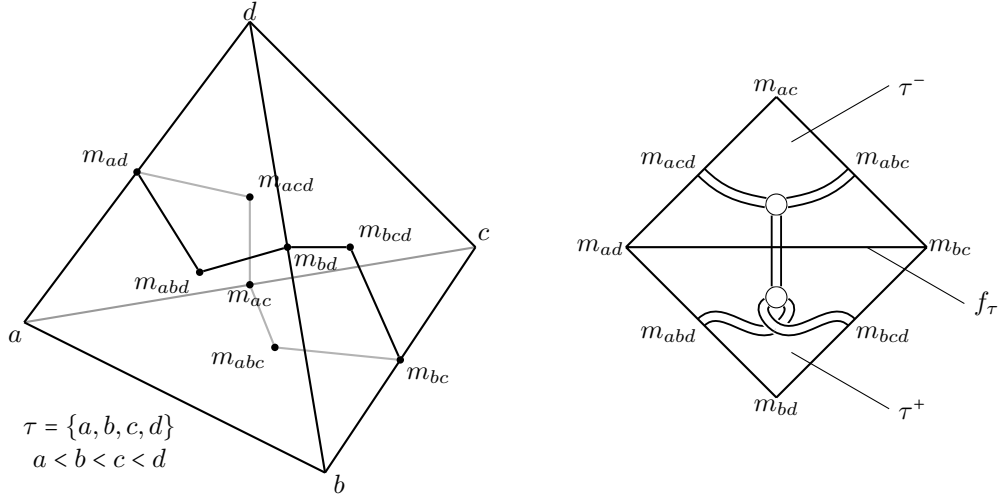

It remains to check that, for a 4-simplex $\xi$,
the anchors in $\del \xi$ glue up to be equivalent to
the graph of \figref{f:network-evaluation},
and that the labelings by objects and morphisms agree.

\begin{figure}[ht]
\begin{tikzpicture}
\node at (1,0.3) {$\hat{v}_4$};
\node at (0.65,2.3) {\smallerer $\hat{v}_4^+$};
\node at (0.85,1.75) {\smallerer $\hat{v}_4^-$};
\node[emptynode, label=-90:0] (40) at (0,0) {};
\node[emptynode, label=0:1] (41) at (2,4) {};
\node[emptynode, label=-135:2] (42) at (0,2) {};
\node[emptynode, label=180:3] (43) at (-2,4) {};
\node[emptynode] (4012) at (0.9,1.4) {};
\node[emptynode] (4013) at (-0.4,2.7) {};
\node[emptynode] (4023) at (-0.5,1) {};
\node[emptynode] (4123) at (0.2,3.2) {};
\node[ellipse_morphism={20},minimum size=10pt] (v4+) at (0.2,2.3) {\;\;\;};
\node[ellipse_morphism={20},minimum size=10pt] (v4-) at (0.4,1.7) {\;\;\;};
\draw[ribbon] (v4-) to[out=-120,in=60] (4023);
\draw[ribbon] (v4-) to[out=-45,in=135] (4012);
\draw[ribbon] (v4-) -- (v4+);
\draw[ribbon] (v4+) to[out=60,in=-30] (4013);
\draw[ribbon,overline] (v4+) to[out=180,in=-135] (4123);
\draw[thin_overline={1}] (40) -- (42) -- (43) -- (41) -- (42);
\draw (40) to[out=30,in=-90] (41);
\draw (40) to[out=150,in=-90] (43);
\draw[opacity=0.3] (0,0.8)
	to[out=30,in=-135] (4012)
	to[out=45,in=-60] (1,3)
	to[out=180,in=-45] (4123)
	to[out=135,in=-80] (-0.3,4)
	to[out=-90,in=60] (4013)
	to[out=-120,in=30] (-1.6,1.8)
	to[out=-45,in=150] (4023)
	to[out=-30,in=150] (0,0.8);
\begin{scope}[shift={(-3,-4)}]
\node at (-1,0.5) {$\hat{v}_1$};
\node at (0,2.) {\smallerer $\hat{v}_1^+$};
\node at (-1.15,2.25) {\smallerer $\hat{v}_1^-$};
\node[emptynode, label=-90:0] (10) at (0,0) {};
\node[emptynode, label=0:2] (12) at (1,4) {};
\node[emptynode, label=-135:3] (13) at (-2,6) {};
\node[emptynode, label=180:4] (14) at (-0.3,3) {};
\node[emptynode] (1024) at (0.4,1.5) {};
\node[emptynode] (1034) at (-0.8,1.6) {};
\node[emptynode] (1234) at (0.1,3.8) {};
\node[emptynode] (1023) at (-1.3,3.3) {};
\node[ellipse_morphism={90},minimum size=10pt] (v1+) at (0,2.5) {\;\;\;};
\node[ellipse_morphism={120},minimum size=10pt] (v1-) at (-0.8,2.5) {\;\;\;};
\draw[ribbon] (v1-) to[out=10,in=170] (v1+);
\draw[ribbon] (v1-) .. controls +(-120:0.3cm) and +(45:0.3cm) .. (1034);
\draw[ribbon] (v1-) to[out=140,in=-10] (1023);
\draw[ribbon] (v1+) .. controls +(60:0.3cm) and +(120:0.3cm) ..
	(0.5,2.5) to[out=-80,in=150] (1024);
\draw[ribbon,overline] (v1+) .. controls +(-60:0.3cm) and +(-100:0.3cm) ..
	(0.3,2.5) to[out=80,in=-135] (1234);
\draw[thin_overline={1}] (10) to[out=100,in=-90] (14);
\draw[thin_overline={1}] (12) -- (14);
\draw (10) to[out=60,in=-90] (12);
\draw (10) to[out=130,in=-90] (13) -- (12);
\draw (14) to[out=135,in=-80] (13);
\draw[opacity=0.3] (-0.15,1.05)
	to[out=30,in=-120] (1024)
	to[out=60,in=-45] (0.4,3.5)
	-- (1234)
	to[out=135,in=-80] (-0.4,4.9)
	to[out=-150,in=80] (1023)
	-- (-1.4,2.2)
	-- (1034)
	to[out=-45,in=150] (-0.15,1.05);
\end{scope}
\begin{scope}[shift={(3,-3)}]
\node at (1,0.5) {$\hat{v}_3$};
\node at (0.5,1.85) {\smallerer $\hat{v}_3^+$};
\node at (-0.4,2.15) {\smallerer $\hat{v}_3^-$};
\node[emptynode, label=-90:0] (30) at (0,0) {};
\node[emptynode, label=180:2] (32) at (-1,4) {};
\node[emptynode, label=135:1] (31) at (2,6) {};
\node[emptynode, label=180:4] (34) at (0.3,3) {};
\node[emptynode] (3024) at (-0.4,1.7) {};
\node[emptynode] (3014) at (0.8,1.6) {};
\node[emptynode] (3012) at (-0.1,2.8) {};
\node[emptynode] (3124) at (0.7,4.3) {};
\node[ellipse_morphism={120},minimum size=10pt] (v3+) at (0.5,2.3) {\;\;\;};
\node[ellipse_morphism={80},minimum size=10pt] (v3-) at (0,2.2) {\;\;\;};
\draw[ribbon] (v3-) .. controls +(0:0.2cm) and +(-150:0.2cm) .. (v3+);
\draw[ribbon] (v3-) .. controls +(-130:0.5cm) and +(45:0.3cm) .. (3024);
\draw[ribbon] (v3-) .. controls +(130:0.4cm) and +(-45:0.3cm) .. (3012);
\draw[ribbon] (v3+) .. controls +(80:0.3cm) and +(120:0.3cm) ..
	(0.9,2.3) to[out=-60,in=150] (3014);
\draw[ribbon,overline] (v3+) .. controls +(-30:0.3cm) and +(-100:0.3cm) ..
	(0.9,2.8) to[out=80,in=-135] (3124);
\draw[thin_overline={1}] (30) to[out=80,in=-90] (34);
\draw[thin_overline={1}] (34) to[out=45,in=-100] (31);
\draw (30) to[out=120,in=-90] (32);
\draw (30) to[out=50,in=-90] (31);
\draw (31) -- (32) -- (34);
\draw[opacity=0.3] (0.2,1.2)
	to[out=20,in=-120] (3014)
	to[out=60,in=-90] (1.2,4)
	to[out=170,in=-45] (3124)
	to[out=135,in=-90] (0.5,4.9)
	to[out=-110,in=45] (3012)
	to[out=-135,in=60] (-0.8,2)
	to[out=-30,in=135] (3024)
	to[out=-45,in=150] (0.2,1.2);
\end{scope}
\begin{scope}[shift={(-4.5,3)}]
\node at (-2,0.8) {$\hat{v}_2$};
\node at (-0.4,2.3) {\smallerer $\hat{v}_2^+$};
\node at (-0.85,1.7) {\smallerer $\hat{v}_2^-$};
\node[emptynode, label=-90:0] (20) at (0,0) {};
\node[emptynode, label=-30:3] (23) at (0,2.5) {};
\node[emptynode, label=0:1] (21) at (1,4) {};
\node[emptynode, label=180:4] (24) at (-3,3.5) {};
\node[emptynode] (2013) at (0.5,2) {};
\node[emptynode] (2014) at (-1.5,2.5) {};
\node[emptynode] (2034) at (-0.8,1.1) {};
\node[emptynode] (2134) at (-0.6,3.4) {};
\node[ellipse_morphism={30},minimum size=10pt] (v2+) at (-0.8,2.2) {\;\;\;};
\node[ellipse_morphism={30},minimum size=10pt] (v2-) at (-0.5,1.8) {\;\;\;};
\draw[ribbon] (v2-) -- (v2+);
\draw[ribbon] (v2-) .. controls +(-30:0.3cm) and +(170:0.3cm) .. (2013);
\draw[ribbon] (v2-) .. controls +(-100:0.4cm) and +(75:0.2cm) .. (2034);
\draw[ribbon] (v2+) .. controls +(80:0.4cm) and +(-30:0.3cm) .. (2014);
\draw[ribbon,overline] (v2+) .. controls +(170:0.5cm) and +(-140:0.3cm) .. (2134);
\draw[thin_overline={1}] (20) -- (23) -- (24);
\draw (23) -- (21);
\draw (20) to[out=150,in=-80] (24);
\draw (20) to[out=60,in=-90] (21);
\draw (21) to[out=170,in=20] (24);
\draw[opacity=0.3] (0,1)
	to[out=45,in=-100] (2013)
	to[out=80,in=-90] (0.5,3.25)
	to[out=180,in=-40] (2134)
	to[out=140,in=-70] (-1,4.05)
	to[out=-100,in=60] (2014)
	to[out=-120,in=60] (-2,1.5)
	to[out=-30,in=165] (2034)
	to[out=-15,in=180] (0,1);
\end{scope}
\begin{scope}[shift={(0,5)}]
\node at (-1.4,3.4) {$\hat{v}_0$};
\node at (-0.4,1.4) {\smallerer $\hat{v}_0^+$};
\node at (0.4,1.9) {\smallerer $\hat{v}_0^-$};
\node[emptynode, label=-90:2] (02) at (0,0) {};
\node[emptynode, label=0:1] (01) at (2,1) {};
\node[emptynode, label=180:3] (03) at (-2,1) {};
\node[emptynode, label=90:4] (04) at (0,4) {};
\node[emptynode] (0123) at (0.2,0.7) {};
\node[emptynode] (0124) at (0.7,2.5) {};
\node[emptynode] (0134) at (1,1.7) {};
\node[emptynode] (0234) at (-0.7,2.2) {};
\node[ellipse_morphism={45},minimum size=10pt] (v0+) at (-0.3,1.8) {\;\;\;};
\node[ellipse_morphism={45},minimum size=10pt] (v0-) at (0.3,1.5) {\;\;\;};
\draw[thin_overline={1}] (03) -- (01);
\draw[ribbon] (v0-) -- (v0+);
\draw[ribbon] (v0-) .. controls +(-90:0.3cm) and +(120:0.3cm) .. (0123);
\draw[ribbon] (v0-) .. controls +(-10:0.3cm) and +(150:0.3cm) .. (0134);
\draw[ribbon] (v0+) .. controls +(70:0.3cm) and +(-30:0.3cm) .. (0234);
\draw[ribbon] (v0+) .. controls +(170:0.4cm) and +(-150:0.2cm) ..
	(-0.3,2.2) .. controls +(30:0.3cm) and +(-90:0.2cm) .. (0124);
\draw[thin_overline={1}] (02) -- (04);
\draw (03) -- (02) -- (01);
\draw (04) to[out=-150,in=90] (03);
\draw (04) to[out=-30,in=90] (01);
\draw[opacity=0.3] (0123)
	to[out=30,in=-135] (0.6,1)
	to[out=60,in=-120] (0134)
	to[out=60,in=-120] (1.45,2.7)
	to[out=-150,in=0] (0124)
	to[out=180,in=-20] (0,2.7)
	to[out=-150,in=60] (0234)
	to[out=-120,in=90] (-1,0.5)
	to[out=0,in=-150] (0123);
\end{scope}
\draw[opacity=0.1, line width=3pt] (1023)
	.. controls +(150:1cm) and +(150:1cm) .. (-3,1);
\draw[opacity=0.1, line width=3pt, overline] (-3,1)
	.. controls +(-30:0.5cm) and +(-120:0.4cm) .. (4023);
\draw[opacity=0.1, line width=3pt, overline] (1024) to[out=-30,in=-135] (3024);
\draw[opacity=0.1, line width=3pt, overline] (1034) to[out=-135,in=-120] (2034);
\draw[opacity=0.1, line width=3pt, overline] (4012)
	.. controls +(-45:0.2cm) and +(135:0.2cm) .. (1.5,1);
\draw[opacity=0.1, line width=3pt] (1.5,1)
	.. controls +(-45:0.2cm) and +(135:0.2cm) .. (3012);
\draw[opacity=0.1, line width=3pt, overline] (4123)
	.. controls +(45:0.5cm) and +(-90:0.5cm) .. (0.7,4.5);
\draw[opacity=0.1, line width=3pt] (0.7,4.5)
	.. controls +(90:0.5cm) and +(-60:0.5cm) .. (0123);
\draw[opacity=0.1, line width=3pt] (4013)
	.. controls +(150:1cm) and +(-30:1cm) .. (-2,4.5);
\draw[opacity=0.1, line width=3pt, overline] (-2,4.5)
	.. controls +(150:1cm) and +(0:0.5cm) .. (2013);
\draw[opacity=0.1, line width=3pt, overline] (3014)
	.. controls +(-45:2cm) and +(-60:5cm) .. (5,5);
\draw[opacity=0.1, line width=3pt] (5,5)
	.. controls +(120:6cm) and +(0:2cm) .. (-2,10)
	.. controls +(180:4cm) and +(150:2cm) .. (2014);
\draw[opacity=0.1, line width=3pt, overline] (2134)
	.. controls +(45:3cm) and +(170:2.5cm) .. (0,9.4);
\draw[opacity=0.1, line width=3pt] (0,9.4)
	.. controls +(-10:3cm) and +(-30:1cm) .. (0134);
\draw[opacity=0.1, line width=3pt, overline] (1234)
	.. controls +(60:1cm) and +(-90:3cm) .. (-3,4)
	.. controls +(90:3cm) and +(150:2cm) .. (0234);
\draw[opacity=0.1, line width=3pt, overline] (3124)
	.. controls +(60:1cm) and +(-70:4cm) .. (2.5,7)
	.. controls +(110:2cm) and +(80:1cm) .. (0124);
\draw (-2,9.35) to[out=30,in=-150] +(0:0.5cm);
\draw (-2.67,6.5) to[out=60,in=-120] +(90:0.5cm);
\draw (3.3,5) to[out=-60,in=120] +(-90:0.5cm);
\draw (0.8,4.3) to[out=90,in=-90] +(115:0.5cm);
\end{tikzpicture}
\caption{Anchors in $\overline{\del \xi}$ fitting together. The co-orientation of the
co-ribbon graph is pointing towards us (except near $\hat{v}_0$).
The squiggles on the gray ribbons going into $\hat{v}_0$
indicate that there should be a half-twist. }
\label{f:anchors-xi}
\end{figure}
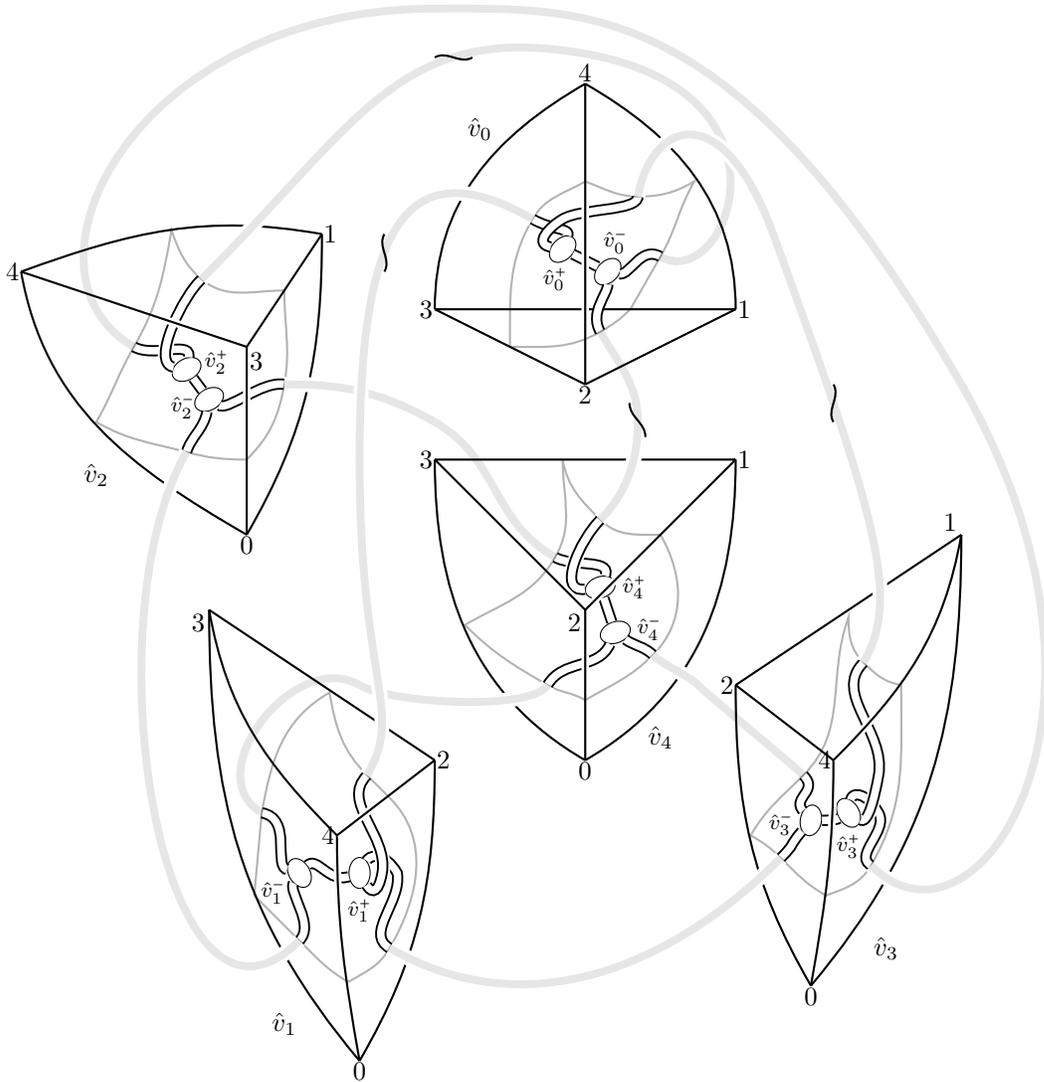
\begin{figure}[ht]
\begin{tikzpicture}
\node[small_morphism, label={[yshift=0.05cm]0:$\hat{v}_4^-$}] (v4-) at (0,0) {};
\node[small_morphism, label={[yshift=-0.1cm]0:$\hat{v}_4^+$}] (v4+) at (0,1) {};
\node[small_morphism, label={[xshift=0.1cm]-90:$\hat{v}_1^-$}] (v1-) at (-3,-1) {};
\node[small_morphism, label={[yshift=-0.00cm]0:$\hat{v}_1^+$}] (v1+) at (-2,-1) {};
\node[small_morphism, label={[xshift=0.05cm]-90:$\hat{v}_3^-$}] (v3-) at (2,-1) {};
\node[small_morphism, label={[xshift=-0.1cm]90:$\hat{v}_3^+$}] (v3+) at (3,-1) {};
\node[small_morphism, label={[yshift=0.05cm]0:$\hat{v}_2^-$}] (v2-) at (-2,3) {};
\node[small_morphism, label={[yshift=-0.1cm]0:$\hat{v}_2^+$}] (v2+) at (-2,4) {};
\node[small_morphism, label={[yshift=-0.1cm]0:$\hat{v}_0^-$}] (v0-) at (2,4) {};
\node[small_morphism, label={[yshift=0.05cm]0:$\hat{v}_0^+$}] (v0+) at (2,3) {};
\draw (v0-) -- (v0+);
\draw (v1-) -- (v1+);
\draw (v2-) -- (v2+);
\draw (v3-) -- (v3+);
\draw (v4-) -- (v4+);
\draw (v1-) to[out=120,in=-150] (v4-);
\draw (v1+) to[out=60,in=-120] (v3-);
\draw (v1-) to[out=-120,in=-150] (v2-);
\draw (v4-) to[out=-30,in=120] (v3-);
\draw (v4+) to[out=30,in=-30] (v2-);
\draw (v3+) to[out=60,in=-45] (1,2) to[out=135,in=30] (v2+);
\draw (v2+) to[out=150,in=30] (v0-);
\draw (v3+) to[out=-60,in=-45] (3,1) to[out=135,in=-150] (v0+);
\draw[overline] (v1+) to[out=-60,in=-30] (v0+);
\draw[overline] (v4+) to[out=150,in=150] (v0-);
\end{tikzpicture}
\caption{Simplified version of \figref{f:anchors-xi}.}
\label{f:anchors-xi-simplified}
\end{figure}

\end{proof}

\clearpage
\section{Crane-Yetter Invariant as an Extended TQFT}
\label{s:cy-extended-plcw}
\vspace{0.8in}

\subsection{$\ZCY$ for 4-Manifolds with Corners}
\par \noindent

We discuss an extension of the definition of the Crane-Yetter invariant
to 4-manifolds with corners.
The modifications we make to the original state sum
seem very similar to \ocite{barrett-obs}
(they also work with the category of representations of
quantum $\mathfrak{s}\mathfrak{k}_2$),
but we were not aware of their work until recently.


\begin{definition}
Let $N$ be a closed surface.
A \emph{colored marked PLCW decomposition} $(\cN,l)$ of $N$
is a marked PLCW decomposition $\cN$ with a simple labeling $l$.
\label{d:colored-marked-PLCW}
\end{definition}

It is best to think of $N$ as an unoriented
(but possibly orientable) surface
unless the orientation is explicitly needed;
for example,
$N$ may be a closed surface in the boundary of
an oriented 4-manifold,
say $\del W = M \cup_N M'$,
wherein the outward orientations on $N$ with respect to $M$ and $M'$,
themselves given the outward orientation with respect to $W$,
are different.

\begin{definition}
For an oriented PLCW 3-manifold $\cM$
with boundary $\cN$, let $l_\cN$ be a simple labeling of $\cN$.
We define the \emph{relative pre-state space} to be
\[
H(\cM; (\cN,l_\cN)) := \bigoplus_{l|_\cN = l_\cN} H(\cM; l)
\]
i.e. the sum of pre-state spaces over all labelings on $\cM$
agreeing with $l_\cN$ on $\cN$.
\end{definition}

The corresponding relative state space will be defined in
\defref{d:state-space-boundary}.

\begin{definition}
\label{d:relative-state-space-pairing}
The \emph{relative pre-state space pairing}
is the perfect pairing
\begin{align}
\evdel : H(\cM;(\cN,l_\cN)) \tnsr H(\overline{\cM};(\cN,l_\cN))
	&\to \kk
\\
\bigotimes \vphi_C \tnsr \bigotimes \vphi_C'
	&\mapsto d_{l_\cN}^{1/2} \prod_C (\vphi_C,\vphi_C')
\label{e:relative-state-space-pairing}
\end{align}
that applies 
the local state space pairing of \lemref{l:local-state-space-dual}
to each corresponding pair of 3-cells.
(Note we do not have to deal with the orientation on $\cN$,
using the same colored marked PLCW decomposition $(\cN,l_\cN)$
for both $M$ and $\ov{M}$.)
\end{definition}

\begin{definition}
Let $\cW$ be an oriented PLCW 4-manifold,
and let $\cN \subset \del \cW$ be a PLCW closed surface in its boundary.
Recall from \eqnref{e:zcy-label}:
\begin{equation}
	Z(\cW,l) =
		\ev (\bigotimes_T Z(T,l)) \in H(\del \cW, l)
\label{e:zcy-label-repeat}
\end{equation}

Given a simple labeling $l_\cN$ of $\cN$,
the \emph{restricted Crane-Yetter invariant} $\ZCY(\cW;(\cN,l_\cN))$
is
\begin{equation}
	\ZCY(\cW; (\cN,l_\cN)) :=
	\cD^{\vme(\cW \backslash \del \cW) + \frac{1}{2}\vme(\del \cW \backslash \cN)}
		\sum_l \prod_f d_{l(f)}^{n_f} Z(\cW,l)
	\in H(\del \cW; (\cN, l_\cN))
\label{e:cy-restr}
\end{equation}
where
\begin{itemize}
\item the sum is taken over all equivalence classes of simple labelings $l$
	that agree with $l_\cN$ on $\cN$,
\item $f$ runs over the set of unoriented 2-cells of $\cW$ not in $\cN$,
\item $n_f =
	\begin{cases}
		1 &\textrm{if } f \textrm{ is an internal 2-face} \\
		\frac{1}{2} & f \in \del \cW \backslash \cN \\
		0 & f \in \cN \;\; (\text{redundant by previous point})
	\end{cases}
	$
\item $\cD$ is the dimension of the category (see \eqnref{e:dimC})
\item $\vme(X) = v(X) - e(X) =
	\textrm{number of 0-cells} - \textrm{number of 1-cells}$
\item $d_{l(f)}$ is the categorical dimension of $l(f)$
\end{itemize}

The \emph{extended Crane-Yetter invariant} $\ZCY(\cW; \cN)$ is
\begin{equation}
\ZCY(\cW; \cN) := \sum_{l_\cN} \ZCY(\cW; (\cN,l_\cN))
\in H(\del \cW)
\label{e:cy-extn}
\end{equation}

\label{d:zcy-corner}
\end{definition}

Essentially, the definition is the same as \defref{d:zcy},
but the coefficient for each labeling is weighted as if
$\cN$ was removed from $\cW$;
this was cooked up to make the gluing along manifolds with boundary
to hold
(similar to the $1/2$'s that appear in $\ZCY(\cW)$
(\defref{d:zcy}) to account for gluing along boundaries).

\begin{theorem}
For an oriented PLCW 4-manifold $\cW$,
with corner $\cN \subset \del \cW$,
$\ZCY(\cW;\cN)$ is independent of the choice of PLCW decomposition
in the interior;
that is, if $\cW'$ is another PLCW decomposition that agrees with $\cW$
on the boundary, then
$\ZCY(\cW';\cN) = \ZCY(\cW;\cN) \in H(\del \cW)$.
\label{t:invariance-zcy-corner}
\end{theorem}

\begin{proof}
Follows from \thmref{t:invariance-zcy};
we just have to note that in the proof of invariance under
elementary moves, only coefficients $d_{l(f)}$ for 2-cells $f$
in the interior are involved.
\end{proof}

The gluing result holds:

\begin{proposition}
Let $\cW$ be an oriented PLCW 4-manifold,
and suppose the 3-cells of $\del \cW$ are partitioned
into three 3-manifolds, $\del \cW = \cM_0 \cup \cM \cup \cM'$,
such that $\cM$ and $\cM'$ are disjoint,
and $\cM' \simeq \overline{\cM}$.
In particular, $\del \cM_0 = \del \cM \sqcup \del \cM'$
as unoriented manifolds.
Let $\cW'$ be the PLCW manifold obtained from $\cW$ by identifying
$\cM, \cM'$.
Then
\begin{align*}
\ZCY(\cW'; \del \cM_0)
	&= \ev (\ZCY(\cW; \del \cM_0))
\\
\ZCY(\cW')
	&= \evdel (\ZCY(\cW))
\end{align*}
where $\ev$ applies the local state space pairing 
(\lemref{l:local-state-space-dual}) 
to each 3-cell of $\cM$,
and $\evdel$ is the relative pre-state space pairing
\defref{d:relative-state-space-pairing}.
\label{p:zcy-corner-gluing}
\end{proposition}
\begin{proof}
Essentially by definition.
\end{proof}

It is often convenient to state the gluing result as the
composition of maps:

\begin{proposition}
Let $\cW, \cW'$ be an oriented PLCW 4-manifolds,
and suppose $\del \cW = \overline{\cM} \cup_{\cN} \cM'$,
$\del \cW' = \overline{\cM'} \cup_{\cN} \cM''$.
We say $\cW$ is a \emph{cornered cobordism over $\cN$
from $\cM$ to $\cM'$},
denoted $\cW : \cM \to_\cN \cM'$
(likewise $\cW' : \cM' \to_\cN \cM''$;
see also \defref{d:cornered-cobordism}).

Then for a simple labeling $l_\cN$ of $\cN$,
\[
\ZCY(\cW \cup_{\cM'} \cW'; (\cN,l_\cN))
= \ev(\ZCY(\cW'; (\cN,l_\cN)) \tnsr \ZCY(\cW; (\cN,l_\cN)))
\]
where the evaluation applies the local state space pairing
(\lemref{l:local-state-space-dual}) 
to each 3-cell of $\cM'$.
Thus
\[
\ZCY(\cW \cup_{\cM'} \cW'; \cN)
= \ev(\ZCY(\cW'; \cN) \tnsr \ZCY(\cW; \cN))
\]

We may interpret $\ZCY(\cW; (\cN,l_\cN))$ as a map
\[
\ZCY(\cW; (\cN,l_\cN)) : H(\cM; (\cN,l_\cN)) \to H(\cM'; (\cN,l_\cN))
\]
and similarly, $\ZCY(\cW; \cN)$ as a map
\[
\ZCY(\cW; \cN) : H(\cM) \to H(\cM')
\]
Then
\begin{align*}
\ZCY(\cW \cup_{\cM'} \cW'; (\cN,l_\cN))
&= \ZCY(\cW'; (\cN,l_\cN)) \circ \ZCY(\cW; (\cN,l_\cN))
: H(\cM; (\cN, l_\cN)) \to H(\cM'; (\cN, l_\cN))
\\
\ZCY(\cW \cup_{\cM'} \cW'; \cN)
&= \ZCY(\cW'; \cN) \circ \ZCY(\cW; \cN)
: H(\cM) \to H(\cM')
\end{align*}

\label{p:zcy-corner-gluing-map}
\end{proposition}

The extended $\ZCY$ also respects the
extended composition (see \defref{d:composition-extended})
wherein the corners can be different,
which in fact is a generalization of the above result:

\begin{proposition}
\label{p:zcy-corner-gluing-alt}
Let $\cW_1,\cW_2$ be cornered cobordisms
\begin{align*}
\cW_1 &: \cM_1 \to_{\cN_1} \cM_1'
\\
\cW_2 &: \cM_2 \to_{\cN_2} \cM_2'
\end{align*}
Suppose $\cM_1' \subseteq \cM_2$ is a PLCW submanifold.
Let $\cW = \cW_1 \cup_{\cM_1'} \cW_2$,
and $\cM = (\cM_2 \backslash \cM_1') \cup \cM_1$,
so that $\cW$ is naturally a cornered cobordism
\[
\cW : \cM \to_{\cN_2} \cM_2'
\]

Then $\ZCY(\cW_1; \cN_1)$ and $\ZCY(\cW_2;\cN_2)$ compose to give
$\ZCY(\cW; \cN_2)$; more precisely,
\begin{align*}
\ZCY(\cW; \cN_2) &= \ZCY(\cW_2; \cN_2) \circ
	(\id \tnsr \ZCY(\cW_1; \cN_1))
\\
\ZCY(\cW;\cN_2)(\vphi \tnsr \psi)
&= \ZCY(\cW_2;\cN_2) (\vphi \tnsr \ZCY(\cW_1;\cN_1) (\psi))
\end{align*}
for $\psi \in H(\cM_1;(\cN_1,l)),
	\psi \in \bigotimes_{C \in \cM_2 \backslash \cM_1'} H(C,l)$,
$l$ a simple labeling of $\cM$.

Similarly, suppose that we have the reverse inclusion
$\cM_2 \subseteq \cM_1'$.
Let $\cW = \cW_1 \cup_{\cM_2} \cW_2$,
and $\cM' = (\cM_1' \backslash \cM_2) \cup \cM_2'$,
so that $\cW$ is naturally a cornered cobordism
\[
\cW : \cM_1 \to_{\cN_1} \cM'
\]

Then $\ZCY(\cW_1; \cN_1)$ and $\ZCY(\cW_2;\cN_2)$ compose to give
$\ZCY(\cW; \cN_1)$; more precisely,
\begin{align*}
\ZCY(\cW; \cN_1) &= (\id \tnsr \ZCY(\cW_2; \cN_2))
	\circ (\ZCY(\cW_1; \cN_1))
\\
\ZCY(\cW;\cN_1)(\Phi)
&= \sum_a \vphi^{(a)} \tnsr \ZCY(\cW_2;\cN_2) (\psi^{(a)})
\end{align*}
where $\ZCY(\cW_1;\cN_1) (\Phi) = \sum_a \vphi^{(a)} \tnsr \psi^{(a)}$,
with $\psi^{(a)} \in H(\cM_2)$.

\begin{equation}
\begin{tikzpicture}
\node[dotnode] (N2L) at (0,0) {};
\node[dotnode] (N2R) at (4,0) {};
\node[dotnode] (N1) at (2,0) {};
\draw (N2L) -- (N2R);
\draw (N2L) to[out=-80,in=-100] (N2R);
\draw (N1) to[out=80,in=100] (N2R);
\node at (2,-0.8) {$\cW_2$};
\node at (3,0.3) {$\cW_1$};
\node at (1,-0.25) {\smallerer $\cM_2 \backslash \cM_1'$};
\node at (3,-0.25) {\smallerer $\cM_1'$};
\node at (3,0.85) {\smallerer $\ov{\cM_1}$};
\node at (2,-1.4) {\smallerer $\cM_2'$};
\node at (-0.3,0) {\tiny $\cN_2$};
\node at (4.3,0) {\tiny $\cN_2$};
\node at (2,-0.2) {\tiny $\cN_1$};
\end{tikzpicture}
\;\;\;\;
\begin{tikzpicture}
\begin{scope}[shift={(0,-0.5)}]
\node[dotnode] (N2L) at (0,0) {};
\node[dotnode] (N2R) at (4,0) {};
\node[dotnode] (N1) at (2,0) {};
\draw (N2L) -- (N2R);
\draw (N2L) to[out=80,in=100] (N2R);
\draw (N1) to[out=-80,in=-100] (N2R);
\node at (2,0.7) {$\cW_1$};
\node at (3,-0.3) {$\cW_2$};
\node at (2,1.4) {\smallerer $\ov{\cM_1}$};
\node at (1,-0.25) {\smallerer $\cM_1' \backslash \cM_2$};
\node at (3,0.25) {\smallerer $\ov{\cM_2}$};
\node at (3,-0.85) {\smallerer $\cM_2'$};
\node at (-0.3,0) {\tiny $\cN_2$};
\node at (4.3,0) {\tiny $\cN_2$};
\node at (2,0.2) {\tiny $\cN_1$};
\end{scope}
\end{tikzpicture}
\end{equation}
(We label the 3-manifolds with the orientation that makes
the side on which the label appears the ``outside''.)
\footnote{``Are those ... invariants ... on a cob?''}
\end{proposition}

\begin{proof}
Note that we have a slightly more general situation than
\defref{d:composition-extended}, as $\cN,\cN'$ can meet here.
Once again, this is a straightforward consequence of the definitions.
As mentioned before, $\ZCY(\cW;\cN)$ is essentially the same as
$\ZCY(\cW)$, but the weights are designed to ``omit'' $\cN$.
Thus, intuitively, in the composition
$\ZCY(\cW_2; \cN_2) \circ (\id \tnsr \ZCY(\cW_1; \cN_1))$,
$\cN_1$ appears in both terms, but contributes to the $\cD$ and $d_{l(f)}$
weights in the $\ZCY(\cW_2; \cN_2)$ and not in
$(\id \tnsr \ZCY(\cW_1; \cN_1))$.

%
\end{proof}

Note that $\cN_1$ and $\cN_2$ may not be disjoint,
so $\cM_2 \backslash \cM_1'$ may not be a submanifold.

The following lemma is useful for computations:

\begin{lemma}
Let $\cW$ be a cell-like 4-ball,
and suppose $\del \cW = \overline{\cM}_{in} \cup_\cN \cM_{out}$,
where $\cM_-$ are 3-balls,
so $\cW : \cM_{in} \to_\cN \cM_{out}$.
Furthermore,
suppose the bottom hemisphere $\cM_{out}$ consist of exactly one 3-cell $C_{out}$.

Choose some anchoring of $\cW$ and simple labeling $l$.
Consider some $\Phi = \bigotimes \vphi_C \in H(\cM_{in}; (\cN,l))$.
Assemble the anchors in $\cM_{in}$ to a ribbon graph $\Gamma$,
and color the coupons by $\vphi_C$,
giving us a colored ribbon graph $(\Gamma, \Phi)$.
Under an identification $\cM_{in} \simeq \cM_{out}$
as PL manifolds,
$(\Gamma, \Phi)$ may be viewed in $C_{out}$,
and has some Reshetikhin-Turaev evaluation (\defref{d:zrt-eval})
$\vphi_{ev}\in H(C_{out},l) = H(\cM_{out}; (\cN,l))$.
Then
\[
\ZCY(\cW; \cN) (\Phi) =
\cD^{\frac{1}{2} \vme(\cM_{in} \backslash \cN)}
d_{l^\circ}^{1/2}
	\cdot \vphi_{ev}
\]
where $d_{l^\circ}^{1/2} = \prod_{f \in \cM_{in} \backslash \cN} d_{l(f)}$.

Thus, if $C_{out}$ is a 3-cell in some PLCW 3-manifold $\cM$,
and $\cM' = (\cM \backslash C_{out} ) \cup \cM_{in}$,
then for simple labeling $l$ of $\cM'$,
and $\Psi \in \prod_{C \in \cM \backslash C_{out}} H(C,l)$,
by \prpref{p:zcy-corner-gluing-alt},
$\Psi \tnsr \Phi$ and
$\Psi \tnsr (\cD^{\vme(\cM_{in} \backslash \del \cM_{in})/2}
	d_{l^\circ}^{1/2} \cdot \vphi_{ev})$
are equivalent as states (\defref{d:pre-state-equiv}),
i.e.
\begin{equation}
A_{\cM', \cM''}(\Psi \tnsr \Phi)
= A_{\cM,\cM''} (\Psi \tnsr (\cD^{\vme(\cM_{in} \backslash \del \cM_{in})/2}
	d_{l^\circ}^{1/2} \cdot \vphi_{ev}))
\label{e:lifts-relate}
\end{equation}
for any $\cM''$.

In words, if $\cM$ is obtained from $\cM'$ by ``merging''
several 3-cells into one 3-cell,
then a pre-state of $\cM'$ is equivalent (up to a factor)
to the pre-state
obtained by ``evaluating'' the subgraph in those 3-cells.
\label{l:state-space-moves}
\end{lemma}

\begin{proof}
Note that
$\vme(\cW \backslash \del \cW) = 0$
and $\vme(\del \cW \backslash \cN) = \vme(\cM_{in} \backslash \cN)$.
We have, for any $\vphi' \in H(\overline{C_{out}},l)$,
\begin{align*}
\ev (\ZCY(\cW;\cN) (\Phi), \vphi')
&= \cD^{\vme(\cW \backslash \del \cW) + \frac{1}{2}\vme(\del \cW \backslash \cN)}
		d_{l(f)}^{n_f}
	\ev (Z(\cW,l), \Phi \tnsr \vphi')
\\
&= \cD^{\frac{1}{2} \vme(\cM_{in} \backslash \cN)} d_{l^\circ}^{1/2}
	\ev (Z(\cW,l), \Phi \tnsr \vphi')
\\
&= \cD^{\frac{1}{2} \vme(\cM_{in} \backslash \cN)} d_{l^\circ}^{1/2}
	\ZRT(\Gamma \cup \Gamma', l, \Phi \tnsr \vphi')
\text{ where }\Gamma' \text{ is anchor for }C_{out}
\\
&= \cD^{\frac{1}{2} \vme(\cM_{in} \backslash \cN)} d_{l^\circ}^{1/2}
	\ZRT(\Gamma' \cup \Gamma', l, \vphi_{ev} \tnsr \vphi')
\\
&= \ev (\cD^{\frac{1}{2} \vme(\cM_{in} \backslash \cN)} d_{l^\circ}^{1/2}
	\cdot \vphi_{ev}, \vphi').
\end{align*}
(recall \defref{d:ZCY3cell}).
\end{proof}

This fact is useful for dealing with lifts of states in $H(\cM)$,
as it relates them to states for different PLCW decomposition.

Recall the ``empty states'' $\emptystate{\cM}$
from \eqnref{e:emptystate}.
It follows from \lemref{l:state-space-moves} that,
for PLCW structures $\cM,\cM',\cM''$ on a closed 3-manifold,
\begin{equation}
\label{e:empty-state-relate}
A_{\cM,\cM''} (\cD^{-\frac{1}{2}\vme(\cM)} \emptystate{\cM})
=
A_{\cM',\cM''} (\cD^{-\frac{1}{2}\vme(\cM')} \emptystate{\cM'})
\end{equation}
Indeed, by \thmref{t:elementary-boundary},
PLCW decompositions $\cM,\cM'$ are related by elementary moves
fixing the boundary.
The operation of changing $C_{out}$ to $\cM_{in}$
is a single-cell subdivision,
and by \prpref{p:single-cell-subdivision},
can relate any two PLCW decompositions.
It is clear that the relation above holds for $\cM,\cM'$
related by a single-cell move, so we are done.
This will be made more clear in \secref{s:sk-equiv}.

\begin{remark}
\label{r:state-space-null-graph}
\lemref{l:state-space-moves} is the analog of
``null graphs with respect to $D$'' in the skeins context
(see \secref{s:skein-cat}),
in the sense that \lemref{l:state-space-moves}
provides an explicit description of how states transform when the
PLCW decomposition changes locally,
while a null graph defines an equivalence between graphs
that are the same outside a local region.
\end{remark}

We may define the relative state space similar
to \prpref{p:projectors-compose}:

\begin{definition}
Let $M$ be a 3-manifold with boundary $\del M = N$.
Consider $W = M \times I$ with corner $N \times 0$,
so that it is a cornered cobordism from $M$ to $M$.
Let $\cM,\cM'$ be PLCW decompositions of $M$
with the same boundary $\cN$.
Let $\cW$ be a PLCW structure on $W$
that extends $\cM$ on the incoming $M$
and $\cM'$ on the outgoing $M$.

We define, for simple labeling $l_\cN$ on $\cN$,
\begin{align*}
A_{\cM,\cM'; l_\cN}
&:= \ZCY(\cW; (\cN,l_\cN)) :
H(\cM; (\cN,l_\cN)) \to H(\cM'; (\cN,l_\cN))
\\
A_{\cM,\cM'}
&:= \ZCY(\cW; \cN) :
H(\cM) \to H(\cM')
\end{align*}
Note $A_{\cM,\cM'} = \sum_l A_{\cM,\cM';l}$,
summed over all simple labelings of $\cN$.

The maps $A_{\cM,\cM'}$ compose as in \prpref{p:projectors-compose},
i.e. $A_{\cM,\cM''} = A_{\cM',\cM''} \circ A_{\cM,\cM'}$,
so in particular $A_\cM := A_{\cM,\cM}$ is a projection,
and the spaces
\[
\ZCY(\cM; (\cN,l)) := \im(A_{\cM,\cM; l})
\]
are canonically identified, thus we may define
$\ZCY(M; (\cN,l)) = \ZCY(\cM; (\cN,l))$
without ambiguity.
We call $\ZCY(M;(\cN,l))$ the
\emph{relative state space of $M$}.

We will often simple denote $A_{\cM,\cM'}$ by $A$
when the context is clear.
\label{d:state-space-boundary}
\end{definition}

Similar to before, we may talk about a \emph{lift} of a state
in $\ZCY(\cM; (\cN,l))$ to a pre-state in $H(\cM; (\cN,l))$.

We also consider the extension of 
\defref{d:state-space-pairing-reduced} to
3-manifolds with boundary:

\begin{definition}
\label{d:reduced-relative-state-space-pairing}
For $\vphi \in H(\cM;(\cN,l)), \vphi' \in H(\ov{\cM};(\cN,l))$,
the \emph{reduced relative pre-state space pairing} is defined as
\begin{equation}
\evdelbar (\vphi, \vphi')
:= \evdel (\vphi, A(\vphi'))
= \evdel (A(\vphi), \vphi').
\label{e:pairing-A-corner}
\end{equation}
\end{definition}

Just as in \defref{d:state-space-pairing-reduced},
the reduced pairing can also be given a more TQFT-esque flavor:

\begin{lemma}
\label{l:reduced-pairing-TQFT}
Observe that $M \times I$ has boundary $M \cup_N \ov{M}$
(where $N = \del M$);
we treat it as a cobordism
$M\times I : M \cup_N \ov{M} \to \emptyset$.
Give $M \cup_N \ov{M}$ the PLCW structure $\cM \cup_\cN \ov{\cM}$.
Then for $\vphi \in H(\cM;(\cN,l)), \vphi' \in H(\ov{\cM};(\cN,l))$,
$\vphi \tnsr \vphi'$ defines a pre-state in
$H(\cM \cup_\cN \ov{\cM})$,
and
\[
\ZCY(M \times I) (\vphi \tnsr \vphi')
= d_l^{1/2} \ev (\ZCY(M \times I; \cN) (\vphi), \vphi')
= \evdelbar(\vphi,\vphi')
\]
(here the $\ev$ is just $\evdel$ but without
the $d_l^{1/2}$ coefficient).
\end{lemma}
\begin{proof}
Obvious.
\end{proof}

\subsection{Category of Boundary Values: PLCW version}
\par \noindent

As an extended theory, $\ZCY$ should associate to each surface $N$
a category $\ZCY(N)$, known as the category of boundary
values. Here we define such a category;
however, we only develop many of its properties later,
in the more convenient language of skein categories.

\begin{definition}
Let $N$ be an \emph{oriented} closed surface.
Define the category $\hatZCY(N)$ whose objects are
``direct sums'' of colored marked PLCW compositions
$\bigoplus (\cN,l)$.
We describe morphisms for objects $(\cN,l)$ and
extend to direct sums.

Morphisms are given by
\[
\Hom_{\hatZCY(N)}((\cN,l), (\cN',l')) = \ZCY(N\times I; (\cN,l), (\cN',l')),
\]
where $(\cN,l)$ is imposed on $N \times 0$
and $(\cN',l')$ on $N \times 1$,
and $N \times I$ is given the orientation such that
the outward orientation at $N \times 1$ is the orientation of $N$
(under the obvious identification $N \simeq N \times 1$).

The identity morphism for $(\cN,l)$ given by
\[
\id_{(\cN,l)} = 
\cD^{\frac{1}{2} e(\cN)} d_l^{-1/2}
A(\wdtld{\id}_{(\cN,l)})
\]
where
\begin{align*}
\wdtld{\id}_{(\cN,l)}
&= \bigotimes_{f\in \cN} \coev_{l(f)} \in
	H(\cN \times I, l \times I)
\\
d_l^{-1/2}
&= \prod_{f \in \cN} d_{l(f)}^{-1/2}
\end{align*}
where $l \times I$ is the labeling with $l$ on both copies
$\cN \times 0$ and $\cN \times 1$,
and $\one$ on all other 2-faces $e \times I$ of $\cN \times I$;
more intuitively, it is the ``vertical graph'' $m \times I$,
where $m$ is the marking.
(See \lemref{l:zcysk-identity}.)

The composition of morphisms is given by:
\begin{align*}
\Hom_{\hatZCY(N)}((\cN',l'), (\cN'',l''))
\tnsr
\Hom_{\hatZCY(N)}((\cN,l), (\cN',l'))
&\to
\Hom_{\hatZCY(N)}((\cN,l), (\cN'',l''))
\\
\vphi' \tnsr \vphi &\mapsto
\vphi' \circ \vphi :=
A(\vphi' \tnsr \vphi)
\end{align*}

We define $\ZCY(N)$ to be the Karoubi envelope of $\hatZCY(N)$:
\[
\ZCY(N) = \Kar(\hatZCY(N))
\]
\label{d:cat-boundary-values}
\end{definition}

\begin{lemma}
\label{l:zcysk-identity}
The morphism $\id_{(\cN,l)}$ defined in \defref{d:cat-boundary-values}
indeed satisfy the properties of an identity morphism.
\end{lemma}
\begin{proof}
Let $\wdtld{\vphi} \in H(\cM)$
represent some morphism in $\Hom_{\ZCY(N)}((\cN,l),(\cN',l'))$,
where $\cM$ is some PLCW structure on $N\times I$.
Consider $\wdtld{\vphi} \tnsr \wdtld{\id}_{(\cN,l)}
\in H(\cM \cup_\cN \cN \times I)$.
Using \lemref{l:state-space-moves},
we can merge each 3-cell of $\cN \times I$ with the adjacent
3-cell in $\cM$.
We then have essentially $\wdtld{\vphi}$ again,
with a factor of $d_l^{1/2}$ from the merging operations
(number of 0- and 1-cells remains the same, so no $\cD$ factor).
We then perform more single-cell moves (\prpref{p:single-cell-subdivision})
to modify the PLCW structure back to $\cM$;
this accumulates another factor of $\cD^{-\frac{1}{2}e(\cN)}$
(since
$\frac{1}{2} \vme(\cM \cup \cN \times I) - \frac{1}{2} \vme(\cM)
= \frac{1}{2}(\vme(\cN \times I) - \vme(\cN))
= \frac{1}{2}((2\vme(\cN) - v(\cN)) - \vme(\cN))
= -\frac{1}{2} e(\cN)$).
Thus it follows that
$A(\wdtld{\vphi} \tnsr A(\cD^{\frac{1}{2} e(\cN)} d_l^{-1/2} \wdtld{\id}))
= A(\wdtld{\vphi} \tnsr \cD^{\frac{1}{2} e(\cN)} d_l^{-1/2} \wdtld{\id})
= A(\wdtld{\vphi})$.
\end{proof}

\begin{proposition}
\label{p:zcy-surface-functors}
Let $N$ be an oriented closed surface.
A 3-manifold $M$ bounded by $\del M = N$
defines a functor
\begin{align*}
\cF_M: \ZCY(N) &\to \Vect
\\
(\cN,l) &\mapsto \ZCY(M;(\cN,l))
\end{align*}
and for a morphism $f : (\cN,l) \to (\cN',l')$,
\[
\cF_M(f) = f \circ - := A(f \tnsr - )
\]
Furthermore, a cornered cobordism $W : M \to_\cN M'$
defines a natural transformation
\begin{align*}
\cF_W &: \cF_M \to \cF_{M'}
\\
(\cF_W)_{(\cN,l)} = \ZCY(W;l) &:
\ZCY(M;(\cN,l)) \to \ZCY(M'; (\cN,l))
\end{align*}

\end{proposition}

\begin{proof}
We first prove this for $\hatZCY(N)$.
The functor $\cF_M$ is shown to respect composition as follows.
For $\vphi \in \ZCY(M;(\cN,l))$,
$f \in \Hom_{\ZCY(N)}((\cN,l),(\cN',l'))$,
$g \in \Hom_{\ZCY(N)}((\cN',l'),(\cN'',l''))$,
consider the two diagrams

\begin{equation*}
\begin{tikzpicture}
\node[dotnode] (N1) at (1,0) {};
\node[dotnode] (N2) at (2,-1) {};
\node[dotnode] (N3) at (3,-2) {};
\draw (N1) -- (N3);
\draw (0,0) -- (N1);
\draw (0,-1) -- (N2);
\draw (0,-2) -- (N3);
\draw[dotted] (0,0) -- (-0.5,0);
\draw[dotted] (0,-1) -- (-0.5,-1);
\draw[dotted] (0,-2) -- (-0.5,-2);
\node at (1.6,0) {\smallerer $(\cN,l)$};
\node at (3,-1) {\smallerer $(\cN',l') = \cN_1$};
\node at (4,-2) {\smallerer $(\cN'',l'') = \cN_2$};
\node at (-0.7,0) {\smallerer $M$};
\node at (-0.7,-1) {\smallerer $M$};
\node at (-0.7,-2) {\smallerer $M$};
\node at (-0.3,-0.5) {\smallerer $W_1 = M \times I$};
\node at (-0.3,-1.5) {\smallerer $W_2 = M \times I$};
\node at (1.6,-0.4) {\smallerer $f$};
\node at (2.6,-1.4) {\smallerer $g$};
\node at (0.5,-0.2) {\smallerer $\vphi$};
\node at (1,-1.2) {\smallerer $A(f \tnsr \vphi)$};
\node at (1.5,-2.2) {\smallerer $A(g \tnsr A(f \tnsr \vphi))$};
\end{tikzpicture}
\;\;\;\;
\begin{tikzpicture}
\node[dotnode] (N1) at (1,0) {};
\node[dotnode] (N2) at (2,-1) {};
\node[dotnode] (N3) at (3,-2) {};
\draw (N1) -- (N3);
\draw (0,0) -- (N1);
\draw (0,-2) -- (N3);
\draw (N1) to[out=-90,in=170] (N3);
\draw[dotted] (0,0) -- (-0.5,0);
\draw[dotted] (0,-2) -- (-0.5,-2);
\node (cN) at (1.6,0) {\smallerer $(\cN,l)$};
\node at (3,-1) {\smallerer $(\cN',l') = \cN_1$};
\node (cN'')at (3.6,-2) {\smallerer $(\cN'',l'')$};
\node (cN2) at (4,-0.5) {\smallerer $\cN_2$};
\draw[line width=0.4pt] (cN2) -- (cN);
\draw[line width=0.4pt] (cN2) to[out=-90,in=60] (cN'');
\node at (-0.7,0) {\smallerer $M$};
\node at (-0.7,-2) {\smallerer $M$};
\node at (1.6,-1) {\smallerer $W_1$};
\node at (-0.3,-1) {\smallerer $W_2 = M \times I$};
\node at (1.6,-0.4) {\smallerer $f$};
\node at (2.6,-1.4) {\smallerer $g$};
\node at (1.1,-1.4) {\smallerer $g \circ f$};
\node at (0.5,-0.2) {\smallerer $\vphi$};
\node at (1.5,-2.2) {\smallerer $A((g \circ f) \tnsr \vphi)$};
\end{tikzpicture}
\end{equation*}

To apply \prpref{p:zcy-corner-gluing-alt},
on the left, we give $W_1$ corner $\cN_1 = \cN'$,
and $W_2$ corner $\cN_2 = \cN''$;
on the right, we give $W_1$ corner $\cN_1 = \cN \sqcup \cN''$,
and $W_2$ corner $\cN_2 = \cN''$.
Then,
\begin{equation*}
\cF_M(g) (\cF_M(f) (\vphi))
= A(g \tnsr A(f \tnsr \vphi))
= A(g \tnsr f \tnsr \vphi)
= A((g \circ f) \tnsr \vphi)
= \cF_M(g \circ f) (\vphi)
\end{equation*}

The naturality of $\cF_W$ follows from a similar argument:
from

\begin{equation*}
\begin{tikzpicture}
\node[dotnode] (N1) at (0,0) {};
\node[dotnode] (N2) at (0,-1) {};
\node at (0.3,0) {\smallerer $\cN_1$};
\node at (0.3,-1) {\smallerer $\cN_2$};
\node[emptynode] (A1) at (-2,0.5) {};
\node[emptynode] (A2) at (-2,-0.5) {};
\node[emptynode] (A3) at (-2,-1.5) {};
\draw[dotted] (A1) -- +(left:0.5cm) node[left] {\smallerer $M$};
\draw[dotted] (A2) -- +(left:0.5cm) node[left] {\smallerer $M'$};
\draw[dotted] (A3) -- +(left:0.5cm) node[left] {\smallerer $M'$};
\node at (-2,0) {\smallerer $W_1 = W$};
\node at (-1.8,-1) {\smallerer $W_2 = M' \times I$};
\draw (N1) -- (N2);
\node at (0.2,-0.5) {\smallerer $f$};
\draw (N1) to[out=150,in=0] (A1);
\draw (N2) to[out=-150,in=0] (A3);
\draw (N1) to[out=-150,in=0] (A2);
\end{tikzpicture}
\;\;\;\;
\begin{tikzpicture}
\node[dotnode] (N1) at (0,0) {};
\node[dotnode] (N2) at (0,-1) {};
\node at (0.3,0) {\smallerer $\cN_1$};
\node at (0.3,-1) {\smallerer $\cN_2$};
\node[emptynode] (A1) at (-2,0.5) {};
\node[emptynode] (A2) at (-2,-0.5) {};
\node[emptynode] (A3) at (-2,-1.5) {};
\draw[dotted] (A1) -- +(left:0.5cm) node[left] {\smallerer $M$};
\draw[dotted] (A2) -- +(left:0.5cm) node[left] {\smallerer $M$};
\draw[dotted] (A3) -- +(left:0.5cm) node[left] {\smallerer $M'$};
\node at (-1.8,0) {\smallerer $W_1 = M \times I$};
\node at (-2,-1) {\smallerer $W_2 = W$};
\draw (N1) -- (N2);
\node at (0.2,-0.5) {\smallerer $f$};
\draw (N1) to[out=150,in=0] (A1);
\draw (N2) to[out=-150,in=0] (A3);
\draw (N2) to[out=150,in=0] (A2);
\end{tikzpicture}
\end{equation*}

we have
\[
\cF_{M'}(f) (\cF_{W_1} (\vphi))
= A(f \tnsr \ZCY(W_1;\cN_1)(\vphi))
= \ZCY(W_2;\cN_2)(A(f \tnsr \vphi))
= \cF_{W_2} ( \cF_M(f) (\vphi))
\]

The naturality of $\cF_W$
allows us to extend $\cF_M$ and $\cF_W$
to the Karoubi envelope:
for an object $((\cN,l),P) \in \ZCY(N)$,
\[
\cF_M( ( (\cN,l), P)) = P \cdot \cF_M( (\cN,l))
\]
(here $P$ is acting by $(P \circ -)$)
and for a state
$\vphi = P \circ \vphi \in P \cdot \cF_M((\cN,l))$,
\[
\cF_W (\vphi) = \cF_W (A(P \tnsr \vphi))
= A(P \tnsr \cF_W(\vphi))
\in P \cdot \cF_{M'}( (\cN,l))
\]
\end{proof}



\newpage
\section{Skein Categories}
\label{s:skein-cat}
\vspace{0.8in}

\subsection{Skeins}
\label{s:sk-defn}
\par \noindent



In this section, we give a definition of colored graphs/skeins.
This definition essentially coincides with those given in
\ocite{KT}, \ocite{freyd}, \ocite{cooke},
however, here everything is performed in the PL category.

Throughout this section, all 3-manifolds are assumed to be oriented,
and may be non-compact and/or with boundary.

In \ocite{KT}, we define skeins using framed graphs,
that is, an embedded graph with a normal vector field.
We can convert a (smooth) ribbon graph into a framed graph
by shrinking the coupons to points,
and taking the cores of ribbons to be the edges of the graph,
with normal vector field given by the normal to the surface of the
ribbon graph (these are defined up to a contractible choice).

The constructions in this section mirror those of
\secref{s:cy-extended-plcw};
boundary values (\defref{d:boundary-value-skein})
correspond to
colored marked PLCW decompositions on oriented surfaces,
and skeins (\defref{d:skein_3d})
correspond to states.
This will be made explicit in \secref{s:sk-equiv}.

\begin{definition}
Let $N$ be an (unoriented) surface, possibly noncompact or with boundary.
A \emph{boundary value} $\VV = (B, \{V_{\vec{b}}\}_{\vec{b}\in \vec{B}})$
is a finite collection $B$ of embedded unoriented arcs $b : [-1,1] \to N$,
together with an assignment of an object $V_{\vec{b}} \in \cA$
to each oriented arc $\vec{b}$,
such that $v_{\cev{b}} = v_{\vec{b}}^*$.
\label{d:boundary-value-skein}
\end{definition}

The assignment $V_{\vec{b}}$,
in particular the duality under changing orientations,
is meant to mirror labelings
(see \secref{s:sk-equiv}).

\begin{definition}
The \emph{empty configuration} or \emph{empty boundary value}
on the surface $N$,
denoted $\EE_N$ or simply $\EE$,
is the boundary value with no marked points,
$\EE = (\emptyset,\{\}) \in \ZCYsk(N)$.
\label{d:empty-boundary-value}
\end{definition}

\begin{definition}
Given an embedding of surfaces $f: N \hookrightarrow N'$
and a boundary value $\VV$ on $N$,
we may define the boundary value
$f_*(\VV) = (f(B), \{V_{\vec{b}}\})$ on $N'$.
\end{definition}


\begin{definition}
Let $M$ be an oriented 3-manifold with boundary $N = \del M$.
A colored ribbon graph $(\Gamma,\Psi)$ in $M$ defines a boundary value
$\VV = (B, \{V_{\vec{b}}\})$ on $N$,
with $B = \Gamma \cap N$,
and for $\vec{b} = \vec{e} \cap N$, where $\vec{e}$ is a leg of $\Gamma$
taken with direction pointing outwards at $b$,
and $\vec{b}$ is given the outward orientation with respect to $\Gamma$
as an oriented surface,
we take $V_{\vec{b}} = \Psi(\vec{e})$,
and $V_{\cev{b}} = \Psi(\cev{e})$.
We define
\[
  \Graph(M;\VV) =
    \text{set of all colored ribbon graphs
      in $M$ with boundary value } \VV
\]
and similarly consider formal linear combinations:
\[
\VGraph(M;\VV)=\{\text{formal linear combinations of graphs }
  \Gamma \in \Graph(M;\VV)\}
\]

\label{d:vgraph}
\end{definition}

We denote the empty graph in $M$ by $\emptyset_M^{\Gamma}$;
it is a basis element of $\VGraph(M;\EE)$,
and should not be confused with the 0 element in $\VGraph(M;\EE)$.


\begin{definition}
Let $\Gamma = \sum c_i \Gamma_i \in \VGraph(M;\VV)$
be a formal linear combination of colored graphs,
and $D$ an embedded closed ball in $M$.
(Note $D$ is allowed to touch the boundary $\del M$.)
We say that $\Gamma$ is \emph{null with respect to $D$}
if either
\begin{enumerate}
\item $\Gamma = \zeta^t(\Gamma^0) - \Gamma^0$,
	where $\zeta^t$ is an ambient isotopy supported on $D$,
	or
\item each $\Gamma_i$ meets $\del D$ transversally,
	all $\Gamma_i$ coincide outside of $D$ as colored graphs,
	and
	\[
		\eval{\Gamma}_D = \sum c_i \eval{\Gamma_i}_D = 0
	\]
	where $\eval{\Gamma}_D$ is the Reshetikhin-Turaev evaluation
	of $\Gamma$ in a ball $D$ (see \prpref{p:rt-evaluation-ball};
	some choice of identification $D \simeq B^3$ is needed,
	but it is clear that the choice is irrelevant in this context).
\end{enumerate}
\end{definition}

Note that if we are in the first situation,
and $\Gamma$ is transversal to $\del D$,
then the second condition subsumes the first,
since isotopic graphs have the same Reshetikhin-Turaev evaluation.
Thus, if we imagine all ribbon graphs to be narrow enough
(or even infinitesimally narrow as in
\rmkref{r:narrowing-infinitesimal}),
then we can essentially disregard the first condition.

\begin{definition}\label{d:skein_3d}
For an oriented 3-manifold $M$ and boundary condition
$\VV=(B, \{V_b\})$ on $\del M$, 
we define the
\emph{skein module of $M$ with boundary value $\VV$} as
\[
  \ZCYsk(M; \VV)=\VGraph(M;\VV)/Q 
\]
where $Q$ is the subspace spanned by all null graphs
(for all possible embedded balls).
An element of $\ZCYsk(M;\VV)$ is called a
\emph{skein with boundary value $\VV$},
or simple a skein.
\end{definition}

\begin{definition}
The \emph{empty skein}, denoted $\emptyskein{M}$,
is the equivalence class of the empty graph
$[\emptyset_M^\Gamma] \in \ZCYsk(M,\EE)$.
\label{d:emptyskein}
\end{definition}

For example, $\ZCYsk(S^3) \simeq \kk$,
as a ribbon graph $\Gamma \subset S^3$ is always contained in
the interior of some ball $D$,
which has some Reshetikhin-Turaev evaluation
$\ZRT(\Gamma) \in \kk$,
and so $\Gamma$ is equivalent, as a skein,
to $\ZRT(\Gamma) \cdot \emptyskein{M}$.

We also have $\ZCYsk(\emptyset) \simeq \kk$,
as expected from a 4-dimensional TQFT,
since the empty graph is the only graph in the empty 3-manifold,
and there are no relations.

Ribbons labeled with $\one$ can essentially be ignored:

\begin{equation*}
\begin{tikzpicture}
\begin{scope}[shift={(0,-0.25)}]
\draw[fill=gray, line width=0] (1,0.8) circle (0.8cm);
\node at (1.7,1.5) {\smallerer $D$};
\node at (2.1,0.5) {\smallerer $\one$};
\node[small_morphism] (ph1) at (0,0) {\tiny $\vphi_1$};
\node[small_morphism] (ph2) at (2,0) {\tiny $\vphi_2$};
\draw (ph1) to[out=90,in=90] (ph2);
\draw (ph1) -- +(-120:1cm);
\draw (ph1) -- +(-60:1cm);
\draw (ph2) -- +(-60:1cm);
\draw (ph2) -- +(-90:1cm);
\draw (ph2) -- +(-120:1cm);
\end{scope}
\end{tikzpicture}
\;\;=\;\;
\begin{tikzpicture}
\begin{scope}[shift={(0,-0.25)}]
\draw[fill=gray, line width=0] (1,0.8) circle (0.8cm);
\node at (1.7,1.5) {\smallerer $D$};
\node at (2.1,0.5) {\smallerer $\one$};
\node at (-0.1,0.5) {\smallerer $\one$};
\node[small_morphism] (ph1) at (0,0) {\tiny $\vphi_1$};
\node[small_morphism] (ph2) at (2,0) {\tiny $\vphi_2$};
\node[dotnode] (d1) at (0.5,0.7) {};
\draw (d1) to[out=-150,in=90] (ph1);
\node[dotnode] (d2) at (1.5,0.7) {};
\draw (d2) to[out=-30,in=90] (ph2);
\draw (ph1) -- +(-120:1cm);
\draw (ph1) -- +(-60:1cm);
\draw (ph2) -- +(-60:1cm);
\draw (ph2) -- +(-90:1cm);
\draw (ph2) -- +(-120:1cm);
\end{scope}
\end{tikzpicture}
\;\;\;\;\;,\;\;\;\;\;
\begin{tikzpicture}
\draw[fill=gray, line width=0] (0,0) circle (0.8cm);
\node at (0.8,-0.6) {\smallerer $D$};
\node at (-0.1,0.5) {\smallerer $\one$};
\node[small_morphism] (ph1) at (0,0) {\tiny $\vphi_1$};
\node[dotnode] (d1) at (0.3,0.5) {};
\draw (d1) to[out=-150,in=90] (ph1);
\draw (ph1) -- +(-120:1cm);
\draw (ph1) -- +(-60:1cm);
\end{tikzpicture}
\;\;=\;\;
\begin{tikzpicture}
\draw[fill=gray, line width=0] (0,0) circle (0.8cm);
\node at (0.8,-0.6) {\smallerer $D$};
\node[small_morphism] (ph1) at (0,0) {\tiny $\vphi_1$};
\draw (ph1) -- +(-120:1cm);
\draw (ph1) -- +(-60:1cm);
\end{tikzpicture}
\end{equation*}

\begin{definition}
Given a homeomorphism $f : M \simeq M'$
and a boundary value $\VV$ on $M$,
we define
\[
f_* : \ZCYsk(M;\VV) \simeq \ZCYsk(M';f_*(\VV))
\]
by simply applying $f$ to graphs.
\label{d:skein-pushforward}
\end{definition}

\begin{definition}
Let $M'$ be a 3-manifold obtained from $M$
by gluing two boundary components $N,N'$
together via a diffeomorphism $f : \ov{N} \simeq N'$.
Let $\VV$ be a boundary value on $N$,
and $\VV' = f_*(\VV)$.

Then for $\vphi \in \ZCYsk(M; \VV,\VV',\WW)$,
where $\WW$ is some boundary value on the other boundary components
of $M'$, we define $\vphi/_N$ to be the skein
obtained as the image of $\vphi$ under the gluing
$M \to M/f = M'$.
If $N$ separates $M = M_1 \cup_N M_2$,
and $\vphi_1 \in \ZCYsk(M_1; \VV,..),
\vphi_2 \in \ZCYsk(M_2;,\VV',..)$,
we also use the notation
\[
\vphi_1 \cup_N \vphi_2 = (\vphi_1 \tnsr \vphi_2)/_N
\]
\label{d:glue-skeins}
\end{definition}

We can now define the category of boundary conditions. 

\begin{definition}
Let $N$ be an oriented surface, possibly non-compact.
Suppose first $N$ has no boundary. Define $\hatZCYsk(N)$
as the category whose objects are boundary values on $N$,
and the morphisms are given by:
\[
  \Hom_{\hatZCYsk(N)}(\VV,\VV')
  = \ZCYsk(N \times [0,1]; \iota_*^0(\VV),\iota_*^1(\VV')),
\]
where $\iota^j : N \simeq N \times \{j\}$,
and $N \times [0,1]$ is oriented so that $N \times 1$
(with the orientation from $N$) has the outward orientation.
The identity morphism $\id_\VV$, for $\VV = (B,\{V_{\vec{b}}\})$,
is given by the ``vertical'' graph $B \times [0,1]$
with the obvious coloring.

Composition of morphisms $\vphi \in \Hom_{\hatZCYsk}(\VV,\VV')$,
$\vphi' \in \Hom_{\hatZCYsk}(\VV',\VV'')$
is given by
\[
\vphi' \circ \vphi = \vphi' \cup_N \vphi
\]
i.e. joining the underlying graphs along $\VV'$.

$\hatZCYsk$ is additive and $\kk$-linear.
We define
\begin{equation}
  \label{e:zcy}
\ZCYsk(N) = \Kar(\hatZCYsk(N)),
\end{equation}
its Karoubi envelope.

For $N$ with boundary, we define $\hatZCYsk(N) = \hatZCYsk(N \backslash \del N)$,
$\ZCYsk(N) = \ZCYsk(N \backslash \del N)$.

We call $\ZCYsk(N)$ the \emph{skein category of $N$},
or the \emph{category of boundary values on $N$}.
\end{definition}

It is immediate from the definition that for a 2-disk $\DD^2$,
$\ZCYsk(\DD^2) \simeq \cA$;
by choosing some arc $b$ in $\DD^2$,
for $A,A' \in \Hom_\cA(A,A')$,

\begin{equation}
\label{e:A-ZCYD}
\begin{tikzpicture}
\node[small_morphism] (f) at (0,0) {\smallerer $f$};
\node (A) at (0,0.7) {$A$};
\node (A')at (0,-0.7) {$A'$};
\draw (f) -- (A);
\draw (f) -- (A');
\node at (0.5,0.7) {$\mapsto$};
\node at (0.5,-0.7) {$\mapsto$};
\node at (0.5,0) {$\mapsto$};
\begin{scope}[shift={(1,-0.1)}]
\draw (0,0.7) -- (1,0.7);
\draw (0,0.7) -- (0,-0.7);
\draw (0,-0.7) -- (1,-0.7);
\draw (1,0.7) -- (1,-0.7);
\draw (0.2,0.9) -- (1.2,0.9);
\draw (0.2,0.9) -- (0.2,-0.5);
\draw (0.2,-0.5) -- (1.2,-0.5);
\draw (1.2,0.9) -- (1.2,-0.5);
\draw (0,0.7) -- +(0.2,0.2);
\draw (1,0.7) -- +(0.2,0.2);
\draw (0,-0.7) -- +(0.2,0.2);
\draw (1,-0.7) -- +(0.2,0.2);
\node[small_morphism] (f) at (0.6,0.1) {\smallerer $f$};
\draw[ribbon] (f) -- (0.6,0.8);
\draw[ribbon] (f) -- (0.6,-0.6);
\draw[line width=0.3pt] (0.525,0.8) -- (0.675,0.8);
\draw[line width=0.3pt] (0.525,-0.6) -- (0.675,-0.6);
\node at (0.4,0.8) {\tiny $b$};
\node at (0.4,-0.6) {\tiny $b$};
\node at (0.6,1.1) {\smallerer $A$};
\node at (0.6,-0.9) {\smallerer $A'$};
\draw (0,0.7) -- (1,0.7);
\end{scope}
\end{tikzpicture}
\end{equation}

Next let us address the 4-dimensional component of $\ZCYsk$.

\begin{definition}
Let $W = \cH_k \wdtld{\circ} \id_M : M \to_N M'$
be an elementary cornered cobordism of index $k$.
(see \defref{d:elementary-cobordism}).
For a boundary value $\VV \in \ZCYsk(N)$,
we define
\[
\ZCYsk(W;N) : \ZCYsk(M;\VV) \to \ZCYsk(M';\VV)
\]
case-by-case: given a colored ribbon graph
$\Gamma \in \ZCYsk(M;\VV)$,
$\ZCYsk(W;N)(\Gamma)$ is constructed as follows:
\begin{itemize}
\item $k=0$: $\cH_0$ adds a new $S^3$ component to $M$;
	we define
	\[
		\ZCYsk(W;N)(\Gamma) :=
	\cD \cdot \Gamma \cup \emptyskein{S^3}
	\]
\item $k=4$: $\cH_4$ kills off an $S^3$ component of $M$;
	writing $\Gamma = \Gamma' \sqcup \Gamma''$
	with $\Gamma'$ in the $S^3$ component
	and $\Gamma''$ in the other component of $M$,
	\[
		\ZCYsk(W;N)(\Gamma) :=
		\ZRT(\Gamma') \cdot \Gamma''
	\]
\item $k=1$: the attaching region of $\cH_1$ is a pair of balls;
	by an isotopy, we may arrange that $\Gamma$ is disjoint
	from the attaching region,
	and regard $\Gamma$ as a graph in $M'$,
	and define
	\[
		\ZCYsk(W;N)(\Gamma) := \cD^\inv \cdot \Gamma
	\]
\item $k=2$: similar to the $k=1$ case,
	arrange $\Gamma$ to be disjoint from the attaching region.
	The belt sphere of $\cH_2$ (and its neighborhood)
	defines an embedding of the solid torus $B^2 \times \del B^2 \to M'$.
	Let $\gamma = [-\eps,\eps] \times \del B^2 \subset B^2 \times \del B^2$
	be the core of the solid torus, trivially framed,
	and let $\Gamma' = (\gamma, \text{regular})$
	be the colored ribbon graph obtained by applying the regular coloring
	to $\gamma$.
	Then, sending $\Gamma'$ to $M'$ under the embedding,
	\[
		\ZCYsk(W;N)(\Gamma) = \Gamma \cup \Gamma'
	\]
\item $k=3$: first suppose that $\Gamma$ intersects
	the co-core of $\cH_3$ transversally and in exactly one ribbon,
	and suppose the label of this ribbon is a simple object $i$.
	If $i=\one$, we may ignore this ribbon and isotope
	$\Gamma$ to be disjoint from $\cH_3$;
	then we may define
	\[
		\ZCYsk(W;N)(\Gamma) = \delta_{i,\one} \cdot \Gamma
	\]
	In general, by isotopy, we may arrange $\Gamma$ to be transverse
	to the co-core of $\cH_3$, and then apply \lemref{l:summation}
	to get $\Gamma = \sum_j \Gamma_j$,
	where each $\Gamma_j$ satisfies the previous assumption
	(see \eqnref{e:ups-isom-01}).
	Then we extend to this case by linearity.
\end{itemize}
\label{d:zcyksk-elementary}
\end{definition}

It is not hard to see that this definition is well-posed:
\begin{itemize}
\item for $k=0,3,4$, it is clear that $\ZCYsk(W;N)(\Gamma)$
is well-defined,
\item for $k=2$, well-defined-ness follows from the sliding lemma
\lemref{lem:sliding}, and
\item for $k=3$, well-defined-ness follows from
	\lemref{l:summation-alpha-tunnel}
	(see also \eqnref{e:ups-isom-01}).
\end{itemize}

\begin{definition}
Let $W : M \to_N M'$ be a cornered cobordism,
and let $W = W_l \circ \cdots \circ W_1$
be a handle decomposition
(\defref{d:handle-decomposition}).
We define
\[
\ZCYsk(W;N) = \ZCYsk(W_l;N) \circ \cdots \circ \ZCYsk(W_1;N)
\]
\label{d:zcysk-corner}
\end{definition}

\begin{proposition}
The map $\ZCYsk(W;N)$ defined in \defref{d:zcysk-corner}
is independent of handle decomposition of $W$.
\label{p:zcysk-corner-indep}
\end{proposition}

\begin{proof}
By \prpref{p:handle-decomp-relate},
it suffices to check that any modification as in
\defref{d:handle-decomp-modification} leaves
the overall $\ZCYsk(W;N)$ unchanged.

It is clear that reordering handles does not affect $\ZCYsk(W;N)$.

Let us consider handle pair cancellation.
We consider a graph $\Gamma$ in $M$,
then consider a pair of canceling elementary cobordisms,
and observe what happens to $\Gamma$.
Canceling a 0-1 pair is clear, as the $\cD^\pm$ factors cancel.
A 1-2 canceling pair looks like
\[
\frac{1}{\cD} \;\;\;
\begin{tikzpicture}
\draw[regular] (1,0.2) to[out=0,in=0] (1,-0.3);
\draw[overline] (0.5,-0.05) -- (2,-0.05);
\draw[regular, overline] (1,0.2) to[out=180,in=180] (1,-0.3);
\draw (0,0) circle (0.5cm);
\draw (2.5,0) circle (0.5cm);
\draw (-0.5,0) to[out=-90,in=-90] (0.5,0);
\draw (2,0) to[out=-90,in=-90] (3,0);
\node at (1.7,0.1) {\smallerer $n$};
\node (a) at (-0.2,-1) {\smallerer 1-handle};
\node at (-0.2,-1.3) {\smallerer attaching region};
\node (b) at (2.2,-1) {\smallerer 2-handle};
\node at (2.2,-1.3) {\smallerer attaching sphere};
\draw[line width=0.4pt] (a) -- (0,0);
\draw[line width=0.4pt] (a) -- (2.5,0);
\draw[line width=0.4pt] (b) -- (1.6,-0.05);
\end{tikzpicture}
\]
The dashed loop, coming from the belt sphere of $\cH_2$,
can be pull out, and evaluated to $\cD$,
thus canceling the $1/\cD$ factor.
For a 2-3 pair, we first have $\Gamma \cup \Gamma'$
from $\cH_2$. Then the graph $\gamma$ from $\cH_2$ must meet
the co-core of $\cH_3$ exactly once,
and is cut, leaving only the $i=\one$ strand,
which leaves us with $\Gamma$ again.
For a 3-4 pair, the attaching region of $\cH_3$ is
$\del B^3 \times B^1$,
and the fact that $\cH_3$ cancels with $\cH_4$
implies this attaching region separates
$M$ into $M' \sqcup B^3$
(the $B^3$, after attaching $\cH_3$,
becomes an $S^3$ component that gets killed by $\cH_4$).
Thus, $\Gamma$ can be isotoped away from $B^3$,
and then away from the attaching region,
thus it is not affected by attaching both $\cH_3$ and $\cH_4$.

Finally let us consider handle slides.
There is nothing to prove for 0- and 4-handles,
and for 1-handles it is clear.
For 2-handles,
independence follows from the sliding lemma
(\lemref{lem:sliding}),
which is the same way that the Reshetikhin-Turaev invariant
for 3-manifolds is shown to be invariant with respect to
handle slides in Kirby calculus.
For 1-handles, we have
\[
\begin{tikzpicture}
\draw[densely dotted] (-0.8,0) to[out=90,in=90] (0.8,0);
\draw[densely dotted] (1.4,0) to[out=90,in=90] (3,0);
\draw[fill=med-gray,draw=white,line width=0pt] (0,0) circle (0.4cm);
\draw[fill=med-gray,draw=white,line width=0pt] (2.2,0) circle (0.4cm);
\draw[draw=med-gray,line width=7pt,opacity=0.3] (0,0) circle (0.8cm);
\draw[draw=med-gray,line width=7pt,opacity=0.3] (2.2,0) circle (0.8cm);
\draw (0,0) circle (0.8cm);
\draw (2.2,0) circle (0.8cm);
\draw (-0.8,0) to[out=-90,in=-90] (0.8,0);
\draw (1.4,0) to[out=-90,in=-90] (3,0);
\node[small_morphism] (a1) at (-1.05,0) {\tiny $\al$};
\node[small_morphism] (a2) at (-0.55,0) {\tiny $\al$};
\draw (a1) -- +(180:0.5cm);
\draw (a1) -- +(150:0.5cm);
\draw (a1) -- +(-150:0.5cm);
\draw (a2) -- +(0:0.5cm);
\draw (a2) -- +(30:0.5cm);
\draw (a2) -- +(-30:0.5cm);
\node[small_morphism,inner sep=0.5pt] (b1) at (3.25,0) {\tiny $\be$};
\node[small_morphism,inner sep=0.5pt] (b2) at (2.75,0) {\tiny $\be$};
\draw (b1) -- +(0:0.5cm);
\draw (b1) -- +(30:0.5cm);
\draw (b1) -- +(-30:0.5cm);
\draw (b2) -- +(180:0.5cm);
\draw (b2) -- +(150:0.5cm);
\draw (b2) -- +(-150:0.5cm);
\end{tikzpicture}
\;\;\;
=
\;\;\;
\begin{tikzpicture}
\draw[densely dotted] (1.4,0) to[out=90,in=90] (3,0);
\draw[fill=med-gray,draw=white,line width=0pt] (0,0) circle (0.4cm);
\draw[fill=med-gray,draw=white,line width=0pt] (2.2,0) circle (0.4cm);
\draw[draw=med-gray,line width=7pt,opacity=0.3] (2.2,0) circle (0.8cm);
\draw[preaction={draw=med-gray,line width=7pt,opacity=0.3,-}] (-0.8,0)
	.. controls +(-90:0.5cm) and +(180:0.5cm) .. (0,-0.8)
	.. controls +(0:0.6cm) and +(180:0.2cm) .. (0.9,-0.2)
	.. controls +(0:0.3cm) and +(180:1cm) .. (2.2,-1.1)
	.. controls +(0:0.7cm) and +(-90:0.7cm) .. (3.3,0)
	.. controls +(90:0.7cm) and +(0:0.7cm) .. (2.2,1.1)
	.. controls +(180:1cm) and +(0:0.3cm) .. (0.9,0.2)
	.. controls +(180:0.2cm) and +(0:0.6cm) .. (0,0.8)
	.. controls +(180:0.5cm) and +(90:0.5cm) .. (-0.8,0);
\draw (2.2,0) circle (0.8cm);
\draw (1.4,0) to[out=-90,in=-90] (3,0);
\node[small_morphism] (a1) at (-1.05,0) {\tiny $\al$};
\node[small_morphism] (a2) at (-0.55,0) {\tiny $\al$};
\draw (a1) -- +(180:0.5cm);
\draw (a1) -- +(150:0.5cm);
\draw (a1) -- +(-150:0.5cm);
\draw (a2) -- +(0:0.5cm);
\draw (a2) -- +(30:0.5cm);
\draw (a2) -- +(-30:0.5cm);
\node[small_morphism,inner sep=0.5pt] (b1) at (2.75,0) {\tiny $\be$};
\node[small_morphism,inner sep=0.5pt] (b2) at (3.55,0) {\tiny $\be$};
\draw (b1) -- +(180:0.5cm);
\draw (b1) -- +(150:0.5cm);
\draw (b1) -- +(-150:0.5cm);
\draw (b2) -- +(0:0.5cm);
\draw (b2) -- +(30:0.5cm);
\draw (b2) -- +(-30:0.5cm);
\end{tikzpicture}
\]
\end{proof}

Note that \thmref{t:sk-equiv-4mfld}
provides an alternative proof to \prpref{p:zcysk-corner-indep}.

It is easy to get a formula relating connected sums
to their components:

\begin{proposition}
For non-empty closed 4-manifolds $W_1,W_2$,
\[
\ZCYsk(W_1 \# W_2) = \cD^\inv \cdot \ZCYsk(W_1) \cdot \ZCYsk(W_2)
\]
More generally, for non-empty 4-manifolds $W_1,W_2$,
possibly with boundary,
\[
\ZCYsk(W_1 \# W_2) = \cD^\inv \cdot \ZCYsk(W_1) \tnsr \ZCYsk(W_2)
\in \ZCYsk(\del W_1) \tnsr \ZCYsk(\del W_2)
\]
\label{p:zcysk-connectsum}
\end{proposition}
\begin{proof}
Let $m = 1,2$.
Choose some handle decomposition for $W_m$
as a cobordism $W_m:\emptyset \to \del W_m$
that contains at least one handle
(we can create one using handle pair creation;
see \defref{d:handle-decomp-modification}
or proof of \prpref{p:zcysk-corner-indep});
let $\cH_4^{(m)}$ be such a 4-handle.

Let $W_m' = W_m \backslash \cH_4^{(m)}$.
Note that $\del W_m' = \del W_m \sqcup S^3$.
It is clear from \defref{d:zcyksk-elementary} that
\[
\ZCYsk(W_m') = \ZCYsk(W_m) \tnsr \emptyskein{S^3}
\in \ZCYsk(\del W_m) \tnsr \ZCYsk(S^3)
\simeq \ZCYsk(\del W_m')
\]
We can obtain $W_1 \# W_2$
by attaching a 1-handle that connects the $S^3$ components
from $\del W_1'$ and $\del W_2'$,
then capping the resulting $S^3$ component off with a 4-handle.
Together these operations give
\begin{align*}
\ZCYsk(W_1') \tnsr \ZCYsk(W_2')
&= (\ZCYsk(W_1) \tnsr \emptyskein{S^3})
	\tnsr (\ZCYsk(W_2) \tnsr \emptyskein{S^3})
\\
&\overset{\text{1-handle}}{\mapsto}
	\cD^\inv \cdot \ZCYsk(W_1) \tnsr \tnsr \ZCYsk(W_2) \tnsr \emptyskein{S^3}
\\
&\overset{\text{4-handle}}{\mapsto}
	\cD^\inv \cdot \ZCYsk(W_1) \tnsr \tnsr \ZCYsk(W_2)
\end{align*}
(Note that the arguments above work for a self-connect sum,
i.e. if $W_1 = W_2$.)
\end{proof}

\begin{lemma}
Let $W: M_0 \to_N M_1$ be a cornered cobordism,
and let $M : N \to N'$ be a cobordism.
Recall the cornered cobordism obtained from $W$ by extending along $M$
from \defref{d:cornered-cobordism-extend}
is the cornered cobordism
$W_{M: N \to N'} = W \wdtld{\circ} \id_{M_0 \cup_N M}
	: M_0 \cup_N M \to_{N'} M_1 \cup_N M$.
Then for $\vphi \in \ZCYsk(M_0;\VV), \vphi' \in \ZCYsk(M;\VV,\VV')$,
\[
\ZCYsk(W_{M: N \to N'};N')(\vphi' \cup_N \vphi)
= \vphi' \cup_N \ZCYsk(W)(\vphi)
\]
\label{l:cornered-cobordism-extend-zcysk}
\end{lemma}
\begin{proof}
By construction.
\end{proof}

We consider a
skein-theoretic analog of \prpref{p:zcy-corner-gluing-alt}:

\begin{proposition}
\label{p:zcy-corner-gluing-skein}
Let $W_1,W_2$ be cornered cobordisms
\begin{align*}
W_1 &: M_1 \to_{N_1} M_1'
\\
W_2 &: M_2 \to_{N_2} M_2'
\end{align*}
Suppose $M_1' \subseteq M_2$ is a submanifold.
Let $W = W_1 \cup_{M_1'} W_2$,
and $M = (M_2 \backslash M_1') \cup M_1$,
so that $W$ is the extended cobordism (\defref{d:composition-extended})
\[
W : M \to_{N_2} M_2'
\]

Then $\ZCYsk(W_1; N_1)$ and $\ZCYsk(W_2;N_2)$ compose to give
$\ZCYsk(W; N_2)$; more precisely,
\begin{align*}
\ZCYsk(W; N_2) &= \ZCYsk(W_2; N_2) \circ
	(\id \cup \ZCYsk(W_1; N_1))
\\
\ZCYsk(W;N_2)(\Phi)
&= \ZCYsk(W_2;N_2) (\psi \cup \ZCYsk(W_1;N_1) (\vphi))
\end{align*}
where $\Phi \in \ZCYsk(M;\VV)$
can be obtained by gluing colored graphs
$\vphi \subset M_1, \psi \subset M \backslash M_1$,
and $\VV$ is some boundary value on $N_2$,

Similarly, suppose that we have the reverse inclusion
$M_2 \subseteq M_1'$.
Let $W = W_1 \cup_{M_2} W_2$,
and $M' = (M_1' \backslash M_2) \cup M_2'$,
so that $W$ is the extended cobordism (\defref{d:composition-extended})
\[
W : M_1 \to_{N_1} M_2'
\]

Then $\ZCYsk(W_1; N_1)$ and $\ZCYsk(W_2;N_2)$ compose to give
$\ZCYsk(W; N_1)$; more precisely,
\begin{align*}
\ZCYsk(W; N_1) &= (\id \tnsr \ZCYsk(W_2; N_2))
	\circ (\ZCYsk(W_1; N_1))
\\
\ZCYsk(W;N_1)(\Phi)
&= \sum_a \vphi^{(a)} \cup \ZCYsk(W_2;N_2) (\psi^{(a)})
\end{align*}
where $\ZCYsk(W_1;N_1) (\Phi) = \sum_a \vphi^{(a)} \cup \psi^{(a)}$,
with $\psi^{(a)} \in \ZCYsk(M_2;\VV)$ for some boundary value
$\VV$ on $N_2$.
\end{proposition}

\begin{proof}
Also by construction.
\end{proof}

We have an analog of \prpref{p:zcy-surface-functors}:

\begin{proposition}
\label{p:zcy-surface-functors-skein}
Let $N$ be an oriented closed surface.
A 3-manifold $M$ bounded by $\del M = N$
defines a functor
\begin{align*}
\Fsk{M} : \ZCY(N) &\to \Vect
\\
\VV &\mapsto \ZCYsk(M;\VV)
\end{align*}

and for a morphism $f : \VV \to \VV'$,
\[
\Fsk{M}(f) = f \cup_N -
\]
Furthermore, a cornered cobordism $W : M \to_N M'$
defines a natural transformation
\begin{align*}
\Fsk{W} &: \Fsk{M} \to \Fsk{M'}
\\
(\Fsk{W})_{\VV} = \ZCYsk(W;N) &:
\ZCYsk(M;\VV) \to \ZCYsk(M';\VV)
\end{align*}
\end{proposition}
\begin{proof}
Similar to \prpref{p:zcy-surface-functors}.
\end{proof}

We also consider the analog of the reduced relative state space pairing
(\defref{d:reduced-relative-state-space-pairing}),
based on the more TQFT-esque interpretation
(see \lemref{l:reduced-pairing-TQFT}):

\begin{definition}
Observe that $W = M \times I$ has boundary $M \cup_N \ov{M}$
(where $N = \del M$);
we treat it as a cobordism
$W : M \cup_N \ov{M} \to \emptyset$
(see \lemref{l:reduced-pairing-TQFT}).
For a boundary value $\VV$ on $N$
and skeins $\vphi \in \ZCYsk(M;\VV), \vphi' \in \ZCYsk(\ov{M};\VV)$,
we define the \emph{skein pairing}
\begin{equation}
\evsk (\vphi, \vphi') = \ZCYsk(W) (\vphi \cup_N \vphi')
\label{e:pairing-skein}
\end{equation}
\label{d:pairing-skein}
\end{definition}

We will see in \secref{s:sk-equiv} that this pairing
is equivalent to the pairing in 
\defref{d:reduced-relative-state-space-pairing},
which is non-degenerate.

\begin{lemma}
A cornered cobordism $W : M \to_N M'$ defines a map
\[
\ZCYsk(W;N) : \ZCYsk(M;\VV) \to \ZCYsk(M';\VV)
\]
but as a cobordism $W : M \cup_N \ov{M'} \to \emptyset$,
we have
\[
\ZCYsk(W) : \ZCYsk(M;\VV) \tnsr \ZCYsk(\ov{M'};\VV)
	\to \ZCYsk(M \cup_N \ov{M'}) \to \kk
\]
Then for $\vphi \in \ZCYsk(M;\VV)$,
$\vphi' \in \ZCYsk(\ov{M'};\VV)$,
\[
\evsk(\ZCYsk(W;N)(\vphi), \vphi')
= \ZCYsk(W)(\vphi \cup \vphi')
\]
\label{l:skein-pairing-dual}
\end{lemma}

\begin{proof}
Follows immediately from \prpref{p:zcy-corner-gluing-skein}.
\end{proof}

\begin{example}
\label{x:skein-pairing-handlebody}
Let us give a more explicit description
of the skein pairing for handlebodies,
and show that it agrees (up to a factor of $\cD^{g/2}$)
with \ocite{BK}*{(4.4.5)}.
Consider the following skeins
in the genus 1 handlebodies (i.e. solid tori) $H$ and $\ov{H}$:
\begin{equation}
\label{e:skein-pairing-handlebodies-anchor}
\begin{tikzpicture}
\node[small_morphism] (a) at (0,-1) {\smallerer $\vphi$};
\draw[ribbon] (a)
	.. controls +(0:1cm) and +(-90:0.6cm) .. (1.9,0)
	.. controls +(90:0.6cm) and +(0:1cm) .. (0,1)
	.. controls +(180:1cm) and +(90:0.6cm) .. (-1.9,0)
	.. controls +(-90:0.6cm) and +(180:1cm) .. (a);
\draw[->] (0.3,-0.99) to[out=0,in=-175] (0.7,-0.96);
\node at (0.6,-1.15) {\smallerer $i$};
\draw[ribbon] (a) -- (0,-0.5);
\draw[->] (0,-0.8) -- (0,-0.55);
\node at (-0.25,-0.65) {\smallerer $V$};
\draw (-2.5,0)
	.. controls +(-90:0.6cm) and +(180:2cm) .. (0,-1.5)
	.. controls +(0:2cm) and +(-90:0.6cm) .. (2.5,0)
	.. controls +(90:0.6cm) and +(0:2cm) .. (0,1.5)
	.. controls +(180:2cm) and +(90:0.6cm) .. (-2.5,0);
\draw (-1.23,0)
	.. controls +(60:0.3cm) and +(180:0.7cm) .. (0,0.5)
	.. controls +(0:0.7cm) and +(120:0.3cm) .. (1.23,0);
\draw (-1.4,0.2)
	.. controls +(-60:0.3cm) and +(180:0.8cm) .. (0,-0.4)
	.. controls +(0:0.8cm) and +(-120:0.3cm) .. (1.4,0.2);
\node at (-2.3,-1.3) {$H$};
\end{tikzpicture}
\;\;\;\;\;\;\;\;\;
\begin{tikzpicture}
\node[small_morphism] (a) at (0,-0.95) {\smallerer $\vphi^,$};
\draw[ribbon] (a)
	.. controls +(0:1cm) and +(-90:0.6cm) .. (1.9,0)
	.. controls +(90:0.6cm) and +(0:1cm) .. (0,1)
	.. controls +(180:1cm) and +(90:0.6cm) .. (-1.9,0)
	.. controls +(-90:0.6cm) and +(180:1cm) .. (a);
\draw[->] (0.3,-0.94) to[out=0,in=-175] (0.7,-0.91);
\node at (0.6,-1.15) {\smallerer $j$};
\draw[ribbon] (a) -- (0,-1.45);
\draw[->] (0,-1.2) -- (0,-1.4);
\node at (-0.3,-1.3) {\smallerer $V^*$};
\draw (-2.5,0)
	.. controls +(-90:0.6cm) and +(180:2cm) .. (0,-1.5)
	.. controls +(0:2cm) and +(-90:0.6cm) .. (2.5,0)
	.. controls +(90:0.6cm) and +(0:2cm) .. (0,1.5)
	.. controls +(180:2cm) and +(90:0.6cm) .. (-2.5,0);
\draw (-1.23,0)
	.. controls +(60:0.3cm) and +(180:0.7cm) .. (0,0.5)
	.. controls +(0:0.7cm) and +(120:0.3cm) .. (1.23,0);
\draw (-1.4,0.2)
	.. controls +(-60:0.3cm) and +(180:0.8cm) .. (0,-0.4)
	.. controls +(0:0.8cm) and +(-120:0.3cm) .. (1.4,0.2);
\node at (2.3,-1.3) {$\ov{H}$};
\end{tikzpicture}
\end{equation}

Note that $\ov{H}$ is drawn after reflecting $H$
across a horizontal plane,
so that its interior appears with the same orientation
as $H$.

Let $N = \del H$.
Observe that $H$ is built from $N$ by first attaching
a 2-handle (its core is a vertical disc),
then a 3-handle (a ball that fills the rest of the interior of $H$).

The 4-manifold defining the skein pairing
(i.e. $H \times I$) is built similarly,
but with each handle bumped up one index.
We start with $H \cup_N \ov{H}$,
which we can visualize as follows:

\begin{equation}
\begin{tikzpicture}
\node[small_morphism,minimum size=0pt] (b) at (0,-1.4) {\tiny $\vphi$};
\node[small_morphism,inner sep=0,minimum size=0pt]
	(a) at (0,-0.8) {\tiny $\vphi^,$};
\draw[fill=light-gray] (0.8,-0.6)
	.. controls +(90:0.2cm) and +(-130:0.2cm) .. (0.9,-0.1)
	.. controls +(50:0.4cm) and +(90:0.5cm) .. (1.6,-1.2)
	.. controls +(-90:0.3cm) and +(35:0.2cm) .. (1.4,-1.76)
	.. controls +(-145:0.3cm) and +(-90:0.4cm) .. (0.8,-0.6);
\draw (1.6,-1.31) -- (0.8,-0.53);
\draw (-2.5,0)
	.. controls +(-90:1cm) and +(180:2cm) .. (0,-2)
	.. controls +(0:2cm) and +(-90:1cm) .. (2.5,0)
	.. controls +(90:1cm) and +(0:2cm) .. (0,2)
	.. controls +(180:2cm) and +(90:1cm) .. (-2.5,0);
\draw (-2.5,0.2)
	.. controls +(-90:0.6cm) and +(180:2cm) .. (0,-1.7)
	.. controls +(0:2cm) and +(-90:0.6cm) .. (2.5,0.2);
\draw (-1,0)
	.. controls +(60:0.3cm) and +(180:0.7cm) .. (0,0.5)
	.. controls +(0:0.7cm) and +(120:0.3cm) .. (1,0);
\draw (-1.2,0.4)
	.. controls +(-80:0.3cm) and +(180:0.8cm) .. (0,-0.4)
	.. controls +(0:0.8cm) and +(-100:0.3cm) .. (1.2,0.4);
\draw[midarrow={0.7}] (b) -- (a);
\draw[midarrow={0.9}] (b)
	.. controls +(0:1.5cm) and +(-90:0.6cm) .. (1.9,-0.2)
	.. controls +(90:0.6cm) and +(0:1.5cm) .. (0,1)
	.. controls +(180:1.5cm) and +(90:0.6cm) .. (-1.9,-0.2)
	.. controls +(-90:0.6cm) and +(180:1.5cm) .. (b);
\node[small_dotnode] at (1.2,-1.22) {};
\draw[midarrow={0.9}] (a)
	.. controls +(0:1.5cm) and +(-90:0.6cm) .. (1.9,0.4)
	.. controls +(90:0.6cm) and +(0:1.5cm) .. (0,1.6)
	.. controls +(180:1.5cm) and +(90:0.6cm) .. (-1.9,0.4)
	.. controls +(-90:0.6cm) and +(180:1.5cm) .. (a);
\node[small_dotnode] at (1.2,-0.62) {};
\draw[thin_overline={0.8}] (1.79,-0.1) to[out=70,in=-95] (1.885,0.2);
\draw[thin_overline={0.8}] (-1.79,-0.1) to[out=110,in=-85] (-1.885,0.2);
\node at (-1,-1.1) {\tiny $i$};
\node at (-1,-0.5) {\tiny $j$};
\node at (-0.2,-1.1) {\tiny $V$};
\node at (-2.5,1) {\smallerer $\ov{H}$};
\node at (-2.5,-1) {\smallerer $H$};
\draw[<->] (2.6,-0.4) to[out=0,in=-90] (3,0) to[out=90,in=0] (2.6,0.4);
\node at (3.3,0) {\smallerer glue};
\end{tikzpicture}
\end{equation}
The gray disks are the cores of the (3-dimensional) 2-handles making up
$H$ and $\ov{H}$;
they glue up in $H \cup_N \ov{H}$ to form a 2-sphere,
to which we attach a (4-dimensional) 3-handle.
Similarly, the (3-dimensional) 3-handles
making up the rest of $H$ and $\ov{H}$
glue up in $H \cup_N \ov{H}$ to form a 3-sphere,
to which we attach a (4-dimensional) 4-handle.
Thus, based on \defref{d:zcyksk-elementary},
we have that
(writing $\vphi$ for both the coloring of the coupon
and the skein that results from the coloring)
\begin{equation}
\evsk(\vphi,\vphi') = \ZCYsk(H \times I) (\vphi \cup \vphi')
=
\begin{tikzpicture}
\begin{scope}[shift={(0,-0.8)}]
\node[small_morphism,minimum size=0pt] (b) at (0,0) {\tiny $\vphi$};
\node[small_morphism,inner sep=0,minimum size=0pt]
	(a) at (0,1) {\tiny $\vphi^,$};
\node[small_morphism,minimum size=0pt] (a1) at (0.5,0.5) {\tiny $\al$};
\node[small_morphism,minimum size=0pt] (a2) at (1,0.5) {\tiny $\al$};
\draw[midarrow={0.7}] (a) to[out=0,in=120] (a1);
\draw[midarrow={0.7}] (b) to[out=0,in=-120] (a1);
\draw[midarrow={0.6}] (b) -- (a);
\node at (-0.2,0.5) {\tiny $V$};
\node at (0.5,0.85) {\tiny $j$};
\node at (0.5,0.15) {\tiny $i$};
\draw (a)
	.. controls +(180:0.4cm) and +(180:0.4cm) .. (0,1.4)
	.. controls +(0:0.7cm) and +(60:0.5cm) .. (a2);
\draw (b)
	.. controls +(180:0.8cm) and +(180:0.8cm) .. (0,1.7)
	.. controls +(0:1cm) and +(90:0.6cm) .. (1.5,0.5)
	.. controls +(-90:0.4cm) and +(-60:0.5cm) .. (a2);
\end{scope}
\end{tikzpicture}
= \delta_{i*,j} \cdot d_i^\inv \cdot
\begin{tikzpicture}
\begin{scope}[shift={(0,-0.5)}]
\node[small_morphism,minimum size=0pt] (b) at (0,0) {\tiny $\vphi$};
\node[small_morphism,inner sep=0,minimum size=0pt]
	(a) at (0,1) {\tiny $\vphi^,$};
\draw (a) -- (b);
\draw (a) to[out=-30,in=30] (b);
\draw (a) to[out=-150,in=150] (b);
\end{scope}
\end{tikzpicture}
= \delta_{i*,j} \cdot d_j^\inv \cdot \ev(\vphi,\vphi')
\end{equation}

The computation above easily generalizes to handlebodies
of arbitrary genus:
in a handlebody $H$ of genus $g$,
with a boundary value $\VV = (B,\{V_{\vec{b}}\})$,
choose a co-ribbon graph with a single coupon
like in \eqnref{e:skein-pairing-handlebodies-anchor}
(more precisely, choose PLCW structure on $H$ with exactly one 3-cell,
and appropriate number of 2-cells on the boundary,
and take the anchor as our co-ribbon graph).
Then for labeling
$\vphi \in \Hom_\cA(\one,
	(\bigotimes_b V_{\vec{b}}) \tnsr
	(X_{i_1} \tnsr X_{i_1}^*) \tnsr \cdots \tnsr
	(X_{i_g} \tnsr X_{i_g}^*))$
of the coupon in $H$,
where $i_1,\ldots,i_g$ are simple labels of ribbons
in the interior,
and labeling
$\vphi' \in \Hom_\cA(\one,
	(X_{j_g}^* \tnsr X_{j_g}) \tnsr \cdots
	(X_{j_1}^* \tnsr X_{j_1}) \tnsr
	(\bigotimes_b V_{\vec{b}}^*))$
of the coupon in $\ov{H}$,
we have
\begin{equation}
\evsk(\vphi,\vphi')
= \prod_{k=1}^g \delta_{i_k^*,j_k} \cdot d_{j_g}^\inv \cdot
	\ev(\vphi,\vphi')
\end{equation}

It is easy to see that this pairing agrees
(up to a factor of $\cD^{g/2}$) with \ocite{BK}*{(4.4.5)}.
\end{example}

\subsection{Equivalence with $\ZCY$}
\label{s:sk-equiv}
\par \noindent

We show an equivalence between the theory of skeins and
the theory established with PLCW decompositions.
The proof of \lemref{l:ups-projections-intertwine}
and \lemref{l:ups-isom}
mirror the methods used in \ocite{kirillov-stringnet},
except that all dimensions increase by 1,
and we deal with manifolds with boundaries directly.

\begin{definition}
Let $N$ be a closed surface.
For a colored marked PLCW decomposition $(\cN,l)$,
we assign a boundary value
$\Ups((\cN,l)) = (B,\{V_{\vec{b}}\})$
where $B$ is the marking of $(\cN,l)$
(forgetting the co-orientation),
and, for the marking $b$,
with orientation $\vec{b}$, of some 2-cell $f$
with the orientation $(\vec{b},\hat{n})$,
we set $V_{\vec{b}} = l(f)$.
(Note this is opposite from the convention discussed in
\rmkref{r:co-ribbon},
but the reason for violating the convention will be clear soon.)
\label{d:marked-to-boundary-value}
\end{definition}

\begin{lemma}
Let $\cM$ be an oriented PLCW 3-manifold with boundary $\cN$.
Arbitrarily choose an anchoring of $\cM$,
giving us a marking of $\cN$.
Let $l_\cN$ be a simple labeling of $\cN$,
and let $\VV = \Ups( (\cN,l_\cN))$.

Consider some vector
$\Phi = \bigotimes \vphi_C \in \bigotimes H(C,l) \subset
H(\cM;(\cN,l_\cN))$,
where $l$ is a labeling agreeing with $l_\cN$ on the boundary.
The anchors assemble into a ribbon graph $\Gamma$,
and $\Phi$ naturally defines a coloring of $\Gamma$
which we also denote $\Phi$;
we consider a differently-normalized coloring
$\cD^{\frac{1}{2} \vme(\cM \backslash \cN) + \frac{1}{4} \vme(\cN)}
	d_{l^\circ}^{1/2} d_{l_\cN}^{1/4} \Phi$,
where $d_{l^\circ}^{1/2} = \prod d_{l(f)}^{1/2}$,
the product taken over unoriented internal 2-cells of $\cM$,
$d_{l_\cN}^{1/4} = \prod_{f \in \cN} d_{l(f)}^{1/4}$,
and $\vme$ is from \lemref{l:state-space-moves},
This defines a linear map
\begin{align}
\ups : H(\cM; (\cN,l_\cN))
&\to \ZCYsk(M;\VV)
\\
\Phi &\mapsto (\Gamma,
	\cD^{\frac{1}{2} \vme(\cM \backslash \cN) + \frac{1}{4} \vme(\cN)}
	d_{l^\circ}^{1/2} d_{l_\cN}^{1/4} \Phi
\label{e:equiv-map-morphism}
\end{align}

This map factors through another map.
Let $M^\circ = M \backslash \sk^1(\cM)$,
i.e. remove the 1-skeleton of $\cM$ from $M$.
Then \eqnref{e:equiv-map-morphism}
factors through

\begin{align}
\ups : H(\cM; (\cN,l_\cN))
&\to \ZCYsk(M^\circ; \VV)
\\
\Phi &\mapsto (\Gamma,
	\cD^{\frac{1}{2} \vme(\cM \backslash \cN) + \frac{1}{4} \vme(\cN)}
	d_{l^\circ}^{1/2} d_{l_\cN}^{1/4} \Phi
\label{e:equiv-map-morphism-factors}
\end{align}

Now for an interior 1-cell $e$ of $\cM$,
define the map
$B_e : \ZCYsk(M^\circ; \VV) \to \ZCYsk(M^\circ; \VV)$,
where we add a ribbon loop, label with $\frac{1}{\cD}$
times the regular coloring.
By \lemref{lem:sliding},
$B_e$ is a projection.
It is also clear that $B_e, B_{e'}$ commute for any two interior
1-cells $e,e'$.
Thus, we may consider the projection
\begin{equation}
B_\cM = \prod B_e : \ZCYsk(M^\circ; \VV) \to \ZCYsk(M^\circ; \VV)
\label{e:edge-projection}
\end{equation}

Then $\ups$ intertwines $A_\cM$ and $B_\cM$,
i.e. $\ups \circ A_\cM = B_\cM \circ \ups$.

\label{l:ups-projections-intertwine}
\end{lemma}

\begin{proof}
Recall that $A_\cM = \ZCY(W;\cN)$,
where $W = M \times I$ with corner $N \times 0$
(see \defref{d:state-space-boundary}).
Consider the PLCW decomposition $\cW$
that has exactly one interior $(k+1)$-cell $C^W$
for each $k$-cell $C$ of $\cM$
(essentially, $\cW = \cM \times I$, but the boundary $N \times I$ is squashed;
see \eqnref{e:AM}).

\begin{equation}
\centering
\begin{tikzpicture}
\draw (0,0) -- (2,0) -- (1,2) -- (-1,2) -- (-2,1) -- (0,0);
\draw (0,0) -- +(-135:0.5cm);
\draw (2,0) -- +(-63:0.5cm);
\draw (-1,2) -- +(135:0.5cm);
\draw (-2,1) -- +(180:0.5cm);
\begin{scope}[shift={(0,1)}]
\draw (0,0) -- (2,0) -- (1,2) -- (-1,2) -- (-2,1) -- (0,0);
\draw (0,0) -- +(-135:0.5cm);
\draw (2,0) -- +(-63:0.5cm);
\draw (-1,2) -- +(135:0.5cm);
\draw (-2,1) -- +(180:0.5cm);
\end{scope}
\draw (0,0) -- (0,1);
\draw (2,0) -- (2,1);
\draw (1,2) -- (1,3);
\draw (-1,2) -- (-1,3);
\draw (-2,1) -- (-2,2);
\draw (2,0) to[out=0,in=-135] (3.5,0.5);
\draw (2,1) to[out=0,in=135] (3.5,0.5);
\draw (1,2) to[out=0,in=-135] (2.5,2.5);
\draw (1,3) to[out=0,in=135] (2.5,2.5);
\draw[line width=2pt] (2,3.5) -- (4,-0.5);
\node at (3.4,1.5) {$\cN$};
\node at (0.7,0.5) {$F^W$};
\node at (0.6,-0.2) {$F$};
\end{tikzpicture}
\;\;\;\;
\begin{tikzpicture}
\draw (0,0) -- (2,0) -- (1,2) -- (-1,2) -- (-2,1) -- (0,0);
\draw (0,0) -- +(-135:0.5cm);
\draw (2,0) -- +(-63:0.5cm);
\draw (-1,2) -- +(135:0.5cm);
\draw (-2,1) -- +(180:0.5cm);
\draw (1,2) to[out=0,in=-135] (2.5,2.5);
\draw (2,0) to[out=0,in=-135] (3.5,0.5);
\draw (0.2,1.2) -- (1.7,1.2) -- (0.8,2.8) -- (-0.8,2.8) -- (-1.8,2) -- (0.2,1.2);
\draw (0,0) -- (0.2,1.2);
\draw (2,0) -- (1.7,1.2);
\draw (1,2) -- (0.8,2.8);
\draw (-1,2) -- (-0.8,2.8);
\draw (-2,1) -- (-1.8,2);
\draw (2,0) -- (2,0.8) -- (0.3,0.8) -- (0,0);
\draw (2,0.8) -- +(-63:0.5cm);
\draw (0.3,0.8) -- +(-135:0.5cm);
\draw (0,0) -- (-0.2,1) -- (-2.2,1.8) -- (-2,1);
\draw (-0.2,1) -- +(-135:0.5cm);
\draw (-2.2,1.8) -- +(180:0.5cm);
\draw (-2,1) -- (-2.2,2.2) -- (-1.2,3) -- (-1,2);
\draw (-2.2,2.2) -- +(180:0.5cm);
\draw (-1.2,3) -- +(135:0.5cm);
\draw (-1,2) -- (-0.9,3.1) -- (1.1,3.1) to[out=0,in=135] (2.5,2.5);
\draw (1,2) -- (0.9,3.1);
\draw (-0.9,3.1) -- +(135:0.5cm);
\draw (1,2) -- (1.2,2.8) to[out=0,in=135] (2.5,2.5);
\draw (1.2,2.8) -- (2.1,1.2);
\draw (2,0) -- (2.1,1.2) to[out=0,in=135] (3.5,0.5);
\draw (2,0) -- (2.3,0.8) to[out=0,in=135] (3.5,0.5);
\draw (2.3,0.8) -- +(-63:0.5cm);
\draw[line width=2pt] (2,3.5) -- (4,-0.5);
\end{tikzpicture}
\label{e:AM}
\end{equation}

Recall
\[
Z(\cW,l) = \ev(\bigotimes_C Z(C^W,l))
\]
where $\ev$ applies the pairing \eqnref{e:pairing}
to each interior 3-cell of $\cW$;
by construction, each interior 3-cell is of the form
$F^W = F \times I$ for some interior 2-cell $F$ of $\cM$.
Choose dual bases $\{\vphi_{F,\al}\},\{\vphi_{\ov{F},\al}\}$
for $H(F^W,l), H(\ov{F^W},l)$ for each such $F$.
Then, for $\Phi = \tnsr \vphi_C$ as in the proposition statement,
\[
\ev(Z(\cW,l), \Phi)
= \ev(\bigotimes_C Z(C^W,l),
	\bigotimes_C \vphi_C \tnsr \bigotimes_F \vphi_{F,\al} \tnsr
	\vphi_{\ov{F},\al})
= \sum_\al \bigotimes_C \ev(Z(C^W,l),
		\vphi_C \tnsr \bigotimes_F \vphi_{\ov{F},\al})
\]
where, in the middle expression,
the sum over $\al$ is implicit,
and in the last expression,
the $F$ runs over interior 2-cells in $\del C$.
(If $F$ has outward orientation with respect to $C$,
then $F^W = F \times I$ has the outward orientation with respect to $C^W$,
so we should evaluate $Z(C^W,l)$ against
$\vphi_{\ov{F},\al} \in H(\ov{F^W},l)$,
not $\vphi_{F,\al} \in H(F^W,l)$.)

Now each term
$\ev(Z(C^W,l), \vphi_C \tnsr \bigotimes_F \vphi_{\ov{F},\al})$
in the last expression is a local state in
$H(C,l)$,
and is $\vphi_{ev}$ as discussed in
\lemref{l:state-space-moves}.

Thus, if $\ups(\Phi)$ is of the form

\begin{equation}
\ups(\Phi) = 
	\cD^{y}
	d_{l^\circ}^{1/2} d_{l_\cN}^{1/4}
\begin{tikzpicture}
\begin{scope}[shift={(0,-1.5)}]
\draw[opacity=0.3] (0,0) -- (2,0) -- (2,2) -- (0,2.5) -- (-1,1) -- (0,0);
\draw[opacity=0.3] (2,0) -- (4,0);
\draw[opacity=0.3] (2,2) -- (4,2);
\draw[opacity=0.3] (0,0) -- +(-135:0.5cm);
\draw[opacity=0.3] (2,0) -- +(-90:0.5cm);
\draw[opacity=0.3] (0,2.5) -- +(90:1.2cm);
\draw[opacity=0.3] (-1,1) -- +(180:0.5cm);
\draw[line width=2pt] (4,-0.5) -- (4,3.5);
\node[small_morphism] (ph1) at (0.7,1.1) {};
\node[small_morphism] (ph2) at (3,1) {};
\node[small_morphism] (ph3) at (2.1,3) {};
\draw (ph1) -- (ph2);
\draw (ph1) to[out=80,in=-150] (ph3);
\draw (ph2) to[out=90,in=-45] (ph3);
\draw (ph1) -- +(150:2cm);
\draw (ph1) -- +(-150:2cm);
\draw (ph1) to[out=-70,in=90] +(-80:1.6cm);
\draw (ph1) -- +(150:2cm);
\draw (ph2) -- +(-90:1.6cm);
\draw (ph2) -- +(0:1cm);
\draw (ph3) -- (4,3);
\draw (ph3) -- +(170:2.7cm);
\draw (ph3) -- +(90:0.5cm);
\end{scope}
\end{tikzpicture}
\label{e:ups-01}
\end{equation}
where the circles should be labeled by $\vphi_C$'s,
and $y = \frac{1}{2} \vme(\cM \backslash \cN)
	+ \frac{1}{4} \vme(\cM)$,
then we have,
writing
$l_i = l^\circ$,
$l_o = l_{out}^\circ$,
$l_m = l_{mid}$
(for the labeling on the 2-cells in the incoming boundary,
outgoing boundary, and interior of $\cW$, respectively),
and
$y' = \vme(\cW \backslash \del \cW)
	+ \frac{1}{2} \vme(\del \cW \backslash \cN)
	= -v(\cM \backslash \cN) + \vme(\cM \backslash \cM)
	= -e(\cM \backslash \cN)$,

\begin{align}
\begin{split}
\label{e:ups-02}
\ups(A(\Phi))
&=
\sum_{l_m,l_o}
\underbrace{\cD^y d_{l_o}^{1/2} d_{l_\cN}^{1/4}}_{\text{from } \ups}
\;
\underbrace{\cD^{y'} d_{l_i}^{1/2} d_{l_m} d_{l_o}^{1/2}}_{\text{from } \ZCY(\cW;\cN)}
\;\;\;
\begin{tikzpicture}
\begin{scope}[scale=1.6,shift={(0,-1.5)}]
\draw[opacity=0.3] (0,0) -- (2,0) -- (2,2) -- (0,2.5) -- (-1,1) -- (0,0);
\draw[opacity=0.3] (2,0) -- (4,0);
\draw[opacity=0.3] (2,2) -- (4,2);
\draw[opacity=0.3] (0,0) -- +(-135:0.5cm);
\draw[opacity=0.3] (2,0) -- +(-90:0.5cm);
\draw[opacity=0.3] (0,2.5) -- +(90:1.2cm);
\draw[opacity=0.3] (-1,1) -- +(180:0.5cm);
\draw[line width=2pt] (4,-0.5) -- (4,3.7);
\node at (4.2,2.5) {$l_\cN$};
\node[small_morphism] (ph1) at (0.7,1.1) {};
\node[small_morphism] (ph2) at (3,1) {};
\node[small_morphism] (ph3) at (2.1,3) {};
\node[dotnode] (a10) at (1.8,1) {};
\node[dotnode] (a11) at (1,2) {};
\node[dotnode] (a12) at (-0.3,1.6) {};
\node[dotnode] (a13) at (-0.3,0.6) {};
\node[dotnode] (a14) at (1,0.2) {};
\node[dotnode] (a20) at (3,1.8) {};
\node[dotnode] (a21) at (2.2,1) {};
\node[dotnode] (a22) at (3,0.2) {};
\node[dotnode] (a30) at (0.2,3.4) {};
\node[dotnode] (a31) at (1.2,2.4) {};
\node[dotnode] (a32) at (3,2.2) {};
\node[dotnode] (a40) at (-0.2,3.4) {};
\node[dotnode] (a41) at (-0.7,1.9) {};
\node[dotnode] (a42) at (-0.6,0.3) {};
\node[dotnode] (a43) at (1,-0.2) {};
\node[dotnode] (a44) at (3,-0.2) {};
\draw (ph1) -- (a10);
\node at (1.3,1.2) {$l_i$};
\draw (ph1) -- (a11);
\node at (0.7,1.6) {$l_i$};
\draw (ph1) -- (a12);
\draw (ph1) -- (a13);
\draw (ph1) -- (a14);
\draw (ph2) -- (a20);
\node at (2.9,1.4) {$l_i$};
\draw (ph2) -- (a21);
\node at (2.6,1.15) {$l_i$};
\draw (ph2) -- (a22);
\draw (ph2) -- (4,1);
\draw (ph3) -- (a30);
\draw (ph3) -- (a31);
\draw (ph3) -- (a32);
\draw (ph3) -- (4,3);
\draw (ph3) -- +(90:0.7cm);
\draw (a10) -- (a21);
\node at (2,1.15) {$l_o$};
\draw (a11) -- (a31);
\node at (1.25,2.1) {$l_o$};
\draw (a12) -- (a41);
\draw (a13) -- (a42);
\draw (a14) -- (a43);
\draw (a20) -- (a32);
\node at (2.85,2) {$l_o$};
\draw (a22) -- (a44);
\draw (a30) -- (a40);
\draw (a40) -- +(180:0.5cm);
\draw (a41) -- +(150:0.5cm);
\draw (a42) -- +(-150:0.5cm);
\draw (a43) -- +(-90:0.3cm);
\draw (a44) -- +(-90:0.3cm);
\draw (a10) to[out=90,in=-20] (a11);
\node at (1.75,1.7) {$l_m$};
\draw (a11) to[out=160,in=60] (a12);
\draw (a12) to[out=-120,in=135] (a13);
\draw (a13) to[out=-45,in=180] (a14);
\draw (a14) to[out=0,in=-90] (a10);
\draw (a20) to[out=180,in=90] (a21);
\node at (2.3,1.7) {$l_m$};
\draw (a21) to[out=-90,in=180] (a22);
\draw (a30) -- +(90:0.3cm);
\draw (a30) to[out=-90,in=160] (a31);
\draw (a31) to[out=-20,in=180] (a32);
\node at (2.1,2.35) {$l_m$};
\draw (a40) -- +(90:0.3cm);
\draw (a40) to[out=-90,in=60] (a41);
\draw (a41) to[out=-120,in=0] (-1.5,1.2);
\draw (-1.5,0.8) to[out=0,in=135] (a42);
\draw (a42) to[out=-45,in=70] (-0.5,-0.2);
\draw (-0.2,-0.5) to[out=45,in=180] (a43);
\draw (a43) to[out=0,in=120] (1.7,-0.5);
\draw (2.3,-0.5) to[out=60,in=180] (a44);
\end{scope}
\end{tikzpicture}
\\
&=
\sum_{l_m}
\cD^{-e(\cM \backslash \cN)} d_{l_m} \cdot \cD^y d_{l_i}^{1/2} d_{l_\cN}^{1/4}
\;\;\;
\begin{tikzpicture}
\begin{scope}[scale=1.6,shift={(0,-1.5)}]
\draw[opacity=0.3] (0,0) -- (2,0) -- (2,2) -- (0,2.5) -- (-1,1) -- (0,0);
\draw[opacity=0.3] (2,0) -- (4,0);
\draw[opacity=0.3] (2,2) -- (4,2);
\draw[opacity=0.3] (0,0) -- +(-135:0.5cm);
\draw[opacity=0.3] (2,0) -- +(-90:0.5cm);
\draw[opacity=0.3] (0,2.5) -- +(90:1.2cm);
\draw[opacity=0.3] (-1,1) -- +(180:0.5cm);
\draw[line width=2pt] (4,-0.5) -- (4,3.7);
\node at (4.2,2.5) {$l_\cN$};
\draw (0,0) circle (0.3cm);
\draw (2,0) circle (0.3cm);
\draw (2,2) circle (0.3cm);
\node at (1.7,1.7) {$l_m$};
\draw (0,2.5) circle (0.3cm);
\draw (-1,1) circle (0.3cm);
\node[small_morphism] (ph1) at (0.7,1.1) {};
\node[small_morphism] (ph2) at (3,1) {};
\node[small_morphism] (ph3) at (2.1,3) {};
\draw (ph1) -- (ph2);
\draw (ph1) to[out=80,in=-150] (ph3);
\draw (ph2) to[out=90,in=-45] (ph3);
\draw (ph1) -- +(150:2cm);
\draw (ph1) -- +(-150:2cm);
\draw (ph1) to[out=-70,in=90] +(-80:1.6cm);
\draw (ph1) -- +(150:2cm);
\draw (ph2) -- +(-90:1.6cm);
\draw (ph2) -- +(0:1cm);
\draw (ph3) -- (4,3);
\draw (ph3) -- +(170:2.7cm);
\draw (ph3) -- +(90:0.5cm);
\end{scope}
\end{tikzpicture}
\\
&=
B(\ups(\Phi))
\end{split}
\end{align}

\end{proof}

\begin{lemma}
The maps $\ups$ is ``invariant under $A$'', that is,
if $\ups_\cM = \ups : H(\cM; (\cN,l)) \to \ZCYsk(M;\VV)$,
then $\ups_\cM = \ups_{\cM'} \circ A_{\cM,\cM'}$.
In other words, $\ups$ gives a well-defined map
$\ups : \ZCY(\cM;(\cN,l)) \to \ZCYsk(M;\VV)$.
\label{l:ups-A-invariant}
\end{lemma}

\begin{proof}
It is clear from \lemref{l:state-space-moves}
that if $\cM,\cM'$ are related by a single-cell move,
then the lemma holds;
then by \prpref{p:single-cell-subdivision},
a sequence of single-cell moves relate any two PLCW decompositions,
so we are done.
\end{proof}

\begin{lemma}
The map $\ups : H(\cM; (\cN,l_\cN)) \to \ZCYsk(M^\circ; \VV)$
from \lemref{l:ups-projections-intertwine}
is an isomorphism.
\label{l:ups-isom}
\end{lemma}

\begin{proof}
We define an inverse map $\ups^\inv$.
Let $\Gamma$ be the graph built from arbitrarily chosen anchors,
as in \eqnref{e:ups-01}.

Let $\Gamma'$ be some arbitrary colored ribbon graph in
$\ZCYsk(M^\circ;\VV)$.
We define $\ups^\inv(\Gamma')$ as follows.
Apply an isotopy to make $\Gamma'$ weakly transverse to every 2-cell,
then by \lemref{l:narrow-transverse},
there is a strict narrowing that makes $\Gamma'$ transverse to every 2-cell
(we discuss how this choice of isotopy is irrelevant).
Using \lemref{l:summation},
we can make the graph meet each 2-cell in a single ribbon,
or more precisely,
$\Gamma'$ is equivalent as a skein to $\sum c_j \Gamma_j''$,
where each $\Gamma_j'$ agrees exactly with $\Gamma$
in a neighborhood of each 2-cell
(see \eqnref{e:ups-isom-01});
this is well-defined by \lemref{l:summation-alpha-tunnel}.
Then, for each 3-cell $C$, take a ball $D$ that is a just a
slightly shrunk $C$, and evaluate the graphs $\Gamma_j''$,
and replace it with a subgraph that looks like $\Gamma$
but with different colorings.

\begin{align}
\label{e:ups-isom-01}
\begin{split}
\begin{tikzpicture}
\begin{scope}[shift={(0,-0.8)}]
\tikzmath{
	\lf = -0.25;
	\rt = 0.25;
	\md = 0;
	\lfmost = -1;
	\rtmost = 1;
	\ya = 0.3;
	\yb = 0.6;
	\yc = 0.9;
	\yd = 1.2;
	\ymid = \yb/2 + \yc/2;
}
\draw[fill=gray, fill opacity=0.2]
	(\lf,0) -- (\rt,-0.3) -- (\rt,1.5) -- (\md,2) -- (\lf,1.75) -- (\lf,0);
\node[small_morphism] (f) at (-0.6,\ya) {\tiny $f$};
\draw (\lfmost,\ya) -- (f) -- (\md,\ya);
\draw[opacity=0.3] (\md,\ya) -- (\rt,\ya);
\draw (\rt,\ya) -- (\rtmost,\ya);
\draw (\lfmost,\yb)
	.. controls +(0:0.5cm) and +(-90:0.1cm) .. (-0.4,\ymid)
	.. controls +(90:0.1cm) and +(0:0.5cm) .. (\lfmost,\yc);
\draw (\lfmost,\yd) -- (\md,\yd);
\draw[opacity=0.3] (\md,\yd) -- (\rt,\yd);
\draw (\rt,\yd) -- (\rtmost,\yd);
\end{scope}
\end{tikzpicture}
\;\;\;
\;\;\;\;\;\;
&=
\;\;\;\;\;\;
\;\;\;
\begin{tikzpicture}
\begin{scope}[shift={(0,-0.8)}]
\tikzmath{
	\lf = -0.25;
	\rt = 0.25;
	\md = 0;
	\lfmost = -1;
	\rtmost = 1;
	\ya = 0.3;
	\yb = 0.6;
	\yc = 0.9;
	\yd = 1.2;
	\ymid = \yb/2 + \yc/2;
}
\draw[fill=gray, fill opacity=0.2]
	(\lf,0) -- (\rt,-0.3) -- (\rt,1.5) -- (\md,2) -- (\lf,1.75) -- (\lf,0);
\node[small_morphism] (f) at (0.6,\ya) {\tiny $f$};
\draw (\lfmost,\ya) -- (\md,\ya);
\draw[opacity=0.3] (\md,\ya) -- (\rt,\ya);
\draw (\rt,\ya) -- (f) -- (\rtmost,\ya);
\draw (\lfmost,\yb) -- (\md,\yb);
\draw[opacity=0.3] (\md,\yb) -- (\rt,\yb);
\draw (\lfmost,\yc) -- (\md,\yc);
\draw[opacity=0.3] (\md,\yc) -- (\rt,\yc);
\draw (\rt,\yb)
	.. controls +(0:0.5cm) and +(-90:0.1cm) .. (0.85,\ymid)
	.. controls +(90:0.1cm) and +(0:0.5cm) .. (\rt,\yc);
\draw (\lfmost,\yd) -- (\md,\yd);
\draw[opacity=0.3] (\md,\yd) -- (\rt,\yd);
\draw (\rt,\yd) -- (\rtmost,\yd);
\end{scope}
\end{tikzpicture}
\\[10pt]
\begin{tikzpicture}
\begin{scope}[shift={(0,-0.8)}]
\tikzmath{
	\lf = -0.25;
	\rt = 0.25;
	\md = 0;
	\lfmost = -1.5;
	\rtmost = 1.5;
	\ya = 0.3;
	\yb = 0.6;
	\yc = 0.9;
	\yd = 1.2;
	\ymid = \yb/2 + \yc/2;
}
\draw[fill=gray, fill opacity=0.2]
	(\lf,0) -- (\rt,-0.3) -- (\rt,1.5) -- (\md,2) -- (\lf,1.75) -- (\lf,0);
\node[small_morphism] (a1) at (-0.5,\ymid) {\tiny $\al$};
\node[small_morphism] (a2) at (0.5,\ymid) {\tiny $\al$};
\draw (a1) -- (\md,\ymid);
\draw[opacity=0.3] (\md,\ymid) -- (\rt,\ymid);
\draw (\rt,\ymid) -- (a2);
\node[small_morphism] (f) at (-1.1,\ya) {\tiny $f$};
\draw (\lfmost,\ya) -- (f) to[out=0,in=-120] (a1)
	to[out=120,in=0] (-1.1,\yd) -- (\lfmost,\yd);
\draw (\rtmost,\ya) -- (1.1,\ya)
	to[out=180,in=-60] (a2) to[out=60,in=180] (1.1,\yd) -- (\rtmost,\yd);
\draw (\lfmost,\yb)
	.. controls +(0:0.5cm) and +(-90:0.1cm) .. (-0.9,\ymid)
	.. controls +(90:0.1cm) and +(0:0.5cm) .. (\lfmost,\yc);
\end{scope}
\end{tikzpicture}
\;\;\;
&=
\;\;\;
\begin{tikzpicture}
\begin{scope}[shift={(0,-0.8)}]
\tikzmath{
	\lf = -0.25;
	\rt = 0.25;
	\md = 0;
	\lfmost = -1.5;
	\rtmost = 1.5;
	\ya = 0.3;
	\yb = 0.6;
	\yc = 0.9;
	\yd = 1.2;
	\ymid = \yb/2 + \yc/2;
}
\draw[fill=gray, fill opacity=0.2]
	(\lf,0) -- (\rt,-0.3) -- (\rt,1.5) -- (\md,2) -- (\lf,1.75) -- (\lf,0);
\node[small_morphism] (a1) at (-0.5,\ymid) {\tiny $\al$};
\node[small_morphism] (a2) at (0.5,\ymid) {\tiny $\al$};
\draw (a1) -- (\md,\ymid);
\draw[opacity=0.3] (\md,\ymid) -- (\rt,\ymid);
\draw (\rt,\ymid) -- (a2);
\draw (\lfmost,\ya) -- (-1.1,\ya) to[out=0,in=-120] (a1)
	to[out=120,in=0] (-1.1,\yd) -- (\lfmost,\yd);
\node[small_morphism] (f) at (1.1,\ya) {\tiny $f$};
\draw (\rtmost,\ya) -- (f)
	to[out=180,in=-60] (a2) to[out=60,in=180] (1.1,\yd) -- (\rtmost,\yd);
\draw (\lfmost,\yb) to[out=0,in=-160] (a1);
\draw (\lfmost,\yc) to[out=0,in=160] (a1);
\draw (a2)
	.. controls +(-20:0.5cm) and +(-90:0.1cm) .. (1.2,\ymid)
	.. controls +(90:0.1cm) and +(20:0.5cm) .. (a2);
\end{scope}
\end{tikzpicture}
\end{split}
\end{align}

Now let us discuss why the isotopies of $\Gamma'$
to make it transverse to 2-cells is irrelevant.
It is easy to see that the choice of strict narrowing is irrelevant,
thus we will assume that all ribbon graphs are already ``narrow enough''
(see also \rmkref{r:narrowing-infinitesimal}).
Suppose we have isotoped $\Gamma'$ to
$\Gamma_1'$ and $\Gamma_2'$, each of which is weakly transverse
to every 2-cell.
Let $\Phi^t$ be an ambient isotopy that throws
$\Gamma_1'$ onto $\Gamma_2'$.

For a 2-cell $F$ of $\cM$,
let $U_F$ be a small open neighborhood of $\Int(F) \subset M^\circ$
(say, take a point on either side of $F$ close to it,
and take $U_F$ to be the cone over $F$, then remove its boundary).
Then $\{\Int(C),U_F\}_{C,F}$ is an open cover of $M^\circ$,
where $C,F$ run over all 3-,2-cells of $\cM$, respectively.
By \thmref{t:isotopy-open-cover-boundary},
the isotopy $\Phi^t$
can be replace with a sequence of moves, each of which is supported
in a ball in some $\Int(C)$ or $U_F$.

Let $\Gamma_0 = \Gamma_1',\Gamma_1, \ldots, \Gamma_l = \Gamma_2'$
be the graphs between the moves,
and let $h_k^t$ be the move that takes $\Gamma_k$ to $\Gamma_{k+1}$.

We show that we can assume that all $\Gamma_k$'s are weakly transverse to
all 2-cells, so that it suffices to show
that $\ups^\inv(\Gamma_k) = \ups^\inv(\Gamma_{k+1})$.
Consider a move supported in a ball $D \subset U_F$ for some 2-cell $F$;
say this move $h_k^t$ takes $\Gamma_k$ to $\Gamma_{k+1}$.
Suppose $\Gamma_k$ is weakly transverse to $F$,
If $\Gamma_{k+1}$ is not weakly transverse to $F$,
we apply a small isotopy $\phi^t$  (in particularly, supported in $D$
and away from all other supporting balls of moves supported in $\Int(C)$),
so that $\phi^1(\Gamma_{k+1})$ is weakly transverse to $F$.
Somewhere ``down the line'', there is some minimal $k' > k$ such that
the $k'$-th move $h_{k'}^t$ is supported in $U_F$
(because $\Gamma_{k+1}$ is not weakly transversal
while the target $\Gamma_l$ is,
and moves supported in $\Int(C)$ cannot help).
We modify all the graphs and moves up to $k'$
so that we have
\[
\ldots,\Gamma_k, \phi^1(\Gamma_{k+1}), \phi^1(\Gamma_{k+2}),\ldots,
\phi^1(\Gamma_{k'}), \Gamma_{k'+1},\ldots
\]
and for isotopies,
the $k$-th isotopy is post-concatenated with $\phi^t$,
while the $k'$-th isotopy is pre-concatenated with $\phi^{1-t}$.
Thus, we may assume that all $\Gamma_k$'s are weakly transverse to
all 2-cells.

Clearly the moves supported in $\Int(C)$
do not affect $\ups^\inv(\Gamma_k)$.
Now consider some $h_k^t$ supported in a ball $D \subset U_F$
for some 2-cell $F$ as above.
Let $h'^t$ be an isotopy supported in $U_F$
that moves $D$ into the interior of an adjacent cell $C$ of $F$.
We choose $h'^t$ to be composed of simple ``displacements'',
say, if $F$ is given a collar neighborhood $F \times [-1,1]$,
then such ``displacements'' should be essentially of the form
$(x,s) \mapsto (x,s+\al)$ for some fixed $\al$
(except perhaps near $\del F$);
furthermore, we ensure that each such ``displacement'' ends
with the graph being weakly transverse to $F$.
These ``displacements'' look like \eqnref{e:ups-isom-01},
and thus $h'$ do not affect $\ups^\inv(\Gamma_k)$.
Finally, the move that was supported in $D$ is now
the move $h'^1 \circ h_k^t$,
supported in $\Int(C)$.
Thus, we can replace $h_k^t$ by moves
$h'^t, h'^1 \circ h_k^t, h'^{1-t}$,
each of which does not affect the outcome of $\ups^\inv$,
so we are done.

(The arguments above on the irrelevance of isotopies
are very similar to the proof of \lemref{lem:alpha_independence}.)
\end{proof}

\begin{theorem}
For an oriented closed surface $N$,
the assignments $\Ups, \ups$ from 
\defref{d:marked-to-boundary-value},
\lemref{l:ups-projections-intertwine}
define an equivalence
\[
\Upsilon : \ZCY(N) \simeq \ZCYsk(N)
\]
\label{t:sk-equiv}
\end{theorem}

\begin{proof}
It suffices to prove equivalence for the hat versions:
\[
\Ups : \hatZCY(N) \simeq \hatZCYsk(N)
\]
$\Ups$ is clearly essentially surjective.
By the previous lemmas applied to $M = N \times I$,
we see that $\ups$ defines isomorphisms
$\ups : \Hom_{\ZCY(N)}( (\cN,l),(\cN',l') )
\simeq \Hom_{\ZCYsk(N)}( \Ups( (\cN,l) ), \Ups( (\cN',l') ) )$.
It is straightforward to check that $\Ups$ respects composition
(this is the reason for the $\cD^{\frac{1}{4}(\vme(\cN))} d_{l_\cN}^{1/4}$
factors in $\ups$),
and that the identity morphism $\id_{(\cN,l)}$
is sent to the identity morphism.
\end{proof}

\begin{theorem}
\label{t:sk-equiv-4mfld}
We have
\[
\ups : \cF_M \simeq \cF_M^{sk} \circ \Ups
\]
where $\ups, \Ups$ are from \thmref{t:sk-equiv} above,
and $\cF_M,\Fsk{M}$ are from
Propositions \ref{p:zcy-surface-functors},
\ref{p:zcy-surface-functors-skein}.

Furthermore, $\ups$ intertwines $\Fsk{W} \circ \Ups$ and $\cF_W$:
for $\VV = \Ups((\cN,l))$,
we have the commutative diagram
\[
\begin{tikzcd}[column sep=2cm]
\ZCY(M;(\cN,l)) \ar[r,"\ZCY(W;N)"] \ar[d,"\ups"']
	& \ZCY(M';(\cN,l)) \ar[d,"\ups"]
\\
\ZCYsk(M;\VV) \ar[r,"\ZCYsk(W;N)"]
	& \ZCYsk(M';\VV)
\end{tikzcd}
\]
\end{theorem}

\begin{proof}
The first part follows directly from \thmref{t:sk-equiv}
and the preceding lemmas.

To prove the second part, it suffices to consider
elementary cornered cobordisms $W$.
More specifically, we just have to consider
handles $\cH_k$ with (mostly) ``empty'' input.
For a cornered cobordism $W: M \to_N M'$,
let $\Ftld{W} : \Fsk{M} \circ \Ups \to \Fsk{M'} \circ \Ups$
be the natural transformation
that is obtained from ``transferring'' $\Fpl{W}$ via $\Ups$;
more precisely, for $(\cN,l) \in \ZCY(N)$,
we define
\[
(\Ftld{W})_{(\cN,l)} = \ups_{(\cN,l)} \circ (\Fpl{W})_{(\cN,l)}
	\circ \ups_{(\cN,l)}^\inv
\]
(In other words, the second part claims that $\Ftld{W} = \Fsk{W} \circ \Ups$.)

We compute $\Ftld{W}(\cH_k)$ case-by-case.
We choose different PLCW decompositions on $B^4 = B^k \times B^{4-k}$
for different cases,
but in each case, the interior of $\cH_k$ will consist of just
one 4-cell, i.e. we will choose a cell-like PLCW decomposition,
and we denote by $\cM_{in}, \cM_{out}$ the PLCW decomposition
on $B^k \times \del B^{4-k}, \del B^k \times B^{4-k}$.

\textbf{Case $k=4$:}
$\Ftld{\cH_4} (\emptyskein{S^3}) = 1$.
Take any PLCW structure $\cM_{in}$ on $S^3$.
The empty skein corresponds to
$\ups^\inv(\emptyskein{S^3}) = \cD^{-\frac{1}{2}x(\cM_{in})}
	\cdot \emptystate{\cM_{in}}$.
Then
\begin{align*}
\Ftld{\cH_4}(\emptyskein{S^3})
&= \ZCY(\cH_4)(\cD^{-\frac{1}{2}x(\cM_{in})} \cdot \emptystate{\cM_{in}})
\\
&= \ev(\ZCY(\cH_4), \cD^{-\frac{1}{2}x(\cM_{in})} \cdot \emptystate{\cM_{in}})
\\
&= \ev(\cD^{\frac{1}{2}x(\cM_{in})} Z(B^4,l\equiv \one),
	\cD^{-\frac{1}{2}x(\cM_{in})} \cdot \emptystate{\cM_{in}})
\\
&= \ZRT(\emptystate{\cM_{in}}) = 1
\end{align*}
where in the last line, $\emptystate{\cM_{in}}$ is interpreted as a
colored graph as in \defref{d:ZCY3cell}.

\textbf{Case $k=0$:} $\Ftld{\cH_0} = \cD \cdot \emptyskein{S^3}$.
Consider the PLCW structure $\cM_{out}$ on $S^3 = \del \cH_0$
with exactly one $j$-cell for each dimension $j=0,1,2,3$
as follows:

\begin{equation}
\begin{tikzpicture}
\draw (0,-0.5) circle (2.5cm);
\node[dotnode] (d) at (-1,-1) {};
\node at (2.3,-2.3) {$S^3$};
\node[small_morphism] (ph) at (0.5,-1.5) {};
\draw[ribbon] (ph)
	.. controls +(130:0.3cm) and +(-90:0.3cm) .. (0,-0.5);
\node at (1.1,-1) {\smallerer $\psi$};
\draw (d)
	to[out=90,in=180] (0,0)
	to[out=0,in=60] (0.5,-0.5)
	to[out=-120,in=0] (d);
\draw[fill=med-gray,opacity=0.5] (d)
	to[out=90,in=180] (0,0)
	to[out=0,in=60] (0.5,-0.5)
	to[out=-120,in=0] (d);
\draw[ribbon] (ph)
	.. controls +(60:1cm) and +(0:0.7cm) .. (0.5,0.2)
	.. controls +(180:0.5cm) and +(90:0.3cm) .. (0,-0.5);
\draw[line width=0.4pt] (-1,0) -- (-0.5,-0.5)
	node[pos=0,left] {\smallerer 2-cell};
\node at (-1,-2) {\smallerer 3-cell $C$};
\end{tikzpicture}
\end{equation}

For the labeling $l(f) = i$,
the state $\vphi = \coev_i \in H(C,l)$
should have coefficient $\frac{1}{d_i} \cdot d_i = 1$
in $\ZCY(\cH_0,l)$,
because the dual to $\coev_i$ is
$\frac{1}{d_i} \cdot \coev_i \in H(\overline{C},l)$,
and that graph evaluates to $d_i$.
Thus
\[
\ZCY(\cH_0) = \sum_i d_i^{1/2} \coev_i
\overset{\ups}{\mapsto}
\sum_i d_i^{1/2} (\Gamma, d_i^{1/2} \coev_i)
= \sum_i d_i^2 \cdot \emptyskein{S^3} = \cD \cdot \emptyskein{S^3}
\]
where $\Gamma$ is the ribbon graph obtained from the anchor.
(Note $(\Gamma,\coev_i) = d_i$.)

\textbf{Case $k=1$:}
$\Ftld{\cH_1}(\emptyskein{B^3 \sqcup B^3})
= \cD^{-1} \cdot \emptyskein{S^2 \times B^1}$

(Strictly speaking, $\emptyskein{B^3 \sqcup B^3}$
has empty boundary value $\EE$,
which is not in the image of $\Ups$;
however, an object $(B,\{\one\}) = \Ups((\cN,l\cong \one))$
is canonically isomorphic to $\EE$
by the simplest possible graph
(one ribbon and one coupon for each marking,
labeled with $\one$ and $\coev_{\one}$ respectively.)

A 1-handle $\cH_1$ is a cornered cobordism
$\cH_1 : \cD^3 \sqcup \cD^3 \to_{S^2 \sqcup S^2} S^2 \times B^1$.
We use the following PLCW decomposition:

\begin{equation}
\label{e:H1-plcw}
\begin{tikzpicture}
\node at (2.4,-2.2) {$S^3$};
\draw (0,0) circle (2.8cm);
\node[small_morphism] (a) at (-1.4,0.1) {};
\node[small_morphism] (b) at (1.4,0.1) {};
\node[small_morphism] (c) at (0,-1.2) {};
\draw[ribbon] (a)
	.. controls +(-90:0.3cm) and +(120:0.3cm) .. (-1.1,-0.7);
\draw[ribbon] (b)
	.. controls +(-90:0.3cm) and +(60:0.3cm) .. (1.1,-0.7);
\draw[ribbon] (c) .. controls +(60:0.8cm) and +(-45:0.2cm) .. (0.1,0);
\filldraw[med-gray,path fading=fade inside]
	(-1.4,0.1) circle (1cm);
\draw (-1.4,0.1) circle (1cm);
\filldraw[med-gray,path fading=fade inside]
	(1.4,0.1) circle (1cm);
\node[small_dotnode] at (-0.6,0.2) {};
\node[small_dotnode] at (-1,-0.2) {};
\node[small_dotnode] at (0.6,-0.2) {};
\node[small_dotnode] at (1,0.2) {};
\draw[fill=med-gray,opacity=0.5]
	(-0.6,0.2) -- (-1,-0.2) -- (0.6,-0.2) -- (1,0.2) -- (-0.6,0.2);
\draw (0.4,0.2) --
	(-0.6,0.2) -- (-1,-0.2) -- (0.45,-0.2); -- (1,0.2); -- (-0.6,0.2);
\draw (1.4,0.1) circle (1cm);
\draw[ribbon] (-1.1,-0.7) .. controls +(-60:0.5cm) and +(180:0.3cm) .. (c);
\draw[ribbon] (1.1,-0.7) .. controls +(-120:0.5cm) and +(0:0.3cm) .. (c);
\draw[ribbon] (c) .. controls +(120:1.2cm) and +(135:0.4cm) .. (0.1,0);
\node[emptynode] (Min) at (-3.5,1.5) {$\cM_{in}$};
\node[emptynode] (Mout) at (0,1.5) {$\cM_{out}$};
\draw[line width=0.4pt] (Min) -- (-2,0.5);
\draw[line width=0.4pt] (Min) -- (1,0.5);
\end{tikzpicture}
\end{equation}

We have
$\ups(\emptystate{B^3}) = 
\cD^{\frac{1}{4}} \cdot \emptyskein{B^3}$,
so
$\ups(\emptystate{\cM_{in}}) = 
\cD^{\frac{1}{2}} \cdot \emptyskein{B^3 \sqcup B^3}$,
and 
$\ups(\emptystate{\cM_{out}}) =
\cD^{-\frac{1}{2}} \cdot \emptyskein{S^2 \times B^1}$.
Then

\begin{align*}
\Ftld{\cH_1}(\emptyskein{B^3 \sqcup B^3})
&= \ZCY(\cH_1; S^2 \sqcup S^2)(\cD^{-1/2} \emptystate{\cM_{in}})
\\
&= \cD^{-1/2} \sum_i \ev(\cD^\inv d_i^{1/2} Z(\cH_1, l(f)=i), \emptystate{\cM_{in}})
\\
&= \cD^{-3/2} \sum_i d_i^{1/2} \coev_i
\\
&\overset{\ups}{\mapsto}
\cD^{-3/2} \sum_i d_i^2  \cD^{-1/2} \cdot \emptyskein{S^2 \times B^1}
\\
&=
\cD^{-1} \cdot \emptyskein{S^2 \times B^1}
\end{align*}

\textbf{Case $k=3$:}
$\Ftld{\cH_3}(\zeta_i)
= \delta_{i,\one} \cdot \emptyskein{B^3 \sqcup B^3}$,
where $\zeta_i$ is the identity morphism
for an object $(\{b\},\{X_i\}) \in \ZCYsk(S^2)$.

We will use the same PLCW decomposition as for $\cH_1$
(but $\cM_{in},\cM_{out}$ are reversed).
For the pre-state $\ups^\inv(\zeta_i)$
corresponding to the skein $\zeta_i$,
the 2-cell in the center of the diagram \eqnref{e:H1-plcw}
should be labeled $\one$,
and the 2-cells bounding the 3-balls should be labeled
$i$ and $i^*$ respectively.
Then for $i\neq \one$,
the local state space for the two 3-balls is 0,
thus we must have the coefficient $\delta_{i,\one}$.

For $i=\one$, $\zeta_{\one}$ is essentially
the empty skein $\emptyskein{B^1 \times S^2}$,
which corresponds to
$\ups^\inv(\emptyskein{B^1 \times S^2}) = \cD^{1/2} \cdot \emptystate{\cM_{in}}$
as before.
Thus 
\begin{align*}
\Ftld{\cH_3}(\emptyskein{B^1 \times S^2})
&= \ZCY(\cH_3; S^2 \sqcup S^2)(\cD^{1/2} \cdot \emptystate{\cM_{in}})
\\
&= \cD^{1/2} \ev(\cD^\inv Z(\cH_3, l\equiv \one), \emptystate{\cM_{in}})
\\
&= \cD^{-1/2} \cdot \emptystate{\cM_{out}}
\\
&\overset{\ups}{\mapsto}
\emptyskein{B^3 \sqcup B^3}
\end{align*}

\textbf{Case $k=2$:}
$\Ftld{\cH_2}(\emptyskein{B^2 \times S^1})
= \sum_i d_i \phi_i$,
where $\phi_i$ is the belt sphere labeled with object $i$.
(Technically, $\phi_i$ should be
$S^1 \times [-\eps,\eps] \times \{0\} \subset S^1 \times B^2$,
i.e. framed with 0 self-intersection
when included into $S^3 = \del \cH_2$.)

We use the following PLCW decomposition,
with corresponding anchor:
\begin{equation}
\begin{tikzpicture}
\node at (3,-2.7) {$S^3$};
\draw (0,0) circle (3.5cm);
\node[small_morphism] (a) at (-3,0) {};
\node[small_morphism] (b) at (-1.9,0) {};
\draw[ribbon] (b) -- (-2.4,0);
\draw[ribbon] (b)
	.. controls +(90:0.6cm) and +(180:1cm) .. (0,1.1)
	.. controls +(0:1cm) and +(90:0.6cm) .. (1.9,0)
	.. controls +(-90:0.6cm) and +(0:1cm) .. (0,-1.1);
\draw[ribbon] (b)
	.. controls +(-90:0.6cm) and +(160:0.3cm) .. (-0.95,-0.9);
\draw[ribbon] (a)
	.. controls +(-90:0.9cm) and +(180:0.9cm) .. (-1.5,-1.5)
	.. controls +(0:0.9cm) and +(-90:0.9cm) .. (0,0);
\draw (-2.5,0)
	.. controls +(-90:0.6cm) and +(180:2cm) .. (0,-1.5)
	.. controls +(0:2cm) and +(-90:0.6cm) .. (2.5,0)
	.. controls +(90:0.6cm) and +(0:2cm) .. (0,1.5)
	.. controls +(180:2cm) and +(90:0.6cm) .. (-2.5,0);
\draw (-1.23,0)
	.. controls +(60:0.3cm) and +(180:0.7cm) .. (0,0.5)
	.. controls +(0:0.7cm) and +(120:0.3cm) .. (1.23,0);
\draw[fill=med-gray,opacity=0.7] (-1.4,0)
	.. controls +(-90:0.3cm) and +(180:1cm) .. (0,-0.7)
	.. controls +(0:1cm) and +(-90:0.3cm) .. (1.4,0)
	.. controls +(90:0.3cm) and +(0:1cm) .. (0,0.7)
	.. controls +(180:1cm) and +(90:0.3cm) .. (-1.4,0);
\node[dotnode] at (-0.55,-0.66) {};
\draw[fill=med-gray,opacity=0.7] (-0.7,-0.31)
	.. controls +(150:0.4cm) and +(150:0.4cm) .. (-1,-1.4)
	.. controls +(-30:0.4cm) and +(-30:0.4cm) .. (-0.7,-0.31);
\draw (-0.7,-0.31)
	.. controls +(150:0.4cm) and +(150:0.4cm) .. (-1,-1.4);
\draw (-1.37,0.15)
	.. controls +(70:0.3cm) and +(180:0.8cm) .. (0,0.7)
	.. controls +(0:0.8cm) and +(110:0.3cm) .. (1.37,0.15);
\draw (-1.4,0.2)
	.. controls +(-60:0.3cm) and +(180:0.8cm) .. (0,-0.4)
	.. controls +(0:0.8cm) and +(-120:0.3cm) .. (1.4,0.2);
\draw[ribbon] (0,-1.1)
	.. controls +(180:0.3cm) and +(-20:0.3cm) .. (-0.95,-0.9);
\draw[fill=med-gray,opacity=0.3] (-2.5,0)
	.. controls +(-90:0.6cm) and +(180:2cm) .. (0,-1.5)
	.. controls +(0:2cm) and +(-90:0.6cm) .. (2.5,0)
	.. controls +(90:0.6cm) and +(0:2cm) .. (0,1.5)
	.. controls +(180:2cm) and +(90:0.6cm) .. (-2.5,0);
\draw[ribbon] (a)
	.. controls +(90:0.9cm) and +(180:0.9cm) .. (-1.5,1.5)
	.. controls +(0:0.9cm) and +(90:0.9cm) .. (0,0);
\draw[ribbon] (a) -- (-2.4,0);
\end{tikzpicture}
\end{equation}

We have $\ups(\emptystate{B^2 \times S^1})
	= \cD^{-1/4} \cdot \emptyskein{B^2 \times S^1}$ for both solid tori;
then
\begin{align*}
\ZCYsk(\cH_2; S^1 \times S^1)(\emptyskein{B^2 \times S^1})
&= \ZCY(\cH_2; S^1 \times S^1)(\cD^{1/4} \cdot \emptystate{B^2 \times S^1})
\\
&= \cD^{1/4} \sum_i d_i^{1/2} \coev_i
\\
&\overset{\ups}{\mapsto}
\sum_i d_i \phi_i
\end{align*}

It follows easily from the above computations that
$\Ftld{W} = \Fsk{W} \circ \Ups$ for elementary cornered cobordisms $W$,
thus we are done.
\end{proof}

\begin{corollary}
The skein pairing $\evsk$ (\defref{d:pairing-skein})
is equivalent to the reduced relative pre-state space pairing
$\evdelbar$
(\defref{d:reduced-relative-state-space-pairing}),
i.e. for $\wdtld{\vphi} \in H(\cM;(\cN,l)),
\wdtld{\vphi'} \in H(\ov{\cM};(\cN,l))$,
\[
\evdelbar(\wdtld{\vphi},\wdtld{\vphi'})
= \evsk(\ups(\vphi), \ups(\vphi'))
\]
In particular, since $\evdelbar$, when restricted
to $\ZCY(M;(\cN,l))$, is non-degenerate,
it follows that $\evsk$ is non-degenerate as well.
\label{c:pairing-equiv}
\end{corollary}
\begin{proof}
Follows immediately from \thmref{t:sk-equiv-4mfld}
and \lemref{l:reduced-pairing-TQFT}.
\end{proof}

\begin{remark}
To draw very loose parallels,
the PLCW definition of Crane-Yetter is to the skein definition
as simplicial homology is to singular homology.
$H(\cM)$ would play the role of the simplicial chain complex $C_*^\Delta(M)$,
both of which are already finite-dimensional,
but are dependent on the (PLCW) combinatorial structure;
$\VGraph(M)$ would play the role of the singular chain complex $C_*(M)$,
both of which are manifestly independent of
additional structures on $M$,
but are infinite-dimensional and unwieldy.
\end{remark}

\newpage
\section{Properties of Skein Categories}
\label{s:skein-prop}
\vspace{0.8in}

In this section, we consider properties
of skein modules and categories.
Subsection~\ref{s:skein_modules}
is focused on the space
of relations (i.e. the null graphs $Q \subset \VGraph(Y,\VV)$),
in particular how they are generated.
In Subsection~\ref{s:skprop-stacking},
we exhibit a ``stacking'' monoidal structure
on the category of boundary values
of manifolds of the form $P\times (0,1)$,
and show it to be pivotal.
Finally in Subsection~\ref{s:skprop-excision},
we show that skein categories satisfy excision.


Most of this section is taken from \ocite{KT},
but adapted to the PL category.
All PL topology results used in this section
is quoted in \secref{s:bg-PLtop}
for the reader's convenience.

All surfaces in this section will be of finite type,
i.e. closed surfaces with finitely many (open or closed) disks removed.

\subsection{Skein Modules}
\label{s:skein_modules}
\par \noindent

Recall that a null graph in $M$ is null
with respect to some 3-ball $D$,
and $D$ is allowed to touch the boundary $\del M$.
In future applications,
it will be convenient to only consider balls
$D$ that do not meet $\del M$,
since such balls can be displaced by
ambient isotopy but balls meeting $\del M$ may not.
If we exclude balls $D$ that meet $\del M$,
the resulting space of null graphs $Q'$
will be strictly smaller than $Q$,
but not by much; the following lemma says
we just need to include
equivalence of graphs under ambient isotopy rel boundary:

\begin{lemma}
\label{l:null-boundary-ball}
Let $M$ be a 3-manifold,
possibly with boundary or non-compact,
and let $\VV = \in \Obj \hatZCYsk(\del M)$
be a fixed boundary value.
Define $Q' \subset Q \subset \VGraph(M,\VV)$
to be the subspace generated by
graphs that are null with respect to a ball
that does not meet the boundary $\del M$.
Let $U$ be an open neighborhood of $\VV$ (more precisely,
of the set of arcs of $\VV$),
and let $Q'' \subset \VGraph(M,\VV)$
be relations obtained by ambient isotopy supported on $U$
fixing the boundary,
i.e. generated by graphs $\Gamma^1 - \Gamma^0$,
where $\Gamma^t = \zeta^t(\Gamma)$,
$\zeta^t$ is a compactly-supported ambient isotopy fixing $\del M$
and is supported on $U$.

Then $Q = Q' + Q''$.
\end{lemma}

\begin{proof}
It is clear that $Q',Q'' \subset Q$, so
it suffices to show that $Q \subset Q' + Q''$.
Let $\Gamma = \sum c_i\Gamma_i$ be a null graph
with boundary value $\VV$,
null with respect to a ball $D \subset M$,
and suppose $D$ meets the boundary $\del M$.
We would like to shrink $D$ to not meet $\del M$
while maintaining that $\Gamma$ be null with respect to it.
Clearly if $D$ does not meet any point in $\VV$
then we can do this, and then $\Gamma \in Q'$.

Suppose $D$ does contain some arc $b \in \VV$.
For each $i$, apply a small ambient isotopy $\zeta_i^t$ supported
in a small neighborhood of $b$ so that
the resulting graphs $\zeta_i^1(\Gamma_i)$
agree in a (possibly smaller) neighborhood of $b$
(this follows from uniqueness of collar neighborhoods,
see \thmref{t:collar-exist-unique}).
\[
\begin{tikzpicture}
\draw (0,1) -- (2,1)
  node[pos=0.9,above] {$\del M$};
\node[label={$b$}] (b) at (1,1) {};
\draw (0.6,1) to[out=180,in=60] (0,0.5);
\draw (1.4,1) to[out=0,in=120] (2,0.5);
\node at (0.2,0.5) {$D$};
\node[label={[shift={(0.4,-0.3)}] $\Gamma_1$}] (ga1) at (0.5,0.1) {};
\draw (1,1) to[out=-130,in=50] (ga1);
\node[label={[shift={(0.1,0)}] $\Gamma_2$}] (ga2) at (1.5,0) {};
\draw (1,1) to[out=-50,in=130] (ga2);
\end{tikzpicture}
\;\;
\tikz \node at (0,0.5) {$\longrightarrow$};
\;\;
\begin{tikzpicture}
\draw (0,1) -- (2,1)
  node[pos=0.9,above] {$\del M$};
\node[label={$b$}] (b) at (1,1) {};
\draw (0.6,1) to[out=180,in=60] (0,0.5);
\draw (1.4,1) to[out=0,in=120] (2,0.5);
\draw (0.7,1) to[out=0,in=180] (1.05,0.9);
\draw (1.05,0.9) to[out=0,in=180] (1.3,1);
\node at (0.2,0.5) {$D$};
\node[label={[shift={(0.2,-0.3)}] \tiny $\zeta_1^1(\Gamma_1)$}] (ga1) at (0.5,0.1) {};
\draw (1,1) to[out=-50,in=50] (ga1);
\node[label={[shift={(0.35,-0.1)}] \tiny $\zeta_2^1(\Gamma_2)$}] (ga2) at (1.5,0) {};
\draw (1,1) to[out=-50,in=130] (ga2);
\end{tikzpicture}
\;\;
\tikz \node at (0,0.5) {$\longrightarrow$};
\;\;
\begin{tikzpicture}
\draw (0,1) -- (2,1)
  node[pos=0.9,above] {$\del M$};
\node[label={$b$}] (b) at (1,1) {};
\draw (0,0.5) to[out=60,in=180] (0.5,0.9);
\draw (0.5,0.9) -- (1.5,0.9);
\draw (1.5,0.9) to[out=0,in=120] (2,0.5);
\node at (0.2,0.5) {$D$};
\node[label={[shift={(0.2,-0.3)}] \tiny $\zeta_1^1(\Gamma_1)$}] (ga1) at (0.5,0.1) {};
\draw (1,1) to[out=-50,in=50] (ga1);
\node[label={[shift={(0.35,-0.1)}] \tiny $\zeta_2^1(\Gamma_2)$}] (ga2) at (1.5,0) {};
\draw (1,1) to[out=-50,in=130] (ga2);
\end{tikzpicture}
\]
Then we can push $D$ slightly inwards away from the boundary at $b$,
and note that this new graph
$\Gamma' = \sum c_i \zeta_i^1(\Gamma_i)$
will be null with respect to the deformed $D$.
This reduces the number of arcs of $\VV$ that $D$ meets,
so after performing this finitely many times,
we are back to the case considered above where $D$
does not meet $\VV$.
Thus we see that repeated applications of isotopies
(i.e. relations in $Q''$) takes $\Gamma$ to another graph
$\Gamma' \in Q'$; in other words,
$\Gamma \in Q'' + Q'$.
\end{proof}

%
%

\begin{proposition}
\label{prp:null_cover}
Let $M$ be an 3-manifold,
possibly with boundary or non-compact.
Let $\VV \in \Obj \hatZCYsk(\del M)$
be a fixed boundary value.
Let $\{U_\al\}$ be a finite open cover of $M$
such that each arc of $\VV$ is contained in some $U_\al$.
Define $Q_\al \subset Q \subset \VGraph(M,\VV)$
to be the subspace of null graphs
in $M$ with boundary value $\VV$
that are null with respect to some closed ball $D$
contained in $U_\al$.
Then the space of null graphs is generated by $Q_\al$'s,
i.e.
\[
  Q = \sum Q_\al
\]
\end{proposition}

\begin{proof}
Let $\Gamma \in Q$ be a null graph.
Let $V_b$ be an open neighborhood of $b \in \VV$
that is contained in some $U_\al$,
and let $V = \bigcup V_b$.
By \lemref{l:null-boundary-ball},
$\Gamma$ can be written as a sum of null graphs
$\Gamma' + \Gamma''$,
where $\Gamma' = \sum c_j' \Gamma_j'$
is a sum of graphs, each $\Gamma_j'$
is null with respect to some ball
not meeting $\del M$,
and $\Gamma'' = \sum ((\Gamma_j'')^1 - (\Gamma_j'')^0)$
for some smooth isotopies $(\Gamma_j'')^t$ supported on $V$.
Clearly $\Gamma'' \in \sum Q_\al$,
so it suffices to consider $\Gamma'$.


Consider a term $\Gamma_j'$ in $\Gamma'$,
and suppose it is null with respect to some ball
$D$ not meeting $\del M$.
There exists an ambient isotopy
$\zeta^t: M \to M$ that moves $D$
into some open set $U_a$.
Then $\zeta^1(\Gamma_j') \in Q_a$.
But by \thmref{t:isotopy-open-cover-boundary}
the isotopy $\zeta$ can be chosen to be a sequence of
moves, each supported in some $U_\al$,
thus $\zeta^1(\Gamma_j') - \Gamma_j' \in \sum Q_\al$.
\end{proof}

\subsection{Properties of Categories of Boundary Values}
\label{s:skprop-stacking}
\par \noindent

Let $I' = (0,1)$, the open interval.

\begin{lemma}
Let $N_1,N_2$ be surfaces without boundary,
possibly non-compact.
Let $\vphi: N_1 \to N_2$
be an orientation-preserving embedding.
Then $\vphi$ induces an obvious inclusion functor
\[
  \vphi_*: \hatZCYsk(N_1) \to \hatZCYsk(N_2)
\]
that sends objects to their image under $\vphi$,
and sends morphisms to their image under $\vphi \times \id_I$.
This extends to the Karoubi envelopes
\[
  \vphi_*: \ZCYsk(N_1) \to \ZCYsk(N_2)
\]
Furthermore, an isotopy $\vphi^t: N_1 \to N_2$
induces a natural isomorphism from $\vphi_*^0$ to $\vphi_*^1$,
and isotopic isotopies induce the same natural isomorphisms.
\end{lemma}

\begin{proof}
Clear.
\end{proof}

\begin{lemma}
\label{lem:disjoint_union}
Under the same hypothesis above,
\begin{align*}
  \hatZCYsk(N_1 \sqcup N_2) \simeq \hatZCYsk(N_1) \boxtimes \hatZCYsk(N_2) \\
  \ZCYsk(N_1 \sqcup N_2) \simeq \ZCYsk(N_1) \boxtimes \ZCYsk(N_2)
\end{align*}
\end{lemma}

\begin{proof}
The proof for $\hatZCYsk$ is clear:
the inclusions of $N_1$ and $N_2$ into $N_1\sqcup N_2$
together induce
$\hatZCYsk(N_1) \boxtimes \hatZCYsk(N_2) \to \hatZCYsk(N_1 \sqcup N_2)$,
and this is easily seen to be an isomorphism of categories.
The equivalence for $\ZCYsk$ then follows by universal property,
and the fact that the Deligne-Kelly tensor product
of two finite semisimple abelian categories
is also a finite semisimple  abelian category.
\end{proof}

Next we discuss the ``stacking'' monoidal structure
of some special surfaces.
Let $\iI = (0,1)$,
and let $m : \iI \sqcup \iI \to \iI$
be $x/2$ on the first $\iI$ and $(x+1)/2$
on the second $\iI$.
This can be made part of an $A_\infty$-space structure,
as defined in \ocite{stasheff}:
$m$ is not associative,
but there is an isotopy
$m_3^t: \iI \sqcup \iI \sqcup \iI \to \iI$
from $m_3^0 = m \circ (m \sqcup \id_\iI)$
to $m_3^1 = m \circ (\id_\iI \sqcup m)$,
relating the two ways of including three intervals
into one.
Note that the ``straight line isotopy'' that is often used
to define $m_3^t$ is not a PL isotopy;
we may explicitly define $m_3^t = \vphi^t \circ m_3^0$,
where
\[
\vphi^t(x) = 
\begin{cases}
	2x & \text{ if } x < t/4
	\\
	(1+x)/2 & \text{ if } x > 1 - t/2
	\\
	x + t/4 & \text{ else}
\end{cases}
\]

Let $P$ be a finite collection of open intervals and circles.
By abuse of notation, we denote $\id_P \times m$,
$\id_P \times m_3^t$ by $m,m_3^t$, respectively.

\begin{proposition}
\label{prp:stacking}
There is a monoidal structure on $\hatZCYsk(P\times \iI)$
given as follows:
\begin{itemize}
  \item The tensor product is
  \[
    \tnsr := m_* : \hatZCYsk(P\times \iI) \boxtimes \hatZCYsk(P\times \iI)
      \to \hatZCYsk(P\times \iI)
  \]
  \item The unit $\one$ is the empty configuration $\EE$.
    (Left, right unit constraints are given in proof.)
  \item The associativity constraint $\alpha$ is the natural isomorphism
    that is induced by $m_3^t$.
\end{itemize}
This extends to a monoidal structure on $\ZCYsk(P\times \iI)$.
\end{proposition}

\begin{proof}
Left unit constraint $l_A : A \tnsr \one \to A$
is given by a ``straight line" graph,
likewise for right unit constraint.
That $\alpha$ satisfies the pentagon relations
follows from the fact that any two inclusions
$\iI^{\sqcup 4} \hookrightarrow \iI$
are isotopic, and any two isotopies are themselves isotopic.
The result for $\ZCYsk(P\times \iI)$ follows from universal property.
\end{proof}

\begin{proposition}
\label{p:stack_pivotal}
The monoidal structure on $\hatZCYsk(P\times \iI)$ and $\ZCYsk(P\times \iI)$
given in \prpref{prp:stacking}
is pivotal.
\end{proposition}

\begin{remark}
The input category $\cA$ has to be spherical,
but the resulting categories $\ZCYsk(P\times \iI)$ may not be;
we will see in \secref{s:annulus} that $\ZCYsk(S^1 \times \iI)$ is
pivotal but not spherical.
\end{remark}

\begin{proof}
It suffices to prove this for $\hatZCYsk(P\times \iI)$,
since its Karoubi envelope will inherit the pivotal structure.

The rigid and pivotal structures come from topological constructions.
Denote by $\theta: P \times \iI \to P \times \iI$
be the orientation-reversing diffeomorphism
which flips $\iI$,
i.e. $(p,x) \mapsto (p, 1-x)$.
Denote by $\Theta: P \times \iI \times [0,1] \to P \times \iI \times [0,1]$
the orientation-preserving diffeomorphism
that rotates the $\iI \times [0,1]$ rectangle by $180^\circ$,
i.e. $(p,x,t) \mapsto (p,1-x,1-t)$.

Denote by $\nu$ the map that takes $P \times \iI \times [0,1]$,
squeezes it in half along the $\iI$ direction,
bends it like an accordion so that the left side collapses,
and puts it back in $P \times \iI \times [0,1]$
so that the top and bottom are now attached to the top
(see \firef{f:nu-eta});
$\nu', \eta,\eta'$ are defined similarly.

\begin{figure}[ht]
$\nu,\nu',\eta, \eta':$
\begin{tikzpicture}
\begin{scope}[shift={(0,-0.5)}]
\draw (0,0) -- (1,0) node[pos=0, below] {\tiny $A$} node[below] {\tiny $B$};
\draw (0,1) -- (1,1) node[pos=0, above] {\tiny $C$} node[above] {\tiny $D$};
\draw[dotted] (0,0) -- (0,1);
\draw[dotted] (1,0) -- (1,1);
\end{scope}
\end{tikzpicture}
$\to$
\begin{tikzpicture}
\begin{scope}[shift={(0,-0.5)}]
\draw[dotted] (0,1) arc (180:360:0.5cm);
\draw (0,1) -- (1,1)
  node[pos=0, above] {\tiny $B$}
  node[pos=0.5,above] {\tiny $A/C$}
  node[above] {\tiny $D$};
\draw[white] (0.5,0.95) -- (0.5,1.05);
\end{scope}
\end{tikzpicture}
,
\begin{tikzpicture}
\begin{scope}[shift={(0,-0.5)}]
\draw[dotted] (0,1) arc (180:360:0.5cm);
\draw (0,1) -- (1,1)
  node[pos=0, above] {\tiny $C$}
  node[pos=0.5,above] {\tiny $D/B$}
  node[above] {\tiny $A$};
\draw[white] (0.5,0.95) -- (0.5,1.05);
\end{scope}
\end{tikzpicture}
,
\begin{tikzpicture}
\begin{scope}[shift={(0,-0.5)}]
\draw[dotted] (1,0) arc (0:180:0.5cm);
\draw (0,0) -- (1,0)
  node[pos=0,below] {\tiny $A$}
  node[pos=0.5,below] {\tiny $B/D$}
  node[below] {\tiny $C$};
  \draw[white] (0.5,-0.05) -- (0.5,0.05);
\end{scope}
\end{tikzpicture}
,
\begin{tikzpicture}
\begin{scope}[shift={(0,-0.5)}]
\draw[dotted] (1,0) arc (0:180:0.5cm);
\draw (0,0) -- (1,0)
  node[pos=0,below] {\tiny $D$}
  node[pos=0.5,below] {\tiny $C/A$}
  node[below] {\tiny $B$};
\draw[white] (0.5,-0.05) -- (0.5,0.05);
\end{scope}
\end{tikzpicture}
\caption{The maps $\nu,\nu',\eta,\eta'$ for $P = \{*\}$}
\label{f:nu-eta}
\end{figure}
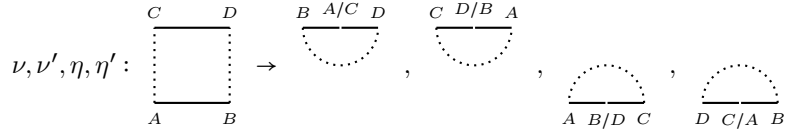

Let $\VV = (B, \{V_b\}) \in \Obj \hatZCYsk(P \times \iI)$.
Its left dual $\VV^*$ is given by
$(\theta(B), \{V_b^*\})$,
that is, apply the flipping diffeomorphism $\theta$
defined above to the marked points,
and label them by the left duals of the original labeling.
Similarly, the right dual is
$\rdual \VV = (\theta(B), \{\rdual V_b\})$.
(It is not too important to distinguish $V_b^*$ from $\rdual V_b$
since $\cA$ itself is pivotal.)

The left evaluation and coevaluation morphisms for $\VV$ are obtained by
applying $\nu$ and $\eta$ to $\id_\VV$, respectively.
Similarly, the right evaluation and coevaluation morphisms for $\VV$
are obtained by applying $\nu'$ and $\eta'$ to $\id_\VV$, respectively.
It is easy to see that these morphisms have the required properties.

Given a morphism
$f \in \Hom_{\hatZCYsk(P\times \iI)}(\VV, \VV')$
represented by a graph $\Gamma$,
it is easy to check that
its left and right duals are given by applying
the rotation $\Theta$ to $\Gamma$,
and keeping all orientations and labels of
the edges of $\Gamma$.

The pivotal structure is essentially the identity morphism,
but with one vertex on each vertical line labeled by $\delta$,
the pivotal structure of $\cA$.
\end{proof}

\begin{example}
  \label{xmp:cy_annulus}
Recall that $\ZCY(\iI \times \iI)  \simeq \ZCY(D^2) \simeq \cA$.
$\iI \times \iI$ can stack in two ways, along the first copy of $\iI$
(horizontal stacking) or the second (vertical stacking).
They both give monoidal structures equivalent to $\cA$'s.
This is explored further in \secref{s:elliptic}
\end{example}

Next we consider module categories over $\ZCYsk(P\times \iI)$.

\begin{definition}
Let $N$ be a surface,
and suppose there is a surface $N'$
with a boundary component $P \subseteq \del N'$,
such that $N = N' \backslash P$.

A \emph{0-collaring} of $N$ is a proper embedding
$\iota : P \times \iI \hookrightarrow N$
that extends to $\overline{\iota} : P \times [0,1) \hookrightarrow N'$;
a \emph{1-collaring}
extends to $P \times (0,1] \hookrightarrow N'$.
\end{definition}

Let $N$ be a 0-collared surface.
By crossing with $P$,
we can carry, via the collaring, the left module structure $n$
on $\iI$ to $N$,
obtaining a left multiplication
$n : P \times \iI \sqcup N \to N$
and an isotopy $n_3^t$ from
$n \circ (\id_{P\times \iI} \sqcup n)$
to $n \circ (m \sqcup \id_N)$.

\begin{proposition}
\label{p:collared_module}
Given a 0-collared surface $N$,
the multiplication $m$ on $P \times \iI$
defines a left module map
\[
n : P \times \iI \sqcup N \to N
\]
and an isotopy
\[
\Big( n_3^t :
n \circ (\id_{P\times \iI} \sqcup n)
\to n \circ (m \sqcup \id_N) \Big)
: P \times \iI \sqcup P \times \iI \sqcup N
\to N
\]
Then there is a left $\hatZCYsk(P \times \iI)$-module category structure on
$\hatZCYsk(N)$
given by
\[
  \lact \, := n_* : \hatZCYsk(P\times \iI) \boxtimes \hatZCYsk(N)
    \to \hatZCYsk(N)
\]
and the associativity constraint is given by the
natural isomorphism induced by the isotopy $n_3^t$.
This extends to a left $\ZCYsk(P\times \iI)$-module category structure on
$\ZCYsk(N)$.
\end{proposition}

\begin{proof}
Similar to \prpref{prp:stacking}.
\end{proof}

There is a similar story for right module structure,
where $N$ is a collared $n$-manifold so that 1 escapes to infinity.
Of course, the results above extend to surfaces with collared boundaries,
as $\ZCY$ is defined by removing the boundaries.

\subsection{Excision for $\ZCYsk$}
\label{s:skprop-excision}
\par \noindent

In this subsection, we prove the main result of this section,
that $\ZCYsk$ satisfies excision.
Let $I' = (0,1)$ and $I = [0,1]$ as before.

Let $N$ be a finite type surface without boundary.
To present $N$ as the quotient of some surface $N'$
by some gluing,
consider a map
$f: N \to S^1 = \RR / 2\ZZ$,
together with a trivialization of $P$-bundles
$\PI \simeq f^\inv(I')$,
where $P$ is some finite collection of open intervals and circles,
as in \secref{s:skprop-stacking}.
Take $N'$ to be the ``preimage of $(0,3)$ under $f$'';
more precisely, pullback $f$ along the universal
covering map $\RR \to \RR/2\ZZ$ to get
$\wdtld{f}: \wdtld{N} \to \RR$,
and take $N' = \wdtld{f}^\inv((0,3))$
(see figure below).
So $N$ is obtained from $N'$ by gluing the parts over
$(0,1)$ and $(2,3)$.

\[
\begin{tikzpicture}
\draw (1.5,1) circle (0.4cm);
\draw (3.5,1) circle (0.4cm);
\draw[line width=1mm, white] (-0.4,1) -- (4.4,1);
\draw (-0.4,1.05) -- (4.4,1.05);
\draw (-0.4,0.95) -- (4.4,0.95);
\draw[fill, white] (1.5,1) circle (0.37cm);
\draw[fill, white] (3.5,1) circle (0.37cm);
\draw[decorate,decoration={brace,amplitude=10pt,raise=4pt}]
 (0,1.5) -- (3,1.5)
 node[pos=0.5, yshift=0.8cm] {$N'$};
\draw[dotted] (0,1.5) -- (0,-1);
\draw[dotted] (3,1.5) -- (3,-1);
\draw[decorate,decoration={brace,amplitude=3pt,raise=2pt}]
 (0,1.1) -- (1,1.1)
 node[pos=0.5, yshift=0.4cm] {\small $P \times I'$};
\draw[decorate,decoration={brace,amplitude=3pt,raise=2pt}]
 (2,1.1) -- (3,1.1)
 node[pos=0.5, yshift=0.4cm] {\small $P \times I'$};

\draw[->] (1.5,0.4) -- (1.5,-0.6) node[pos=0.5,right] {$\tilde{f}$};
\node at (-1.5,1) {$\tilde{N}$};
\node at (-1.5,-1) {$\RR$};

\draw (-0.4,-1) -- (4.4,-1);
\foreach \x in {0,...,4} {
 \node at (\x,-1.3) {\x};
 \draw (\x,-0.9) -- (\x,-1.1);
}
\node at (0.5,-0.7) {$I'$};

\end{tikzpicture}
\hspace{5pt}
\begin{tikzpicture}
\node at (0,1) {$\to$};
\node at (0,-1) {$\to$};
\end{tikzpicture}
\hspace{5pt}
\begin{tikzpicture}

\draw (0.5,1) circle (0.4cm);
\draw[line width=1mm, white] (-0.5,1) -- (1.5,1);
\draw (-0.5,1.05) -- (1.5,1.05);
\draw[dotted] (-0.7,1.05) -- (1.7,1.05);
\draw (-0.5,0.95) -- (1.5,0.95);
\draw[dotted] (-0.7,0.95) -- (1.7,0.95);
\draw[fill, white] (0.5,1) circle (0.37cm);

\draw[->] (0.5,0.4) -- (0.5,-0.6) node[pos=0.5,right] {$f$};
\node at (2.5,1) {$N$};
\node at (2.5,-1) {$\RR/2\ZZ$};

\draw (-0.5,-1) -- (1.5,-1);
\draw[dotted] (-0.7,-1) -- (1.7,-1);
\node at (0,-1.3) {1};
\node at (1,-1.3) {2/0};
\foreach \x in {0,...,1} {
 \draw (\x,-0.9) -- (\x,-1.1);
}
\end{tikzpicture}
\]

\begin{remark}
\label{rmk:gluing}
Excision is usually phrased in terms of gluing
two collared manifolds.
In the above language,
that will correspond to the case when
$N' = N_1 \sqcup N_2$,
where $N_1 = \wdtld{f}^\inv((0,1.5))$,
$N_2 = \wdtld{f}^\inv((1.5,3))$,
so that the pullback map $N' \to N$
is the gluing/overlapping of $N_1$ and $N_2$
over $(0,1)$, the collared neighborhoods.
\end{remark}

Since $\wdtld{f}^\inv((0,1)) \simeq \wdtld{f}^\inv((2,3))$
naturally,
the trivialization $\PI \simeq f^\inv(I')$
gives a left and right $\PI$-module structure on $N'$,
and makes $\hatZCYsk(N')$ a $\hatZCYsk(\PI)$-bimodule category
(likewise for $\ZCYsk$).

The natural gluing map $N' \to N$
is the composition $N' \subset \wdtld{N} \to N$.
We can also embed $N'$ in $N$ as follows:
consider a following sequence of maps
$N' \to \PI \sqcup N' \sqcup \PI \to N' \to N$;
the first map is just the obvious inclusion,
the second is the left and right module maps
``squeezing" $N'$ into itself,
and the third map is the natural quotient map.
It is easy to see that the composition
is an embedding, in fact a diffeomorphism
onto $N \backslash f^\inv(0.5)$.

We denote the composition above by $i$,
and denote $X = f^\inv(0.5)$,
which we call the \emph{seam}.

Since $i:N' \to N$ is an embedding,
it induces a functor
$i_* : \hatZCYsk(N') \to \hatZCYsk(N)$.
Recall that there is a natural functor
$\htr: \hatZCYsk(N') \to \htr(\hatZCYsk(N'))$
that is identity on objects.

\begin{lemma}
The inclusion functor $i_*: \hatZCYsk(N') \to \hatZCYsk(N)$
extends along $\htr$ to a functor
$i_*: \htr(\hatZCYsk(N')) \to \hatZCYsk(N)$.
\end{lemma}
\begin{proof}
Consider a map $\Psi: N' \times I \to N \times I$
described as the following composition:
\begin{equation}
\Psi:\;
\begin{tikzpicture}
\begin{scope}[shift={(0,-0.5)}]
\draw (0,0) -- (0,1);
\draw (0,1) -- (1,1);
\draw (1,1) -- (1,0);
\draw (1,0) -- (0,0);
\draw (0,0) -- (0.5,1);
\draw (0.5,1) -- (1,0);
\draw[dashed] (1,0) -- (2,0);
\draw[dashed] (1,1) -- (2,1);
\draw (2,0) -- (2,1);
\draw (2,1) -- (3,1);
\draw (3,1) -- (3,0);
\draw (3,0) -- (2,0);
\draw (2,1) -- (2.5,0);
\draw (2.5,0) -- (3,1);
\node at (0.2,0.75) {\tiny 1};
\node at (0.5,0.3) {\tiny 2};
\node at (0.8,0.75) {\tiny 3};
\node at (2.2,0.25) {\tiny 4};
\node at (2.5,0.7) {\tiny 5};
\node at (2.8,0.25) {\tiny 6};
\end{scope}
\end{tikzpicture}
\;\;\;\;
\to
\;\;\;\;
\begin{tikzpicture}
\begin{scope}[shift={(0.5,-0.5)}]
\draw (0.5,0) -- (0.5,1);
\draw (0.5,1) -- (1,1);
\draw (1,1) -- (1,0);
\draw (1,0) -- (0.5,0);
\draw (0.5,1) -- (1,0);
\draw (0.5,1) -- (0.75,0);
\draw[dashed] (1,0) -- (2,0);
\draw[dashed] (1,1) -- (2,1);
\draw (2,0) -- (2,1);
\draw (2,1) -- (2.5,1);
\draw (2.5,1) -- (2.5,0);
\draw (2.5,0) -- (2,0);
\draw (2.5,0) -- (2,1);
\draw (2.5,0) -- (2.25,1);
\node at (0.6,0.15) {\tiny 1};
\node at (0.8,0.15) {\tiny 2};
\node at (0.8,0.75) {\tiny 3};
\node at (2.2,0.25) {\tiny 4};
\node at (2.2,0.85) {\tiny 5};
\node at (2.4,0.85) {\tiny 6};
\end{scope}
\end{tikzpicture}
\;\;
\to
\;\;\;\;\;\;
\begin{tikzpicture}
\begin{scope}[shift={(0.5,-0.5)}]
\draw (0.5,0) -- (0.5,1);
\draw (0.5,1) -- (1,1);
\draw (1,1) -- (1,0);
\draw (1,0) -- (0.5,0);
\draw (0.5,1) -- (1,0);
\draw (0.5,1) -- (0.75,0);
\draw[dashed] (1,0) -- (1.5,0);
\draw[dashed] (1,1) -- (1.5,1);
\draw[dashed] (0,0) -- (-0.5,0);
\draw[dashed] (0,1) -- (-0.5,1);
\draw (0,0) -- (0,1);
\draw (0,1) -- (0.5,1);
\draw (0.5,1) -- (0.5,0);
\draw (0.5,0) -- (0,0);
\draw (0.5,0) -- (0,1);
\draw (0.5,0) -- (0.25,1);
\node at (0.6,0.15) {\tiny 1};
\node at (0.8,0.15) {\tiny 2};
\node at (0.8,0.75) {\tiny 3};
\node at (0.2,0.25) {\tiny 4};
\node at (0.2,0.85) {\tiny 5};
\node at (0.4,0.85) {\tiny 6};
\end{scope}
\end{tikzpicture}
\end{equation}
The vertical direction corresponds to the $I$ factor in $N' \times I$,
while a horizontal slice is $N$ or $N'$
(the diagram depicts cross-sections of the form
$\{*\} \times I' \times I \subset N \times I$).
The numbered regions are in the $(P \times I') \times I$ portions
of $N' \times I$ (i.e. the collared part $\times I$);
$\Psi$ is linear on these regions.
The first map sends $N' \times I$ into itself,
and the second map is simply the gluing map $N' \to N$ (crossed with $I$);
the left and right vertical sides on the middle figure
get identified as $f^\inv(0.5) \subset N \times I$.

Below we depict a graph $\Gamma$ in $N' \times I$
with incoming boundary value $A \lact \VV$ and
outgoing boundary value $\VV' \ract A$,
representing an element of $\Hom_{\htr(\hatZCYsk(N'))}(\VV,\VV')$.
It is sent to a graph $\Psi(\Gamma)$ in $N \times I$
with incoming boundary value $\VV$ and outgoing boundary value $\VV'$,
representing an element of $\Hom_{\hatZCYsk(N)}(\VV,\VV')$.

\begin{equation}
\begin{tikzpicture}
\begin{scope}[shift={(0,-0.5)}]
\node at (-0.5,1) {0};
\draw (0,1) -- (1,1) node[pos=0.25,above] {$A$};
\filldraw (0.5,1) circle (0.5pt);
\draw[dashed] (1,1) -- (2,1) node[pos=0.5,above] {$\VV$};
\draw (2,1) -- (3,1);
\node at (-0.5,0) {1};
\draw (0,0) -- (1,0);
\draw[dashed] (1,0) -- (2,0) node[pos=0.5,below] {$\VV'$};
\draw (2,0) -- (3,0) node[pos=0.75,below] {$A$};
\filldraw (2.5,0) circle (0.5pt);
\draw[dotted] (0,1) -- (0,0);
\draw[dotted] (3,0) -- (3,1);
\node at (1.5,0.5) {\small $\Gamma$};
\node[dotnode] (c) at (0.45,0.5) {}; 
\node[dotnode] (b) at (0.75,0.5) {}; 
\draw (c) -- (b);
\draw (c) to[out=-120,in=90] (0.35,0); 
\draw (b) to[out=-120,in=90] (0.75,0); 
\draw (c) to[out=120,in=-90] (0.25,1); 
\draw (b) -- +(10:0.3cm); 
\draw[dotted] (b) -- +(10:0.5cm); 
\draw (b) -- +(-10:0.3cm); 
\draw[dotted] (b) -- +(-10:0.5cm); 
\node[dotnode] (a) at (2.25,0.5) {}; 
\node[dotnode] (d) at (2.25,0.2) {}; 
\node[dotnode] (e) at (2.5,0.7) {}; 
\draw (a) -- (d);
\draw (a) to[out=60,in=-90] (2.35,1); 
\draw (a) to[out=-10,in=-100] (e);
\draw (d) to[out=0,in=90] (2.75,0); 
\draw (d) -- (2.25,0); 
\draw (a) -- +(180:0.3cm); 
\draw[dotted] (a) -- +(180:0.5cm); 
\end{scope}
\end{tikzpicture}
\;\;
\mapsto
\;\;
\begin{tikzpicture}
\begin{scope}[shift={(0,-0.5)}]
\draw[dashed] (-0.5,1) -- (0,1);
\draw (0,1) -- (1,1);
\node at (1,1.3) {$i_*(\VV)$};
\draw[dashed] (1,1) -- (1.5,1);
\filldraw (0.5,1) circle (0.5pt); 
\draw[dashed] (-0.5,0) -- (0,0);
\draw (0,0) -- (1,0);
\node at (1,-0.3) {$i_*(\VV')$};
\draw[dashed] (1,0) -- (1.5,0);
\filldraw (0.5,0) circle (0.5pt); 
\draw[densely dotted] (0.5,0) -- (0.5,1);
\node at (0.6,0.75) {\tiny $A$};
\node at (1.7,0.5) {\small $\Psi(\Gamma)$};
\node[emptynode] (A) at (0.5,0.5) {}; 
\node[dotnode] (a) at (0.2,0.5) {}; 
\node[dotnode] (d) at (0.2,0.2) {}; 
\node[dotnode] (e) at (0.3,0.7) {}; 
\draw (d) .. controls +(0:0.3cm) and +(180:0.2cm) .. (A); 
\draw (a) -- (d) -- (0.2,0);
\draw (a) to[out=90,in=-80] (0.15,1); 
\draw (a) to[out=0,in=-90] (e);
\draw (a) -- +(180:0.3cm); 
\draw[dotted] (a) -- +(180:0.5cm); 
\node[dotnode] (b) at (0.85,0.5) {}; 
\node[dotnode] (c) at (0.7,0.5) {}; 
\draw (c) -- (A); 
\draw (c) -- (b);
\draw (c) to[out=-95,in=90] (0.7,0); 
\draw (b) to[out=-100,in=80] (0.85,0); 
\draw (b) -- +(10:0.3cm); 
\draw[dotted] (b) -- +(10:0.5cm); 
\draw (b) -- +(-10:0.3cm); 
\draw[dotted] (b) -- +(-10:0.5cm); 
\end{scope}
\end{tikzpicture}
\end{equation}

The only points in $N \times I$ that are hit more than once
are in $X \times I$;
we sometimes also call this the seam.
We denote $L_0 = X \times I$.

In the figure above, the seam $L_0$ is depicted as the vertical dotted line
labeled with $A$ in the right figure.
The seam is also the image of the top left and bottom right boundary pieces
(the parts labeled $A$ in the left figure).
The image of $\Psi|_{N' \times I'}$ is
exactly $N \times I' \backslash L_0$.

We claim that the following map is well-defined:
\begin{align*}
\Hom_{\htr(\hatZCYsk(N'))}(\VV,\VV') &\to \Hom_{\hatZCYsk(N)}(i_*(\VV),i_*(\VV'))
\\
\Gamma &\mapsto \Psi(\Gamma)
\end{align*}
It is not hard to see that the assignment
$\Gamma \mapsto \Psi(\Gamma)$
yields a well-defined map
$\ihom{\hatZCYsk(N')}{A}(\VV,\VV') \to \Hom_{\hatZCYsk(N)}(i_*(\VV),i_*(\VV'))$;
a graph $\Gamma=\sum c_i\Gamma_i$
that is null with respect to some ball $D$
would have image $\Psi(\Gamma)$ null with respect to $\Psi(D)$.
We need to check that the relations $\sim$ in
$\Hom_{\htr(\hatZCYsk(N'))}(\VV,\VV') = \bigoplus
\ihom{\hatZCYsk(N')}{A}(\VV,\VV')/\sim$
are satisfied.
Recall that relations are generated by
$\Theta \circ (\psi \lact \id_\VV) - (\id_{\VV'} \ract \psi) \circ \Theta$,
where $\Theta \in \ihom{\hatZCYsk(N')}{A,B}(\VV,\VV')$
and $\psi \in \Hom_{\hatZCYsk(P\times I')}(B,A)$.
We see that
\begin{equation}
\label{e:gluing-relation-zero}
\begin{tikzpicture}
\begin{scope}[shift={(0,-0.5)}]
\node at (1.5,0.5) {$\Theta$};
\draw (0,1) -- (1,1)
  node[pos=0.35,above] {$B$};
\node[dotnode] at (0.7,1) {};
\draw (0,1) -- (0.35,0.7) -- (0.7,1);
\node at (0.35,0.9) {\tiny $\psi$};
\draw[regular] (1,1) -- (2,1)
  node[pos=0.5,above] {$\VV$};
\draw (2,1) -- (3,1);
\draw (0,0) -- (1,0);
\draw[regular] (1,0) -- (2,0)
  node[pos=0.5,below] {$\VV'$};
\draw (2,0) -- (3,0)
  node[pos=0.65,below] {$B$};
\node[dotnode] at (2.3,0) {};
\end{scope}
\end{tikzpicture}
\; - \;
\begin{tikzpicture}
\begin{scope}[shift={(0,-0.5)}]
\node at (1.5,0.5) {$\Theta$};
\draw (0,1) -- (1,1)
  node[pos=0.35,above] {$B$};
\node[dotnode] at (0.7,1) {};
\draw[regular] (1,1) -- (2,1)
  node[pos=0.5,above] {$\VV$};
\draw (2,1) -- (3,1);
\draw (0,0) -- (1,0);
\draw[regular] (1,0) -- (2,0)
  node[pos=0.5,below] {$\VV'$};
\draw (2,0) -- (3,0)
  node[pos=0.65,below] {$B$};
\node[dotnode] at (2.3,0) {};
\draw (2.3,0) -- (2.65,0.3) -- (3,0);
\node at (2.65,0.1) {\tiny $\psi$};
\end{scope}
\end{tikzpicture}
\; \mapsto \;
\begin{tikzpicture}
\begin{scope} [shift={(0,-0.5)}]
\draw[regular] (0,1) -- (2,1);
\draw (0.5,1) -- (1.5,1);
\draw[regular] (0,0) -- (2,0);
\draw (0.5,0) -- (1.5,0);
\draw[dotted] (1,0) -- (1,1);
\draw[dotted] (1,0) -- (0.75,0.5) -- (1,1);
\node at (0.9,0.5) {\tiny $\psi$};
\end{scope}
\end{tikzpicture}
\; - \;
\begin{tikzpicture}
\begin{scope} [shift={(0,-0.5)}]
\draw[regular] (0,1) -- (2,1);
\draw (0.5,1) -- (1.5,1);
\draw[regular] (0,0) -- (2,0);
\draw (0.5,0) -- (1.5,0);
\draw[dotted] (1,0) -- (1,1);
\draw[dotted] (1,0) -- (1.25,0.5) -- (1,1);
\node at (1.1,0.5) {\tiny $\psi$};
\end{scope}
\end{tikzpicture}
= 0
\end{equation}
We leave checking that composition is respected as
a simple exercise.
\end{proof}

We want to show that $i_*$ is an equivalence,
and will be considering $\Psi^\inv$ applied to graphs.
It is not clear that this is well-defined,
e.g. moving parts of a graph in $N \times I$ across the seam $L_0$.
could result in different graphs with
different boundary conditions in $N' \times I$.
However, the relation
$\Theta \circ (\psi \lact \id_\VV) - (\id_{\VV} \ract \psi) \circ \Theta$
essentially takes care of this ambiguity.

Let us make this precise.
Let $\beta \in (0,0.5)$ be such that
$i_*(\VV)$ and $i_*(\VV')$ do not have any arcs
in the neighborhood
$f^\inv( (0.5-\beta, 0.5+\beta)) \subset N$
of $X$.
Let $\zeta_\al$, $\al \in (-\beta',\beta')$
for some smaller $0 < \beta' < \beta$,
be an ambient isotopy (indexed by $\al$)
of $N \times I$ 
defined, in a smaller neighborhood
\begin{equation}
U_1 = f^\inv( (0.5-\beta',0.5+\beta')) \subset N \times I
\label{e:seam-U1}
\end{equation}
of the seam, by
\begin{equation}
\zeta_\al : (p,x,s)
\mapsto
\begin{cases}
(p,x \pm s,s) \;\;\;\text{if } x < |\al|
\\
(p,x \pm (1-s),s) \;\;\;\text{if } s > 1 - |\al|
\\
(p,x + \al, s) \;\;\;\text{else}
\end{cases}
\end{equation}
where the $\pm$ is the sign of $\al$,
$(p,x) \in P \times I'$ (so that $X = \{x=0\}$),
$s$ parameterizes the up-down axis
(see \eqnref{e:L_al}).
Let $L_\al = \zeta_\al(L_0)$.

Define $\Psi_\al = \zeta_\al \circ \Psi : N' \times I \to N \times I$,
so that the new ``seam'' is $L_\al$.

\begin{equation}
\begin{tikzpicture}
\begin{scope}[shift={(0,-0.5)}]
\draw[dashed] (-1,0) -- (1,0);
\draw[dashed] (-1,1) -- (1,1);
\draw (-0.5,0) -- (0.5,0);
\draw (-0.5,1) -- (0.5,1);
\draw (0,0) -- (0,1); 
\draw (0,0) -- (-0.25,0.25) -- (-0.25,0.75) -- (0,1);
\draw (0,0) -- (0.1,0.1) -- (0.1,0.9) -- (0,1);
\draw (0,0) -- (0.3,0.3) -- (0.3,0.7) -- (0,1);
\draw[line width=0.4pt] (-0.3,-0.15) -- (0,0.3);
\node at (0.1,-0.3) {\tiny $L_0 = $ seam};
\node[emptynode] (x) at (0.5,0.8) {};
\node at (0.7,0.8) {\tiny $L_\al$};
\draw[line width=0.4pt] (x) -- (-0.25,0.7);
\draw[line width=0.4pt] (x) -- (0.1,0.6);
\draw[line width=0.4pt] (x) -- (0.3,0.6);
\draw[->] (-2.5,0.8) -- (-2,0.8);
\draw[->] (-2.5,0.8) -- (-2.5,0.3);
\node at (-1.8,0.8) {\smallerer $x$};
\node at (-2.5,0.1) {\smallerer $s$};
\node[emptynode] at (2.5,0) {};
\end{scope}
\end{tikzpicture}
\label{e:L_al}
\end{equation}

Note that $L_\al$'s overlap near the top and bottom $X \times \{0,1\}$,
but away from there, they form parallel copies of the seam $X \times I$.
Thus, for a compact submanifold
that does not meet $X \times \{0,1\}$,
intersections with $L_\al$ coincide with level sets of $x$.

\begin{lemma}
\label{lem:alpha_independence}
Let $\Gamma$ be a graph in $N \times I$
that represents a morphism in
$\Hom_{\hatZCYsk(N)}(i_*(\VV),i_*(\VV'))$
for some $\VV,\VV' \in \htr(\hatZCYsk(N'))$,
and let $S(\Gamma)$ denote the surface of the graph.

Since $i_*(\VV),i_*(\VV')$ does not meet the seam,
$\al$ defines a function
\[
\al : S(\Gamma) \cap U_1 \to (0.5 - \beta', 0.5 + \beta')
\]

By \prpref{p:regular-level-set}
there exists an $\al$ such that $\Gamma$
is weakly transverse to $L_\al$,
and by \lemref{l:narrow-transverse},
there is a strict narrowing $\Gamma'$ of $\Gamma$
that is transverse to $L_\al$,
thus defining a boundary value $A_\al = \Gamma' \cup L_\al$ on $L_\al$.
We see that $\Psi_\al^\inv(\Gamma')$
is a graph in $N' \times I$ representing a morphism in
$\Hom_{\hatZCYsk(N')}(A_\al \lact \VV, \VV' \ract A_\al)$.
Then as a morphism in $\Hom_{\htr(\hatZCYsk(N'))}(\VV,\VV')$,
$\Psi_\al^\inv(\Gamma')$ is independent of
such a choice of $\al$ and strict narrowing of $\Gamma$.
\end{lemma}

\begin{proof}
First note that if $\Gamma$ is transverse (as a graph)
to $L_\al,L_{\al'}$,
then it is clear from the picture \eqnref{e:gluing-relation-zero}
that $\Psi_\al^\inv(\Gamma) = \Psi_{\al'}^\inv(\Gamma)$.

Next, it is easy to see that if $\Gamma$ is transverse to $L_\al$,
then a further strict narrowing of $\Gamma$
will also not affect $\Psi_\al^\inv(\Gamma)$
(the arcs of the object $A_\al$ just get shorter).
Thus we may assume that the graphs discussed below are
``sufficiently narrow''
(see \rmkref{r:narrowing-infinitesimal}
and proof of \lemref{l:ups-isom}).

Finally, we show that given $\al$,
the choice of strict narrowing is irrelevant;
then to prove the lemma, take $\al,\al'$
with $L_\al,L_{\al'}$ transverse to $S(\Gamma)$,
and find a narrowing of $\Gamma$ that is transverse to both
$L_\al,L_{\al'}$,
then by the first observation, we are done.

To show irrelevance of choice of strict narrowing,
consider the more general situation where
$\Gamma,\Gamma'$ are isotopic graphs
that are transverse to
$L_\al,L_{\al'}$ respectively.
Let $U_\pm = N \times \backslash L_{\pm \beta'}$;
$\{U_+,U_-\}$ forms an open cover of
$N \times I \backslash X \times \{0,1\}$.

By \thmref{t:isotopy-open-cover-boundary},
the isotopy from $\Gamma$ to $\Gamma'$
can be broken into moves supported in $U_\pm$.
In particular, each move is supported in some closed ball
that does not meet some $L_\al$;
more precisely,
if $\{h_k^t\}$ is the sequence of moves,
with $h_k^t$ supported on a closed ball $D_k$,
then there exists $\al_k \in (0.5-\beta',0.5+\beta')$ such that
$D_k$ does not meet $L_{\al_k}$,
and we may take $\al_k$ such that $L_{\al_k}$ is weakly transverse
to the surface of the graph during the isotopy $h_k^t$.

Let $\Gamma_k = h_k^1 \circ \cdots \circ h_1^1 (\Gamma)$
for $0 \leq k \leq l$
(with $\Gamma_0 = \Gamma, \Gamma_l = \Gamma'$),
and let $\al_0 = \al, \al_{l+1} = \al'$.
By the second observation above, we assume $\Gamma$
is narrow enough to begin with,
or more precisely,
we retroactively apply a strict narrowing to $\Gamma$
so that each $\Gamma_k$ is transverse to $L_{\al_k}$;
this does not affect
$\Psi_\al^\inv(\Gamma)$ nor $\Psi_{\al'}^\inv(\Gamma')$.
We consider the sequence of graphs in $N' \times I$:
\[
\Psi_{\al}^\inv(\Gamma)
= \Psi_{\al_0}^\inv(\Gamma_0),
\Psi_{\al_1}^\inv(\Gamma_0),
\Psi_{\al_1}^\inv(\Gamma_1),
\Psi_{\al_2}^\inv(\Gamma_1),
\ldots,
\Psi_{\al_l}^\inv(\Gamma_{l-1}),
\Psi_{\al_l}^\inv(\Gamma_l),
\Psi_{\al_{l+1}}^\inv(\Gamma_l)
=\Psi_{\al'}^\inv(\Gamma')
\]
The adjacent graphs
$\Psi_{\al_l}^\inv(\Gamma_l), \Psi_{\al_{l+1}}^\inv(\Gamma_l)$
are equivalent as skeins by the first observation,
while
$\Psi_{\al_l}^\inv(\Gamma_{l-1}), \Psi_{\al_l}^\inv(\Gamma_l)$
are isotopic.
\end{proof}

\begin{theorem}
  \label{t:hat_excision}
The extension $i_*: \htr(\hatZCYsk(N')) \to \hatZCYsk(N)$
is an equivalence.
\end{theorem}

\begin{proof}
It was already evident from the object map that
$\htr(\hatZCYsk(N')) \to \hatZCYsk(N)$ is essentially surjective
- it only misses objects that have marked points on $f^\inv(0.5)$,
but such an object is isomorphic to an object
with points moved slightly off of $f^\inv(0.5)$.\\

To show that $i_*$ is fully faithful,
fix objects $\VV,\VV' \in \htr(\hatZCYsk(N'))$.
Let $\Gamma$ be a colored ribbon graph
in $N \times I$ with boundary values $i_*(\VV),i_*(\VV')$.

By \lemref{lem:alpha_independence},
we have a well-defined map
$\Psi_-^\inv : \VGraph(N \times I; i_*(\VV),i_*(\VV'))
\to \Hom_{\htr(\hatZCYsk(N'))}(\VV,\VV')$.
This map factors through the quotient map
$\VGraph(N \times I; i_*(\VV),i_*(\VV')) \to
\Hom_{\hatZCYsk(N)}(i_*(\VV),i_*(\VV'))$:
using the same open cover $\{U_\pm\}$
as in the proof of \lemref{lem:alpha_independence},
a null graph in $N \times I$,
null with respect to some $D \subset U_\pm$,
is clearly sent to a null graph in $N' \times I$,
null with respect to $\Psi_\al^\inv(D)$.

It is clear that for a graph $\Gamma'$ in $N' \times I$
representing an element of $\Hom_{\htr(\hatZCYsk(N'))}(\VV,\VV')$,
$\Psi_0^\inv (\Psi(\Gamma')) = \Gamma'$.
It is also clear that for a graph $\Gamma$ in $N \times I$,
$\Psi(\Psi_\al^\inv(\Gamma))$ is isotopic to $\Gamma$.
Thus $i_*$ is fully faithful.
\end{proof}

Combining the topological result above with
the algebraic results of \secref{s:skein_modules},
we have the main result of this section:
\begin{theorem}
\label{t:excision-skein}
There is an equivalence
\[
  \cZ_{\ZCYsk(P\times I)}(\ZCYsk(N')) \simeq \ZCYsk(N)
\]
In particular, when $N = N_1 \cup N_2$
as in \rmkref{rmk:gluing},
\[
  \ZCYsk(N_1) \boxtimes_{\ZCYsk(P\times I)} \ZCYsk(N_2) 
  \simeq \cZ_{\ZCYsk(P\times I)}(\ZCYsk(N_1 \sqcup N_2))\simeq \ZCYsk(N)
\]
\end{theorem}

\begin{proof}
We claim that $\ZCYsk(P\times I)$ is multifusion; we justify this claim later.
By \prpref{p:stack_pivotal}, $\hatZCYsk(P\times I)$ is pivotal.
In reference to the notation in \secref{s:skein_modules},
take $\cC' = \hatZCYsk(P\times I), \cC = Z(P\times I),
\cM' = \hatZCYsk(N'), \cM = Z(N')$.
Then we have
\begin{align*}
\cZ_{\ZCYsk(P\times I)}(\ZCYsk(N'))
&\simeq \Kar(\htr_{\hatZCYsk(P\times I)}(\ZCYsk(N')))
  & \text{ by \corref{cor:center},}\\
&\simeq \Kar(\htr_{\hatZCYsk(P\times I)}(\hatZCYsk(N')))
  & \text{ by \lemref{lem:M_Karoubi},}\\
&\simeq \Kar(\hatZCYsk(N))
  & \text{ by extending $i_*$ from \thmref{t:hat_excision} to Kar} \\
&= \ZCYsk(N)
\end{align*}

The second statement follows from the first
by applying \lemref{lem:disjoint_union} and \eqnref{e:center-product}.

Now we need to justify $\ZCYsk(P\times I)$ being multifusion,
in particular, that it is semisimple,
as \prpref{p:stack_pivotal} guarantees it is rigid and pivotal.
This is true for $P = I$.
We apply this result with $N' = I \times I, N = S^1 \times I$,
and by \prpref{p:ZM_ss}, $\ZCYsk(S^1 \times I)$ is semisimple.
In fact, $\ZCYsk(S^1 \times I) \simeq \cZ_{\ZCYsk(I \times
I)}(\ZCYsk(I \times I)) \simeq \ZA$,
the Drinfeld center of $\cA$, which is known to be modular,
in particular multifusion;
we discuss this equivalence in more detail in \secref{s:ann-cy}.

So $\ZCYsk(P\times I)$ is pivotal multifusion for any connected $P$;
the claim easily follows for a disjoint union of finitely many such $P$'s.
\end{proof}

\begin{corollary}
The functor
$i_* : \ZCYsk(N') \to \ZCYsk(N)$ is dominant.
\label{c:excision-dominant}
\end{corollary}
\begin{proof}
Follows immediately from \prpref{p:I_dominant}
\end{proof}

\begin{corollary} \label{c:skein_ss}
$\ZCYsk(N)$ is a finite semisimple category. 
\end{corollary}
\begin{proof}
Any connected $N$ can be built from $I^n$
by a sequence of gluings of collared manifolds.
For example, for $n=2$,
gluing opposite edges of a square gives an annulus,
and gluing boundaries of the annulus together
gives the torus. 

By \thmref{t:excision-skein}, the corresponding category $\ZCYsk(N)$ thus can be
constructed from the  Deligne product of several copies of  $\ZCYsk(I^n) \simeq \cA$
by repeatedly applying the center construction, replacing a category $\cM$ by
$\cZ_{\ZCYsk(P\times I)}(\cM)$. Since it was shown in the proof of
\thmref{t:excision-skein} that $\ZCYsk(P\times I)$ is pivotal multifusion, it now follows
from \prpref{p:ZM_ss} that applying the center construction always gives a
finite semisimple category. Thus, $\ZCYsk(N)$ is a finite semisimple category.

\end{proof}

\newpage
\section{The Reduced Tensor Product on $\ZCYsk(\Ann)$}
\label{s:annulus}
\vspace{0.8in}

In this section, we focus on the category of boundary values
on the annulus.
The results here are reproduced from \ocite{tham_reduced}.
We will fix a premodular category $\cA$;
concatenating objects denotes tensor product (of $\cA$).

\subsection{Background}
\par \noindent

Recall that $\ZA$, the Drinfeld center of $\cA$,
is the category with

Objects: pairs $(X,\ga)$, where $X\in \cA$ and $\ga$
is a half-braiding, i.e. a natural isomorphism of functors  
$\ga_A : A \tnsr X\to X \tnsr A$, $A\in \cC$
satisfying natural compatibility conditions.

Morphisms: $\Hom_\ZA((X,\ga), (X',\ga'))=\{f\in \Hom_\cA(X,X')\st f\ga=\ga' f\}$.

We recall some well-known properties of $\ZA$.
These do not require the braiding on $\cA$,
only its spherical fusion structure.

The following is standard (see e.g. \ocite{EGNO}*{Corollary 8.20.14}):
\begin{definition}
$\ZA$ is modular, with tensor product
\begin{equation}
(X,\ga) \tnsr (Y,\mu) := (X\tnsr Y,\ga \tnsr \mu)
  \label{e:tnsr_std}
\end{equation}
where
\begin{equation}
  (\ga \tnsr \mu)_A := (\id_X \tnsr \mu_A) \circ (\ga_A \tnsr \id_Y)
  \label{e:tnsr_hfbrd_std}
\end{equation}
\label{d:tnsr_std}
and left dual given by
\begin{equation}
(X,\ga)^* = (X^*,\ga^*),
\text{ where }
(\ga^*)_A =
(\ga_{\rdual{A}})^* = 
\begin{tikzpicture}
\node[small_morphism] (ga) at (0,0) {\tiny $\ga$};
\draw (ga) .. controls +(-45:0.5cm) and +(0:0.5cm) ..
  (-0.3,-0.9) .. controls +(180:0.5cm) and +(down:1cm) .. (-1,1)
  node[above] {$A$};
\draw (ga) .. controls +(135:0.5cm) and +(180:0.5cm) ..
  (0.3,0.9) .. controls +(0:0.5cm) and +(up:1cm) .. (1,-1)
  node[below] {$A$};
\draw (ga) .. controls +(-120:0.5cm) and +(0:0.2cm) ..
  (-0.25,-0.6) .. controls +(180:0.2cm) and +(down:1cm) .. (-0.5,1)
  node[above] {$X^*$};
\draw (ga) .. controls +(60:0.5cm) and +(180:0.2cm) ..
  (0.25,0.6) .. controls +(0:0.2cm) and +(up:1cm) .. (0.5,-1)
  node[below] {$X^*$};
\end{tikzpicture}
,
\end{equation}
and similarly the right dual is
$\rdual{(X,\ga)} = (\rdual{X}, \rdual{\ga})$,
where $(\rdual{\ga})_A = \rdual{(\ga_{A^*})}$.

The pivotal structure is given by that of $\cA$.
\defend
\end{definition}

The following is taken from \ocite{kirillov-stringnet}*{Theorem 8.2}:
(these are direct specializations of results in 
\secref{s:graphical-multifusion} and \secref{s:skein_modules},
with $\cM = \cA$ as a bimodule over itself.)

\begin{proposition}\label{p:center}\par\noindent
Let $F : \ZA \to \cA$ be the natural forgetful functor
$F: (X,\ga) \mapsto X$. 
Then it has a two-sided adjoint functor
$I : \cA \to \ZA$, given by 
\begin{equation}
\label{e:induction}
I(A)=\big( \bigoplus_{i\in \Irr(\cA)} X_i \tnsr A \tnsr X_i^*, \Gamma \big)
\end{equation}
where $\Gamma$ is the half-braiding given by
\begin{equation}
\Gamma_B = 
\displaystyle{\sum_{i,j \in \Irr(\cA)}\sqrt{d_i}\sqrt{d_j}}\quad 
\begin{tikzpicture}
\draw (0,1)--(0,-1);
\node[above] at (0,1) {$A$};
\node[below] at (0,-1) {$A$};
\node[small_morphism] (L) at (-0.5,0) {$\al$};
\node[small_morphism] (R) at (0.5,0) {$\al$};
\draw (L)-- +(0,1) node[above] {$i$}; \draw (L)-- +(0,-1) node[below] {$j$}; 
\draw (R)-- +(0,1)node[above] {$i^*$}; \draw (R)-- +(0,-1)node[below] {$j^*$}; 
\draw (-1, 1) .. controls +(down:0.5cm) and +(135:0.5cm) .. (L);
\draw (1, -1) .. controls +(up:0.5cm) and +(-45:0.5cm) .. (R);
\node[above] at (-1,1) {$B$};
\node[below] at (1,-1) {$B$};
\end{tikzpicture}
.
\end{equation}
We refer the reader to \ocite{kirillov-stringnet} for more details.
\end{proposition}

\begin{proposition}
The adjoint functor $I: \cA \to \ZA$ above is dominant.
More explicitly,
the object $(X,\ga)$ is a direct summand of $I(X)$,
given by the projection $P_{(X,\ga)}$, described below:

\begin{equation}
P_{(X,\ga)} :=
\sum_{i,j \in \Irr(\cA)} \frac{\sqrt{d_i}\sqrt{d_j}}{\cD}
\begin{tikzpicture}
\tikzmath{
  \toprow = 1;
  \bottom = -1;
  \lf = -0.7;
  \rt = 0.7;
}
\node[small_morphism] (ga) at (0,0.3) {$\ga$}; 
\node[small_morphism] (ga2) at (0,-0.3) {$\ga$};
\draw (ga) -- (0,\toprow) node[above] {$X$};
\draw (ga) -- (ga2);
\draw (ga2) -- (0,\bottom) node[below] {$X$}; 
\draw[midarrow_rev] (ga) to[out=180,in=-90] (\lf,\toprow);
\node at (-0.6,0.35) {\tiny $i$};
\draw[midarrow] (ga) to[out=0,in=-90] (\rt,\toprow);
\draw[midarrow] (ga2) to[out=180,in=90] (\lf,\bottom);
\node at (-0.6,-0.25) {\tiny $j$};
\draw[midarrow_rev] (ga2) to[out=0,in=90] (\rt,\bottom);
\end{tikzpicture}
.
\end{equation}
\end{proposition}

\subsection{Reduced Tensor Product on $\ZA$}
\label{s:reduced}
\par \noindent

From here on, we will assume that $\cA$ is premodular.
We define a different monoidal structure on $\ZA$,
which we call the \emph{reduced tensor product},
denoted by $\tnsrbar$;
we emphasize that \emph{braiding} is required
to define this monoidal structure.
We will see in the coming sections that
the definitions and results proved here
have topological origin,
but we first present them purely algebraically.

We also note again that \ocite{wassermansymm}
defines a similar tensor product that coincides with ours
when $\cA$ is symmetric.

The definition of $\tnsrbar$ will be given in several steps.

\begin{definition}
\label{d:tnsr_red}
Let $X,Y$ be objects in $\cA$,
and let $\ga,\mu$ be half-braidings
on $X,Y$ respectively.
The \emph{reduced tensor product of $X$ and $Y$
with respect to $\ga,\mu$}
is defined as the image of the
projection $Q_{\ga,\mu}:X \tnsr Y \to X \tnsr Y$ defined
as follows:
\begin{equation} \label{e:Qproj}
X \tnsrproj{\ga}{\mu} Y :=
\im (Q_{\ga,\mu})
\hspace{5pt},\hspace{5pt}
Q_{\ga,\mu} :=
\frac{1}{\cD}
\begin{tikzpicture}
\draw (0.3,-0.8) -- (0.3,0.8)
  node[above] {$Y$};
\draw[regular, overline] (0,0) circle (0.5cm); 
\node[small_morphism] (n2) at (0.3,0.4) {\tiny $\mu$};
\draw[overline] (-0.3,-0.8) -- (-0.3,0.8)
  node[above] {$X$};
\node[small_morphism] (n1) at (-0.3,-0.4) {\tiny $\ga$};
\end{tikzpicture}
\end{equation}
It is easy to check that $Q^2_{\ga,\mu} = Q_{\ga,\mu}$.
\defend
\end{definition}
(Compare \ocite{wassermansymm}*{Equation (11)})

There is an accompanying definition of $\tnsrbar$
for half-braidings:
\begin{definition}
Let $\ga,\mu$ be half-braidings on $X,Y$ respectively.
Define $\ga \tnsrbar \mu$ to be natural transformation
$- \tnsr XY \to XY \tnsr -$ given by
\[
(\ga \tnsrbar \mu)_A =
\frac{1}{\cD}
\begin{tikzpicture}
\draw (0.3,-0.8) -- (0.3,0.8)
  node[above] {$Y$}
  node[pos=0,below] {$Y$};
\draw[regular, overline] (0,0) circle (0.5cm); 
\node[small_morphism] (n2) at (0.3,0.4) {\tiny $\mu$};
\draw[overline] (-0.3,-0.8) -- (-0.3,0.8)
  node[above] {$X$}
  node[pos=0,below] {$X$};
\node[small_morphism] (n1) at (-0.3,-0.4) {\tiny $\ga$};
\node[small_morphism] (a1) at (-0.5,0) {\tiny $\al$};
\draw (a1) to[out=150,in=-90] (-1,0.8)
  node[above] {$A$};
\node[small_morphism] (a2) at (0.5,0) {\tiny $\al$};
\draw (a2) to[out=-30,in=90] (1,-0.8)
  node[below] {$A$};
\end{tikzpicture}
\]
\defend
\end{definition}

In general, $\ga \tnsrbar \mu$ fails to be a
half-braiding, but only insofar as it is not an isomorphism.
Observe that $\ga \tnsrbar \mu$ commutes with
$Q_{\ga,\mu}$, so it descends to a
natural transformation
$- \tnsr (X\tnsrproj{\ga}{\mu} Y)
\to (X\tnsrproj{\ga}{\mu} Y) \tnsr -$.
This is in fact a half-braiding
on $X\tnsrproj{\ga}{\mu} Y$:

\begin{lemma}
Let $\ga,\mu$ be half-braidings on $X,Y$ respectively.
Consider the half-braidings $\ga \tnsr c$ and $c^\inv \tnsr \mu$
on $X \tnsr Y$, where recall $c$ is the braiding on $\cA$.
Observe that the projection $Q_{\ga,\mu}$ intertwines
both $\ga \tnsr c$ and $c^\inv \tnsr \mu$,
thus they restrict to half-braiding
on $X \tnsrproj{\ga}{\mu} Y$.
Then as half-braidings on
$X \tnsrproj{\ga}{\mu} Y$,
we have
\begin{equation*}
\ga \tnsrbar \mu = \ga \tnsr c = c^\inv \tnsr \mu
\end{equation*}
\label{l:hfbrd_red}
\end{lemma}

\begin{proof}
This follows from the following computation:
\begin{equation}
\label{e:hfbrd_sym}
\begin{tikzpicture}
\draw (0.3,-1) -- (0.3,1);
\draw[regular, overline] (0,0.2) circle (0.5cm); 
\node[small_morphism] (n2) at (0.3,0.6) {\tiny $\mu$};
\draw[overline] (-0.3,-1) -- (-0.3,1);
\node[small_morphism] (n1) at (-0.3,-0.2) {\tiny $\ga$};
\node[small_morphism] (n3) at (-0.3,-0.7) {\tiny $\ga$};
\draw (n3) to[out=150,in=-90] (-1,1);
\draw[overline] (n3) to[out=-20,in=90] (1,-1);
\end{tikzpicture}
=
\begin{tikzpicture}
\draw (0.3,-1) -- (0.3,1);
\draw[regular, overline] (0,0.2) circle (0.5cm); 
\node[small_morphism] (n2) at (0.3,0.6) {\tiny $\mu$};
\draw[overline] (-0.3,-1) -- (-0.3,1);
\node[small_morphism] (n1) at (-0.3,-0.2) {\tiny $\ga$};
\node[small_morphism] (n3) at (-0.3,-0.7) {\tiny $\ga$};
\draw (n3) to[out=150,in=-90] (-1,1);
\node[small_morphism] (a1) at (0.1,-0.25) {\tiny $\al$};
\draw (n3) to[out=-20,in=-90] (a1);
\node[small_morphism] (a2) at (0.5,0.2) {\tiny $\al$};
\draw (a2) to[out=-50,in=90] (1,-1);
\end{tikzpicture}
=
\begin{tikzpicture}
\draw (0.3,-1) -- (0.3,1);
\draw[regular, overline] (0,0) circle (0.5cm); 
\node[small_morphism] (n2) at (0.3,0.4) {\tiny $\mu$};
\draw[overline] (-0.3,-1) -- (-0.3,1);
\node[small_morphism] (n1) at (-0.3,-0.4) {\tiny $\ga$};
\node[small_morphism] (a1) at (-0.5,0) {\tiny $\al$};
\draw (a1) to[out=150,in=-90] (-1,1);
\node[small_morphism] (a2) at (0.5,0) {\tiny $\al$};
\draw (a2) to[out=-30,in=90] (1,-1);
\end{tikzpicture}
=
\begin{tikzpicture}
\node[small_morphism] (n3) at (0.3,0.7) {\tiny $\mu$};
\draw (n3) to[out=160,in=-90] (-1,1);
\draw (n3) to[out=-30,in=90] (1,-1);
\draw (n3) -- (0.3,1);
\draw (n3) -- (0.3,-1);
\draw[regular, overline] (0,-0.2) circle (0.5cm); 
\node[small_morphism] (n2) at (0.3,0.2) {\tiny $\mu$};
\draw[overline] (-0.3,-1) -- (-0.3,1);
\node[small_morphism] (n1) at (-0.3,-0.6) {\tiny $\ga$};
\end{tikzpicture}
\end{equation}
\end{proof}
(Compare \ocite{wassermansymm}*{Lemma 10}.)

\begin{definition}
\label{d:tnsr_red_ZA}
Let $(X,\ga),(Y,\mu)$ be objects in $\ZA$.
Their \emph{reduced tensor product}
is defined as follows:
\begin{equation*}
  (X,\ga) \tnsrbar (Y,\mu) := (X \tnsrproj{\ga}{\mu} Y, \ga \tnsrbar \mu)
\end{equation*}

For $f : (X,\ga) \to (X,\ga'), g : (Y,\mu) \to (Y',\mu')$,
their reduced tensor product is
\[
  f \tnsrbar g := Q_{\ga',\mu'} \circ (f \tnsr f') \circ Q_{\ga,\mu}
\]
or more simply, it is $f \tnsr f'$ restricted to $X \tnsrproj{\ga}{\mu} Y$.
\defend
\end{definition}

\begin{lemma}
The reduced tensor product of \defref{d:tnsr_red_ZA}
is associative. More precisely,
if $a : (X_1 \tnsr X_2) \tnsr X_3 \simeq X_1 \tnsr (X_2 \tnsr X_3)$
is the associativity constraint of $\cA$,
and $\ga_1,\ga_2,\ga_3$ are half-braidings on $X_1,X_2,X_3$
respectively,
then $a$ restricts to an isomorphism
\[
  a : (X_1 \tnsrproj{\ga_1}{\ga_2} X_2) \tnsrproj{\ga_1 \tnsrbar \ga_2}{\ga_3} X_3
  \simeq
  X_1 \tnsrproj{\ga_1}{\ga_2 \tnsrbar \ga_3} (X_2 \tnsrproj{\ga_2}{\ga_3} X_3)
\]
and hence
\[
  a : \big( (X_1,\ga_1) \tnsrbar (X_2,\ga_2) \big) \tnsrbar (X_3,\ga_3)
  \simeq
  (X_1,\ga_1) \tnsrbar \big( (X_2,\ga_2) \tnsrbar (X_3,\ga_3) \big)
\]
Furthermore, $a$ is natural in $(X_i,\ga_i)$,
and satisfies the pentagon equation.
\label{l:assoc}
\end{lemma}

\begin{proof}
Follows easily from \lemref{l:hfbrd_red}.
(Compare \ocite{wassermansymm}*{Lemma 17}.)
\end{proof}

\begin{proposition}
\label{p:reduced_pivotal}
$(\ZA, \tnsrbar)$ is a pivotal multifusion category.
More precisely,
\begin{itemize}
\item the associativity constraint is given by
  the associativity constraint of $\cA$
  (see \lemref{l:assoc});
\item the unit object, denoted $\onebar$, is $I(\one)$
  (see \prpref{p:center}),
  with left and right unit constraints given by
\begin{equation}
l_{(X,\ga)} :=
\sum_i \frac{\sqrt{d_i}}{\sqrt{\cD}}
\begin{tikzpicture}
\node[small_morphism] (ga) at (0,0) {\tiny $\ga$};
\draw (ga) -- (0,0.8) node[above] {$X$};
\draw (ga) -- (0,-0.6) node[below] {$X$};
\draw[midarrow_rev] (ga) to[out=180,in=-45] (-1,0.8);
\draw[overline] (ga) to[out=0,in=-45] (-0.5,0.8);
\node at (-0.8,0.35) {\tiny $i$};
\node[above] at (-0.75,0.8) {$I(\one)$};
\end{tikzpicture}
\hspace{10pt}
r_{(X,\ga)} := \sum_i \frac{\sqrt{d_i}}{\sqrt{\cD}}
\begin{tikzpicture}
\node[small_morphism] (ga) at (0,0) {\tiny $\ga$};
\draw (ga) -- (0,-0.6) node[below] {$X$};
\draw (ga) to[out=180,in=-135] (0.5,0.8);
\draw[midarrow] (ga) to[out=0,in=-135] (1,0.8);
\node at (0.8,0.3) {\tiny $i$};
\draw[overline] (ga) -- (0,0.8) node[above] {$X$};
\node[above] at (0.75,0.8) {$I(\one)$};
\end{tikzpicture}
\end{equation}
\item the duals are given by
\begin{equation*}
(X,\ga)^\vee := (X^*,\ga^\vee),
\hspace{10pt}
\rvee (X,\ga) := (\rdual{X}, \rvee \ga)
\end{equation*}

where

\begin{equation} \label{e:brd_duals}
\ga^\vee := 
\begin{tikzpicture}
\node[small_morphism] (ga) at (0,0) {\tiny $\ga$};
\draw (ga) -- (-0.5,0.8);
\draw[overline={1.5}]
  (ga) .. controls +(-120:0.3cm) and +(down:0.2cm) ..
  (-0.25,0) .. controls +(up:0.2cm) and +(down:0.5cm) ..
  (0,0.8);
\node[above] at (0,0.8) {$X^*$};
\draw 
  (ga) .. controls +(60:0.3cm) and +(up:0.2cm) ..
  (0.25,0) .. controls +(down:0.2cm) and +(up:0.5cm) ..
  (0,-0.8);
\draw[overline] (ga) -- (0.5,-0.8);
\end{tikzpicture}
=
\begin{tikzpicture}
\node[small_morphism] (ga) at (0,0) {\tiny $\ga *$};
\draw (ga) -- (0,-0.8);
\draw[overline={1.5}] (ga) .. controls +(120:0.5cm) and +(150:1.5cm) .. (0.5,-0.8);
\draw (ga) .. controls +(-60:0.5cm) and +(-30:1.5cm) .. (-0.5,0.8);
\draw[overline={1.5}] (ga) -- (0,0.8) node[above] {$X^*$};
\end{tikzpicture}
=
\begin{tikzpicture}
\node[small_morphism] (ga) at (0,0) {\tiny $\ga *$};
\draw (ga) -- (0,0.8);
\node[above] at (0,0.8) {$X^*$};
\draw (ga) .. controls +(-60:0.4cm) and +(-80:1.5cm) .. (-0.5,0.8);
\draw[thin_overline={1}] (ga) -- (0,-0.8);
\draw[thin_overline={1}] (ga) .. controls +(120:0.4cm) and +(100:1.5cm) .. (0.5,-0.8);
\end{tikzpicture}
, \hspace{10pt}
\rvee \ga
:=
\begin{tikzpicture}
\node[small_morphism] (ga) at (0,0) {\tiny $\ga$};
\draw
  (ga) .. controls +(-100:0.3cm) and +(down:0.3cm) ..
  (0.25,0) .. controls +(up:0.2cm) and +(down:0.4cm) ..
  (0,0.8);
\node[above] at (0,0.8) {$\rdual{X}$};
\draw[overline={1.5}] (ga) -- (0.5,-0.8);
\draw (ga) -- (-0.5,0.8);
\draw [overline={1.5}]
  (ga) .. controls +(80:0.3cm) and +(up:0.3cm) ..
  (-0.25,0) .. controls +(down:0.2cm) and +(up:0.4cm) ..
  (0,-0.8);
\end{tikzpicture}
=
\begin{tikzpicture}
\node[small_morphism] (ga) at (0,0) {\tiny $* \ga$};
\draw (ga) -- (0,0.8);
\node[above] at (0,0.8) {$\rdual{X}$};
\draw (ga) .. controls +(-60:0.4cm) and +(-80:1.5cm) .. (-0.5,0.8);
\draw[overline={1.5}] (ga) -- (0,-0.8);
\draw[overline={1.5}] (ga) .. controls +(120:0.4cm) and +(100:1.5cm) .. (0.5,-0.8);
\end{tikzpicture}
=
\begin{tikzpicture}
\node[small_morphism] (ga) at (0,0) {\tiny $* \ga$};
\draw (ga) -- (0,-0.8);
\draw[overline={1.5}] (ga) .. controls +(120:0.5cm) and +(150:1.5cm) .. (0.5,-0.8);
\draw (ga) .. controls +(-60:0.5cm) and +(-30:1.5cm) .. (-0.5,0.8);
\draw[overline={1.5}] (ga) -- (0,0.8) node[above] {$X^*$};
\end{tikzpicture}
\end{equation}

with evaluation and coevaluation maps given by
(the projections $Q_{\ga^\vee,\ga}$ are implicit)

\begin{equation}
\ev_{(X,\ga)} :=
\sum_i \frac{\sqrt{d_i}}{\sqrt{\cD}}
\begin{tikzpicture}
\node[above] at (-0.2,0.7) {$X^*$ \;};
\node[above] at (0.2,0.7) {\;\; $X$};
\draw (-0.2,0.7) .. controls +(down:1cm) and +(180:0.1cm) ..
  (0,-0.3) .. controls +(0:0.1cm) and +(down:1cm) ..  (0.2,0.7);
\node[ellipse_morphism={0}] (brd) at (0,0.2) {\tiny $\ga^\vee \tnsrbar \ga$};
\draw[midarrow={0.8}] (brd) .. controls +(170:0.8cm) and +(up:0.3cm) .. (-0.1,-0.7);
\node at (-0.3,-0.5) {\tiny $i$};
\draw (brd) .. controls +(-10:0.7cm) and +(up:0.2cm) .. (0.1,-0.7);
\node[below] at (0,-0.7) {$I(\one)$};
\end{tikzpicture}
=
\sum_i \frac{\sqrt{d_i}}{\sqrt{\cD}}
\begin{tikzpicture}
\node[small_morphism] (ga) at (0,0.1) {\tiny $\ga$};
\draw[midarrow={0.8}] (ga) .. controls +(120:0.5cm) and +(up:1cm) .. (-0.3,-0.7);
\node at (-0.4,-0.5) {\tiny $i$};
\draw[overline] (ga) .. controls +(-150:0.3cm) and +(down:1cm) .. (-0.7,0.7);
\draw (ga) .. controls +(30:0.3cm) and +(down:0.4cm) .. (0.4,0.7);
\draw[midarrow_rev={0.7}] (ga) .. controls +(-60:0.3cm) and +(up:0.3cm) .. (0,-0.7);
\node at (0.15,-0.5) {\tiny $i$};
\node[above] at (-0.7,0.7) {$X^*$};
\node[above] at (0.4,0.7) {$X$};
\node[below] at (-0.15,-0.7) {$I(\one)$};
\end{tikzpicture}
, \hspace{10pt}
\coev_{(X,\ga)} :=
\sum_i \frac{\sqrt{d_i}}{\sqrt{\cD}}
\begin{tikzpicture}
\node[small_morphism] (gav) at (0,-0.1) {\tiny $\ga$};
\draw[midarrow_rev={0.7}] (gav) .. controls +(120:0.3cm) and +(down:0.3cm) .. (0,0.7);
\node at (-0.15,0.5) {\tiny $i$};
\draw (gav) .. controls +(-150:0.3cm) and +(up:0.4cm) .. (-0.4,-0.7);
\draw (gav) .. controls +(30:0.3cm) and +(up:1cm) .. (0.7,-0.7);
\draw[overline, midarrow={0.8}] (gav) .. controls +(-60:0.5cm) and +(down:1cm) .. (0.3,0.7);
\node at (0.4,0.5) {\tiny $i$};
\node[above] at (0.15,0.7) {$I(\one)$};
\node[below] at (-0.4,-0.7) {$X$};
\node[below] at (0.7,-0.7) {$X^*$};
\end{tikzpicture}
;
\end{equation}
and similarly for right duals.
\item the pivotal structure is given by that of $\cA$.
\end{itemize}
\label{t:ZA_reduced}
\end{proposition}
(Compare \ocite{wassermansymm}*{Theorem 24})

\begin{proof}
The inverses of the unit constraints are given by reflecting the diagrams
vertically. For example, we have
\begin{equation*}
\sum_{i,j} \frac{\sqrt{d_i} \sqrt{d_j}}{\cD}
\begin{tikzpicture}
\draw (0,-1) -- (0,1);
\node[small_morphism] (ga) at (0,0.3) {\tiny $\ga$};
\draw[midarrow_rev] (ga) to[out=-170,in=-45] (-1,1);
\draw[overline] (ga) to[out=10,in=-45] (-0.5,1);
\node at (-0.8,0.4) {\tiny $i$};
\node[small_morphism] (ga2) at (0,-0.3) {\tiny $\ga$};
\draw[midarrow] (ga2) to[out=170,in=45] (-1,-1);
\draw[overline] (ga2) to[out=-10,in=45] (-0.5,-1);
\node at (-0.8,-0.4) {\tiny $j$};
\end{tikzpicture}
=
\sum_{i,j} \frac{\sqrt{d_i} \sqrt{d_j}}{\cD}
\begin{tikzpicture}
\draw (0,-1) -- (0,1);
\node[small_morphism] (a1) at (-0.8,0) {\tiny $\al$};
\node[small_morphism] (a2) at (0.4,0) {\tiny $\al$};
\draw[midarrow_rev] (a1) to[out=90,in=-90] (-1,1);
\node at (-1.1,0.5) {\tiny $i$};
\draw[midarrow] (a1) to[out=-90,in=90] (-1,-1);
\node at (-1.1,-0.5) {\tiny $j$};
\draw[midarrow={0.3},overline] (a2) to[out=90,in=-90] (-0.5,1);
\node at (0.4,0.5) {\tiny $i$};
\draw[midarrow_rev={0.3},overline] (a2) to[out=-90,in=90] (-0.5,-1);
\node at (0.4,-0.5) {\tiny $j$};
\draw[regular] (a1) -- (a2);
\node[small_morphism] at (0,0) {\tiny $\ga$};
\end{tikzpicture}
=
\sum_{i,j} \frac{\sqrt{d_i} \sqrt{d_j}}{\cD}
\begin{tikzpicture}
\node[small_morphism] (b1) at (-0.8,-0.2) {\tiny $\beta$};
\node[small_morphism] (b2) at (0.4,-0.2) {\tiny $\beta$};
\node[small_morphism] (ga) at (0,0.2) {\tiny $\ga$};
\draw[regular] (b1) to[out=160,in=180] (ga);
\draw[regular] (b2) to[out=20,in=0] (ga);
\draw[overline,midarrow_rev] (b1) to[out=90,in=-90] (-1,1);
\node at (-1.05,0.5) {\tiny $i$};
\draw[midarrow] (b1) to[out=-95,in=90] (-1,-1);
\node at (-1.05,-0.5) {\tiny $j$};
\draw (ga) -- (0,1);
\draw (ga) -- (0,-1);
\draw[overline,midarrow={0.4}] (b2) .. controls +(up:1cm) and +(down:0.5cm) .. (-0.5,1);
\node at (0.4,0.5) {\tiny $i$};
\draw[overline,midarrow_rev={0.3}] (b2) to[out=-90,in=90] (-0.5,-1);
\node at (0.4,-0.6) {\tiny $j$};
\end{tikzpicture}
=
Q_{\Gamma,\ga}
=
\id_{I(\one) \tnsrbar (X,\ga)}
\end{equation*}

The last two forms of $\ga^\vee$ are equivalent
from the following consideration:
pulling the diagonal strands across the $\ga^*$ morphism
accumulates a (left) right-hand twist,
i.e. (inverse) Drinfeld morphism,
and these cancel out by the naturality of half-braidings.

It is also easy to check that the (co)evaluation maps
described have the desired properties.
The only potentially confusing thing is that
the ``empty strand" in graphical calculus is no longer
the unit of $\ZA$,
so the unit object and unit constraints have to be explicitly included.
For example, (all dot nodes represent $\ga$,
and sum over $i,j \in \Irr(\cA)$ is implicit)

\begin{equation*}
\ev_{(X,\ga)} \circ \coev_{(X,\ga)}
=
\frac{d_i d_j}{\cD^2}
\begin{tikzpicture}
\draw (1,0) -- (1,1.2) node[above] {\small $X$};
\node[dotnode] (g1) at (1,1) {};
\node[dotnode] (g2) at (0.2,0.3) {};
\draw (0,0) to[out=90,in=-160] (g2);
\draw (g2) to[out=20,in=90] (0.5,0);
\draw[midarrow] (g1) .. controls +(150:0.2cm) and +(120:0.2cm) .. (g2);
\draw[overline] (g2) .. controls +(-60:0.2cm) and +(-120:0.2cm) ..
  (0.6,0.5) .. controls +(60:0.3cm) and +(-30:0.4cm) .. (g1);
\node at (0.5,1) {\tiny $i$};
\draw (0,0) -- (0,-1.2);
\node[dotnode] (g3) at (0,-1) {};
\node[dotnode] (g4) at (0.8,-0.3) {};
\draw[midarrow] (g3) .. controls +(-30:0.2cm) and +(-60:0.2cm) .. (g4);
\draw (g4) .. controls +(120:0.2cm) and +(60:0.2cm) ..
  (0.4,-0.5) .. controls +(-120:0.3cm) and +(150:0.4cm) .. (g3);
\draw (1,0) to[out=-90,in=20] (g4);
\draw[overline] (g4) to[out=-160,in=-90] (0.5,0);
\draw[overline] (g3) -- (0,0);
\node at (0.5,-1) {\tiny $j$};
\end{tikzpicture}
=
\frac{d_i d_j}{\cD^2}
\begin{tikzpicture}
\draw (0,-1.2) .. controls +(up:1.2cm) and +(down:1.2cm) .. (1,1.2)
  node[above] {\small $X$};
\node[dotnode] (g1) at (0.93,0.7) {};
\node[dotnode] (g2) at (0.7,0.25) {};
\node[dotnode] (g3) at (0.3,-0.25) {};
\node[dotnode] (g4) at (0.07,-0.7) {};
\draw (g2) .. controls +(160:0.5cm) and +(160:0.5cm) .. (g4);
\draw[midarrow=0.2] (g4) .. controls +(-20:0.5cm) and +(-20:0.5cm) .. (g2);
\draw[overline, midarrow=0.2] (g1) .. controls +(160:0.5cm) and +(160:0.5cm) .. (g3);
\draw[overline] (g1) .. controls +(-20:0.5cm) and +(-20:0.5cm) .. (g3);
\node at (0.8, 0.9) {\tiny $i$};
\node at (0.3, -0.9) {\tiny $j$};
\end{tikzpicture}
=
\frac{1}{\cD^2}
\begin{tikzpicture}
\draw (0,-1.2) to (0,1.2) node[above] {\small $X$};
\draw[regular] (0,0.6) circle (0.4cm);
\draw[regular] (0,-0.6) circle (0.4cm);
\node[dotnode] at (0,1) {};
\node[dotnode] at (0,0.2) {};
\node[dotnode] at (0,-0.2) {};
\node[dotnode] at (0,-1) {};
\end{tikzpicture}
=
\frac{1}{\cD^2}
\begin{tikzpicture}
\draw (0.5,-1.2) to (0.5,1.2) node[above] {\small $X$};
\draw[regular] (0,0.6) circle (0.4cm);
\draw[regular] (0,-0.6) circle (0.4cm);
\end{tikzpicture}
=
\id_{(X,\ga)}
\end{equation*}

As for the pivotal structure, it is straightforward
to check that $\ga^{\vee \vee} = \ga^{**}$.
\end{proof}

\begin{remark}
$\tnsrbar$ is not braided;
indeed, we will see later
(e.g. \eqnref{e:cAA_not_symm})
that there can be objects
$(X,\ga),(Y,\mu) \in \ZA$ such that
$(X,\ga) \tnsrbar (Y,\mu) \not\simeq 0$
and $(Y,\mu) \tnsrbar (X,\ga) \simeq 0$.
In particular,
this necessitates the ``multi" in multifusion.
\rmkend
\end{remark}

When considering the usual monoidal structure on $\ZA$,
the forgetful functor $F: \ZA \to \cA$
is naturally a tensor functor, but its adjoint,
$I : \cA \to \ZA$ is not.
With the reduced tensor product, however,
$F$ is clearly not tensor,
but $I$ is:

\begin{proposition}
\label{p:I_tensor}
The functor $I$ from \prpref{p:center} is a tensor functor.
More precisely, for $X,Y \in \cA$, define the morphism
\begin{equation*}
J_{X,Y} : I(X) \tnsrbar I(Y) \to I(X \tnsr Y)
\end{equation*}
by
\begin{equation*}
J_{X,Y} :=
\sqrt{d_i} \sqrt{d_j} \sqrt{d_k}
\begin{tikzpicture}
\node[above] at (-0.4,0.7) {\tiny $I(X)$};
\node[above] at (0.4,0.7) {\tiny $I(Y)$};
\node[below] at (0,-0.7) {\tiny $I(X\tnsr Y)$};
\draw (0.4,0.7) .. controls +(down:0.5cm) and +(up:0.5cm) .. (0.1,-0.7);
\node[small_morphism] (al1) at (-0.6,0) {\tiny $\alpha$};
\node[small_morphism] (al2) at (0.6,0) {\tiny $\alpha$};
\draw[midarrow_rev] (al1) .. controls +(100:0.3cm) and +(down:0.3cm) .. (-0.5,0.7);
\node at (-0.7,0.4) {\tiny $i$};
\draw[midarrow] (al1) .. controls +(down:0.3cm) and +(up:0.3cm) .. (-0.2,-0.7);
\node at (-0.6,-0.4) {\tiny $j$};
\draw[midarrow_rev={0.9}] (al1) .. controls +(45:0.3cm) and +(down:0.3cm) .. (0.3,0.7);
\draw[overline] (-0.4,0.7) .. controls +(down:0.5cm) and +(up:0.5cm) .. (-0.1,-0.7);
\draw[midarrow={0.8}] (al2) .. controls +(125:0.3cm) and +(down:0.3cm) .. (0.5,0.7);
\draw[midarrow_rev] (al2) .. controls +(down:0.3cm) and +(up:0.3cm) .. (0.2,-0.7);
\node at (0.6,-0.4) {\tiny $j$};
\draw[thin_overline={1.5},midarrow={0.9}]
  (al2) .. controls +(70:0.3cm) and +(down:0.3cm) .. (-0.3,0.7);
\node at (0.65,0.6) {\tiny $k$};
\end{tikzpicture}
\end{equation*}

Then
\[
(I,J) : (\cA, \tnsr) \to (\ZA,\tnsrbar)
\]
is a pivotal tensor functor.
\end{proposition}

\begin{proof}
Straightforward computations.
\end{proof}

\begin{remark} \label{r:2fold}
In \ocite{wasserman2fold},
Wasserman showed that the Drinfeld center
of a symmetric $\cA$ is a ``bilax 2-fold
tensor category", which loosely means
a category with two monoidal structures that
almost commute. More precisely,
it consists of a pair of natural transformations
\[
\eta : ((X,\ga) \tnsr (X',\ga')) \tnsrbar
  ((Y,\mu) \tnsr (Y',\mu'))
\rightleftharpoons
((X,\ga) \tnsrbar (Y,\mu)) \tnsr
  ((X',\ga') \tnsrbar (Y',\mu'))
: \zeta
\]
such that $\eta \circ \zeta = \id$,
together with several morphisms
relating the units $\one$ and $\onebar$,
and they satisfy a cocktail of compatibility axioms.
We claim that $\ZA$ also has such a
structure when $\cA$ is not symmetric.
The only difference to the structure maps is
where we have to distinguish the braiding $c$
from its inverse:
$\eta = \id_X \tnsr c^\inv_{X',Y} \tnsr \id_{Y'}$
and
$\zeta = \id_X \tnsr c_{Y,X'} \tnsr \id_{Y'}$
(with the various projections implicit).
The proof that they satisfy the various compatibility
axioms is a lengthy calculation that we do not share
here.
There is a more topological approach,
see \rmkref{r:2fold_cy}.
We do not prove this claim in this thesis,
as it will take us too far afield.
\rmkend
\end{remark}


\subsection{Horizontal Trace of $\cA$}
\label{s:ann-htr}
\par \noindent

In this section, we will consider the
\emph{horizontal trace} of $\cA$,
as defined in \ocite{BHLZ}*{Section 2.4},
which is a generalization of Ocneanu's tube algebra \ocite{Ocn}.
We follow the exposition in \secref{s:bg-htr} (see also \ocite{KT}),
where we also considered a minor generalization
to bimodule categories $\cM$;
here we only consider $\cM = \cA$.

\begin{definition}
\label{d:hA}
Consider $\cA$ as a bimodule category over itself
by left and right multiplication.
Its \emph{horizontal trace},
denoted $\htr(\cA)$ or simply $\hA$,
is the category with the following objects and morphisms:

Objects: same as in $\cA$

Morphisms: $\Hom_\hA(X,X') := \bigoplus_A \ihom{\cA}{A}(X,X')/\sim$,
where $\ihom{\cA}{A}(X,X') := \Hom_\cA(A \tnsr X, X' \tnsr A)$,
the sum is over all objects $A\in \cA$,
and $\sim$ is the equivalence relation generated by 
the following:

For any $\psi\in \ihom{\cA}{B,A}(X,X') := \Hom_\cA(B \tnsr X,X' \tnsr A)$
  and $f\in \Hom_\cC(A,B)$, we have 

\begin{equation}
\ihom{\cA}{A}(X,X') \ni
\begin{tikzpicture}
\node[small_morphism] (psi) at (0,0) {\small $\psi$};
\draw (psi) -- +(0,0.8) node[above] {$X$};
\draw (psi) -- +(0,-0.8) node[below] {$X'$};
\coordinate (L) at (-0.8,0.6);
\coordinate (R) at (0.8,-0.6);
\node[left] at (L) {$A$};
\node[right] at (R) {$A$};
\draw (L) -- (psi);
\draw (psi) -- (R);
\node[small_morphism] at (-0.48,0.36) {\tiny $f$};
\node at (-0.35,0.05) {\tiny $B$};
\end{tikzpicture}
\quad \sim \quad
\begin{tikzpicture}
\node[small_morphism] (psi) at (0,0) {\small $\psi$};
\draw (psi) -- +(0,0.8) node[above] {$X$};
\draw (psi) -- +(0,-0.8) node[below] {$X'$};
\coordinate (L) at (-0.8,0.6);
\coordinate (R) at (0.8,-0.6);
\node[left] at (L) {$B$};
\node[right] at (R) {$B$};
\draw (L) -- (psi);
\draw (psi) -- (R);
\node[small_morphism] at (0.48,-0.36) {\tiny $f$};
\node at (0.35,-0.1) {\tiny $A$};
\end{tikzpicture}
\in \ihom{\cA}{B}(X,X')
\end{equation}

In other words, $\Hom_\hA(X,X') =
\displaystyle\int^A \ihom{\cA}{A}(X,X')$.
\defend
\end{definition}

\begin{definition}
\label{d:htr}
Let $\htr: \cA \to \hA$
be the natural inclusion functor
that is identity on objects,
and on morphisms is the natural map
$\Hom_\cA(X,X') = \ihom{\cA}{\one}(X,X') \to \Hom_\hA(X,X')$.
\defend
\end{definition}

The adjective ``inclusion" further justified by the
following proposition
(along with the fact that $I$ is faithful),
which implies $\htr$ is faithful also.

\begin{proposition}{\ocite{KT}*{Theorem 3.9}}
\label{p:hA_ZA}
Let $G: \hA \to \ZA$ be defined as follows:
on objects, $G(X) = I(X)$, and on morphisms,
for $\psi \in \ihom{\cM}{A}(X,X')$,

\begin{equation}
G(\psi) = 
\sum_{i,j \Irr(\cA)} \sqrt{d_i}\sqrt{d_j}
\begin{tikzpicture}
\node[small_morphism] (psi) at (0,0) {$\psi$}; 
\node[small_morphism] (L) at (-0.7,0) {$\al$};
\node[small_morphism] (R) at (0.7,0) {$\al$};
\draw (psi)-- +(0,1) node[above] {$X$};
\draw (psi)-- +(0,-1) node[below] {$X'$}; 
\draw[midarrow_rev] (L)-- +(0,1) node[pos=0.5,left] {\tiny $i$};
\draw[midarrow] (L) -- +(0,-1) node[pos=0.5,left] {\tiny $j$}; 
\draw[midarrow] (R) -- +(0,1) node[pos=0.5,right] {\tiny $i$};
\draw[midarrow_rev] (R)-- +(0,-1) node[pos=0.5,right] {\tiny $j$};
\draw[midarrow={0.7}] (L) -- (psi);
\node at (-0.4, -0.2) {\tiny $A$};
\draw[midarrow_rev={0.7}] (R) -- (psi);
\node at (0.4, -0.2) {\tiny $A$};
\end{tikzpicture}
\end{equation}

Then the extension to the Karoubi envelope
is an equivalence of abelian categories:
\[
  \Kar(G) : \Kar(\hA) \simeq \ZA
\]
Under this equivalence, the natural functor $\htr: \cA \to \hA$
is identified with $I: \cA \to \ZA$,
i.e. we have the commutative diagram
\begin{equation}
\label{e:comm_diag}
\begin{tikzcd}
\cA \ar[r,"\htr"] \ar[d,"I"']
& \hA \ar[ld,"G"'] \ar[d,"\Kar"]
\\
\ZA
& \Kar(\hA) \ar[l,"\Kar(G)"] \ar[l,"\simeq"']
\end{tikzcd}
\end{equation}
\end{proposition}

We note that in \ocite{KT}, the proposition
would establish an equivalence
$\Kar(\htr(\cM)) \simeq \cZ(\cM)$,
where $\htr(\cM)$ is the horizontal trace
of an $\cA$-bimodule category,
and $\cZ(\cM)$ is the center of $\cM$
\ocite{GNN}*{Definition 2.1}
which is analogous to the Drinfeld center.

It is useful to construct an inverse to $\Kar(G)$:

\begin{proposition}
\label{p:G_inv}
An inverse to $\Kar(G)$ is given by:
\begin{align*}
\Kar(G)^\inv : \ZA &\simeq \Kar(\hA) \\
(X,\ga) &\mapsto (X, \hP_\ga)
\end{align*}
where
\[
\hP_\ga := \sum_{i \in \Irr(\cA)} \frac{d_i}{\cD} \ga_{X_i} =
\frac{1}{\cD}
\begin{tikzpicture}
\node[small_morphism] (ga) at (0,0) {\tiny $\ga$};
\draw (ga) -- (0,0.6) node[above] {\small $X$};
\draw (ga) -- (0,-0.6);
\draw[regular] (ga) -- +(135:0.6cm);
\draw[regular] (ga) -- +(-45:0.6cm);
\end{tikzpicture}
\]
and on morphisms, for $f\in \Hom_\ZA((X,\ga),(Y,\mu))$,
\[
f \mapsto \hP_\mu \circ f = f \circ \hP_\ga
\]
\end{proposition}

\begin{proof}
Straightforward.
\end{proof}

When $\cA$ is premodular,
in particular when $\cA$ is \emph{braided},
$\hA$ has a natural monoidal structure.
This was discussed in \ocite{KT}*{Example 8.2};
here we spell it out more explicitly.

\begin{proposition}
\label{p:hA-tnsr}
There is a tensor product on $\hA$,
denoted $\htnsr$:
on objects, it is simply the same as $\cA$,
and on morphisms,

\begin{equation*}
\begin{tikzpicture}
\node at (0,3) {$\htnsr: $};
\end{tikzpicture}
\begin{tikzpicture}
\node at (0,3) {$\ihom{\cA}{A}(X,X')$};
\node at (0,2) {$X$};
\node at (0,0) {$X'$};
\node[small_morphism] (ph) at (0,1) {$\ph$};
\draw (ph) -- (0,0.2);
\draw (ph) -- (0,1.8);
\draw (ph) -- (-0.5,1.5) node[pos=1,above left] {$A$};
\draw (ph) -- (0.5,0.5) node[pos=1,below right] {$A$};
\end{tikzpicture}
\begin{tikzpicture}
\node at (0,1) {$\boxtimes$};
\node at (0,3) {$\boxtimes$};
\end{tikzpicture}
\begin{tikzpicture}
\node at (0,3) {$\ihom{\cA}{B}(Y,Y')$};
\node at (0,2) {$Y$};
\node at (0,0) {$Y'$};
\node[small_morphism] (ph) at (0,1) {\small $\psi$};
\draw (ph) -- (0,0.2);
\draw (ph) -- (0,1.8);
\draw (ph) -- (-0.5,1.5) node[pos=1,above left] {$B$};
\draw (ph) -- (0.5,0.5) node[pos=1,below right] {$B$};
\end{tikzpicture}
\begin{tikzpicture}
\node at (0,1) {$\mapsto$};
\node at (0,3) {$\longrightarrow$};
\end{tikzpicture}
\begin{tikzpicture}
\node at (0,3) {$\ihom{\cA}{AB}(XY, X'Y'$)};
\node at (0,2) {$XY$};
\node at (0,0) {$X'Y'$};
\node[small_morphism] (ph1) at (-0.2,0.9) {\tiny $\ph$};
\node[small_morphism] (ph2) at (0.2,1.1) {\tiny $\psi$};
\draw (ph2) -- (0.2,1.8);
\draw (ph2) -- (0.2,0.2);
\draw (ph2) -- +(150:1cm);
\draw (ph2) -- +(-30:0.8cm);
\draw[overline] (ph1) -- (-0.2,1.8);
\draw (ph1) -- (-0.2,0.2);
\draw (ph1) -- +(150:0.8cm);
\draw[overline] (ph1) -- +(-30:1cm);
\node at (-1.3,1.7) {$AB$};
\node at (1.3,0.3) {$AB$};
\end{tikzpicture}
\end{equation*}

The unit object is the same as that of $\cA$.

Furthermore, this tensor product structure
is compatible with the rigid and pivotal structures
of $\cA$.

In other words, this is an
extension the tensor product on $\cA$,
so that the inclusion functor $\htr: \cA \to \hA$
is a pivotal tensor functor.
\end{proposition}

\begin{proof}
Straightforward.
\end{proof}

It is useful to work out the left dual of a morphism
(right dual being similar):
\begin{equation}
\label{e:hA_dual}
\begin{tikzpicture}
\node at (0,1.5) {$\ihom{\cA}{A}(X,X')$};
\node[small_morphism] (ph) at (0,0) {\small $\ph$};
\draw (ph) -- (0,0.6) node[above] {$X$};
\draw (ph) -- (0,-0.6) node[below] {$X'$};
\draw (ph) -- (-0.6,0.3) node[left] {$A$};
\draw (ph) -- (0.6,-0.3) node[right] {$A$};
\end{tikzpicture}
\begin{tikzpicture}
\node at (0,1.5) {$\to$};
\node at (0,0) {$\mapsto$};
\end{tikzpicture}
\begin{tikzpicture}
\node at (0,1.5) {$\ihom{\cA}{A}(X'^*,X^*)$};
\node[small_morphism] (ph) at (0,0) {\tiny $\ph$};
\draw (ph) .. controls +(up:0.3cm) and +(up:0.3cm) ..
  (0.2,0) .. controls +(down:0.2cm) and +(up:0.2cm) ..
  (0,-0.6)
  node[below] {$X^*$};
\draw (ph) -- (-0.6,0.3) node[left] {$A$};
\draw[overline={1.5}] (ph) -- (0.6,-0.3)
  node[right] {$A$};
\draw[overline={1.25}]
  (ph) .. controls +(down:0.3cm) and +(down:0.3cm) ..
  (-0.2,0) .. controls +(up:0.2cm) and +(down:0.2cm) ..
  (0,0.6)
  node[above] {$X'^*$};
\end{tikzpicture}
\end{equation}

In particular, when $X' = X$, and $\ph=\ga$ is a half-braiding on $X$,
the right side is nothing but $\ga^\vee$.

The following proposition is an upgrade of
\prpref{p:hA_ZA}:

\begin{proposition}
\label{p:hA_ZA_tnsr}
When $\ZA$ is endowed with the reduced tensor product,
the equivalence in \prpref{p:hA_ZA} is an equivalence
of pivotal multifusion categories.
More precisely, let $J$ be the tensor structure on $I$
from \prpref{p:I_tensor}.
Then
\[
(G,J) : (\hA, \htnsr) \to (\ZA, \tnsrbar)
\]
is a pivotal tensor functor,
and hence its completion to $\Kar(\hA)$
is a pivotal tensor equivalence.
\end{proposition}

\begin{proof}
As \prpref{p:I_tensor} takes care of objects,
it remains to check naturality of $J$ with respect to morphisms
of $\hA$. This is easy to do:
for example, for $\ph \in \ihom{\cA}{A}(X,X')$,
we see that
$J_{X',Y} \circ (I(\ph) \tnsrbar I(\id_Y))
=
I(\ph \tnsr \id_Y) \circ J_{X,Y}$ by

\begin{equation*}
\sqrt{d_i} \sqrt{d_j} \sqrt{d_k} d_l
\begin{tikzpicture}
\tikzmath{ \toprow = 1; \bottom = -1; }
\node[above] at (0,\toprow) {\tiny $I(X) \tnsrbar I(Y)$};
\node[below] at (0,\bottom) {\tiny $I(X'\htnsr Y)$};
\node[small_morphism] (be1) at (-0.6,0) {\tiny $\beta$};
\node[small_morphism] (be2) at (0.6,0) {\tiny $\beta$};
\node[small_morphism] (al1) at (-0.7,0.7) {\tiny $\al$};
\node[small_morphism] (al2) at (-0.1,0.7) {\tiny $\al$};
\node[small_morphism] (ph) at (-0.4,0.7) {\tiny $\ph$};
\draw (0.4,\toprow) .. controls +(down:0.5cm) and +(up:0.5cm) .. (0.1,\bottom);
\draw (be1) .. controls +(100:0.3cm) and +(down:0.3cm) .. (al1);
\node at (-0.8,0.4) {\tiny $l$};
\draw (be1) .. controls +(down:0.3cm) and +(up:0.3cm) .. (-0.3,\bottom);
\node at (-0.6,-0.5) {\tiny $j$};
\draw (be1) .. controls +(45:0.3cm) and +(down:0.3cm) .. (0.3,\toprow);
\draw (ph) -- (-0.4,\toprow);
\draw[thin_overline={1.5}] (ph) .. controls +(down:0.5cm) and +(up:0.5cm) .. (-0.1,\bottom);
\draw (be2) .. controls +(125:0.3cm) and +(down:0.3cm) .. (0.5,\toprow);
\draw (be2) .. controls +(down:0.3cm) and +(up:0.3cm) .. (0.3,\bottom);
\draw[thin_overline={1.5}]
  (be2) .. controls +(70:0.3cm) and +(down:0.3cm) .. (al2);
\node at (0.6,0.6) {\tiny $k$};
\draw (al1) -- (-0.7,\toprow);
\node at (-0.8,0.9) {\tiny $i$};
\draw (al2) -- (-0.1,\toprow);
\draw (al1) -- (ph);
\draw (ph) -- (al2);
\end{tikzpicture}
=
\sqrt{d_i} \sqrt{d_j} \sqrt{d_k}
\begin{tikzpicture}
\tikzmath{ \toprow = 1; \bottom = -1; }
\node[above] at (0,\toprow) {\tiny $I(X) \tnsrbar I(Y)$};
\node[below] at (0,\bottom) {\tiny $I(X'\htnsr Y)$};
\node[small_morphism] (ga1) at (-0.6,0) {\tiny $\ga$};
\node[small_morphism] (ga2) at (0.6,0) {\tiny $\ga$};
\node[small_morphism] (ph) at (-0.2,0) {\tiny $\ph$};
\draw (0.4,\toprow) .. controls +(down:0.5cm) and +(up:0.5cm) .. (0.1,\bottom);
\draw (ga1) .. controls +(100:0.3cm) and +(down:0.3cm) .. (-0.6,\toprow);
\node at (-0.75,0.5) {\tiny $i$};
\draw (ga1) .. controls +(down:0.3cm) and +(up:0.3cm) .. (-0.3,\bottom);
\node at (-0.6,-0.5) {\tiny $j$};
\draw (ga1) .. controls +(45:0.3cm) and +(down:0.3cm) .. (0.3,\toprow);
\draw[thin_overline={1.5}] (ph) .. controls +(100:0.3cm) and +(down:0.5cm) .. (-0.4,\toprow);
\draw (ph) .. controls +(-80:0.5cm) and +(up:0.5cm) .. (-0.1,\bottom);
\draw (ga2) .. controls +(125:0.3cm) and +(down:0.3cm) .. (0.5,\toprow);
\draw (ga2) .. controls +(down:0.3cm) and +(up:0.3cm) .. (0.3,\bottom);
\draw[thin_overline={1.5}]
  (ga2) .. controls +(70:0.3cm) and +(down:0.3cm) .. (-0.2,\toprow);
\node at (0.6,0.6) {\tiny $k$};
\draw (ga1) -- (ph);
\draw[thin_overline={1.5}] (ph) -- (ga2);
\node at (0.1,0.1) {\tiny $A$};
\end{tikzpicture}
=
\sqrt{d_i} \sqrt{d_j} \sqrt{d_k} d_l
\begin{tikzpicture}
\tikzmath{ \toprow = 1; \bottom = -1; }
\node[above] at (0,\toprow) {\tiny $I(X) \tnsrbar I(Y)$};
\node[below] at (0,\bottom) {\tiny $I(X'\htnsr Y)$};
\node[small_morphism] (be1) at (-0.6,0) {\tiny $\beta$};
\node[small_morphism] (be2) at (0.6,0) {\tiny $\beta$};
\node[small_morphism] (al1) at (-0.4,-0.7) {\tiny $\al$};
\node[small_morphism] (al2) at (0.4,-0.7) {\tiny $\al$};
\node[small_morphism] (ph) at (-0.1,-0.7) {\tiny $\ph$};
\draw (0.4,\toprow) .. controls +(down:0.5cm) and +(up:0.5cm) .. (0.1,\bottom);
\draw (be1) .. controls +(100:0.3cm) and +(down:0.3cm) .. (-0.6,\toprow);
\node at (-0.75,0.5) {\tiny $i$};
\draw (be1) .. controls +(down:0.3cm) and +(up:0.3cm) .. (al1);
\node at (-0.6,-0.45) {\tiny $l$};
\draw (be1) .. controls +(45:0.3cm) and +(down:0.3cm) .. (0.3,\toprow);
\draw[thin_overline={1.5}] (ph) .. controls +(up:0.3cm) and +(down:0.5cm) .. (-0.4,\toprow);
\draw (ph) -- (-0.1,\bottom);
\draw (be2) .. controls +(125:0.3cm) and +(down:0.3cm) .. (0.5,\toprow);
\draw (be2) .. controls +(down:0.3cm) and +(up:0.3cm) .. (al2);
\draw[thin_overline={1.5}]
  (be2) .. controls +(70:0.3cm) and +(down:0.3cm) .. (-0.2,\toprow);
\node at (0.6,0.6) {\tiny $k$};
\draw (al1) -- (-0.4,\bottom);
\node at (-0.5,-0.9) {\tiny $j$};
\draw (al2) -- (0.4,\bottom);
\draw (al1) -- (ph);
\draw[thin_overline={1.5}] (ph) -- (al2);
\end{tikzpicture}
\end{equation*}
using \eqnref{e:summation_convention} and \lemref{l:summation}.
\end{proof}

\begin{proposition}
The inverse functor $\Kar(G)^\inv$ defined in
\prpref{p:G_inv}
is naturally a tensor functor.
\end{proposition}
\begin{proof}
For the tensor structure on $\Kar(G)^\inv$,
there is a natural isomorphism
$(X,\hP_\ga) \tnsr (Y,\hP_\mu)
= (XY,\hP_\ga \tnsr \hP_\mu)
\simeq (X\tnsrproj{\ga}{\mu} Y,\hP_{\ga\tnsrbar\mu})$,
which is given by the natural projection
$Q_{\ga,\mu} : XY \to X\tnsrproj{\ga}{\mu} Y$.
This follows from two simple observations:
\[
\hP_\ga \tnsr \hP_\mu = \hP_{\ga\tnsrbar\mu}
\in \End_\hA(XY)
\]
(using \lemref{l:summation}),
which implies $(XY,\hP_\ga \tnsr \hP_\mu)
= (XY,\hP_{\ga\tnsrbar\mu})$,
and
\[
\hP_{\ga\tnsrbar\mu} =
\hP_\ga \circ Q_{\ga,\mu} =
\hP_\ga \circ Q_{\ga,\mu} \circ Q_{\ga,\mu} =
\hP_{\ga\tnsrbar\mu} \circ Q_{\ga,\mu}
\in \End_\hA(XY)
\]
(see \eqnref{e:hfbrd_sym})
which implies $Q_{\ga,\mu} : (XY,\hP_{\ga\tnsrbar\mu})
\to (X\tnsrproj{\ga}{,\mu} Y,\hP_{\ga\tnsrbar\mu})$
is an isomorphism.
It is easy to check that this satisfies
the hexagon axiom for tensor structure.

The isomorphism $\Kar(G)^\inv (\onebar)
\simeq \one$ is given by
\begin{equation}
\label{e:unit_isom_G}
\sum_{i\in \Irr(\cA)}
\frac{\sqrt{d_i}}{\sqrt{\cD}}
\begin{tikzpicture}
\draw (-0.2,0.5) .. controls +(down:0.5cm) and +(-30:0.5cm) .. (-0.6,0.3)
  node[pos=0,above] {\small $X_i$};
\draw (0.2,0.5) .. controls +(down:0.5cm) and +(150:0.5cm) .. (0.6,-0.3)
  node[pos=0,above] {\small $X_i^*$};
 \node at (0,-0.5) {\small $\one$};
\end{tikzpicture}
\end{equation}
\end{proof}

\subsection{Equivalence to Stacking Product}
\label{s:ann-cy}
\par \noindent

As we have seen in the proof of \thmref{t:excision-skein},
$\ZCYsk(\Ann) \simeq \ZA$ as abelian categories.
Here we show that the reduced tensor product on $\ZA$
is equivalent to the stacking product on $\ZCYsk(\Ann)$.
The following is a visual summary of results in \secref{s:skprop-stacking}
specialized to $N = \Ann$:

\begin{proposition}[\ocite{KT}*{Proposition 6.7}, \ocite{tham_reduced}]
The stacking tensor product on $\ZCYsk(\Ann)$,
as defined in \secref{s:skprop-stacking},
denoted by $\tnsrst$,
makes $\hCYA$ a pivotal multifusion category as follows:
\begin{itemize}
\item For an object $\VV = (B, \{V_b\})$,
its left dual is the object $\VV^\vee
:= (\theta(B), \{V_b^*\})$,
where $\theta$ is the operation on $\Ann$ that flips $(0,1)$,
and has the following evaluation and coevaluation maps:
\[
\begin{tikzpicture}
\node (a1) at (0,0) {$\VV^\vee \tnsrst \VV$};
\node (a2) at (0,-1) {$\onest$};
\node (a3) at (0, -2) {$\VV \tnsrst \VV^\vee$};
\draw[->] (a1) -- (a2) node[left,pos=0.5] {$\ev_\VV$};
\draw[->] (a2) -- (a3) node[left,pos=0.5] {$\coev_\VV$};
\end{tikzpicture}
\begin{tikzpicture}
\draw (0,0) ellipse (0.5cm and 0.16cm);
\node[dotnode] (Vb) at (-0.75, 0) {};
\node[dotnode] (Vbp) at (-1.25, 0) {};
\node[dotnode] (Vb1) at (0.75, -0.1) {};
\node[dotnode] (Vb1p) at (1, -0.2) {};
\draw (Vb) .. controls +(down:0.8cm) and +(down:0.8cm) .. (Vbp);
\draw (Vb1) .. controls +(down:0.5cm) and +(down:0.5cm) .. (Vb1p);
\draw (-0.5, 0) -- (-0.5, -1);
\draw (0.5, 0) -- (0.5, -1);
\draw[overline={1}, gray] (0,0) ellipse (1cm and 0.33cm);
\draw[overline={1}] (0,0) ellipse (1.5cm and 0.5cm);
\draw (-1.5, 0) -- (-1.5, -1);
\draw (1.5, 0) -- (1.5, -1);
\node at (-0.7, 0.2) {\tiny $V_b$};
\node at (-1.22,0.2) {\tiny $V_b^*$};
\node at (0.75, 0.1) {\tiny $V_{b'}$};
\node at (1.22, -0.05) {\tiny $V_{b'}^*$};
\draw (0,-2) ellipse (0.5cm and 0.16cm);
\node[dotnode] (Vbp) at (-0.75, -2) {};
\node[dotnode] (Vb) at (-1.25, -2) {};
\node[dotnode] (Vb1p) at (0.75, -2.1) {};
\node[dotnode] (Vb1) at (1, -2.2) {};
\draw[gray] (0,-2) ellipse (1cm and 0.33cm);
\draw (0,-2) ellipse (1.5cm and 0.5cm);
\draw[overline={1}] (Vb) .. controls +(up:0.8cm) and +(up:0.8cm) .. (Vbp);
\draw[overline={1}] (Vb1) .. controls +(up:0.5cm) and +(up:0.5cm) .. (Vb1p);
\draw (-0.5, -1) -- (-0.5, -2);
\draw (0.5, -1) -- (0.5, -2);
\draw (-1.5, -1) -- (-1.5, -2);
\draw (1.5, -1) -- (1.5, -2);
\node at (-0.7, -2.2) {\tiny $V_b^*$};
\node at (-1.22,-2.2) {\tiny $V_b$};
\node at (0.55, -2.3) {\tiny $V_{b'}^*$};
\node at (1.12, -2.4) {\tiny $V_{b'}$};
\end{tikzpicture}
\]
(The gray lines indicate $S^1 \times \{1/2\}$
which separates $\VV^\vee$ from $\VV$
in the tensor product, and play no role in defining
these morphisms.)
We think of the outside half of the annulus as the left side
in the tensor product.
The right dual is $\rvee{\VV}:= (\theta(B), \{\rdual{V_b}\})$,
with essentially the same (co)evaluation maps.

\item The pivotal map
$\delta_\VV : \VV \to \VV^{\vee\vee}$
is simply the graph with vertical strands running down,
and one node on each strand labeled with $\delta_{V_b}$.
\end{itemize}

$\CYA$, as the Karoubi envelope of $\hCYA$,
naturally inherits these structures.
\end{proposition}

It is not hard to see that the left dual of a morphism
is obtained by turning the solid annulus ``inside-out";
for example,

\begin{equation*}
\begin{tikzpicture}
\tikzmath{
  \toprow = 0.8;
  \bottom = -0.8;
  \midrow = 0.2;
  \bigx = 1; 
  \bigy = 0.2;
  \smlx = 0.5;
  \smly = 0.1;
  \medx = 0.75;
  \medy = 0.15;
}
\draw (0,\bottom) ellipse (\bigx cm and \bigy cm);
\draw (\medx,\midrow) arc(0:180:\medx cm and \medy cm);
\draw[overline={1.5}] (-0.5,\bottom) -- (-0.5,\toprow);
\draw[overline={1.5}] (0.5,\bottom) -- (0.5,\toprow);
\draw (0,\bottom) ellipse (\smlx cm and \smly cm); 
\draw[gray,overline={1},->] (-0.4,-0.55)
  to[out=60,in=-60] (-0.4,0.55);
\draw[gray,overline={1},->] (0.4,-0.55)
  to[out=120,in=-120] (0.4,0.55);
\draw[gray,->] (-1.1,0.55)
  to[out=-120,in=120] (-1.1,-0.55);
\draw[gray,->] (1.1,0.55)
  to[out=-60,in=60] (1.1,-0.55);
\node[dotnode] (top) at (-\medx,\toprow) {}; 
\node[dotnode, label={[shift={(0.,-0.6)}] \small $Y'$}] (bottom) at (-\medx,\bottom) {};
\draw[overline={1.5}, midarrow_rev]
  (\medx,\midrow) arc(0:-180:\medx cm and \medy cm)
  node[pos=0.5, below] {\tiny $A$};
\draw[overline={1}] (top) -- (bottom);
\node[small_morphism] at (-\medx,\midrow) {\tiny $\ph$};
\draw[overline={1}] (0,\toprow) ellipse (\bigx cm and \bigy cm);
\draw (0,\toprow) ellipse (\smlx cm and \smly cm);
\draw (-1,\bottom) -- (-1,\toprow);
\draw (1,\bottom) -- (1,\toprow);
\node[label={\small $Y^*$}] at (-\medx,\toprow) {};
\end{tikzpicture}
\mapsto
\begin{tikzpicture}
\tikzmath{
  \toprow = 0.8;
  \bottom = -0.8;
  \midrow = -0.2;
  \bigx = 1; 
  \bigy = 0.2;
  \smlx = 0.5;
  \smly = 0.1;
  \medx = 0.75;
  \medy = 0.15;
}
\draw (0,\bottom) ellipse (\bigx cm and \bigy cm);
\draw (\medx,\midrow) arc(0:180:\medx cm and \medy cm);
\draw[overline={1.5}] (-0.5,\bottom) -- (-0.5,\toprow);
\draw[overline={1.5}] (0.5,\bottom) -- (0.5,\toprow);
\draw (0,\bottom) ellipse (\smlx cm and \smly cm); 
\node[dotnode] (top) at (-\medx,\toprow) {}; 
\node[dotnode, label={[shift={(0,-0.6)}] \small $Y^*$}] (bottom) at (-\medx,\bottom) {};
\draw[overline={1.5}, midarrow_rev]
  (\medx,\midrow) arc(0:-180:\medx cm and \medy cm)
  node[pos=0.5, above] {\tiny $A$};
\draw[overline={1}] (top) -- (bottom);
\node[small_morphism] at (-\medx,\midrow)
  {\tiny $\rotatebox[origin=c]{180}{\(\ph\)}$};
\draw[overline={1}] (0,\toprow) ellipse (\bigx cm and \bigy cm);
\draw (0,\toprow) ellipse (\smlx cm and \smly cm);
\draw (-1,\bottom) -- (-1,\toprow);
\draw (1,\bottom) -- (1,\toprow);
\node[label={\small $Y'$}] at (-\medx,\toprow) {};
\end{tikzpicture}
=
\begin{tikzpicture}
\tikzmath{
  \toprow = 0.8;
  \bottom = -0.8;
  \midrow = 0;
  \bigx = 1; 
  \bigy = 0.2;
  \smlx = 0.5;
  \smly = 0.1;
  \medx = 0.75;
  \medy = 0.15;
}
\draw (0,\bottom) ellipse (\bigx cm and \bigy cm);
\draw (\medx,\midrow) arc(0:180:\medx cm and \medy cm);
\draw[overline={1.5}] (-0.5,\bottom) -- (-0.5,\toprow);
\draw[overline={1.5}] (0.5,\bottom) -- (0.5,\toprow);
\draw (0,\bottom) ellipse (\smlx cm and \smly cm); 
\node[dotnode] (top) at (-\medx,\toprow) {}; 
\node[dotnode, label={[shift={(0.,-0.6)}] \small $Y^*$}] (bottom) at (-\medx,\bottom) {};
\draw[overline={1}]
  (top) .. controls +(down:0.5cm) and +(up:0.3cm) ..
  (-0.95,\midrow) .. controls +(down:0.3cm) and +(down:0.3cm) ..
  (-\medx,\midrow) .. controls +(up:0.3cm) and +(up:0.3cm) ..
  (-0.55,\midrow) .. controls +(down:0.3cm) and +(up:0.5cm) .. (bottom);
\draw[overline={1.5}, midarrow_rev]
  (\medx,\midrow) arc(0:-180:\medx cm and \medy cm)
  node[pos=0.5, below] {\tiny $A$};
\node[small_morphism] at (-\medx,\midrow) {\tiny $\ph$};
\draw[overline={1}] (0,\toprow) ellipse (\bigx cm and \bigy cm);
\draw (0,\toprow) ellipse (\smlx cm and \smly cm);
\draw (-1,\bottom) -- (-1,\toprow);
\draw (1,\bottom) -- (1,\toprow);
\node[label={\small $Y'$}] at (-\medx,\toprow) {};
\end{tikzpicture}
\end{equation*}

The gray arrows indicate the ``inside-out" operation.
\footnote{Imagine pulling your hand out of the sleeve of
a tight sweater - the second diagram is
inside-out the same way the sleeve is.}
Note the upside-down `$\ph$' in the second diagram;
the last diagram can be turned into the second
by pulling on the upward and downward strands,
forcing the $\ph$ node to turn upside-down.

The last diagram is reminiscent of duals in $\hA$
(see \eqnref{e:hA_dual}).
In particular, when $Y' = Y$,
and $\ph$ is a half-braiding on $Y$,
the extra bending in the last diagram can be incorporated into
the node to become $\ph^\vee$ (see
\prpref{p:reduced_pivotal}).

In \ocite{KT}, in Example 8.2, we provided an explicit equivalence
$H = H_p:\hA \simeq \hCYA$, where $p\in \Ann$, given as follows:
\footnote{
$H_p$ is dependent on the choice of a point $p\in \Ann$,
but all $H_p$ are naturally isomorphic (by non-unique
natural isomorphism).
One can consider the full subcategory with objects
of the form $(\{p\},\{X\})$. This is strictly speaking not
a tensor subcategory, since the tensor product of such objects
would have two arcs.
However, one can put a different tensor product,
$(\{p\},\{X\}) \tnsrst' (\{p\},\{Y\}) = (\{p\},\{XY\})$,
and the inclusion can be made a pivotal tensor equivalence.
}

\begin{equation}
\label{e:def_H}
\begin{tikzpicture}
\node at (0,2) {$\hA$};
\node at (0,1.2) {$Y$};
\node at (0,-1.2) {$Y'$};
\node[small_morphism] (ph) at (0,0) {$\vphi$};
\draw (ph) -- (0,-1);
\draw (ph) -- (0,1);
\draw (ph) -- (-0.5,0.5) node[left] {$A$};
\draw (ph) -- (0.5,-0.5) node[right] {$A$};
\end{tikzpicture}
\begin{tikzpicture}
\node at (0,0) {$\mapsto$};
\node at (0,-1.2) {$\mapsto$};
\node at (0,1.2) {$\mapsto$};
\node at (0,2) {$\overset{H}{\simeq}$};
\end{tikzpicture}
\begin{tikzpicture}
\node at (0,2) {$\hCYA$};
\tikzmath{
  \toprow = 0.7;
  \bottom = -0.7;
  \midrow = 0;
  \bigx = 1; 
  \bigy = 0.2;
  \smlx = 0.5;
  \smly = 0.1;
  \medx = 0.75;
  \medy = 0.15;
}
\draw (0,1.2) ellipse (\bigx cm and \bigy cm);
\draw (0,1.2) ellipse (\smlx cm and \smly cm);
\node[dotnode, label={[shift={(0,0.1)}] \small $Y$}] at (-\medx,1.2) {};
\draw (0,-1.2) ellipse (\bigx cm and \bigy cm);
\draw (0,-1.2) ellipse (\smlx cm and \smly cm);
\node[dotnode, label={[shift={(0,-0.6)}] \small $Y'$}] at (-\medx,-1.2) {};
\draw (0,\bottom) ellipse (\bigx cm and \bigy cm);
\draw (\medx,\midrow) arc(0:180:\medx cm and \medy cm);
\draw[overline={1.5}] (-0.5,\bottom) -- (-0.5,\toprow);
\draw[overline={1.5}] (0.5,\bottom) -- (0.5,\toprow);
\draw (0,\bottom) ellipse (\smlx cm and \smly cm); 
\node[dotnode] (top) at (-\medx,\toprow) {}; 
\node[dotnode] (bottom) at (-\medx,\bottom) {};
\draw[overline={1.5}, midarrow_rev]
  (\medx,\midrow) arc(0:-180:\medx cm and \medy cm)
  node[pos=0.5, below] {\tiny $A$};
\draw[overline={1},midarrow={0.3},midarrow={0.9}]
  (top) -- (bottom);
\node[small_morphism] at (-\medx,\midrow) {\tiny $\ph$};
\draw[overline={1}] (0,\toprow) ellipse (\bigx cm and \bigy cm);
\draw (0,\toprow) ellipse (\smlx cm and \smly cm);
\draw (-1,\bottom) -- (-1,\toprow);
\draw (1,\bottom) -- (1,\toprow);
\node[style={fill=white,circle,inner sep=0pt,
  minimum size=0pt, scale=0.8}] at (-0.95,0.4) {\tiny $Y$};
\node[style={fill=white,circle,inner sep=0pt,
  minimum size=0pt, scale=0.8}] at (-0.95,-0.4) {\tiny $Y'$};
\node at (-0.6,1.2) {\tiny $p$};
\node at (-0.6,-1.2) {\tiny $p$};
\end{tikzpicture}
\end{equation}

This is an equivalence by \thmref{t:hat_excision},
and it is not hard to see that,
endowing $\hA$ with the tensor product $\htnsr$,
this equivalence is in fact tensor and pivotal.
For example, for $X,Y \in \Obj \hA$,
the tensor product in $\hCYA$, $H(X) \tnsrst H(Y)$,
is an object with two arcs, labeled with $X$ and $Y$.
The tensor structure on $H$ would be a trivalent
graph connecting $H(X) \tnsrst H(Y)$ to $H(XY)$,
with the unique vertex labeled by
$\id_{XY}$
(which is naturally identified with
$\coev \in \Hom_\cA(\one,XY(XY)^*)$
).
The unit object $\onest$ in $\hCYA$, i.e. the empty configuration,
is isomorphic to the object $H(\one) = (\{p\},\{\one\})$.
Hence we have the following:

\begin{theorem}
\label{t:ann-main}
We have the following commutative diagram,
where all functors are pivotal tensor functors:
\begin{equation*}
\begin{tikzcd}
(\cA,\tnsr) \ar[r,"\htr"] \ar[d,"I"']
& (\hA,\htnsr) \ar[ld,"G"'] \ar[d,"\Kar"]
    \ar[r,"H"] \ar[r,"\simeq"']
& (\hCYA,\tnsrst) \ar[d, "\Kar"]
\\
(\ZA,\tnsrbar)
& (\Kar(\hA),\htnsr) \ar[l,"\Kar(G)"] \ar[l,"\simeq"']
    \ar[r, "\Kar(H)"'] \ar[r,"\simeq"]
& (\CYA,\tnsrst)
\end{tikzcd}
\end{equation*}
In other words, the reduced tensor product $\tnsrbar$ on $\ZA$
encodes the stacking tensor product on $\CYA$.
\end{theorem}

Together with \prpref{p:G_inv},
the pivotal tensor equivalence $(\ZA,\tnsrbar) \simeq \CYA$
is given by

\begin{align}
\Kar(H) \circ \Kar(G)^\inv :
  (\ZA,\tnsrbar) &\simeq \CYA
\\
\begin{tikzpicture}
\node (X) at (0,0.65) {$(X,\ga)$};
\node (Y) at (0,-0.65) {$(Y,\mu)$};
\draw[midarrow={0.55}] (X) -- (Y) node[pos=0.5,right] {$f$};
\end{tikzpicture}
&
\begin{tikzpicture}
\node at (0,0.65) {$\mapsto$};
\node at (0,0) {$\mapsto$};
\node at (0,-0.65) {$\mapsto$};
\end{tikzpicture}
\begin{tikzpicture}
\begin{scope}[shift={(0,-1)}]
\draw (0,0) ellipse (1cm and 0.2cm);
\draw[regular] (0.75,1.5) arc(0:180:0.75cm and 0.15cm);
\draw[regular] (0.75,0.5) arc(0:180:0.75cm and 0.15cm);
\draw[overline={1}] (-0.5,0) -- (-0.5,2);
\draw[overline={1}] (0.5,0) -- (0.5,2);
\draw (0,0) ellipse (0.5cm and 0.1cm); 
\node[dotnode] (top) at (-0.75,2) {}; 
\node[dotnode] (bottom) at (-0.75,0) {};
\node at (-0.75,-0.3) {\small $Y$};
\draw[regular, overline={1.5}]
  (0.75,1.5) arc(0:-180:0.75cm and 0.15cm);
\draw[regular, overline={1.5}]
  (0.75,0.5) arc(0:-180:0.75cm and 0.15cm);
\draw[overline={1}] (top) -- (bottom);
\node[small_morphism] at (-0.75,1.5) {\tiny $\ga$};
\node[small_morphism] at (-0.75,1) {\tiny $f$};
\node[small_morphism] at (-0.75,0.5) {\tiny $\mu$};
\draw[overline={1}] (0,2) ellipse (1cm and 0.2cm);
\draw (0,2) ellipse (0.5cm and 0.1cm);
\draw (-1,0) -- (-1,2);
\draw (1,0) -- (1,2);
\node at (-0.75,2.3) {\small $X$};
\draw[decoration={brace,amplitude=3pt},decorate] (-1.2,1.3) -- (-1.2,2);
\node at (-1.55,1.65) {im};
\draw[decoration={brace,amplitude=3pt},decorate] (-1.2,0) -- (-1.2,0.7);
\node at (-1.55,0.35) {im};
\end{scope}
\end{tikzpicture}
\end{align}

Once again this functor doesn't send unit to unit;
the isomorphism is given by (compare \eqnref{e:unit_isom_G})

\[
\sum_{i\in \Irr(\cA)}
\frac{\sqrt{d_i}}{\sqrt{\cD}}
\begin{tikzpicture}
\begin{scope}[shift={(0,-0.5)}]
\draw (0,0) ellipse (1cm and 0.2cm);
\draw (0.75,0.5) arc(0:180:0.75cm and 0.15cm);
\draw[overline={1}] (-0.5,0) -- (-0.5,1);
\draw[overline={1}] (0.5,0) -- (0.5,1);
\draw (0,0) ellipse (0.5cm and 0.1cm); 
\node[dotnode] (n1) at (-0.75,1) {}; 
\draw[overline={1.5}, midarrow={0.6}]
  (0.75,0.5) arc(0:-180:0.75cm and 0.15cm);
\node at (0,0.5) {\tiny $i$};
\node[small_morphism] (n2) at (-0.75,0.5) {\tiny $\id$};
\draw (n1) -- (n2);
\draw[overline={1}] (0,1) ellipse (1cm and 0.2cm);
\draw (0,1) ellipse (0.5cm and 0.1cm);
\draw (-1,0) -- (-1,1);
\draw (1,0) -- (1,1);
\node at (-0.75,1.3) {\tiny $X_i X_i^*$};
\end{scope}
\end{tikzpicture}
\;\;
\text{, or more intuitively,}\;\;
\sum_{i\in \Irr(\cA)}
\frac{\sqrt{d_i}}{\sqrt{\cD}}
\begin{tikzpicture}
\begin{scope}[shift={(0,-0.5)}]
\draw (0,0) ellipse (1cm and 0.2cm);
\node[dotnode] (n1) at (-0.75,1) {}; 
\node[dotnode] (n2) at (-0.6,0.97) {};
\draw (n1) .. controls +(down:0.4cm) and +(-160:0.1cm) ..
  (-0.65,0.5) .. controls +(20:0.5cm) and +(up:0.1cm) ..
  (0.75,0.5);
\draw[overline={1}] (-0.5,0) -- (-0.5,1);
\draw[overline={1}] (0.5,0) -- (0.5,1);
\draw (0,0) ellipse (0.5cm and 0.1cm); 
\draw[overline={1}]
  (n2) .. controls +(down:0.5cm) and +(170:0.1cm) ..
  (-0.5,0.4) .. controls +(-10:0.3cm) and +(down:0.15cm) ..
  (0.75,0.5);
\draw[overline={1}] (0,1) ellipse (1cm and 0.2cm);
\draw (0,1) ellipse (0.5cm and 0.1cm);
\draw (-1,0) -- (-1,1);
\draw (1,0) -- (1,1);
\node at (-0.75,1.3) {\tiny $X_i X_i^*$};
\end{scope}
\end{tikzpicture}
\]

As an application of \thmref{t:ann-main},
we describe $\ZCY(\torus)$ purely algebraically in terms
of $\cA$.
We can produce $\torus$  from $\Ann$ by
gluing (neighborhoods of) $S^1 \times \{0\}$
and $S^1 \times \{1\}$.
The excision property of $\ZCY$ as stated
in the introduction doesn't work as is,
but \ocite{KT}*{Theorem 7.5}
actually proves an apparently slightly more general but ultimately
equivalent form of excision,
which allows $\Sigma$ to be obtained by
gluing two boundaries of a single surface $\Sigma'$;
the balanced tensor product is then replaced by
the center of $\ZCY(\Sigma')$
as a $\CYA$-bimodule
(as defined in \ocite{GNN}*{Definition 2.1},
repeated in \ocite{KT}*{Definition 3.1};
for applications here,
it suffices to know that this notion of center
for a monoidal category as a bimodule over itself
coincides with the Drinfeld center.)

We can view $\CYA$ as a bimodule category
over itself, thinking of the left and right actions
as ``insertions" from the left
($S^1 \times \{0\}$) and right ($S^1 \times \{0\}$).
Thus we have the following corollary
of \ocite{KT}*{Theorem 7.5}:
\begin{proposition} \label{p:ZCY_torus}
\[
\ZCY(\torus) \simeq \cZ((\CYA,\tnsrst))
\]
\end{proposition}
\begin{proof}
Take $X' = S^1 \times (0,3)$ in \ocite{KT}*{Theorem 7.5}.
\end{proof}

As an immediate corollary of this
and \thmref{t:ann-main}, we have
\begin{corollary} \label{c:torus_ZA}
\[
\ZCY(\torus) \simeq \cZ((\ZA,\tnsrbar))
\]
as abelian categories.
\end{corollary}

\begin{remark} \label{r:tnsr_topology}
Since $\tnsrbar$ on $\ZA$ has a nice topological
interpretation as $\tnsrst$ on $\CYA$,
it is also nice to have a topological interpretation
for the standard tensor product on $\ZA$.
This comes from the (thickened) pair of pants,
denoted $\POP$, in the following manner.
For $(X,\ga),(Y,\mu) \in \ZA$,
with corresponding objects
$\VV = (H(X),H(\hP_\ga)), \WW = (H(Y),H(\hP_\mu))
\in \CYA$,
the object $\VV \tnsr \WW$
which corresponds to $(X,\ga)\tnsr(Y,\mu)$
is characterized by a natural isomorphism
\[
\Hom_\CYA(\VV \tnsr \WW, \mathbf{X})
\simeq \Skein(\POP; \overline{\VV},\overline{\WW},
  \mathbf{X})
=
\Skein (
\begin{tikzpicture}[scale=0.6]
\tikzmath{
  \ellx = 1;
  \elly = 0.2;
  \ellxs = 0.8;
  \ellys = 0.16;
}
\draw (-1.5,1) ellipse (\ellx cm and \elly cm);
\draw (-1.5,1) ellipse (\ellxs cm and \ellys cm);
\draw (1.5,1) ellipse (\ellx cm and \elly cm);
\draw (1.5,1) ellipse (\ellxs cm and \ellys cm);
\draw (0,-1) ellipse (\ellx cm and \elly cm);
\draw (0,-1) ellipse (\ellxs cm and \ellys cm);
\draw (-2.5,1) to[out=-90,in=90] (-1,-1);
\draw (-2.3,1) to[out=-90,in=90] (-0.8,-1);
\draw (2.5,1) to[out=-90,in=90] (1,-1);
\draw (2.3,1) to[out=-90,in=90] (0.8,-1);
\draw (-0.5,1) .. controls +(down:0.5cm) and +(down:0.5cm)
  .. (0.5,1);
\draw (-0.7,1) .. controls +(down:0.6cm) and +(down:0.6cm)
  .. (0.7,1);
\node at (-1.5,1.4) {\tiny $\overline{\VV}$};
\node at (1.5,1.4) {\tiny $\overline{\WW}$};
\node at (0,-1.4) {\tiny $\mathbf{X}$};
\end{tikzpicture}
)
\]
for all $\mathbf{X} \in \CYA$.
(This is well-known for the extended Turaev-Viro theory \ocite{TV},
where given spherical fusion $\cA$, one has
$Z_{\text{TV}}(S^1) = \ZA$ \ocite{kirillov-stringnet},
and the standard tensor product on $\ZA$ is given by
the pair of pants just as above.)
The stacking product can also be described this way,
but instead of the usual pair of pants $\POP$,
we use a different cobordism, $\YPOP$:
\[
\Hom_\CYA(\VV \tnsrst \WW, \mathbf{X})
\simeq \Skein(\YPOP; \overline{\VV},\overline{\WW},
  \mathbf{X})
=
\Skein (
\begin{tikzpicture}[scale=0.6]
\draw (0,1) ellipse (2cm and 0.4cm);
\draw (0,1) ellipse (1.8 cm and 0.36cm);
\draw (0,1) ellipse (1cm and 0.2cm);
\draw (0,1) ellipse (0.8cm and 0.16cm);
\draw (0,-1) ellipse (1.5cm and 0.3cm);
\draw (0,-1) ellipse (1.3cm and 0.26cm);
\draw (-2,1) to[out=-90,in=90] (-1.5,-1);
\draw (-0.8,1) to[out=-90,in=90] (-1.3,-1);
\draw (2,1) to[out=-90,in=90] (1.5,-1);
\draw (0.8,1) to[out=-90,in=90] (1.3,-1);
\draw (-1.8,1) .. controls +(down:1cm) and +(down:1cm)
  .. (-1,1);
\draw (1.8,1) .. controls +(down:1cm) and +(down:1cm)
  .. (1,1);
\node at (-2.2,1) {\tiny $\overline{\VV}$};
\node at (-1.2,1) {\tiny $\overline{\WW}$};
\node at (0,-1.5) {\tiny $\mathbf{X}$};
\end{tikzpicture}
)
\]
$\YPOP$ is a thickened `Y' crossed with $S^1$.
We do not prove these claims here,
which are not hard to prove after all
the work in this section.
\rmkend
\end{remark}

\begin{remark} \label{r:2fold_cy}
The topological interpretations of $\tnsr$ and $\tnsrst$
above can also elucidate the structure morphisms
mentioned in \rmkref{r:2fold}.
Consider the two cobordisms below:
\[
\begin{tikzpicture}[scale=0.8]
\tikzmath{
  \md = -0.2;
}
\draw (-2.5,1) ellipse (2cm and 0.4cm);
\draw (-2.5,1) ellipse (1.8 cm and 0.36cm);
\draw (-2.5,1) ellipse (1cm and 0.2cm);
\draw (-2.5,1) ellipse (0.8cm and 0.16cm);
\draw (2.5,1) ellipse (2cm and 0.4cm);
\draw (2.5,1) ellipse (1.8 cm and 0.36cm);
\draw (2.5,1) ellipse (1cm and 0.2cm);
\draw (2.5,1) ellipse (0.8cm and 0.16cm);
\draw (0,\md) ellipse (2cm and 0.4cm);
\draw (0,\md) ellipse (1.8 cm and 0.36cm);
\draw (0,\md) ellipse (1cm and 0.2cm);
\draw (0,\md) ellipse (0.8cm and 0.16cm);
\draw (0,-1) ellipse (1.5cm and 0.3cm);
\draw (0,-1) ellipse (1.3cm and 0.26cm);
\draw (-4.5,1)
  .. controls +(down:0.5cm) and +(up:0.3cm) .. (-2,\md);
\draw (4.5,1)
  .. controls +(down:0.5cm) and +(up:0.3cm) .. (2,\md);
\draw (-4.3,1)
  .. controls +(down:0.5cm) and +(up:0.3cm) .. (-1.8,\md);
\draw (4.3,1)
  .. controls +(down:0.5cm) and +(up:0.3cm) .. (1.8,\md);
\draw (-3.5,1)
  .. controls +(down:0.5cm) and +(up:0.3cm) .. (-1,\md);
\draw (3.5,1)
  .. controls +(down:0.5cm) and +(up:0.3cm) .. (1,\md);
\draw (-3.3,1)
  .. controls +(down:0.5cm) and +(up:0.3cm) .. (-0.8,\md);
\draw (3.3,1)
  .. controls +(down:0.5cm) and +(up:0.3cm) .. (0.8,\md);
\draw (-0.5,1)
  .. controls +(down:0.4cm) and +(down:0.4cm) .. (0.5,1);
\draw (-0.7,1)
  .. controls +(down:0.45cm) and +(down:0.45cm) .. (0.7,1);
\draw (-1.5,1)
  .. controls +(down:0.6cm) and +(down:0.6cm) .. (1.5,1);
\draw (-1.7,1)
  .. controls +(down:0.65cm) and +(down:0.65cm) .. (1.7,1);
\draw (-2,\md) to[out=-90,in=90] (-1.5,-1);
\draw (-0.8,\md) to[out=-90,in=90] (-1.3,-1);
\draw (2,\md) to[out=-90,in=90] (1.5,-1);
\draw (0.8,\md) to[out=-90,in=90] (1.3,-1);
\draw (-1.8,\md)
  .. controls +(down:0.5cm) and +(down:0.5cm) .. (-1,\md);
\draw (1.8,\md)
  .. controls +(down:0.5cm) and +(down:0.5cm) .. (1,\md);
\draw[gray]
(0,0.7) .. controls +(left:0.5cm) and +(left:0.2cm) ..
(-0.3,-0.8) .. controls +(right:0.1cm) and +(left:0.2cm) ..
(0,0.55) .. controls +(right:0.1cm) and +(left:0.1cm) ..
(0.2,-0.3) .. controls +(right:0.2cm) and +(right:0.2cm) ..
(0,0.7);
\node at (-4.7,1) {$\overline{\VV}$};
\node at (-3.7,1) {$\overline{\WW}$};
\node at (0.3,1.05) {$\overline{\VV}'$};
\node at (1.3,1.05) {$\overline{\WW}'$};
\end{tikzpicture}
\;\;\;\;
\begin{tikzpicture}[scale=0.8]
\draw (-2.5,1) ellipse (2cm and 0.4cm);
\draw (-2.5,1) ellipse (1.8 cm and 0.36cm);
\draw (-2.5,1) ellipse (1cm and 0.2cm);
\draw (-2.5,1) ellipse (0.8cm and 0.16cm);
\draw (2.5,1) ellipse (2cm and 0.4cm);
\draw (2.5,1) ellipse (1.8 cm and 0.36cm);
\draw (2.5,1) ellipse (1cm and 0.2cm);
\draw (2.5,1) ellipse (0.8cm and 0.16cm);
\draw (-2.5,0) ellipse (1.5cm and 0.3cm);
\draw (-2.5,0) ellipse (1.3cm and 0.26cm);
\draw (2.5,0) ellipse (1.5cm and 0.3cm);
\draw (2.5,0) ellipse (1.3cm and 0.26cm);
\draw (0,-1) ellipse (1.5cm and 0.3cm);
\draw (0,-1) ellipse (1.3cm and 0.26cm);
\draw (-4.5,1) to[out=-90,in=90] (-4,0);
\draw (-3.3,1) to[out=-90,in=90] (-3.8,0);
\draw (-0.5,1) to[out=-90,in=90] (-1,0);
\draw (-1.7,1) to[out=-90,in=90] (-1.2,0);
\draw (-4.3,1)
  .. controls +(down:0.5cm) and +(down:0.5cm) .. (-3.5,1);
\draw (-0.7,1)
  .. controls +(down:0.5cm) and +(down:0.5cm) .. (-1.5,1);
\draw (4.5,1) to[out=-90,in=90] (4,0);
\draw (3.3,1) to[out=-90,in=90] (3.8,0);
\draw (0.5,1) to[out=-90,in=90] (1,0);
\draw (1.7,1) to[out=-90,in=90] (1.2,0);
\draw (4.3,1)
  .. controls +(down:0.5cm) and +(down:0.5cm) .. (3.5,1);
\draw (0.7,1)
  .. controls +(down:0.5cm) and +(down:0.5cm) .. (1.5,1);
\draw (-4,0)
  .. controls +(down:0.4cm) and +(up:0.2cm) .. (-1.5,-1);
\draw (4,0)
  .. controls +(down:0.4cm) and +(up:0.2cm) .. (1.5,-1);
\draw (-3.8,0)
  .. controls +(down:0.4cm) and +(up:0.2cm) .. (-1.3,-1);
\draw (3.8,0)
  .. controls +(down:0.4cm) and +(up:0.2cm) .. (1.3,-1);
\draw (-1,0)
  .. controls +(down:0.3cm) and +(down:0.3cm) .. (1,0);
\draw (-1.2,0)
  .. controls +(down:0.5cm) and +(down:0.5cm) .. (1.2,0);
\node at (-4.7,1) {$\overline{\VV}$};
\node at (-3.7,1) {$\overline{\WW}$};
\node at (0.3,1) {$\overline{\VV}'$};
\node at (1.3,1) {$\overline{\WW}'$};
\end{tikzpicture}
\]
The cobordism on the left (ignoring the gray curve)
corresponds to
$(\VV \tnsr \VV') \tnsrst (\WW \tnsr \WW')$,
or $((X,\ga) \tnsr (X',\ga')) \tnsrbar
  ((Y,\mu) \tnsr (Y',\mu'))$,
while the cobordism on the right corresponds to
$(\VV \tnsrst \WW) \tnsr (\VV' \tnsrst \WW')$,
or $((X,\ga) \tnsrbar (Y,\mu)) \tnsr
  ((X',\ga') \tnsrbar (Y',\mu'))$.
They are different, but not by much:
the right can be obtained from the left
by surgery, specifically,
by attaching a 2-handle along the gray curve
(the gray curve starts in the outside $\POP$ for $\VV$'s,
goes down into the bottom $\YPOP$,
back up into the $\POP$ for $\WW$'s,
then goes down again into $\YPOP$, and closes up,
all the while staying close to the ``inner boundary",
i.e. keeping as tight as possible like a rubber band).
One way to visualize this is to consider the
reverse process of removing a 2-handle:
start with the right side, push the two ``troughs"
- between $\VV$ and $\WW$ and between $\VV'$ and $\WW'$
- downward and into the bottom pair of pants,
and when they are about to meet in the middle,
drill a hole through the wall.

On the level of skein modules,
this surgery is the same as
adding the gray curve colored by the regular coloring
(up to a factor, see \eqnref{e:sliding});
this is the topological interpretation of $\eta$.
Conversely, graphs in the right cobordism can be
lifted to a graph in the left cobordism
plus the gray curve with regular coloring;
this is the topological interpretation of $\zeta$.

The other structure morphisms can also be described
very easily. For example, one of them is
a morphism $v_2: \onebar \tnsr \onebar \to \onebar$,
which is simply the empty graph in $\POP$
(up to a factor)
interpreted as a morphism
$\onest \tnsrst \onest \to \onest$.

Thus, checking the compatibility axioms
becomes an exercise in topology.
Once again, we do not show the work in this paper.

As mentioned in the introduction, such 2-fold monoidal
structures are related to iterated loop spaces.
It may be interesting to see if these two topological
aspects of $(\ZA,\tnsr,\tnsrbar)$ are directly related.
\rmkend
\end{remark}

\subsection{$\cA$ Modular}
\label{s:ann-modular}
\par \noindent

When $\cA$ is modular, we have
(\ocite{muger}, see also \ocite{EGNO}*{Proposition 8.20.12}):

\begin{align}
\label{e:modular-split}
\begin{split}
\cAAbop &\simeq_{\tnsr, \text{br}} (\ZA,\tnsr) \\
X \boxtimes Y &\mapsto (X \tnsr Y, c^\inv \tnsr c) = (X,c^\inv) \tnsr (Y,c)
\end{split}
\end{align}
where $\cA^\bop$ is $\cA$ with the opposite braiding,
and $c$ is the braiding on $\cA$.
Here the monoidal structure on $\cAAbop$ is defined component-wise.
In particular, duals are given by
$(X\boxtimes Y)^* = X^* \boxtimes Y^*$.

It is natural to ask: what is the reduced tensor product on
$\cAAbop$ under this equivalence?
We claim (proven below in \thmref{t:AA_ZA})
that the following definition is the answer,
which justifies the repeated use of the name ``reduced"
and notation like $\tnsrbar$ and $\onebar$.
(The reduced tensor product on $\ZA$ cannot in general be
braided, as we shall see soon,
so we will ignore the difference in braiding on $\cA$.)

\begin{definition}
\label{d:AA_red}
Let $W_1 \boxtimes Y_1, W_2 \boxtimes Y_2 \in \cAA$.
Define their \emph{reduced tensor product}
to be
\[
(W_1 \boxtimes Y_1) \tnsrbar (W_2 \boxtimes Y_2)
:= \eval{Y_1,W_2} \cdot W_1 \boxtimes Y_2
\]
where recall $\eval{Y_1,W_2} :=
  \Hom_\cA(\one,Y_1 W_2)$.
$\tnsrbar$ naturally extends to direct sums,
and is clearly associative.

For morphisms $f_1 \boxtimes g_1 :
W_1 \boxtimes Y_1 \to W_1' \boxtimes Y_1'$,
$f_2 \boxtimes g_2 :
W_2 \boxtimes Y_2 \to W_2' \boxtimes Y_2'$,
their \emph{reduced tensor product} is
given by
\[
(f_1 \boxtimes g_1) \tnsrbar (f_2 \boxtimes g_2)
:= \eval{g_1,f_2} \cdot f_1 \boxtimes g_2
=
\begin{tikzpicture}
\draw (-0.2,0.5) -- (-0.2,-0.5)
  node[pos=0,above] {\small $Y_1$}
  node[below] {\small $Y_1'$};
\node[small_morphism] at (-0.2,0) {\tiny $g_1$};
\draw (0.2,0.5) -- (0.2,-0.5)
  node[pos=0,above] {\small $W_2$}
  node[below] {\small $W_2'$};
\node[small_morphism] at (0.2,0) {\tiny $f_2$};
\end{tikzpicture}
\cdot
\begin{tikzpicture}
\draw (-0.3,0.5) -- (-0.3,-0.5)
  node[pos=0,above] {\small $W_1$}
  node[below] {\small $W_1'$};
\node[small_morphism] at (-0.3,0) {\tiny $f_1$};
\draw (0.3,0.5) -- (0.3,-0.5)
  node[pos=0,above] {\small $Y_2$}
  node[below] {\small $Y_2'$};
\node[small_morphism] at (0.3,0) {\tiny $g_2$};
\node at (0,0) {$\boxtimes$};
\end{tikzpicture}
\]
where the left side, the ``coefficient" $\eval{g_1,f_2}$,
is to be interpreted as a linear map
$\eval{Y_1,W_2} \to \eval{Y_1',W_2'}$
by composition.
\defend
\end{definition}

For example,
\begin{equation} \label{e:cAA_simple}
(X_i \boxtimes X_j^*) \tnsrbar (X_k \boxtimes X_l^*)
\simeq \delta_{j,k} X_i \boxtimes X_l^*
\end{equation}
In particular, when $i\neq j$,
\begin{equation} \label{e:cAA_not_symm}
\begin{split}
(X_i \boxtimes X_i^*) \tnsrbar (X_i \boxtimes X_j^*)
&\simeq X_i \boxtimes X_j^* \\
(X_i \boxtimes X_j^*) \tnsrbar (X_i \boxtimes X_i^*)
&\simeq 0
\end{split}
\end{equation}
so $\tnsrbar$ cannot be braided.

\begin{proposition}
$(\cAA, \tnsrbar)$ is a pivotal multifusion category.
More precisely,
\begin{itemize}
\item The unit object, denoted $\onebar$, is
  $\bigoplus_{i\in \Irr(\cA)} X_i \boxtimes X_i^*$;

\item $(X \boxtimes Y)^\vee = Y^* \boxtimes X^*$,
  $\rvee (X \boxtimes Y) = \rvee Y \boxtimes \rdual{X}$,
  the (co)evaluation maps are described in the proof;

\item The pivotal structure is defined component-wise:
  $\delta_{X\boxtimes Y} = \delta_X \boxtimes \delta_Y$.

\end{itemize}
\end{proposition}

\begin{proof}
Fix $X \boxtimes Y \in \cAA$.
The left and right unit constraint are given by
\[
l_{X\boxtimes Y} :=
\sum_{k \in \Irr(\cA)} \sqrt{d_k}
\begin{tikzpicture}
\draw (-0.3,0.3) .. controls +(-90:0.5cm) and +(-90:0.5cm) .. (0.3,0.3)
  node[pos=0,above] {\small $X_k^*$}
  node[above] {\small $X$};
\node[small_morphism] at (-0.25,0.1) {\tiny $\al$};
\end{tikzpicture}
\cdot
\begin{tikzpicture}
\draw (-0.3,-0.3) -- (-0.3,0.3)
  node[above] {\small $X_k$}
  node[pos=0,below] {\small $X$};
\node[small_morphism] at (-0.3,0) {\tiny $\al$};
\node at (0,0) {$\boxtimes$};
\draw (0.3,-0.3) -- (0.3,0.3)
  node[above] {\small $Y$}
  node[pos=0,below] {\small $Y$};
\end{tikzpicture}
\;\; ; \;\;
r_{X\boxtimes Y} :=
\sum_{k \in \Irr(\cA)} \sqrt{d_k}
\begin{tikzpicture}
\draw (-0.3,0.3) .. controls +(-90:0.5cm) and +(-90:0.5cm) .. (0.3,0.3)
  node[pos=0,above] {\small $Y$}
  node[above] {\small $X_k$};
\node[small_morphism] at (-0.25,0.1) {\tiny $\al$};
\end{tikzpicture}
\cdot
\begin{tikzpicture}
\draw (-0.3,-0.3) -- (-0.3,0.3)
  node[above] {\small $X$}
  node[pos=0,below] {\small $X$};
\node at (0,0) {$\boxtimes$};
\draw (0.3,-0.3) -- (0.3,0.3)
  node[above] {\small $X_k^*$}
  node[pos=0,below] {\small $Y$};
\node[small_morphism] at (0.3,0) {\tiny $\al$};
\end{tikzpicture}
\]
where we recall that $\al$ is a sum over
a pair of dual bases
(see \eqnref{e:summation_convention}).
Their inverses are given by
flipping the diagram upside down.

The left (co)evaluation maps are given by
\[
\ev_{X\boxtimes Y} :=
\sum_{k \in \Irr(\cA)} \sqrt{d_k}
\begin{tikzpicture}
\draw (-0.3,0.3) .. controls +(-90:0.5cm) and +(-90:0.5cm) .. (0.3,0.3)
  node[pos=0,above] {\small $X^*$}
  node[above] {\small $X$};
\end{tikzpicture}
\cdot
\begin{tikzpicture}
\draw (-0.3,-0.3) -- (-0.3,0.3)
  node[above] {\small $Y^*$}
  node[pos=0,below] {\small $X_k$};
\node[small_morphism] at (-0.3,0) {\tiny $\al$};
\node at (0,0) {$\boxtimes$};
\draw (0.3,-0.3) -- (0.3,0.3)
  node[above] {\small $Y$}
  node[pos=0,below] {\small $X_k^*$};
\node[small_morphism] at (0.3,0) {\tiny $\al$};
\end{tikzpicture}
\;\; ; \;\;
\coev_{X \boxtimes Y} :=
\sum_{k \in \Irr(\cA)} \sqrt{d_k}
\begin{tikzpicture}
\draw (-0.3,-0.3) .. controls +(90:0.5cm) and +(90:0.5cm) .. (0.3,-0.3)
  node[pos=0,below] {\small $Y$}
  node[below] {\small $Y^*$};
\end{tikzpicture}
\cdot
\begin{tikzpicture}
\draw (-0.3,-0.3) -- (-0.3,0.3)
  node[above] {\small $X_k$}
  node[pos=0,below] {\small $X$};
\node[small_morphism] at (-0.3,0) {\tiny $\al$};
\node at (0,0) {$\boxtimes$};
\draw (0.3,-0.3) -- (0.3,0.3)
  node[above] {\small $X_k^*$}
  node[pos=0,below] {\small $X^*$};
\node[small_morphism] at (0.3,0) {\tiny $\al$};
\end{tikzpicture}
\]
The right (co)evaluation maps are given by
similar diagrams.
It is straightforward to check that these have
the right properties.
\end{proof}

\begin{theorem} \label{t:AA_ZA}
There is an equivalence of pivotal multifusion categories
\begin{align*}
K: \cAA &\simeq_\tnsr (\ZA,\tnsrbar) \\
X \boxtimes Y &\mapsto (XY, c^\inv c)
\end{align*}
\end{theorem}

\begin{proof}
The tensor structure $L$ on $K$ is given as follows:
for $W_1 \boxtimes Y_1, W_2 \boxtimes Y_2$,
the isomorphism
$L : K(W_1\boxtimes Y_1) \tnsrbar K(W_2\boxtimes Y_2)
\simeq K((W_1\boxtimes Y_1) \tnsrbar (W_2 \boxtimes Y_2))$
is given by
\begin{align*}
L: (W_1 Y_1, c^\inv c) \tnsrbar (W_2 Y_2, c^\inv c)
&\simeq
\eval{Y_1,W_2} \cdot (W_1 Y_2, c^\inv c)
\\
\begin{tikzpicture}
\node[small_morphism] (al) at (0,0) {\tiny $\al$};
\draw (al) to[out=-135,in=90] (-0.2,-0.5)
  node[below] {\tiny $Y_1$};
\draw (al) to[out=-45,in=90] (0.2,-0.5)
  node[below] {\tiny $W_2$};
\end{tikzpicture}
&\cdot
\begin{tikzpicture}
\node[small_morphism] (al) at (0,0) {\tiny $\al$};
\draw (al) to[out=135,in=-90] (-0.2,0.5)
  node[above] {\tiny $Y_1$};
\draw (al) to[out=45,in=-90] (0.2,0.5)
  node[above] {\tiny $W_2$};
\draw (-0.5,0.5) .. controls +(down:0.3cm) and +(up:0.3cm) .. (-0.3,-0.5)
  node[pos=0,above] {\tiny $W_1$}
  node[below] {\tiny $W_1$};
\draw (0.5,0.5) .. controls +(down:0.3cm) and +(up:0.3cm) .. (0.3,-0.5)
  node[pos=0,above] {\tiny $Y_2$}
  node[below] {\tiny $Y_2$};
\end{tikzpicture}
\end{align*}

The inverse to $L$ is given by flipping the diagram upside down.
The following observation is helpful:
for $(W_1 Y_1, c^\inv c),(W_2 Y_2, c^\inv c) \in \ZA$,
we have
\[
W_1 Y_1 \tnsrproj{c^\inv c}{c^\inv c} W_2 Y_2
=
\im\Big(
\frac{1}{\cD}
\begin{tikzpicture}
\tikzmath{
  \toprow = 0.6;
  \bottom = -0.6;
}
\draw (0.5,\bottom) -- (0.5,\toprow)
  node[above] {\tiny $Y_2$};
\draw (-0.2,\bottom) -- (-0.2,\toprow)
  node[above] {\tiny $Y_1$};
\draw (0.2,\bottom) -- (0.2,\toprow)
  node[above] {\tiny $W_2$};
\draw[regular,overline={1.5}] (0,0) ellipse (0.7cm and 0.3cm);
\draw[overline={1.5}] (-0.5,\bottom) -- (-0.5,\toprow)
  node[above] {\tiny $W_1$};
\draw[overline={1.5}] (-0.2,0) -- (-0.2,\toprow);
\draw[overline={1.5}] (0.2,0) -- (0.2,\toprow);
\end{tikzpicture}
\Big)
=
\im\Big(
\begin{tikzpicture}
\tikzmath{
  \toprow = 0.6;
  \bottom = -0.6;
}
\draw (-0.5,\bottom) -- (-0.5,\toprow);
\draw (0.5,\bottom) -- (0.5,\toprow);
\node[small_morphism] (al1) at (0,0.2) {\tiny $\al$};
\node[small_morphism] (al2) at (0,-0.2) {\tiny $\al$};
\draw (al1) to[out=135,in=-90] (-0.2,\toprow);
\draw (al1) to[out=45,in=-90] (0.2,\toprow);
\draw (al2) to[out=-135,in=-90] (-0.2,\bottom);
\draw (al2) to[out=-45,in=-90] (0.2,\bottom);
\end{tikzpicture}
\Big)
\]
It is easy to check that $L$ satisfies the hexagon axiom.

Note that $K$ does not send unit to unit - the half-braiding on
$K(\onebar) = (\bigoplus X_i X_i^*,c^\inv c)$
is not the same as
$\onebar = (\bigoplus X_i X_i^*, \Gamma)$;
the isomorphism $K(\onebar) \simeq \onebar$
is essentially given by the $S$-matrix:
\[
S = \sum_{i,j\in \Irr(\cA)} \sqrt{d_i}\sqrt{d_j}
\begin{tikzpicture}
\tikzmath{
  \tp = 0.4;
  \bt = -0.4;
  \lf = -0.15;
  \rt = 0.15;
  \midtp = 0.1;
  \middn = -0.1;
}
\draw (\rt,\tp) to[out=-90,in=0] (0,\middn);
\draw[midarrow_rev={0.4}] (\lf,\bt) to[out=90,in=180] (0,\midtp);
\node at (-0.3,0.3) {\tiny $i$};
\draw[overline={1.5}] (\rt,\bt) to[out=90,in=0] (0,\midtp);
\draw[overline={1.5},midarrow={0.4}]
  (\lf,\tp) to[out=-90,in=180] (0,\middn);
\node at (-0.3,-0.3) {\tiny $j$};
\end{tikzpicture}
\]
Clearly the pivotal structures agree.
\end{proof}

The equivalence given above has a nice interpretation
in $\CYA$. Namely, the composition
$\Kar(H) \circ \Kar(G)^\inv \circ K :
  \cAA \simeq \CYA$
is naturally isomorphic to the following functor:
\begin{align*}
\cAA &\simeq \CYA
\\
\begin{tikzpicture}
\node (a) at (0,0.65) {$X\boxtimes Y$};
\node (b) at (0,-0.65) {$X' \boxtimes Y'$};
\draw[midarrow={0.55}] (a) -- (b) node[pos=0.5,right]
  {$f \boxtimes g$};
\end{tikzpicture}
&
\begin{tikzpicture}
\node at (0,0.65) {$\mapsto$};
\node at (0,0) {$\mapsto$};
\node at (0,-0.65) {$\mapsto$};
\end{tikzpicture}
\begin{tikzpicture}
\tikzmath{
  \toprow = 1;
  \bottom = -1;
  \midrow = 0;
  \qttop = 0.5;
  \qtbot = -0.5;
}
\draw (0,\bottom) ellipse (1cm and 0.2cm);
\draw[regular] (0.75,\qttop) arc(0:180:0.75cm and 0.15cm);
\draw[regular] (0.75,\qtbot) arc(0:180:0.75cm and 0.15cm);
\draw[overline={1}] (-0.5,\bottom) -- (-0.5,\toprow);
\draw[overline={1}] (0.5,\bottom) -- (0.5,\toprow);
\draw (0,\bottom) ellipse (0.5cm and 0.1cm); 
\node[dotnode] (n1) at (-0.85,\toprow) {};
\node[dotnode] (n2) at (-0.65,\toprow) {};
\node[dotnode] (n3) at (-0.85,\bottom) {};
\node[dotnode] (n4) at (-0.65,\bottom) {};
\draw[overline={1}] (n1) -- (n3);
\draw[overline={1}] (n2) -- (n4);
\node[dotnode] at (-0.85,\midrow) {};
\node at (-0.95,0.05) {\tiny $f$};
\node[dotnode] at (-0.65,\midrow) {};
\node at (-0.55,0.05) {\tiny $g$};
\draw[regular, overline={1.5}]
  (-0.75,\qttop) arc(180:360:0.75cm and 0.15cm);
\draw[regular, overline={1.5}]
  (-0.75,\qtbot) arc(180:360:0.75cm and 0.15cm);
\draw[overline={1}] (0,\toprow) ellipse (1cm and 0.2cm);
\draw (0,\toprow) ellipse (0.5cm and 0.1cm);
\draw (-1,\bottom) -- (-1,\toprow);
\draw (1,\bottom) -- (1,\toprow);
\node at (-0.75,1.3) {\tiny $X Y$};
\node at (-0.75,-1.35) {\tiny $X' Y'$};
\draw[decoration={brace,amplitude=3pt},decorate]
  (-1.2,0.3) -- (-1.2,\toprow);
\node at (-1.55,0.65) {im};
\node at (1.3,0.65) {$\frac{1}{\cD}$};
\draw[decoration={brace,amplitude=3pt},decorate]
  (-1.2,\bottom) -- (-1.2,-0.3);
\node at (-1.55,-0.65) {im};
\node at (1.3,-0.65) {$\frac{1}{\cD}$};
\end{tikzpicture}
\end{align*}

The composition $\Kar(H) \circ \Kar(G)^\inv \circ K$
itself only hits objects with one marked point $p$.
The functor presented above is more intuitive
from the following perspective.
Restricting to the first factor,
i.e. setting $Y=\one$,
this is like including $\cA$ into $\CYA$
along the outer boundary; likewise,
the second factor is including $\cA$ into $\CYA$
along the inner boundary:
\[
X \mapsto (X, c^\inv) \mapsto
\im \big( \frac{1}{\cD}
\begin{tikzpicture}
\tikzmath{
  \toprow = 0.5;
  \bottom = -0.5;
  \midrow = 0;
}
\draw (0,\bottom) ellipse (1cm and 0.2cm);
\draw[regular] (0.6,\midrow) arc(0:180:0.6cm and 0.1cm);
\draw[overline={1}] (-0.5,\bottom) -- (-0.5,\toprow);
\draw[overline={1}] (0.5,\bottom) -- (0.5,\toprow);
\draw (0,\bottom) ellipse (0.5cm and 0.1cm); 
\node[dotnode] (n1) at (-0.75,\toprow) {};
\node[dotnode] (n2) at (-0.75,\bottom) {};
\draw[overline={1}] (n1) -- (n2);
\draw[regular, overline={1.5}]
  (0.6,\midrow) arc(0:-180:0.6cm and 0.1cm);
\draw[overline={1}] (0,\toprow) ellipse (1cm and 0.2cm);
\draw (0,\toprow) ellipse (0.5cm and 0.1cm);
\draw (-1,\bottom) -- (-1,\toprow);
\draw (1,\bottom) -- (1,\toprow);
\node at (-0.75,0.8) {\tiny $X$};
\node at (-0.75,-0.8) {\tiny $X$};
\end{tikzpicture}
\big)
\;\;\;
;
\;\;\;
Y \mapsto (Y, c) \mapsto
\im \big( \frac{1}{\cD}
\begin{tikzpicture}
\tikzmath{
  \toprow = 0.5;
  \bottom = -0.5;
  \midrow = 0;
}
\draw (0,\bottom) ellipse (1cm and 0.2cm);
\draw[regular] (0.875,\midrow) arc(0:180:0.875cm and 0.15cm);
\draw[overline={1}] (-0.5,\bottom) -- (-0.5,\toprow);
\draw[overline={1}] (0.5,\bottom) -- (0.5,\toprow);
\draw (0,\bottom) ellipse (0.5cm and 0.1cm); 
\node[dotnode] (n1) at (-0.75,\toprow) {};
\node[dotnode] (n2) at (-0.75,\bottom) {};
\draw[overline={1}] (n1) -- (n2);
\draw[regular, overline={1.5}]
  (0.875,\midrow) arc(0:-180:0.875cm and 0.1cm);
\draw[overline={1}] (0,\toprow) ellipse (1cm and 0.2cm);
\draw (0,\toprow) ellipse (0.5cm and 0.1cm);
\draw (-1,\bottom) -- (-1,\toprow);
\draw (1,\bottom) -- (1,\toprow);
\node at (-0.75,0.8) {\tiny $Y$};
\node at (-0.75,-0.8) {\tiny $Y$};
\end{tikzpicture}
\big)
\]
The functor presented before
is the $\tnsrst$-tensor product of these two.
This picture also elucidates the definition of
$\tnsrbar$ on $\cAA$: if we take the tensor product
of the two functors above in opposite order,
essentially looking at
$(\one \boxtimes X) \tnsrbar (Y \boxtimes \one)$,
we get
\[
K(\one \boxtimes X) \tnsrst K(Y \boxtimes \one)
=
\im \big(
\frac{1}{\cD^2}
\begin{tikzpicture}
\tikzmath{
  \toprow = 0.5;
  \bottom = -0.5;
  \midrow = 0;
}
\draw (0,\bottom) ellipse (1cm and 0.2cm);
\draw[regular] (0.55,\midrow) arc(0:180:0.55cm and 0.1cm);
\draw[regular] (0.95,\midrow) arc(0:180:0.95cm and 0.2cm);
\draw[overline={1}] (-0.5,\bottom) -- (-0.5,\toprow);
\draw[overline={1}] (0.5,\bottom) -- (0.5,\toprow);
\draw (0,\bottom) ellipse (0.5cm and 0.1cm); 
\node[dotnode] (n1) at (-0.85,\toprow) {};
\node[dotnode] (n2) at (-0.65,\toprow) {};
\node[dotnode] (n3) at (-0.85,\bottom) {};
\node[dotnode] (n4) at (-0.65,\bottom) {};
\draw[overline={1}] (n1) -- (n3);
\draw[overline={1}] (n2) -- (n4);
\draw[regular, overline={1.5}]
  (-0.55,\midrow) arc(180:360:0.55cm and 0.1cm);
\draw[regular, overline={1.5}]
  (-0.95,\midrow) arc(180:360:0.95cm and 0.2cm);
\draw[overline={1}] (0,\toprow) ellipse (1cm and 0.2cm);
\draw (0,\toprow) ellipse (0.5cm and 0.1cm);
\draw (-1,\bottom) -- (-1,\toprow);
\draw (1,\bottom) -- (1,\toprow);
\node at (-0.75,0.8) {\tiny $XY$};
\node at (-0.75,-0.85) {\tiny $XY$};
\end{tikzpicture}
\big)
=
\im \big(
\frac{1}{\cD^2}
\begin{tikzpicture}
\tikzmath{
  \toprow = 0.5;
  \bottom = -0.5;
  \midrow = 0;
}
\draw (0,\bottom) ellipse (1cm and 0.2cm);
\draw[regular] (0.55,-0.1) arc(0:180:0.55cm and 0.1cm);
\draw[regular] (-0.55,0.1) arc(0:180:0.2cm and 0.04cm);
\draw[overline={1}] (-0.5,\bottom) -- (-0.5,\toprow);
\draw[overline={1}] (0.5,\bottom) -- (0.5,\toprow);
\draw (0,\bottom) ellipse (0.5cm and 0.1cm); 
\node[dotnode] (n1) at (-0.85,\toprow) {};
\node[dotnode] (n2) at (-0.65,\toprow) {};
\node[dotnode] (n3) at (-0.85,\bottom) {};
\node[dotnode] (n4) at (-0.65,\bottom) {};
\draw[overline={1}] (n1) -- (n3);
\draw[overline={1}] (n2) -- (n4);
\draw[regular, overline={1.5}]
  (-0.55,-0.1) arc(180:360:0.55cm and 0.1cm);
\draw[regular, overline={1.5}]
  (-0.95,0.1) arc(180:360:0.2cm and 0.04cm);
\draw[overline={1}] (0,\toprow) ellipse (1cm and 0.2cm);
\draw (0,\toprow) ellipse (0.5cm and 0.1cm);
\draw (-1,\bottom) -- (-1,\toprow);
\draw (1,\bottom) -- (1,\toprow);
\node at (-0.75,0.8) {\tiny $XY$};
\node at (-0.75,-0.85) {\tiny $XY$};
\end{tikzpicture}
\big)
=
\im \big(
\frac{1}{\cD}
\begin{tikzpicture}
\tikzmath{
  \toprow = 0.5;
  \bottom = -0.5;
  \midrow = 0;
}
\draw (0,\bottom) ellipse (1cm and 0.2cm);
\draw[regular] (0.55,0) arc(0:180:0.55cm and 0.1cm);
\draw[overline={1}] (-0.5,\bottom) -- (-0.5,\toprow);
\draw[overline={1}] (0.5,\bottom) -- (0.5,\toprow);
\draw (0,\bottom) ellipse (0.5cm and 0.1cm); 
\node[dotnode] (n1) at (-0.85,\toprow) {};
\node[dotnode] (n2) at (-0.65,\toprow) {};
\node[dotnode] (n3) at (-0.85,\bottom) {};
\node[dotnode] (n4) at (-0.65,\bottom) {};
\node[dotnode] (a1) at (-0.75,0.1) {};
\node at (-0.9,0.1) {\tiny $\al$};
\node[dotnode] (a2) at (-0.75,-0.1) {};
\node at (-0.9,-0.1) {\tiny $\al$};
\draw[overline={1}] (n1) to[out=-90,in=120] (a1);
\draw[overline={1}] (n2) to[out=-90,in=60] (a1);
\draw[overline={1}] (n3) to[out=90,in=-120] (a2);
\draw[overline={1}] (n4) to[out=90,in=-60] (a2);
\draw[regular, overline={1.5}]
  (-0.55,0) arc(180:360:0.55cm and 0.1cm);
\draw[overline={1}] (0,\toprow) ellipse (1cm and 0.2cm);
\draw (0,\toprow) ellipse (0.5cm and 0.1cm);
\draw (-1,\bottom) -- (-1,\toprow);
\draw (1,\bottom) -- (1,\toprow);
\node at (-0.75,0.8) {\tiny $XY$};
\node at (-0.75,-0.85) {\tiny $XY$};
\end{tikzpicture}
\big)
\simeq
\eval{X,Y} \cdot \onest
\]
where we used \eqnref{e:sliding} and
\lemref{l:killing}.

Note that the equivalence
$\cAAbop \simeq_{\tnsr, \text{br}} (\ZA,\tnsr)$
mentioned in the beginning of this section
is also built by tensoring the same two functors together;
it just happens that
\footnote{Coincidence? I think NOT!}
\[
(X,c^\inv) \tnsr (Y,c) = (X,c^\inv) \tnsrbar (Y,c)
\]

\begin{remark}
One can consider a similar tensor product
on $\cAA$ when $\cA$ is not modular.
\defref{d:AA_red} of $\tnsrbar$ will look the same,
but $\eval{Y_1,W_2}$ would be replaced by
the symmetric part of $Y_1 W_2$,
that is, the direct summands of $Y_1 W_2$
that belong to the symmetric center of $\cA$.
The functor $K: \cA \boxtimes \cA \to \ZA$
will still respect $\tnsrbar$,
but $K$ will not be an equivalence.
Furthermore, it is not clear whether
$\tnsrbar$ on $\cAA$ possesses a unit.
As evidence suggestive of this,
recall that in the modular case,
even showing $K(\onebar) \simeq \onebar$
required the non-degeneracy of the $S$-matrix.
\rmkend
\end{remark}

Finally, we compare $(\cAA,\tnsrbar)$ to
$\matvec{}$,
the category of $\Vctsp$-valued matrices
(see \ocite{EGNO}*{Example 4.1.3}).
Let $\cS$ be some finite set.
The objects of $\matvec{\cS}$
are bigraded vector spaces
$V = \bigoplus_{i,j\in \cS} V_i^j$,
and the tensor product is given by
\[
(V \tnsr W)_i^j = \bigoplus_{k\in \cS} V_i^k \tnsr W_k^j
\]
Let $\kk_i^j$ be $\kk$ with bigrading $i,j$.
Then the unit is $\bigoplus_{i\in \cS} \kk_i^i$.
Duals are given by transposing the matrix
and then taking duals componentwise,
i.e. $(V^*)_i^j = (V_j^i)^*$.

\begin{proposition} \label{p:ZA_matvec}
There is a tensor equivalence
\begin{align*}
\matvec{\Irr(\cA)} &\simeq_\tnsr (\cAA,\tnsrbar)
\\
\kk_i^j &\mapsto X_i \boxtimes X_j^*
\end{align*}
\end{proposition}

\begin{proof}
Clearly the functor above is an equivalence of
abelian categories,
and sends unit to unit.

Denote $X_i^j := X_i \boxtimes X_j^*$,
so $(X_i^j)^\vee = X_j^i$.
One has
$X_i^j \tnsrbar X_j^l = \eval{X_j^*,X_k} X_i^l$,
which is 0 if $j \neq k$.
When $j=k$, define $J_k$ to be the map
\[
(\ev_{X_k} \circ -) \cdot \id_{X_i^l}:
  X_i^k \tnsrbar X_k^l \simeq X_i^l
\]
Then it is easy to check that
\[
J_{\kk_i^j,\kk_k^l} = \delta_{j,k} J_k
\]
is a tensor structure on the functor above.
\end{proof}

While $(\cAA,\tnsrbar)$ and $\matvec{\Irr(\cA)}$ are tensor equivalent,
they are not pivotal tensor equivalent
(assuming the natural pivotal structure on $\matvec{\Irr(\cA)}$
coming from $\Vctsp$).
This can be seen from computing dimensions.
For example, the left trace of $\id_{X_i^j}$ is
\[
\sum_{k\in \Irr(\cA)} d_k
\begin{tikzpicture}
\draw[midarrow] (0,0) circle (0.3cm);
\node at (-0.35,0.2) {\tiny $j^*$};
\end{tikzpicture}
\; \cdot
\begin{tikzpicture}
\draw[midarrow={0.55}] (-0.3,0.6) -- (-0.3,-0.6)
  node[pos=0,above] {\small $X_k$}
  node[below] {\small $X_k$};
\node at (-0.45,0) {\tiny $i$};
\node[small_morphism] at (-0.3,0.3) {\tiny $\al$};
\node[small_morphism] at (-0.3,-0.3) {\tiny $\beta$};
\node at (0,0) {$\boxtimes$};
\draw[midarrow_rev={0.55}] (0.3,0.6) -- (0.3,-0.6)
  node[pos=0,above] {\small $X_k^*$}
  node[below] {\small $X_k^*$};
\node at (0.45,0) {\tiny $i$};
\node[small_morphism] at (0.3,0.3) {\tiny $\al$};
\node[small_morphism] at (0.3,-0.3) {\tiny $\beta$};
\end{tikzpicture}
=
\sum_k \delta_{i,k} \frac{d_j}{d_k}
  \id_{X_k^k}
\]
so the left dimension of $X_i^j$ is
$d_{X_i^j}^L = \frac{d_j}{d_i}$,
which cannot always be 1 for all pairs $i,j$.
Its right dimension is
$d_{X_i^j}^R = d_{X_j^i}^L = \frac{d_i}{d_j}$.
Thus $(\cAA,\tnsrbar)$ can be thought of
as a non-pivotal deformation of $\matvec{}$.

\newpage
\section{Computations for Other Surfaces}
\label{s:surfaces}
\vspace{0.8in}

In this section, we consider the categories associated to
other surfaces.
We will dedicate one subsection to study the once-punctured torus.

We want to relate the skein category of a surface $N$
with that of a punctured one $N_0$, that is,
$N_0 = N \backslash \{p\}$.
We will think of $N$ as obtained from $N_0$ by
gluing with an open disk, ``sealing'' the puncture:
$N = N_0 \cup \DD^2$,
implicitly choosing some collared structure on $N_0$ and $\DD^2$.\\

Recall $\hA := \htr(\cA)$ from \secref{s:ann-htr}.
The unit object $\one$ in $\hA$ corresponds to
the empty configuration in $\hatZCYsk(N_0)$.
There is a right action of $\Hom_\hA(\one,\one)$ on the
morphisms of $\hatZCYsk(N_0)$, by
``pushing in'' from the puncture, i.e.
$\Hom_{\hatZCYsk(N_0)}(Y,Y') \tnsr \Hom_\hA(\one,\one)
\to \Hom_{\hatZCYsk(N_0)}(Y \ract \one, Y' \ract \one)
\cong \Hom_{\hatZCYsk(N_0)}(Y, Y')$.
We have that $\Gamma \in \Hom_{\hatZCYsk(N_0)}(Y,Y')$
and $f,g \in \Hom_\hA(\one,\one)$,
$\Gamma \ract (f \circ g) = (\Gamma \ract f) \ract g$.
Moreover, for $\Gamma' \in \Hom_{\hatZCYsk(N_0)}(Y',Y'')$,
$(\Gamma' \circ \Gamma) \ract (f \circ g) =
(\Gamma' \ract f) \circ (\Gamma \ract g)$.\\

Let $\pi = \sum d_i/\cD \cdot \id_{X_i}
\in \bigoplus \ihom{\cA}{X_i}(\one,\one) = \Hom_\hA(\one,\one)$.
(Note: $\cD$ and simples $X_i$ are of $\cA$, and not of $\cZ(\cA)$.)
$\pi$ is an idempotent in $\Hom_\hA(\one,\one)$,
and hence also acts as an idempotent on $\Hom_{\hatZCYsk(N_0)}(Y,Y')$.
As a morphism in $\ZCYsk(\Ann)$ under the equivalence
discussed in \secref{s:annulus} (see \thmref{t:ann-main}),
$\pi$ is a graph in $\Ann \times [0,1] \simeq S^1 \times D^2$
(a solid torus),
consisting of a loop going along a core of the solid torus,
labeled with the regular coloring$/\cD$.

\begin{proposition}
\label{prp:punctured_glue}
Let $N_0 = N \backslash \{p\}$ as above.
Consider the category $\hB$ 
consisting of the same objects as $\hatZCYsk(N_0)$,
but morphisms given by
\[
  \Hom_\hB(Y,Y') = \im(\Hom_{\hatZCYsk(N_0)}(Y,Y') \rcirclearrowleft \pi)
\]
Then the restriction to $\hB$ of the inclusion functor
corresponding to $i: N_0 \hookrightarrow N$
is an equivalence:
\[
  i_*|_\hB : \hB \simeq \hatZCYsk(N)
\]
The Karoubi envelope $\cB := \Kar(\hB)$ is obtained from $\ZCYsk(N_0)$ the same way:
same objects, and morphisms are
\[
  \Hom_\cB(Y,Y') = \im(\Hom_{\ZCYsk(N_0)}(Y,Y') \rcirclearrowleft \pi)
\]
and the equivalence above extends to an equivalence
\[
  i_*|_\cB : \cB \simeq \ZCYsk(N)
\]
\end{proposition}

\begin{proof}
First note that $\hB$ is indeed closed under composition
of morphisms because $\pi$ is idempotent:
\begin{equation}
(\Gamma' \ract \pi) \circ (\Gamma \ract \pi)
= (\Gamma' \circ \Gamma) \ract (\pi \circ \pi)
= (\Gamma' \circ \Gamma) \ract \pi
\label{e:hB-closed-composition}
\end{equation}
This also immediately implies that $\Kar(\hB)$ is indeed
has the description in the proposition statement.

It is clear that $i_*|_\hB$ is essentially surjective.
To prove fully faithfulness, consider two objects
$Y,Y' \in \hatZCYsk(N_0)$.
Abusing notation, we also denote $i_*(Y), i_*(Y')
\in \Obj \hatZCYsk(N_0)$ by $Y,Y'$.
We call the vertical segment
$p\times \clI \subset N \times \clI$
the \emph{pole},
so that $N_0 \times \clI = N \times \clI 
\backslash $pole.

We construct an inverse map to $i_*$.
Let $U$ be a small open neighborhood of $p$ in $N$,
and let $\cN = U \times \clI \subset N \times \clI$
be a small open neighborhood of the pole.
Choose $U$ small enough so that it does not contain any
marked points of $Y, Y'$.
Consider a graph $\Gamma \in \Graph(N \times \clI; Y^*,Y')$.
Define $j(\Gamma)$ as follows: if $\Gamma$ intersects the pole,
then use an isotopy supported in $\cN$
to push $\Gamma$ off of it, resulting in a new graph $\Gamma'$.
Now $\Gamma'$ can be considered a graph in
$\Graph(N_0 \times \clI; Y^*,Y')$.
Then we define $j(\Gamma) = \Gamma' \ract \pi$.

We need to check that $j$ is well-defined.
Firstly, the (linear combination of) graphs
$\Gamma' \ract \pi$ is independent of the choice of isotopy
- this follows from the sliding lemma (\lemref{lem:sliding}).
More generally, it means that for any isotopy $\vphi$ of
$N\times \clI$ supported on $\cN$,
$j(\Gamma) = j(\vphi(\Gamma))$.

Now we check that $j$ sends null graphs to 0.
Take the two set open cover $\{\cN, N_0\times \clI\}$
of $N \times \clI$, and apply \prpref{prp:null_cover}.
Let $\Gamma = \sum c_i\Gamma_i$ be null with respect to
some ball $D$.
If $D \subset N_0\times \clI$, clearly
$j(\Gamma)$ is null with respect to $D$.
If $D \subset \cN$, we may assume $D$ doesn't touch the
boundary (by choice of $U$),
so we can isotope it with some isotopy $\vphi$
supported on $\cN$ so that $\vphi(D)$ doesn't meet the pole.
Then clearly $j(\vphi(\Gamma))$ is null with respect to $\vphi(D)$.

Finally, it is easy to see that $j$ is inverse to $i_*$.
For example, $i_* \circ j$ amounts
to adding a trivial dashed circle, which is equivalent to 1
by \lemref{l:dashed_circle}.
\end{proof}

\subsection{The Sphere, $S^2$}
\par \noindent

\begin{definition}[\ocite{muger}*{Definition 2.9}]
The \emph{M\"uger center} of a braided tensor category $\cA$,
denoted $\ZMu(\cA)$,
is the full subcategory consisting of \emph{transparent objects},
i.e. objects $X$ such that for all objects $Y\in \cA$,
\[
c_{Y,X} \circ c_{X,Y} = \id_{XY}
\]
\label{d:muger}
\end{definition}

\begin{theorem}
The category of boundary values associated to the sphere is
the Muger center:
\[
\ZCYsk(S^2) \simeq \ZMu(\cA)
\]
and in particular, when $\cA$ is modular,
$\ZCYsk(S^2) \simeq \Vect$.
\label{t:sphere-muger}
\end{theorem}

\begin{proof}
Think of the disk $\disk$ as a punctured sphere,
so by \prpref{prp:punctured_glue}, we have that
$\hatZCYsk(S^2) \simeq \hB$,
where, as in \prpref{prp:punctured_glue}, $\hB$ is the category
with the same objects as $\hatZCYsk(\disk) \simeq \cA$,
but morphisms are, for $A,A'\in \cA$,
\[
\Hom_\hB(A,A') =
\Bigg\{
\frac{1}{\cD}
\begin{tikzpicture}
\draw[regular] (-0.3,-0.1) to[out=90,in=90] (0.3,-0.1);
\draw[overline] (0,0.6) -- (0,-0.6);
\node[small_morphism] at (0,0.3) {\tiny $f$};
\draw[overline,regular] (-0.3,-0.1) to[out=-90,in=-90] (0.3,-0.1);
\end{tikzpicture}
\st f \in \Hom_\cA(A,A')
\Bigg\}
\]
In particular, when $A = A' = X_i$ a simple object,
it follows from \ocite{muger}*{Corollary 2.14} that
simple objects that are not transparent,
i.e. not in the M\"uger center,
are killed:
\[
\End_\hB(X_i) =
\begin{cases}
  \End_\cA(X_i) & \text{ if } X_i \in \ZMu(\cA)\\
  0 & \text{ otherwise}
\end{cases}
\]
It follows that $\hB$ coincides with the M\"uger center,
which is already abelian,
and so $\ZCYsk(S^2) \simeq \Kar(\hB) = \ZMu(\cA)$.
When $\cA$ is modular,
only the unit object is transparent, so
$\ZCYsk(S^2) \simeq \ZMu(\cA) \simeq \Vect$.
\end{proof}

\subsection{The Once-Punctured Torus, $\punctorus$}
\label{s:elliptic}
\par \noindent

In \ocite{tham_elliptic}, we defined a category, which we called
the \emph{elliptic Drinfeld center},
and claimed that it is equivalent to the category of boundary values
associated to the once-punctured torus.
We then proved this claim in \ocite{KT}.
We present the results here.

\begin{definition}
\label{d:elliptic-center}
Let $\cA$ be a premodular category.
The \emph{elliptic Drinfeld center of $\cA$},
denoted $\Zel(\cA)$, is defined as follows:
the objects are tuples
$(A,\lambda^1,\lambda^2)$, where $\lambda^1,\lambda^2$
are half-braidings on $A$ that satisfy
\begin{align} \label{e:COMM}
\begin{tikzpicture}
\begin{scope} [shift={(0,-0.7)}]
\node[small_morphism] (lmb1) at (0,0.4) {\tiny $\lmb^1$};
\node[small_morphism] (lmb2) at (0,1) {\tiny $\lmb^2$};
\draw (0,0) -- (lmb1)
  node[pos=0,below] {$A$};
\draw (lmb1) -- (lmb2);
\draw (lmb2) -- (0,1.4)
  node[pos=1,above] {$A$};
\draw (lmb1) to[out=180,in=-90] (-1,1.4);
\draw (lmb1) to[out=0,in=90] (1,0);
\draw (lmb2) to[out=180,in=-90] (-0.5,1.4);
\draw[overline] (lmb2) to[out=0,in=90] (0.5,0);
\end{scope}
\end{tikzpicture}
=
\begin{tikzpicture}
\begin{scope} [shift={(0,-0.7)}]
\node[small_morphism] (lmb1) at (0,1) {\tiny $\lmb^1$};
\node[small_morphism] (lmb2) at (0,0.4) {\tiny $\lmb^2$};
\draw (0,0) -- (lmb2)
  node[pos=0,below] {$A$};
\draw (lmb2) -- (lmb1);
\draw (lmb1) -- (0,1.4)
  node[pos=1,above] {$A$};
\draw (lmb2) to[out=180,in=-90] (-0.5,1.4);
\draw (lmb2) to[out=0,in=90] (0.5,0);
\draw[overline] (lmb1) to[out=180,in=-90] (-1,1.4);
\draw (lmb1) to[out=0,in=90] (1,0);
\end{scope}
\end{tikzpicture}
\end{align}
We call the relation \eqnref{e:COMM} ``COMM''.
The morphisms $\Hom_\ZZA((A,\lmb^1,\lmb^2),(A',\mu^1,\mu^2)$
are morphisms of $\cA$ that intertwine both half-braidings,
i.e.
\[
  \Hom_\ZZA((A,\lmb^1,\lmb^2),(A',\mu^1,\mu^2))
    :=
  \Hom_{\ZA}((A,\lmb^1),(A',\mu^1))
  \cap \Hom_{\ZA}((A,\lmb^2),(A',\mu^2))
\]
\comment{
\[
  \Hom_\ZZA((A,\lmb^1,\lmb^2),(A',\mu^1,\mu^2)
    := \{\vphi \in \Hom_\cA(A,A') \st
        (\id_{A'}\tnsr \vphi) \circ \lmb^i = \mu^i \circ (\vphi \tnsr \id_A)\}
\]
}

\end{definition}

In \ocite{tham_elliptic}, we proved the following:

\begin{proposition}[\ocite{tham_elliptic}, Prop 3.4]
$\ZZA$ is a finite semisimple category.
\end{proposition}

\begin{proposition}[\ocite{tham_elliptic}, Prop 3.5, Prop 3.8]
\label{prp:elliptic_adjoint}
The forgetful functor $\mathcal{F}^\text{el} : \ZZA \to \cA$ has a two-sided
adjoint $\Iel: \cA \to \ZZA$,
where on objects, $\Iel$ sends
\[
  A \mapsto (\dirsum_{i,j} X_iX_jAX_j^*X_i^*, \Gamma^1, \Gamma^2)
\]
where
\[
\Gamma^1 =
\begin{tikzpicture}
\draw (0,-1) -- (0,1);
\draw (-0.3,-1) -- (-0.3,1);
\draw (0.3,-1) -- (0.3,1);
\node[small_morphism] (al1) at (-0.6,0) {\tiny $\albar$};
\draw (al1) -- (-0.6,-1);
\draw (al1) -- (-0.6,1);
\draw (al1) to[out=180,in=-90] (-1.2,1);
\node[small_morphism] (al2) at (0.6,0) {\tiny $\albar$};
\draw (al2) -- (0.6,-1);
\draw (al2) -- (0.6,1);
\draw (al2) to[out=0,in=90] (1.2,-1);
\end{tikzpicture}
,
\Gamma^2 =
\begin{tikzpicture}
\draw (0,-1) -- (0,1);
\draw (0.6,-1) -- (0.6,1);
\node[small_morphism] (al1) at (-0.3,0) {\tiny $\albar$};
\draw (al1) -- (-0.3,-1);
\draw (al1) -- (-0.3,1);
\draw (al1) to[out=180,in=-90] (-1.2,1);
\draw[overline] (-0.6,-1) -- (-0.6,1);
\node[small_morphism] (al2) at (0.3,0) {\tiny $\albar$};
\draw (al2) -- (0.3,-1);
\draw (al2) -- (0.3,1);
\draw[overline] (al2) to[out=0,in=90] (1.2,-1);
\end{tikzpicture}
\]
where $\albar$ is defined in \lemref{l:halfbrd}.

On morphisms, $f\in \Hom_\cA(A,A')$,
\[
  \Iel(f) = \dirsum_{i,j} \id_{X_i X_j} \tnsr f \tnsr \id_{X_j^* X_i^*}
\]

We refer to \ocite{tham_elliptic} for the functorial isomorphisms
giving the adjunction.

Furthermore, $\Iel$ is dominant.
\end{proposition}

\begin{theorem}[\ocite{tham_elliptic}, Theorem 4.3]
\label{t:elliptic-modular}
When $\cA$ is modular,
there is an equivalence
\begin{align*}
\cA &\simeq \ZZA \\
A &\mapsto (\bigoplus_i X_i A X_i^*, \Gamma, \Omega)
\end{align*}
where $\Gamma$ is the half-braiding on $I(X)$ in \thmref{t:center-adjoint},
and $\Omega = c^\inv \tnsr c^\inv \tnsr c$,
where $c_{-,-}$ is the braiding on $\cA$.
\end{theorem}

These results were proved purely algebraically
in \ocite{tham_elliptic}.
We will reprove them here, making use of the dictionary
between algebra and topology developed in previous sections.

\begin{proposition}
\label{p:hA-bimodule}
Let $\hA = \htr(\cA)$ as in \secref{s:ann-htr}.
Consider an $\cA$-bimodule structure on $\hA$,
where on objects it is the usual
$X \lact A \ract Y = X \tnsr A \tnsr Y$,
and on morphisms,
for $f\in \Hom_\cA(X,X'), g \in \Hom_\cA(Y,Y')$
and $\vphi \in \ihom{\hA}{B}(A,A')$,
we have

\[
\begin{tikzpicture}
\begin{scope}[shift={(0,-1.5)}]
\node at (0,3) {$\Hom_\cA(X,X')$};
\node[small_morphism] (f) at (0,1) {\small $f$};
\draw (f) -- +(90:1cm) node[pos=1,above] {$X$};
\draw (f) -- +(-90:1cm) node[pos=1,below] {$X'$};
\end{scope}
\end{tikzpicture}
\begin{tikzpicture}
\begin{scope}[shift={(0,-1.5)}]
\node at (0,3) {$\tnsr$};
\node at (0,1) {$\tnsr$};
\end{scope}
\end{tikzpicture}
\begin{tikzpicture}
\begin{scope}[shift={(0,-1.5)}]
\node at (0,3) {$\Hom_{\hA}(A,A')$};
\node[small_morphism] (ph) at (0,1) {\small $\vphi$};
\draw (ph) -- +(90:1cm) node[pos=1,above] {$A$};
\draw (ph) -- +(-90:1cm) node[pos=1,below] {$A'$};
\draw[midarrow_rev] (ph) -- +(150:1cm) node[pos=1,above left] {\small $B$};
\draw[midarrow] (ph) -- +(-30:1cm) node[pos=1,below right] {\small $B$};
\end{scope}
\end{tikzpicture}
\begin{tikzpicture}
\begin{scope}[shift={(0,-1.5)}]
\node at (0,3) {$\tnsr$};
\node at (0,1) {$\tnsr$};
\end{scope}
\end{tikzpicture}
\begin{tikzpicture}
\begin{scope}[shift={(0,-1.5)}]
\node at (0,3) {$\Hom_\cA(X,X')$};
\node[small_morphism] (g) at (0,1) {\small $g$};
\draw (g) -- +(90:1cm) node[pos=1,above] {$Y$};
\draw (g) -- +(-90:1cm) node[pos=1,below] {$Y'$};
\end{scope}
\end{tikzpicture}
\begin{tikzpicture}
\begin{scope}[shift={(0,-1.5)}]
\node at (0,3) {$\to$};
\node at (0,1) {$\mapsto$};
\end{scope}
\end{tikzpicture}
\begin{tikzpicture}
\begin{scope}[shift={(0,-1.5)}]
\node at (0,3) {$\Hom_{\hA}(X \tnsr A \tnsr Y, X' \tnsr A' \tnsr Y')$};
\node[small_morphism] (f) at (-0.4,0.55) {\tiny $f$};
\draw (f) -- (-0.4,2) node[above] {$X$};
\draw (f) -- (-0.4,0) node[below] {$X'$};
\node[small_morphism] (ph) at (0,1) {\tiny $\vphi$};
\draw[overline] (ph) -- +(90:1cm) node[pos=1,above] {$A$};
\draw (ph) -- +(-90:1cm) node[pos=1,below] {$A'$};
\draw[overline] (ph) -- +(150:1cm) node[pos=1,above left] {Y};
\draw (ph) -- +(-30:1cm) node[pos=1,below right] {Y};
\node[small_morphism] (g) at (0.4,1.45) {\tiny $g$};
\draw (g) -- (0.4,2) node[above] {$Y$};
\draw[overline] (g) -- (0.4,0) node[below] {$Y'$};
\end{scope}
\end{tikzpicture}
=:
f \lact \vphi \ract g
\]

The $\cA$-bimodule structure extends to $\ZA = \Kar(\hA)$,
\[
X \lact (A,\gamma) \ract Y = (X \tnsr A \tnsr Y, c \tnsr \gamma \tnsr c^\inv)
\]
(see \eqnref{e:tnsr_hfbrd_std} for $\tnsr$ of half-braidings),
and the action of morphisms is simply by tensor product.

Then we have
\[
\Kar(\htr(\hA)) =
\cZ_\cA(\hA) \simeq \cZ_\cA(\ZA) \simeq \Zel(\cA)
\]
In particular, by \prpref{p:ZM_ss}, $\Zel(\cA)$ is finite semisimple.
\label{p:hA-bimodule}
\end{proposition}

\begin{proof}
Let us first establish the last equivalence,
$\cZ_\cA(\ZA) \simeq \Zel(\cA)$.
An object in $\cZ_\cA(\ZA)$ is of the form $((A,\gamma),\gamma')$,
where $\gamma'$ is a half-braiding on $(A,\gamma)$ as an $\cA$-bimodule
category:
for $X \in \cA$, $\gamma'$ gives an isomorphism
\[
\gamma_X' : X \lact (A,\gamma) = (X \tnsr A, c \tnsr \gamma)
\simeq (A \tnsr X, \gamma \tnsr c^\inv) = (A,\gamma) \ract X
\]
Moreover, $\gamma_X'$ must be a $\ZA$-morphism,
and it is easy to see that
that is equivalent to $\gamma,\gamma'$ satisfying
the COMM relation \eqnref{e:COMM}.
Thus, $((A,\gamma),\gamma')$ defines an object
$(A,\gamma,\gamma') \in \Zel(\cA)$, and vice versa.
Hence, there is a bijection between the objects of
$\cZ_\cA(\ZA)$ and $\Zel(\cA)$.
It is also clear that
$\Hom_{\cZ_\cA(\ZA)}(((A,\gamma),\gamma'), ((B,\mu),\mu'))
= \Hom_{\Zel(\cA)}((A,\gamma,\gamma'), (B,\mu,\mu'))$
as subspaces of $\Hom_\cA(A,B)$.

The functor $G : \hA \to \ZA$
from \prpref{p:hA_ZA} admits a bimodule structure $J'$,
The structure $J'$ should provide a natural isomorphism
$J' : G(B \lact A \ract C) \simeq B \lact G(A) \ract C$.
It is easy to check that
$J' = c \tnsr \id \tnsr c$ works.
Then we have a functor
$\htr(G,J') : \htr(\hA) \to \htr(\ZA)$
which is dominant (because $G$ is dominant);
thus, by \lemref{lem:M_Karoubi},
the Karoubi envelopes are equivalent:
$\Kar(\htr(G,J')) : \cZ_\cA(\hA) \simeq \cZ_\cA(\ZA)$.

It is useful to work out the functors
$\htr(\hA) \to \htr(\ZA) \to \cZ_\cA(\ZA)$ explicitly.
For a morphism $\vphi \in \ihom{\hA}{X}(A,A')$,
Recall that, for $A,A' \in \Obj \cA$,
\[
  \Hom_{\htr(\hA)}(A,A')
  \cong \int^{B_2} \Hom_{\hA}(B_2 \lact A, A' \ract B_2)
  \cong \int^{B_2} \int^{B_1}
      \Hom_\cA(B_1 \lact B_2 \lact A, A' \ract B_2 \ract B_1)
\]
For a morphism $\vphi \in \ihom{\hA}{B_2}(A,A') \to \Hom_{\htr(\hA)}(A,A')$,

\begin{equation}
\label{e:htrhA-zza}
\begin{tikzpicture}
\begin{scope}[shift={(0,-1.5)}]
\node at (0,3.5) {$\ihom{\hA}{B_2}(A,A')$};
\node[morphism] (ph) at (0,1.5) {$\vphi$};
\draw (ph) -- +(90:1cm) node[pos=1,above] {$A$};
\draw (ph) -- +(-90:1cm) node[pos=1,below] {$A'$};
\draw[midarrow_rev] (ph) -- +(-150:1cm)
  node[pos=1,below left] {$B_1$};
\draw[midarrow_rev] (ph) -- +(150:1cm)
  node[pos=1,above left] {$B_2$};
\draw[midarrow] (ph) -- +(30:1cm)
  node[pos=1,above right] {$B_1$};
\draw[midarrow] (ph) -- +(-30:1cm)
  node[pos=1,below right] {$B_2$};
\end{scope}
\end{tikzpicture}
\begin{tikzpicture}
\node at (0,0) {$\mapsto$};
\node at (0,2) {$\to$};
\node at (0,2.3) {$\htr(G,J')$};
\end{tikzpicture}
\begin{tikzpicture}
\node at (0,1.4) {$A$};
\node at (0.6,1.4) {$i^*$};
\node at (-0.6,1.4) {$i$};
\node at (-1.2,1.4) {$B_2$};
\node at (0,-1.4) {$A'$};
\node at (-0.6,-1.4) {$i'$};
\node at (0.6,-1.4) {$i'^*$};
\node at (1.2,-1.4) {$B_2$};
\node at (0,2) {$\ihom{\ZA}{B_2}(G(A),G(A'))$};
\draw (-1.2,1.2) to[out=-90,in=90] (-0.6,0.5);
\draw[overline] (-0.6,1.2) to[out=-90,in=90] (-1.2,0.5);
\draw (0.6,1.2) to[out=-90,in=90] (1.2,0.5);
\draw (0.6,-1.2) to[out=90,in=-90] (1.2,-0.5);
\draw[overline] (1.2,-1.2) to[out=90,in=-90] (0.6,-0.5);
\draw (-0.6,-1.2) to[out=90,in=-90] (-1.2,-0.5);
\node[small_morphism] (ph) at (0,0) {$\vphi$};
\node[small_morphism] (al1) at (-1.2,-0.3) {\small $\bar{\al}$};
\node[small_morphism] (al2) at (1.2,0.3) {\small $\bar{\al}$};
\draw (ph) -- (0,1.2);
\draw (ph) -- (0,-1.2);
\draw[midarrow_rev] (ph) -- (al1)
  node[pos=0.5,above] {\tiny $B_1$};
\draw (ph) -- (al2);
\draw (al1) -- (-1.2,0.5);
\draw (al2) -- (1.2,-0.5);
\draw (-0.6,0.5) to[out=-90,in=135] (ph);
\draw (0.6,-0.5) to[out=90,in=-45] (ph);
\end{tikzpicture}
\begin{tikzpicture}
\node at (0,0) {$\mapsto$};
\node at (0,2) {$\to$};
\node at (0,2.3) {$I$};
\end{tikzpicture}
\begin{tikzpicture}
\node at (0,2) {$\Hom_{\cZ_\cA(\ZA)}(I(G(A)), I(G(A')))$};
\node at (0,1.4) {$A$};
\node at (-0.6,1.4) {$i$};
\node at (0.6,1.4) {$i^*$};
\node at (-1.2,1.4) {$j$};
\node at (1.2,1.4) {$j^*$};
\node at (0,-1.4) {$A'$};
\node at (-0.6,-1.4) {$i'$};
\node at (0.6,-1.4) {$i'^*$};
\node at (-1.2,-1.4) {$j'$};
\node at (1.2,-1.4) {$j'^*$};
\node[small_morphism] (ph) at (0,0) {$\vphi$};
\draw (ph) -- +(90:1.2cm);
\draw (ph) -- +(-90:1.2cm);
\node[small_morphism] (be1) at (-1.2,0.5) {\small $\bar{\beta}$};
\draw[midarrow_rev={0.3}] (ph) -- (be1)
  node[pos=0.2,above] {\tiny $B_2$};
\draw (be1) -- (-1.2,1.2);
\draw (be1) -- (-1.2,-1.2);
\node[small_morphism] (al1) at (-0.6,-0.5) {$\albar$};
\draw[midarrow_rev] (ph) -- (al1)
  node[pos=0.4,below] {\tiny $B_1$};
\draw[overline] (al1) -- (-0.6,1.2);
\draw (al1) -- (-0.6,-1.2);
\node[small_morphism] (al2) at (0.6,0.5) {$\albar$};
\draw (ph) -- (al2);
\draw (al2) -- (0.6,1.2);
\draw (al2) -- (0.6,-1.2);
\node[small_morphism] (be2) at (1.2,-0.5) {\small $\bar{\beta}$};
\draw[overline] (ph) -- (be2);
\draw (be2) -- (1.2,1.2);
\draw (be2) -- (1.2,-1.2);
\end{tikzpicture}
\end{equation}
where $I$ is the two-sided adjoint to the forgetful functor
(see \thmref{t:center-adjoint}, \thmref{t:htr-center}).
\end{proof}

It is easy to see that $I(G(A))$ is naturally isomorphic to
$\Iel(A)$.

\begin{proposition}
\label{prp:zza_punctorus}
Let $\punctorus$ be the once-punctured torus.
There is an equivalence
\[
  \ZCY(\punctorus) \simeq \ZZA
\]
Under this equivalence, the inclusion functor
$\cA \simeq \ZCY(\DD^2) \to \ZCY(\punctorus)$
is identified with $\Iel: \cA \to \ZZA$.
\end{proposition}

\begin{proof}
Think of the once-punctured torus as an open disk,
drawn like a `+' sign, with opposite sides identified
($\Ann = S^1 \times I$):
\[
\begin{tikzpicture}
\node at (1,1) {$\DD^2$};
\draw[fill=gray!30] (0.5,2) rectangle (1.5,2.5);
\draw[fill=gray!30,draw=none] (0.51,1.9) rectangle (1.49,2.1);
\node at (1,2.25) {2};
\draw (1.5,2) to[out=-90,in=180] (2,1.5);
\draw[fill=gray!30] (2,1.5) rectangle (2.5,0.5);
\draw[fill=gray!30,draw=none] (1.9,1.49) rectangle (2.1,0.51);
\node at (2.25,1) {3};
\draw (2,0.5) to[out=180,in=90] (1.5,0);
\draw[fill=gray!30] (1.5,0) rectangle (0.5,-0.5);
\draw[fill=gray!30,draw=none] (1.49,0.1) rectangle (0.51,-0.1);
\node at (1,-0.25) {4};
\draw (0.5,0) to[out=90,in=0] (0,0.5);
\draw[fill=gray!30] (0,0.5) rectangle (-0.5,1.5);
\draw[fill=gray!30,draw=none] (0.1,0.51) rectangle (-0.1,1.49);
\node at (-0.25,1) {1};
\draw (0,1.5) to[out=0,in=-90] (0.5,2);
\end{tikzpicture}
\begin{tikzpicture}
  \draw[->] (0,1) -- (1,1) node[pos=0.5,above] {glue 1,3};
\end{tikzpicture}
\begin{tikzpicture}
\node at (1,1) {$\Ann$};
\draw[fill=gray!30] (0.5,2) rectangle (1.5,2.5);
\draw[fill=gray!30,draw=none] (0.51,1.9) rectangle (1.49,2.1);
\node at (1,2.25) {2};
\draw (1.5,2) to[out=-90,in=180] (2,1.5);
\draw[regular] (2,1.5) -- (2.5,1.5);
\draw[regular] (2,0.5) -- (2.5,0.5);
\draw (2,0.5) to[out=180,in=90] (1.5,0);
\draw[fill=gray!30] (1.5,0) rectangle (0.5,-0.5);
\draw[fill=gray!30,draw=none] (1.49,0.1) rectangle (0.51,-0.1);
\node at (1,-0.25) {4};
\draw (0.5,0) to[out=90,in=0] (0,0.5);
\draw[regular] (0,0.5) -- (-0.5,0.5);
\draw[regular] (0,1.5) -- (-0.5,1.5);
\draw (0,1.5) to[out=0,in=-90] (0.5,2);
\end{tikzpicture}
\begin{tikzpicture}
  \draw[->] (0,1) -- (1,1) node[pos=0.5,above] {glue 2,4};
\end{tikzpicture}
\begin{tikzpicture}
\node at (1,1) {$A$};
\draw[regular] (1.5,2) -- (1.5,2.5);
\draw (1.5,2) to[out=-90,in=180] (2,1.5);
\draw[regular] (2,1.5) -- (2.5,1.5);
\draw[regular] (2,0.5) -- (2.5,0.5);
\draw (2,0.5) to[out=180,in=90] (1.5,0);
\draw[regular] (1.5,0) -- (1.5,-0.5);
\draw[regular] (0.5,0) -- (0.5,-0.5);
\draw (0.5,0) to[out=90,in=0] (0,0.5);
\draw[regular] (0,0.5) -- (-0.5,0.5);
\draw[regular] (0,1.5) -- (-0.5,1.5);
\draw (0,1.5) to[out=0,in=-90] (0.5,2);
\draw[regular] (0.5,2) -- (0.5,2.5);
\end{tikzpicture}
\]

The left most figure shows how $\ZCY(\DD^2) \simeq \cA$
is a module category over $\ZCY(I\times I) \simeq \cA$
in four ways;
we think of the 1,2 edges as acting on the left,
3,4 edges as acting on the right.
The actions are just usual left and right multiplication.

By \thmref{t:hat_excision},
the first ``glue 1,3" arrow induces an equivalence
\begin{equation*}
i_*^{1,3} : \htr_{\hatZCY(I\times I)}(\hatZCY(\DD^2)) \simeq \hatZCY(\Ann)
\end{equation*}

Again by \thmref{t:hat_excision},
the second ``glue 2,4" arrow induces an equivalence
\begin{equation*}
i_*^{2,4} :
\htr_{\hatZCY(I\times I)}(\hatZCY(\Ann))
\simeq \hatZCY(\punctorus)
\end{equation*}

It is easy to see that, under the equivalences
$\cA \simeq \hatZCY(I\times I) \simeq \hatZCY(\DD^2)$,
$\hA \simeq \hatZCY(\Ann)$,
the $\cA$-bimodule structure on $\hA$
and the $\hatZCY(I \times I)$-bimodule structure on
$\hatZCY(\Ann)$ are equivalent.
Thus we have, by \prpref{p:hA-bimodule},
\[
\begin{tikzcd}
\cA \ar[r,"\htr"] \ar[d,"\simeq"]
& \hA \ar[r,"\htr(G{,}J')"] \ar[d,"\simeq"]
& \htr(\hA) \ar[d,"\simeq"] \ar[rr,"\Kar"]
& & \ZZA \ar[d,"\simeq"]
\\
\hatZCY(\DD^2) \ar[r,"i_*^{1,3}"]
& \hatZCY(\Ann) \ar[r,"\htr"]
& \htr_{\hatZCY(I \times I)}(\hatZCY(\Ann)) \ar[r,"i_*^{2,4}", "\simeq"']
& \hatZCY(\punctorus) \ar[r,"\Kar"]
& \ZCY(\punctorus)
\end{tikzcd}
\]
As noted before, the top arrows compose to a functor
that is naturally isomorphic to $\Iel$.

Let us give a more explicit description of the equivalence
$\ZCY(\punctorus) \simeq \ZZA$.
The equivalence $\htr(\hA) \simeq \hatZCY(\punctorus)$
sends a morphism
$\vphi \in \Hom_\cA(B_1 \lact B_2 \lact A, A' \ract B_2 \ract B_1)
\to \Hom_{\htr(\hA)}(A,A')$
to the graph in $\punctorus \times [0,1]$ shown on the right
(which we represent more conveniently by the diagram in the middle):
\begin{equation}
\label{eqn:morphism_punctorus}
\begin{tikzpicture}
\begin{scope}[shift={(0,-1.5)}]
\node[morphism] (ph) at (0,1.5) {$\vphi$};
\draw (ph) -- +(90:1cm) node[pos=1,above] {$A$};
\draw (ph) -- +(-90:1cm) node[pos=1,below] {$A'$};
\draw[midarrow_rev] (ph) -- +(-150:1cm)
  node[pos=1,below left] {$B_1$};
\draw[midarrow_rev] (ph) -- +(150:1cm)
  node[pos=1,above left] {$B_2$};
\draw[midarrow] (ph) -- +(30:1cm)
  node[pos=1,above right] {$B_1$};
\draw[midarrow] (ph) -- +(-30:1cm)
  node[pos=1,below right] {$B_2$};
\end{scope}
\end{tikzpicture}
\mapsto
\begin{tikzpicture}
\begin{scope}[shift={(0,-1.5)}]
\node[morphism] (ph) at (0,1.5) {$\vphi$};
\draw (ph) -- +(90:1cm) node[pos=1,above] {$A$};
\draw (ph) -- +(-90:1cm) node[pos=1,below] {$A'$};
\draw[midarrow_rev] (ph) -- +(-150:1cm)
  node[pos=1,below left] {1}
  node[pos=0.5,below] {\tiny $B_1$};
\draw[midarrow_rev] (ph) -- +(150:1cm)
  node[pos=1,above left] {2}
  node[pos=0.5,above] {\tiny $B_2$};
\draw[midarrow] (ph) -- +(30:1cm)
  node[pos=1,above right] {3}
  node[pos=0.5,above] {\tiny $B_1$};
\draw[midarrow] (ph) -- +(-30:1cm)
  node[pos=1,below right] {4}
  node[pos=0.5,below] {\tiny $B_2$};
\end{scope}
\end{tikzpicture}
:=
\begin{tikzpicture}
\begin{scope}[shift={(0,-1.5)}]
\draw (-0.2,2.75) to[out=0,in=-60] (0.3,3);
\draw (1.3,3) to[out=-60,in=180] (2,2.75);
\draw (2.2,2.25) to[out=180,in=120] (1.7,2);
\draw (0.7,2) to[out=120,in=0] (0,2.25);
\draw (-0.2,0.75) to[out=0,in=-60] (0.3,1);
\draw (1.3,1) to[out=-60,in=180] (2,0.75);
\draw (2.2,0.25) to[out=180,in=120] (1.7,0);
\draw (0.7,0) to[out=120,in=0] (0,0.25);
\draw[dotted] (-0.2,0.75) -- (-0.2,2.75);
\draw[dotted] (0.3,1) -- (0.3,3);
\draw[dotted] (1.3,1) -- (1.3,3);
\draw[dotted] (2,0.75) -- (2,2.75);
\draw[dotted] (2.2,0.25) -- (2.2,2.25);
\draw[dotted] (1.7,0) -- (1.7,2);
\draw[dotted] (0.7,0) -- (0.7,2);
\draw[dotted] (0,0.25) -- (0,2.25);
\node[small_morphism] (ph) at (1,1.5) {\tiny $\vphi$};
\node[dotnode] (top) at (1,2.5) {};
\node[dotnode] (bottom) at (1,0.5) {};
\draw (ph) -- (top) node[pos=1,above] {$A^*$};
\draw (ph) -- (bottom) node[pos=1,below] {$A'$};
\draw[midarrow_rev] (ph) -- (-0.1,1.5)
  node[pos=1,left] {\small 1}
  node[pos=0.4,above] {\tiny $B_1$};
\draw (ph) -- (0.8,2)
  node[pos=1,above] {\small 2};
\draw (ph) -- (2.1,1.5)
  node[pos=1,right] {\small 3};
\draw[midarrow={0.6}] (ph) -- (1.2,1)
  node[pos=1,below] {\small 4};
\node at (1.3, 1.3) {\tiny $B_2$};
\end{scope}
\end{tikzpicture}
\end{equation}
As we can see, the two diagrams on the left are essentially the same,
except for the extra labels 1,2,3,4 in the middle diagram.
We may sometimes abuse notation and conflate them.

Then $\ZZA \simeq \ZCY(\punctorus)$
is given as follows:
on objects,
\begin{equation}
\label{e:P_elliptic}
(A,\lmb^1,\lmb^2) \mapsto
\im(P_{(A,\lmb^1,\lmb^2)})
\text{, where }
P_{(A,\lmb^1,\lmb^2)} :=
\frac{1}{\cD^2}
\begin{tikzpicture}
\begin{scope}[shift={(0,-1)}]
\node[small_morphism] (lmb1) at (0,1.3) {\tiny $\lmb^1$};
\node[small_morphism] (lmb2) at (0,0.7) {\tiny $\lmb^2$};
\draw (lmb1) -- (0,2) node[pos=1,above] {$A$};
\draw (lmb1) -- (lmb2);
\draw (lmb2) -- (0,0) node[pos=1,below] {$A$};
\draw[regular] (lmb2) -- +(150:1.5cm) node[pos=1,above left] {2};
\draw[regular] (lmb2) -- +(-30:0.6cm) node[pos=1,below right] {4};
\draw[thick_overline,regular] (lmb1) -- +(-150:1.5cm) node[pos=1,below left] {1};
\draw[regular] (lmb1) -- +(30:0.6cm) node[pos=1,above right] {3};
\end{scope}
\end{tikzpicture}
=
\frac{1}{\cD^2}
\begin{tikzpicture}
\begin{scope}[shift={(0,-1)}]
\node[small_morphism] (lmb1) at (0,0.7) {\tiny $\lmb^1$};
\node[small_morphism] (lmb2) at (0,1.3) {\tiny $\lmb^2$};
\draw (lmb2) -- (0,2) node[pos=1,above] {$A$};
\draw (lmb2) -- (lmb1);
\draw (lmb1) -- (0,0) node[pos=1,below] {$A$};
\draw[regular] (lmb1) -- +(-150:0.6cm) node[pos=1,below left] {1};
\draw[regular] (lmb1) -- +(30:1.5cm) node[pos=1,above right] {3};
\draw[regular] (lmb2) -- +(150:0.6cm) node[pos=1,above left] {2};
\draw[thick_overline,regular] (lmb2) -- +(-30:1.5cm) node[pos=1,below right] {4};
\end{scope}
\end{tikzpicture}
\end{equation}
where the equality of diagrams follows from the COMM requirement
\eqnref{e:COMM},
and the dashed line represents a weighted sum over simples
(see Notation~\ref{n:dashed}).
On morphisms,
\[
\Hom_\ZZA((A,\lmb^1,\lmb^2),(A',\mu^1,\mu^2)) \ni f
  \mapsto P_{(A',\mu^1,\mu^2)} \circ f \circ P_{(A,\lmb^1,\lmb^2)}
\]

Thus we have a 2-commutative diagram
\begin{equation}
\label{e:2comm}
\begin{tikzcd}
\cA \ar[r,"\Iel"] \ar[d,"\simeq"]
& \ZZA \ar[d,"\simeq"]
\\
\ZCY(\disk) \ar[r,"\text{incl.}_*"]
& \ZCY(\punctorus)
\end{tikzcd}
\end{equation}
\end{proof}

Next we will prove \thmref{t:elliptic-modular},
that, when $\cA$ is modular, $\ZZA \simeq \cA$.
We present a different proof,
one that uses the equivalence $\ZCY(\punctorus) \simeq \ZZA$
from \prpref{prp:zza_punctorus}.

\begin{proof}[Proof of \thmref{t:elliptic-modular}]
Under the equivalence \eqnref{e:P_elliptic},
the object $(\dirsum X_i A X_i^*, \Gamma, \Omega)$
is sent to
\begin{equation}
\label{e:P-Omega}
(\dirsum X_i A X_i^*, \Gamma, \Omega) \mapsto
\im \Bigg( \frac{1}{\cD^2}
\begin{tikzpicture}
\draw (0.3,1) -- (0.3,-1) 
  node[pos=0,above] {\small \text{ }$i^*$}
  node[pos=1,below] {\small \text{ }$j^*$};
\draw[overline,regular] (-1,0.5) -- (0.86,-0.7)
  node[pos=0,above left] {$2$}
  node[pos=1,below right] {$4$};
\draw[overline] (0,1) -- (0,-1)
  node[pos=0,above] {\small $A$}
  node[pos=1,below] {\small $A$};
\draw[overline] (-0.3,1) -- (-0.3,-1)
  node[pos=0,above] {\small $i$}
  node[pos=1,below] {\small $j$};
\node[small_morphism] (al) at (-0.3,-0.15) {\tiny $\albar$};
\draw[regular] (al) -- (-1,-0.5) node[pos=1,below left] {1};
\node[small_morphism] (al2) at (0.3,0.15) {\tiny $\albar$};
\draw[regular] (al2) -- (1,0.5) node[pos=1,above right] {3};
\end{tikzpicture}
\Bigg)
\cong
\im \Bigg(
\frac{1}{\cD}
\begin{tikzpicture}
\draw[regular] (-0.86,0.5) -- (0.86,-0.5)
  node[pos=0,above left] {2}
  node[pos=1,below right] {4};
\draw[thick_overline] (0,1) -- (0,-1)
  node[pos=0,above] {$A$}
  node[pos=1,below] {$A$};
\end{tikzpicture}
=: P_\Omega
\Bigg)
\end{equation}
where the isomorphism essentially follows from
\lemref{l:IM_htr_proj}.

Let $F: \ZCY(\DD^2) \to \ZCY(\punctorus)$ be the functor
that sends objects $A \mapsto \im(P_\Omega)$,
and morphisms $f \in \Hom_\cA(A,A')$
to $P_\Omega \circ f \circ P_\Omega$
(here we've identified $\cA \simeq \ZCY(\DD^2)$
by choosing a marked point,
so $A \in \ZCY(\DD^2)$ is the object with one marked point
labeled by $A$).
Note that in fact $P_\Omega \circ f \circ P_\Omega
= P_\Omega \circ f = f \circ P_\Omega$,
but this is not true in general if $f$ is replaced by some
arbitrary morphism of $\Hom_{\ZCY(\punctorus)}(F(A),F(A'))$.

We first show that $F$ is full.
Denote by $i_*(A) \in \ZCY(\punctorus)$ the object
with one marked point labeled by $A$,
that is, the object pushed forward under an inclusion
$i : \DD^2 \hookrightarrow \punctorus$.
By definition, $F(A)$ is a summand of $i_*(A)$,
as $P_\Omega$ is a projection on $i_*(A)$,
$P_\Omega : i_*(A) \to i_*(A)$.
Now
\[
\Hom_{\ZCY(\punctorus)}(F(A),F(A'))
= P_\Omega \circ \Hom_{\hatZCY(\punctorus)}(i_*(A),i_*(A')) \circ P_\Omega
\]
as subspaces of $\Hom_{\hatZCY(\punctorus)}(i_*(A),i_*(A'))$.
In other words, any morphism
$\vphi' \in \Hom_{\ZCY(\punctorus)}(F(A),F(A'))$
is of the form
$P_\Omega \circ \vphi \circ P_\Omega$.
Then

\begin{align}
\label{e:punctorus-modular-cut}
\begin{split}
&\vphi' = P_\Omega \circ \vphi \circ P_\Omega
=
\frac{1}{\cD^2}
\begin{tikzpicture}
\node[small_morphism] (ph) at (0,0) {$\vphi$};
\draw[regular] (0,-0.5) -- +(150:1.7cm)
  node[pos=1,left] {2};
\draw[regular,overline] (0,-0.5) -- +(-30:1.3cm)
  node[pos=1,right] {4};
\draw[overline] (ph) -- +(-90:1.5cm) node[pos=1,below] {$A'$};
\draw (ph) -- +(150:1.5cm)
  node[pos=1,left] {2};
\draw (ph) -- +(30:1.5cm)
  node[pos=1,right] {3};
\draw[regular] (0,0.5) -- +(150:1.3cm)
  node[pos=1,left] {2};
\draw[regular,overline] (0,0.5) -- +(-30:1.7cm)
  node[pos=1,right] {4};
\draw[overline] (ph) -- +(90:1.5cm) node[pos=1,above] {$A$};
\draw (ph) -- +(-30:1.5cm)
  node[pos=1,right] {4};
\draw[overline] (ph) -- +(-150:1.5cm)
  node[pos=1,left] {1};
\end{tikzpicture}
=
\frac{1}{\cD^2}
\begin{tikzpicture}
\node[small_morphism] (ph) at (0,0) {$\vphi$};
\node[small_morphism] (a1) at (-0.5,0.788) {\tiny $\al$};
\node[small_morphism] (a2) at (0.7,0.14) {\tiny $\al$};
\draw[regular] (a1) to[out=-90,in=150] (0.1,-0.5) to[out=-30,in=-90] (a2);
\draw[overline] (ph) -- +(-90:1.5cm) node[pos=1,below] {$A'$};
\draw (ph) -- +(150:1.5cm)
  node[pos=1,left] {2};
\draw (ph) -- +(30:1.5cm)
  node[pos=1,right] {3};
\draw[regular] (a1) -- +(150:0.7cm)
  node[pos=1,left] {2};
\draw[regular,overline] (a2) -- +(-30:0.9cm)
  node[pos=1,right] {4};
\draw[regular,overline] (a1) -- (a2);
\draw[overline] (ph) -- +(90:1.5cm) node[pos=1,above] {$A$};
\draw (ph) -- +(-30:1.5cm)
  node[pos=1,right] {4};
\draw[overline] (ph) -- +(-150:1.5cm)
  node[pos=1,left] {1};
\end{tikzpicture}
=
\frac{1}{\cD^2}
\begin{tikzpicture}
\node[small_morphism] (ph) at (0,0) {$\vphi$};
\draw[overline] (ph) -- +(-90:1.5cm) node[pos=1,below] {$A'$};
\draw (ph) -- +(150:1.5cm)
  node[pos=1,left] {2};
\draw (ph) -- +(30:1.5cm)
  node[pos=1,right] {3};
\draw[regular,overline] (0.35,0.2) circle (0.15cm);
\draw[overline] (ph) -- +(30:0.4cm);
\draw[regular] (0,0.7) -- +(150:1.3cm)
  node[pos=1,left] {2};
\draw[regular,overline] (0,0.7) -- +(-30:1.7cm)
  node[pos=1,right] {4};
\draw[overline] (ph) -- +(90:1.5cm) node[pos=1,above] {$A$};
\draw (ph) -- +(-30:1.5cm)
  node[pos=1,right] {4};
\draw[overline] (ph) -- +(-150:1.5cm)
  node[pos=1,left] {1};
\end{tikzpicture}
\\
&\;\;\;\;=
\frac{1}{\cD}
\begin{tikzpicture}
\node[small_morphism] (ph) at (0,0) {$\vphi$};
\draw[overline] (ph) -- +(-90:1.5cm) node[pos=1,below] {$A'$};
\draw (ph) -- +(150:1.5cm)
  node[pos=1,left] {2};
\draw[regular] (0,0.5) -- +(150:1.3cm)
  node[pos=1,left] {2};
\draw[regular,overline] (0,0.5) -- +(-30:1.7cm)
  node[pos=1,right] {4};
\draw[overline] (ph) -- +(90:1.5cm) node[pos=1,above] {$A$};
\draw (ph) -- +(-30:1.5cm)
  node[pos=1,right] {4};
\end{tikzpicture}
=
\frac{1}{\cD}
\begin{tikzpicture}
\node[small_morphism] (ph) at (0,0) {$\vphi$};
\node[small_morphism] (a1) at (-0.5,0.788) {\tiny $\al$};
\node[small_morphism] (a2) at (0.7,0.14) {\tiny $\al$};
\draw[overline] (ph) -- +(-90:1.5cm) node[pos=1,below] {$A'$};
\draw[regular] (a1) -- +(150:0.7cm)
  node[pos=1,left] {2};
\draw[regular,overline] (a2) -- +(-30:0.9cm)
  node[pos=1,right] {4};
\draw[regular,overline] (a1) -- (a2);
\draw[overline] (ph) -- +(90:1.5cm) node[pos=1,above] {$A$};
\draw (ph) to[out=150,in=-90] (a1);
\draw (ph) to[out=-30,in=-90] (a2);
\end{tikzpicture}
=
\frac{1}{\cD}
\begin{tikzpicture}
\node[small_morphism] (ph) at (0,-0.3) {$\vphi$};
\draw[overline] (ph) -- +(-90:1.2cm) node[pos=1,below] {$A'$};
\draw[regular] (0,0.5) -- +(150:1.5cm)
  node[pos=1,left] {2};
\draw[regular] (0,0.5) -- +(-30:1.5cm)
  node[pos=1,right] {4};
\draw (ph) to[out=150,in=180] (0,0.2) to[out=0,in=-30] (ph);
\draw[overline] (ph) -- +(90:1.8cm) node[pos=1,above] {$A$};
\draw[decorate,decoration={brace,amplitude=3pt,,raise=4pt}]
	(-0.2,-0.5) -- (-0.2,0.2);
\node at (-0.7,-0.15) {\smallerer $\vphi''$};
\end{tikzpicture}
=\vphi'' \circ P_\Omega = F(\vphi'')
\end{split}
\end{align}

so any $\vphi' \in \Hom_{\ZCY(\punctorus)}(F(A),F(A'))$
is equal to 
$\vphi'' \circ P_\Omega = F(\vphi'')$
for some $\vphi'' \in \Hom_\cA(A,A')$.

Next we show that $F$ is faithful.
Suppose some $f \in \Hom_\cA(A,A')$
is sent to $F(f) = f \circ P_\Omega = 0
\in \Hom_{\ZCY(\punctorus)}(F(A), F(A'))$,
viewed as a subspace of
$\Hom_{\hatZCY(\punctorus)}(i_*(A), i_*(A'))$
as above.
By \prpref{prp:zza_punctorus}
and \lemref{l:isom1},
\begin{align*}
\Hom_{\hatZCY(\punctorus)}(i_*(A), i_*(A'))
\cong \int^{B_2} \int^{B_1}
      \Hom_{\ZCY(\DD^2)}(B_1 \lact B_2 \lact A, A' \ract B_2 \ract B_1)
\cong
\bigoplus_{i, j \in \Irr(\cA)} 
      \Hom_\cA(X_i \lact X_j \lact A, A' \ract X_j \ract X_i)
\end{align*}
It is clear that
$f \circ P_\Omega \in 
\bigoplus_{j \in \Irr(\cA)} 
	\Hom_\cA(\one \lact X_j \lact A, A' \ract X_j \ract \one)$,
and in fact the $j=0$ (i.e. $X_j=\one$) component
of $f \circ P_\Omega$ is simply
$1/\cD \cdot f$.
Thus, since $f \circ P_\Omega = 0$, $f = 0$ as well.

Finally, we show that $F$ is essentially surjective.
Since $i_* : \ZCY(\DD^2) \to \ZCY(\punctorus)$ is dominant
by \corref{c:excision-dominant},
and $\ZCY(\DD^2) \simeq \cA$ is semisimple,
it suffices to show that any $i_*(A)$
is isomorphic to some $F(A')$.

Consider 
\begin{equation}
\label{e:proj-A-IA}
\sum_{i\in \Irr(\cA)}
\frac{\sqrt{d_i}}{\sqrt{\cD}}
\begin{tikzpicture}
\node at (0,0.8) {$A$};
\node at (0.1,-0.8) {$i\, A \, i^*$};
\draw (0,0.5) -- (0,-0.5);
\draw (0.2,-0.5) .. controls +(up:0.5cm) and +(150:0.5cm) .. (0.6,-0.3)
	node[pos=1,right] {4};
\draw (-0.2,-0.5) .. controls +(up:0.5cm) and +(-30:0.5cm) .. (-0.6,0.3)
	node[pos=1,left] {2};
\end{tikzpicture}
\;\;,\;\;
\sum_{i\in \Irr(\cA)}
\frac{\sqrt{d_i}}{\sqrt{\cD}}
\begin{tikzpicture}
\node at (0.1,0.8) {$i\, A \, i^*$};
\node at (0,-0.8) {$A$};
\draw (0,0.5) -- (0,-0.5);
\draw (-0.2,0.5) .. controls +(down:0.5cm) and +(-30:0.5cm) .. (-0.6,0.3)
	node[pos=1,left] {2};
\draw (0.2,0.5) .. controls +(down:0.5cm) and +(150:0.5cm) .. (0.6,-0.3)
	node[pos=1,right] {4};
\end{tikzpicture}
\end{equation}

which are analogs of the
$\hat{P}'_{(M,\gamma)}, \check{P}'_{(M,\gamma)}$,
in \lemref{l:IM_htr_proj}.
By similar reasoning, these morphisms are isomorphisms
to and from the object
\[
\im \Bigg(
\frac{1}{\cD}
\begin{tikzpicture}
\node at (0.1,0.8) {$i\, A \, i^*$};
\node at (0.1,-0.8) {$j\, A \, j^*$};
\node[small_morphism] (al1) at (-0.2,0.1) {\tiny $\albar$};
\node[small_morphism] (al2) at (0.2,-0.1) {\tiny $\albar$};
\draw (0,0.5) -- (0,-0.5);
\draw (al1) -- (-0.2,0.5);
\draw (al1) -- (-0.2,-0.5);
\draw (al2) -- (0.2,0.5);
\draw (al2) -- (0.2,-0.5);
\draw[regular] (al1) -- +(150:0.5cm)
	node[pos=1,left] {2};
\draw[regular] (al2) -- +(-30:0.5cm)
	node[pos=1,right] {4};
\end{tikzpicture}
\Bigg)
\;\;=\;\;
\im \Bigg(
\frac{1}{\cD}
\begin{tikzpicture}
\node at (0,0.8) {$B$};
\node at (0,-0.8) {$B$};
\node[small_morphism] (ga) at (0,0) {\tiny $\ga$};
\draw (ga) -- (0,0.5);
\draw (ga) -- (0,-0.5);
\draw[regular] (ga) -- +(150:0.5cm)
	node[pos=1,left] {2};
\draw[regular] (ga) -- +(-30:0.5cm)
	node[pos=1,right] {4};
\end{tikzpicture}
\Bigg)
\]
where of course $(B,\ga) = (\dirsum X_i A X_i^*, \Gamma)$.
Now, from \eqnref{e:modular-split},
there is an isomorphism
\[
\psi : (B,\ga) \simeq \dirsum_k (B_k' \tnsr B_k'',c^\inv \tnsr c)
\]
for some finite collection of objects $B_k',B_k''$.
Thus, the object above is a direct sum of objects
\[
\im \Bigg(
\frac{1}{\cD}
\begin{tikzpicture}
\node at (0,0.8) {$B_k' \, B_k''$};
\node at (0,-0.8) {$B_k' \, B_k''$};
\draw (0.2,0.5) -- (0.2,-0.5);
\draw[overline,regular] (-0.6,0.3) -- (0.6,-0.3)
	node[pos=0,left] {2}
	node[pos=1,right] {4};
\draw[overline] (-0.2,0.5) -- (-0.2,-0.5);
\end{tikzpicture}
\Bigg)
\;\;
\;\;
\begin{tikzpicture}
\begin{scope}[shift={(0,1)}]
\draw (-0.2,0.5) to[out=-90,in=90] (0.2,-0.5);
\draw (-0.2,-0.5) .. controls +(up:0.5cm) and +(30:0.5cm) .. (-0.6,-0.3)
	node[pos=1,left] {1};
\draw (0.2,0.5) .. controls +(down:0.5cm) and +(-150:0.5cm) .. (0.6,0.3)
	node[pos=1,right] {3};
\draw[-stealth] (-1.2,-1) -- (1.2,-1);
\node at (0,-1.2) {$\simeq$};
\end{scope}
\end{tikzpicture}
\;\;
\;\;
\im \Bigg(
\frac{1}{\cD}
\begin{tikzpicture}
\node at (0,0.8) {$B_k'' \, B_k'$};
\node at (0,-0.8) {$B_k'' \, B_k'$};
\draw[regular] (-0.6,0.3) -- (0.6,-0.3)
	node[pos=0,left] {2}
	node[pos=1,right] {4};
\draw[overline] (-0.2,0.5) -- (-0.2,-0.5);
\draw[overline] (0.2,0.5) -- (0.2,-0.5);
\end{tikzpicture}
\Bigg)
=
F(B_k'' \tnsr B_k')
\]

In summary, we have thing following chain of isomorphisms:
\[
\begin{tikzpicture}
\begin{scope}[shift={(0,0)}]
\node at (0,0.8) {$A$};
\draw (0,0.5) -- (0,-0.5);
\draw (0.2,-0.5) .. controls +(up:0.5cm) and +(150:0.5cm) .. (0.6,-0.3)
	node[pos=1,right] {4};
\draw (-0.2,-0.5) .. controls +(up:0.5cm) and +(-30:0.5cm) .. (-0.6,0.3)
	node[pos=1,left] {2};
\end{scope} 
\begin{scope}[shift={(0,-1)}] 
\node[small_morphism] (al1) at (-0.2,0.1) {\tiny $\albar$};
\node[small_morphism] (al2) at (0.2,-0.1) {\tiny $\albar$};
\draw (0,0.5) -- (0,-0.5);
\draw (al1) -- (-0.2,0.5);
\draw (al1) -- (-0.2,-0.5);
\draw (al2) -- (0.2,0.5);
\draw (al2) -- (0.2,-0.5);
\draw[regular] (al1) -- +(150:0.5cm)
	node[pos=1,left] {2};
\draw[regular] (al2) -- +(-30:0.5cm)
	node[pos=1,right] {4};
\node at (-1.6,0) {$\im \Bigg( {\LARGE \frac{1}{\cD}}$};
\node at (1.2,0) {$\Bigg)$};
\end{scope} 
\begin{scope}[shift={(0,-2)}] 
\node[small_morphism] (psi) at (0,0) {\tiny $\psi$};
\draw (psi) to[out=120,in=-90] (-0.2,0.5);
\draw (psi) to[out=90,in=-90] (0,0.5);
\draw (psi) to[out=60,in=-90] (0.2,0.5);
\draw (psi) to[out=-120,in=90] (-0.2,-0.5);
\draw (psi) to[out=-60,in=90] (0.2,-0.5);
\end{scope} 
\begin{scope}[shift={(0,-3)}] 
\draw (0.2,0.5) -- (0.2,-0.5);
\draw[overline,regular] (-0.6,0.3) -- (0.6,-0.3)
	node[pos=0,left] {2}
	node[pos=1,right] {4};
\draw[overline] (-0.2,0.5) -- (-0.2,-0.5);
\node at (-1.8,0) {$\dirsum_k \im \Bigg( {\LARGE \frac{1}{\cD}}$};
\node at (1.2,0) {$\Bigg)$};
\end{scope} 
\begin{scope}[shift={(0,-4)}] 
\draw (-0.2,0.5) to[out=-90,in=90] (0.2,-0.5);
\draw (-0.2,-0.5) .. controls +(up:0.5cm) and +(30:0.5cm) .. (-0.6,-0.3)
	node[pos=1,left] {1};
\draw (0.2,0.5) .. controls +(down:0.5cm) and +(-150:0.5cm) .. (0.6,0.3)
	node[pos=1,right] {3};
\end{scope} 
\begin{scope}[shift={(0,-5)}] 
\draw[regular] (-0.6,0.3) -- (0.6,-0.3)
	node[pos=0,left] {2}
	node[pos=1,right] {4};
\draw[overline] (-0.2,0.5) -- (-0.2,-0.5);
\draw[overline] (0.2,0.5) -- (0.2,-0.5);
\node at (0,-0.8) {$B_k'' \, B_k'$};
\node at (-1.8,0) {$\dirsum_k \im \Bigg( {\LARGE \frac{1}{\cD}}$};
\node at (1.2,0) {$\Bigg)$};
\end{scope} 
\end{tikzpicture}
\]
\end{proof}

Our next task is to upgrade this equivalence to an equivalence
of left $\ZCY(\Ann)$-modules.
In the rightmost figure in \eqnref{e:ell_obj}, the gray area
is a collar neighborhood of the puncture of $\punctorus$.
By \prpref{p:collared_module},
there is a (left) $\hatZCY(\Ann)$-module structure
on $\hatZCY(\punctorus)$: on objects,

\begin{equation}
\label{e:ell_obj}
\begin{tikzpicture}
\draw (0,0) circle (1cm);
\draw (0,0) circle (0.7cm);
\node[dotnode] at (0.6,0.6) {};
\node at (0.9,0.8) {$C$};
\end{tikzpicture}
\boxtimes
\begin{tikzpicture}
\begin{scope}[shift={(0,-1)}]
\node[dotnode] at (1,1) {};
\node at (1.2,1) {$A$};
\draw[regular] (1.5,2) -- (1.5,2.5);
\draw (1.5,2) to[out=-90,in=180] (2,1.5);
\draw[regular] (2,1.5) -- (2.5,1.5);
\draw[regular] (2,0.5) -- (2.5,0.5);
\draw (2,0.5) to[out=180,in=90] (1.5,0);
\draw[regular] (1.5,0) -- (1.5,-0.5);
\draw[regular] (0.5,0) -- (0.5,-0.5);
\draw (0.5,0) to[out=90,in=0] (0,0.5);
\draw[regular] (0,0.5) -- (-0.5,0.5);
\draw[regular] (0,1.5) -- (-0.5,1.5);
\draw (0,1.5) to[out=0,in=-90] (0.5,2);
\draw[regular] (0.5,2) -- (0.5,2.5);
\end{scope}
\end{tikzpicture}
\mapsto
\begin{tikzpicture}
\begin{scope}[shift={(0,-1)}]
\node[dotnode] at (1,1) {};
\node at (1.2,1) {$A$};
\draw[really_thick] (1.35,1.99) -- (1.35,2.5);
\draw[regular] (1.5,2) -- (1.5,2.5);
\draw[really_thick] (1.35,2) to[out=-90,in=180] (2,1.35);
\draw (1.5,2) to[out=-90,in=180] (2,1.5);
\draw[really_thick] (1.99,1.35) -- (2.5,1.35);
\draw[regular] (2,1.5) -- (2.5,1.5);
\draw[really_thick] (1.99,0.65) -- (2.5,0.65);
\draw[regular] (2,0.5) -- (2.5,0.5);
\draw[really_thick] (2,0.65) to[out=180,in=90] (1.35,0);
\draw (2,0.5) to[out=180,in=90] (1.5,0);
\draw[really_thick] (1.35,0.01) -- (1.35,-0.5);
\draw[regular] (1.5,0) -- (1.5,-0.5);
\draw[really_thick] (0.65,0.01) -- (0.65,-0.5);
\draw[regular] (0.5,0) -- (0.5,-0.5);
\draw[really_thick] (0.65,0) to[out=90,in=0] (0,0.65);
\draw (0.5,0) to[out=90,in=0] (0,0.5);
\draw[really_thick] (0.01,0.65) -- (-0.5,0.65);
\draw[regular] (0,0.5) -- (-0.5,0.5);
\draw[really_thick] (0.01,1.35) -- (-0.5,1.35);
\draw[regular] (0,1.5) -- (-0.5,1.5);
\draw[really_thick] (0,1.35) to[out=0,in=-90] (0.65,2);
\draw (0,1.5) to[out=0,in=-90] (0.5,2);
\draw[really_thick] (0.65,1.99) -- (0.65,2.5);
\draw[regular] (0.5,2) -- (0.5,2.5);
\node[dotnode] at (0.45,0.45) {};
\node at (0.65,0.45) {$C$};
\end{scope}
\end{tikzpicture}
\end{equation}
while on morphisms, the module structure
is given as follows:

\begin{equation}
\begin{tikzpicture}
\begin{scope}[shift={(0,-1.5)}]
\node at (0,3) {$\Hom_{\hatZCY(\Ann)}(C,C')$};
\node[small_morphism] (psi) at (0,1) {\small $\psi$};
\draw (psi) -- +(90:1cm) node[pos=1,above] {$C$};
\draw (psi) -- +(-90:1cm) node[pos=1,below] {$C'$};
\draw[midarrow_rev={0.7}] (psi) -- +(150:0.5cm) node[pos=1,above left] {$D$};
\draw[midarrow={0.7}] (psi) -- +(-30:0.5cm) node[pos=1,below right] {$D$};
\end{scope}
\end{tikzpicture}
\begin{tikzpicture}
\begin{scope}[shift={(0,-1.5)}]
\node at (0,3) {$\tnsr$};
\node at (0,1) {$\tnsr$};
\end{scope}
\end{tikzpicture}
\begin{tikzpicture}
\begin{scope}[shift={(0,-1.5)}]
\node at (0,3) {$\Hom_{\hatZCY(\punctorus)}(A,A')$};
\node[small_morphism] (ph) at (0,1) {\small $\vphi$};
\draw (ph) -- +(90:1cm) node[pos=1,above] {$A$};
\draw (ph) -- +(-90:1cm) node[pos=1,below] {$A'$};
\draw[midarrow_rev] (ph) -- +(-150:1cm) node[pos=1,below left] {1}
                          node[pos=0.5,below] {\small $B_1$};
\draw[midarrow_rev] (ph) -- +(150:1cm) node[pos=1,above left] {2}
                         node[pos=0.5,above] {\small $B_2$};
\draw[midarrow]  (ph) -- +(30:1cm) node[pos=1,above right] {3}
                        node[pos=0.5,above] {\small $B_1$};
\draw[midarrow] (ph) -- +(-30:1cm) node[pos=1,below right] {4}
                         node[pos=0.5,below] {\small $B_2$};
\end{scope}
\end{tikzpicture}
\begin{tikzpicture}
\begin{scope}[shift={(0,-1.5)}]
\node at (0,3) {$\to$};
\node at (0,1) {$\to$};
\end{scope}
\end{tikzpicture}
\begin{tikzpicture}
\begin{scope}[shift={(0,-1.5)}]
\node at (0,3) {$\Hom_{\hatZCY(\punctorus)}(A,A')$};
\draw[midarrow_rev] (-0.93, 0.6) to[out=30,in=-30] (-0.93,1.4); 
\draw[midarrow_rev={0.7}] (-0.8,1.6) to[out=-30,in=-150] (0.8,1.6);
\draw[midarrow_rev] (0.93, 1.4) to[out=-150,in=150] (0.93,0.6);
\node[small_morphism] (ph) at (0,1) {\tiny $\vphi$};
\draw[overline] (ph) -- +(90:1cm) node[pos=1,above] {$A$};
\draw (ph) -- +(-90:1cm) node[pos=1,below] {$A'$};
\draw (ph) -- +(-150:1cm) node[pos=1,below left] {1};
\draw (ph) -- +(150:1cm) node[pos=1,above left] {2};
\draw (ph) -- +(30:1cm) node[pos=1,above right] {3};
\draw (ph) -- +(-30:1cm) node[pos=1,below right] {4};
\node[small_morphism] (psi) at (-0.3,0.55) {\tiny $\psi$};
\draw[overline] (psi) -- (-0.3,2) node[above] {$C$};
\draw (psi) -- (-0.3,0) node[below] {$C'$};
\draw (psi) to[out=-170,in=30] (-0.8,0.4);
\draw[midarrow,overline] (psi) to[out=10,in=150] (0.8,0.4);
\node at (0.4,0.4) {\tiny $D$};
\end{scope}
\end{tikzpicture}
\end{equation}
(The $D$-labeled strand originally goes around the annulus
in $\Ann \times [0,1]$;
after inserting into $\punctorus \times [0,1]$,
it wraps around like the gray area in \eqnref{e:ell_obj}).
This extends to a left $\ZCY(\Ann)$-module structure
on $\ZCY(\punctorus)$.

Similarly, there is a left $\hatZCY(\Ann)$-module structure
on $\hatZCY(\disk)$ (which extends to $\ZCY$):
\begin{equation}
\begin{tikzpicture}
\node at (0,3) {$\Hom_{\hatZCY(\Ann)}(C,C')$};
\node[small_morphism] (psi) at (0,1) {\small $\psi$};
\draw (psi) -- +(90:1cm) node[pos=1,above] {$C$};
\draw (psi) -- +(-90:1cm) node[pos=1,below] {$C'$};
\draw[midarrow_rev={0.7}] (psi) -- +(150:0.5cm) node[pos=1,above left] {$D$};
\draw[midarrow={0.7}] (psi) -- +(-30:0.5cm) node[pos=1,below right] {$D$};
\end{tikzpicture}
\begin{tikzpicture}
\node at (0,3) {$\tnsr$};
\node at (0,1) {$\tnsr$};
\end{tikzpicture}
\begin{tikzpicture}
\node at (0,3) {$\Hom_{\hatZCY(\disk)}(A,A')$};
\node[small_morphism] (ph) at (0,1) {\small $\vphi$};
\draw (ph) -- +(90:1cm) node[pos=1,above] {$A$};
\draw (ph) -- +(-90:1cm) node[pos=1,below] {$A'$};
\end{tikzpicture}
\begin{tikzpicture}
\node at (0,3) {$\to$};
\node at (0,1) {$\to$};
\end{tikzpicture}
\begin{tikzpicture}
\node at (0,3) {$\Hom_{\hatZCY(\disk)}(A,A')$};
\node[small_morphism] (ph) at (0,1) {\tiny $\vphi$};
\draw (ph) -- +(-90:1cm) node[pos=1,below] {$A'$};

\node[small_morphism] (psi) at (-0.3,0.7) {\tiny $\psi$};
\draw (psi) to[out=150,in=150] (0.3,1.3);
\draw[overline,midarrow_rev] (0.3,1.3) to[out=-30,in=-30] (psi);
\draw[overline] (psi) -- (-0.3,2) node[above] {$C$};
\draw (psi) -- (-0.3,0) node[below] {$C'$};
\node at (0.6,1) {\small $D$};

\draw[overline] (ph) -- +(90:1cm) node[pos=1,above] {$A$};
\end{tikzpicture}
\end{equation}

In light of \eqnref{e:2comm},
the following theorem is an upgrade
of \thmref{t:elliptic-modular}:

\begin{theorem}
\label{thm:modular_module}
Let $\cA$ be modular. There is an equivalence of
left $\ZCY(\Ann)$-modules
\[
  \ZCY(\disk) \simeq \ZCY(\punctorus)
\]
\end{theorem}

\begin{proof}
For $\psi \in \Hom_{\hatZCY(\Ann)}(C,C')$,
$\vphi' = \cD \cdot \vphi \in \Hom_{\hatZCY(\DD^2)}(A,A')$,
the action of $\psi$ on $F(\vphi')$
(where $F$ is the functor from
\thmref{t:elliptic-modular};
the factor of $\cD$ is to simplify the diagrams below)
is given by:
\begin{align*}
\psi \cdot F(\vphi) =
\begin{tikzpicture}
\draw (-0.93, -0.4) to[out=30,in=-30] (-0.93,0.4); 
\draw (-0.8,0.6) to[out=-30,in=-150] (0.8,0.6); 
\draw (0.93, 0.4) to[out=-150,in=150] (0.93,-0.4); 
\draw[regular] (-0.86,0.5) -- (0.86,-0.5)
  node[pos=0,above left] {2}
  node[pos=1,below right] {4};
\node[small_morphism] (ph) at (0,0.7) {\tiny $\vphi$};
\draw (ph) -- (0,1) node[pos=1,above] {$A$};
\draw[overline] (ph) -- (0,-1) node[pos=1,below] {$A'$};
\node[small_morphism] (psi) at (-0.3,-0.45) {\tiny $\psi$};
\draw[overline] (psi) -- (-0.3,1) node[above] {$C$};
\draw (psi) -- (-0.3,-1) node[below] {$C'$};
\draw (psi) to[out=-170,in=30] (-0.8,-0.6); 
\draw[overline] (psi) to[out=10,in=150] (0.8,-0.6); 
\node at (-1,-0.7) {1};
\node at (1,0.7) {3};
\end{tikzpicture}
=
\begin{tikzpicture}
\draw (-0.93, -0.4) to[out=30,in=-30] (-0.93,0.4); 
\draw[regular] (-0.86,0.5) -- (0.86,-0.5)
  node[pos=0,above left] {2}
  node[pos=1,below right] {4};
\node[small_morphism] (ph) at (0,0.7) {\tiny $\vphi$};
\draw (ph) -- (0,1) node[pos=1,above] {$A$};
\draw[overline] (ph) -- (0,-1) node[pos=1,below] {$A'$};
\node[small_morphism] (al1) at (-0.43,0.25) {\tiny $\al$};
\node[small_morphism] (al2) at (0.35,-0.15) {\tiny $\al$};
\draw (al1) to[out=0,in=-150] (0.8,0.6); 
\draw (al2) to[out=60,in=-150] (0.93,0.4); 
\node[small_morphism] (psi) at (-0.3,-0.45) {\tiny $\psi$};
\draw[overline] (psi) -- (-0.3,1) node[above] {$C$};
\draw (psi) -- (-0.3,-1) node[below] {$C'$};
\draw (psi) to[out=-170,in=30] (-0.8,-0.6); 
\draw[overline] (psi) to[out=10,in=150] (0.8,-0.6); 
\node at (-1,-0.7) {1};
\node at (1,0.7) {3};
\end{tikzpicture}
=
\begin{tikzpicture}
\draw (-0.93, -0.4) to[out=30,in=-30] (-0.93,0.4);
\draw (0.8,0.6) to[out=-150,in=150] (0.3,0.2); 
\draw (0.3,0.2) to[out=-30,in=-150] (0.93,0.4); 
\draw[regular] (-0.86,0.5) -- (0.86,-0.5)
  node[pos=0,above left] {2}
  node[pos=1,below right] {4};
\node[small_morphism] (ph) at (0,0.6) {\tiny $\vphi$};
\draw (ph) -- (0,1) node[pos=1,above] {$A$};
\draw[overline] (ph) -- (0,-1) node[pos=1,below] {$A'$};
\node[small_morphism] (psi) at (-0.3,-0.45) {\tiny $\psi$};
\draw[overline] (psi) -- (-0.3,1) node[above] {$C$};
\draw (psi) -- (-0.3,-1) node[below] {$C'$};
\draw (psi) to[out=-170,in=30] (-0.8,-0.6);
\draw[overline] (psi) to[out=10,in=150] (0.8,-0.6);
\node at (-1,-0.7) {1};
\node at (1,0.7) {3};
\end{tikzpicture}
=
\begin{tikzpicture}
\draw[regular] (-0.86,0.5) -- (0.86,-0.5)
  node[pos=0,above left] {2}
  node[pos=1,below right] {4};
\node[small_morphism] (ph) at (0,0.6) {\tiny $\vphi$};
\draw (ph) -- (0,1) node[pos=1,above] {$A$};
\draw[overline] (ph) -- (0,-1) node[pos=1,below] {$A'$};
\node[small_morphism] (psi) at (-0.3,-0.45) {\tiny $\psi$};
\draw[overline] (psi) -- (-0.3,1) node[above] {$C$};
\draw (psi) -- (-0.3,-1) node[below] {$C'$};
\draw (psi) to[out=-170,in=-30] (-0.93,0.4); 
\draw[overline] (psi) to[out=10,in=150] (0.8,-0.6);
\end{tikzpicture}
=
\begin{tikzpicture}
\draw[regular] (-0.86,0.5) to[out=-30,in=120] (0.86,-0.3);
\node at (-1,0.6) {\small 2};
\node at (1,-0.4) {\small 4};
\node[small_morphism] (ph) at (0,-0.3) {\tiny $\vphi$};
\draw (ph) -- (0,-1) node[pos=1,below] {$A'$};
\node[small_morphism] (psi) at (-0.3,-0.6) {\tiny $\psi$};
\draw (psi) to[out=150,in=150] (0.3,0);
\draw[overline] (0.3,0) to[out=-30,in=-30] (psi);
\draw[overline] (psi) -- (-0.3,1) node[above] {$C$};
\draw (psi) -- (-0.3,-1) node[below] {$C'$};
\draw[overline] (ph) -- (0,1) node[pos=1,above] {$A$};
\end{tikzpicture}
= F(\psi \cdot \vphi)
\end{align*}
The first and last equalities are by definition.
The second and third equalities follow from \lemref{l:summation}
in order to perform the ``sliding lemma'' as in
proof of \lemref{lem:sliding}.
The fourth equality is by isotopy.
The fifth equality applies the similar method of the
second and third equalities.
Hence, the equivalence does respect the module structure,
and we are done.
\end{proof}

Finally, we state the main result of this section:

\begin{theorem}\label{thm:cy_modular}
Let $\cA$ be modular.
Let $N$ be a connected compact oriented surface with
$b$ boundary components and genus $g$,
and let $S_{0,b} = S^2 \backslash (\DD^2)^{\sqcup b}$ be a genus 0 surface
with $b$ boundary components.
Then
\[
  \ZCY(N) \simeq \ZCY(S_{0,b})
\]
In particular,
$\ZCY(\text{closed surface}) \simeq \ZCY(S^2) \simeq \Vect$
and $\ZCY(\text{once-punctured surface}) \simeq \ZCY(\DD^2) \simeq \cA$.
\end{theorem}

\begin{proof}
Suppose $g>0$, so that we can present $N$
as a connect sum $N' \# \torus$,
where $N'$ is a connected compact oriented surface with
$b$ boundary components and genus $g-1$.
We think of the connect sum as
$N = N_0' \cup_{\Ann} (\punctorus)$,
where $N_0' = N' \backslash \{pt\}$
is a punctured surface.
Then by \thmref{t:excision-skein} and \thmref{thm:modular_module},
$\ZCY(N) \simeq \ZCY(N_0') \boxtimes_{\ZCY(\Ann)} \ZCY(\punctorus)
\simeq \ZCY(N_0') \boxtimes_{\ZCY(\Ann)} \ZCY(\DD^2)
\simeq \ZCY(N_0' \cup_{\Ann} \DD^2) = \ZCY(N')$.
Thus, by induction on the genus, we have
$\ZCY(N) \simeq \ZCY(S_{0,b})$.

The final statements follow from the $b=0,1$ cases.
\end{proof}

\subsection{Closed Surfaces}
\label{s:closed-surfaces}
\par \noindent

Let $\cA$ be modular.

Let us consider $\ZCY(N)$ for a closed surface
in more detail.
By \thmref{thm:cy_modular},
$\ZCY(N) \simeq \Vect$,
so a natural question to ask is: what do the simple objects,
i.e. the objects corresponding to $\kk \in \Vect$,
look like in $\ZCY(N)$?

\begin{definition}
Let $N$ be a surface,
and let $\xi$ be an embedded circle in $N$.
Let $E = i_*(\one) \in \hatZCY(N)$ be the empty configuration.
Define $P_\xi \in \End_{\hatZCY(N)}(E)$ to be the morphism
defined by the graph $\xi \times 1/2 \subset N \times [0,1]$,
colored with $1/\cD \cdot$ regular coloring.

Furthermore, we may extend this definition to
multicurves,
i.e. a finite collection of pairwise disjoint embedded circles
(see \ocite{farb-margalit}*{Section 1.2.4}).
Let $\Xi = \{\xi_i\}$ be a multicurve.
Then we define $P_\Xi = \prod P_{\xi_i}$.
\label{d:proj-multicurve}
\end{definition}

Note that $P_\Omega$ (with $A=\one$) from \eqnref{e:P-Omega}
in the proof of \thmref{t:elliptic-modular}
is $P_\xi$ for some loop.

Clearly, if two multicurves $\Xi,\Xi'$
are isotopic, then $P_{\Xi} = P_{\Xi'}$.
In particular, if a multicurve $\Xi$ has two components
$\xi_1,\xi_2 \in \Xi$ that are isotopic,
then $P_\Xi = P_{\Xi'}$, where $\Xi' = \Xi \backslash \{\xi_1\}$.

\begin{definition}
We say a multicurve in a closed surface $N$ is \emph{full}
if performing surgery along it
(removing a neighborhood and gluing in disks)
results in a sphere.
\label{d:multicurve-full}
\end{definition}

For example, each set of attaching curves in a Heegaard diagram
is full.
It is clear by definition that a full multicurve always looks,
under suitable diffeomorphism of $N$, as follows:
\[
\begin{tikzpicture}
\draw (0,-0.6)
	to[out=0,in=180] (2,-1.2)
	to[out=0,in=-90] (3.5,0)
	to[out=90,in=0] (2,1.2)
	to[out=180,in=0] (0,0.6)
	to[out=180,in=0] (-2,1.2)
	to[out=180,in=90] (-3.5,0)
	to[out=-90,in=180] (-2,-1.2)
	to[out=0,in=180] (0,-0.6);
\draw (-2.7,0.1) to[out=-45,in=-135] (-1.3,0.1);
\draw (-2.6,0) to[out=45,in=135] (-1.4,0);
\draw (2.7,0.1) to[out=-135,in=-45] (1.3,0.1);
\draw (2.6,0) to[out=135,in=45] (1.4,0);
\draw (-2,-0.2) to[out=180,in=180] (-2,-1.2);
\draw[dotted] (-2,-0.2) to[out=0,in=0] (-2,-1.2);
\draw (2,-0.2) to[out=180,in=180] (2,-1.2);
\draw[dotted] (2,-0.2) to[out=0,in=0] (2,-1.2);
\end{tikzpicture}
\]

\begin{proposition}
Let $N$ be a closed surface,
and let $E$ be the empty configuration (as in \defref{d:proj-multicurve}).
For any full multicurve $\Xi$,
the object $(E,P_\Xi) \in \ZCY(N)$ is simple.
In particular, all such objects are isomorphic.
\label{p:simple-object-multicurve}
\end{proposition}

\begin{proof}
It is easy to see that an element of
$\End_{\ZCY(N)}(E)$ is always equivalent
to a sum of graphs as follows:

\begin{equation}
\begin{tikzpicture}
\draw (0,-0.8)
	to[out=0,in=180] (2,-1.4)
	to[out=0,in=-90] (3.7,0)
	to[out=90,in=0] (2,1.4)
	to[out=180,in=0] (0,0.8)
	to[out=180,in=0] (-2,1.4)
	to[out=180,in=90] (-3.7,0)
	to[out=-90,in=180] (-2,-1.4)
	to[out=0,in=180] (0,-0.8);
\draw[opacity=0.4] (0,-0.6)
	to[out=0,in=180] (2,-1.2)
	to[out=0,in=-90] (3.5,0)
	to[out=90,in=0] (2,1.2)
	to[out=180,in=0] (0,0.6)
	to[out=180,in=0] (-2,1.2)
	to[out=180,in=90] (-3.5,0)
	to[out=-90,in=180] (-2,-1.2)
	to[out=0,in=180] (0,-0.6);
\draw[opacity=0.4] (-2.7,0.1) to[out=-55,in=-125] (-1.3,0.1);
\draw[opacity=0.4] (-2.6,0) to[out=55,in=125] (-1.4,0);
\draw (-2.5,0.1) to[out=-45,in=-135] (-1.5,0.1);
\draw (-2.4,0) to[out=45,in=135] (-1.6,0);
\draw[opacity=0.4] (2.7,0.1) to[out=-125,in=-55] (1.3,0.1);
\draw[opacity=0.4] (2.6,0) to[out=125,in=55] (1.4,0);
\draw (2.5,0.1) to[out=-135,in=-45] (1.5,0.1);
\draw (2.4,0) to[out=135,in=45] (1.6,0);
\node[small_morphism] (ph) at (0,-0.2) {\smallerer $\vphi$};
\draw[opacity=0.7] (ph)
	to[out=150,in=0] (-2,0.7)
	to[out=180,in=90] (-3,0)
	to[out=-90,in=180] (-2,-0.7)
	to[out=0,in=-170] (ph);
\draw[opacity=0.7] (-1.2,-1) to[out=0,in=-150] (ph) to[out=180,in=40] (-1.7,-0.15);
\draw[opacity=0.7, densely dotted] (-1.7,-0.15) to[out=-140,in=-180] (-1.2,-1);
\draw[opacity=0.7] (ph)
	to[out=30,in=180] (2,0.7)
	to[out=0,in=90] (3,0)
	to[out=-90,in=0] (2,-0.7)
	to[out=180,in=-10] (ph);
\draw[opacity=0.7] (1.2,-1) to[out=180,in=-30] (ph) to[out=0,in=140] (1.7,-0.15);
\draw[opacity=0.7, densely dotted] (1.7,-0.15) to[out=-40,in=0] (1.2,-1);
\end{tikzpicture}
\end{equation}
(each loop starting and ending at the node
represents a basis element of $H_1$.)

Then in $\End_{\ZCY(N)}((E,P_\Xi))$,
by killing and sliding lemma,
the graph can be ``shrunk'' into a ball:
\begin{equation}
\frac{1}{\cD^2}\;\;
\begin{tikzpicture}
\node[small_morphism] (ph) at (0,0) {\smallerer $\vphi$};
\draw[regular,opacity=0.7] (-1.7,-0.5) to[out=-80,in=80] (-1.7,-1);
\draw[regular,opacity=0.7] (-1.9,-0.5)
	to[out=120,in=0] (-2,-0.45)
	to[out=180,in=180] (-2,-1.05)
	to[out=0,in=-120] (-1.9,-1);
\draw[opacity=0.3, regular] (-1.9,-0.5) to[out=-60,in=60] (-1.9,-1);
\draw[opacity=0.3] (-1.4,-0.4) to[out=-90,in=150] (-1.2,-0.8);
\draw[opacity=0.7,thin_overline={1}] (ph)
	to[out=170,in=0] (-2,0.8)
	to[out=180,in=90] (-3.1,0)
	to[out=-90,in=180] (-2,-0.8)
	to[out=0,in=-170] (ph);
\draw[regular,opacity=0.7,thin_overline={1}] (-1.7,-0.5)
	to[out=110,in=0] (-2,-0.15)
	to[out=180,in=180] (-2,-1.35)
	to[out=0,in=-110] (-1.7,-1);
\draw (0,0.8)
	to[out=180,in=0] (-2,1.4)
	to[out=180,in=90] (-3.7,0)
	to[out=-90,in=180] (-2,-1.4)
	to[out=0,in=180] (0,-0.8);
\draw[opacity=0.4] (0,0.4)
	to[out=180,in=0] (-2,1)
	to[out=180,in=90] (-3.3,0)
	to[out=-90,in=180] (-2,-1)
	to[out=0,in=180] (0,-0.4);
\draw[opacity=0.4] (-2.9,0.1) to[out=-80,in=180]
	(-2,-0.5) to[out=0,in=-100] (-1,0.1);
\draw[opacity=0.4] (-2.87,0) to[out=70,in=180]
	(-2,0.5) to[out=0,in=110] (-1.03,0);
\draw (-2.5,0.1) to[out=-45,in=-135] (-1.5,0.1);
\draw (-2.4,0) to[out=45,in=135] (-1.6,0);
\draw[opacity=0.7] (ph) to[out=180,in=90] (-1.4,-0.4);
\draw[opacity=0.7] (-1.2,-0.8) to[out=-30,in=-120] (ph);
\end{tikzpicture}
\;\;=
\frac{1}{\cD^2}
\;\;
\begin{tikzpicture}
\node[small_morphism] (ph) at (0,0) {\smallerer $\vphi$};
\draw[regular,opacity=0.7] (-1.7,-0.5) to[out=-80,in=80] (-1.7,-1);
\draw[opacity=0.7] (ph)
	to[out=170,in=0] (-2,0.8)
	to[out=180,in=90] (-3.1,0)
	to[out=-90,in=180] (-2,-0.8)
	to[out=0,in=-170] (ph);
\draw[densely dotted,opacity=0.7,thin_overline={1}] (-1.95,-0.75)
	to[out=100,in=-0] (-2.025,-0.7)
	to[out=180,in=90] (-2.1,-0.8)
	to[out=-90,in=180] (-2.025,-0.9)
	to[out=0,in=-100] (-1.95,-0.85);
\draw[color=white, line width=3] (-1.4,-0.7) -- (-1.5,-0.5);
\draw[opacity=0.7] (ph)
	to[out=180,in=30] (-1.4,-0.5)
	to[out=-150,in=-150] (-1.3,-0.75)
	to[out=30,in=-150] (ph);
\draw[regular,opacity=0.7,thin_overline={1}] (-1.7,-0.5)
	to[out=110,in=0] (-2,-0.15)
	to[out=180,in=180] (-2,-1.35)
	to[out=0,in=-110] (-1.7,-1);
\draw (0,0.8)
	to[out=180,in=0] (-2,1.4)
	to[out=180,in=90] (-3.7,0)
	to[out=-90,in=180] (-2,-1.4)
	to[out=0,in=180] (0,-0.8);
\draw[opacity=0.4] (0,0.4)
	to[out=180,in=0] (-2,1)
	to[out=180,in=90] (-3.3,0)
	to[out=-90,in=180] (-2,-1)
	to[out=0,in=180] (0,-0.4);
\draw[opacity=0.4] (-2.9,0.1) to[out=-80,in=180]
	(-2,-0.5) to[out=0,in=-100] (-1,0.1);
\draw[opacity=0.4] (-2.87,0) to[out=70,in=180]
	(-2,0.5) to[out=0,in=110] (-1.03,0);
\draw (-2.5,0.1) to[out=-45,in=-135] (-1.5,0.1);
\draw (-2.4,0) to[out=45,in=135] (-1.6,0);
\end{tikzpicture}
\;\;=
\frac{1}{\cD}
\;\;
\begin{tikzpicture}
\node[small_morphism] (ph) at (0,0) {\smallerer $\vphi$};
\draw[regular,opacity=0.7] (-1.7,-0.5) to[out=-80,in=80] (-1.7,-1);
\draw[opacity=0.7] (ph)
	to[out=180,in=120] (-0.5,-0.2)
	to[out=-60,in=-150] (ph);
\draw[regular,opacity=0.7,thin_overline={1}] (-1.7,-0.5)
	to[out=110,in=0] (-2,-0.15)
	to[out=180,in=180] (-2,-1.35)
	to[out=0,in=-110] (-1.7,-1);
\draw (0,0.8)
	to[out=180,in=0] (-2,1.4)
	to[out=180,in=90] (-3.7,0)
	to[out=-90,in=180] (-2,-1.4)
	to[out=0,in=180] (0,-0.8);
\draw[opacity=0.4] (0,0.4)
	to[out=180,in=0] (-2,1)
	to[out=180,in=90] (-3.3,0)
	to[out=-90,in=180] (-2,-1)
	to[out=0,in=180] (0,-0.4);
\draw[opacity=0.4] (-2.9,0.1) to[out=-80,in=180]
	(-2,-0.5) to[out=0,in=-100] (-1,0.1);
\draw[opacity=0.4] (-2.87,0) to[out=70,in=180]
	(-2,0.5) to[out=0,in=110] (-1.03,0);
\draw (-2.5,0.1) to[out=-45,in=-135] (-1.5,0.1);
\draw (-2.4,0) to[out=45,in=135] (-1.6,0);
\end{tikzpicture}
\end{equation}

Thus, any skein in $\End_{\ZCY(N)}((E,P_\Xi))$
is represented by a multiple of the empty graph in
$N \times [0,1]$,
hence $\End_{\ZCY(N)}((E,P_\Xi)) \cong \kk$.
\end{proof}

\begin{lemma}
\label{l:skein-add-2-handle}
Let $M$ be a 3-manifold with boundary component
$N \subseteq \del M$.
Let $\Xi \subset N$ be a multicurve,
and let $M_\xi$ be the 3-manifold obtained from $M$
by adding 2-handles along $\Xi$.
Then the map
$\ZCYsk(M; E_M) \to \ZCYsk(M_\xi; E_{M_\xi})$
induce by inclusion
restricts to an isomorphism
\begin{equation}
\ZCYsk(M; (E_M,P_\xi)) = P_\xi \cdot \ZCYsk(M; E_M)
\cong \ZCYsk(M_\xi; E_{M_\xi})
\label{e:skein-add-2-handle}
\end{equation}
where $E_M, E_{M_\xi}$ are the empty configurations in $\del M, \del M_\xi$.
\end{lemma}

\begin{proof}
The inverse is given by
the operation that takes a graph $\Gamma$ in $M_\xi$
and gives a graph in $M$ by pushing (i.e. isotoping) the graph
out of the 2-handle. In general, the resulting graph may be
dependent on extra choices.
We claim that this operation gives a well-defined isomorphism
\eqnref{e:skein-add-2-handle}.

This claim is similar to \prpref{prp:punctured_glue}.
Indeed, we may consider the ``pole'' here to be
a cocore of the 2-handle in $M_\xi$.
Then $P_\xi$ corresponds to
the action $\pi$ in \prpref{prp:punctured_glue}.
We leave the details to the reader.
\end{proof}

\begin{lemma}
Let $M,N,\Xi$ be as in \lemref{l:skein-add-2-handle}.
Let $M'$ be a handlebody (i.e. a 3-ball with 1-handles)
such that $\del M' \simeq \overline{N}$.
Furthermore, let the diffeomorphism $\del M' \simeq \overline{N}$
send $\Xi$ to null-homotopic curves in $M'$.
Thus, $M_\Xi^* := M \cup_N M' = M_\Xi \cup 3-\text{balls}$.
Then
\[
\ZCYsk(M, (E_M,P_\Xi)) \cong \ZCYsk(M_\Xi^*, E_{M_\Xi^*})
\]
the map induced by inclusion $M \subset M_\Xi^*$.
\label{l:skein-add-handlebody}
\end{lemma}

\begin{proof}
Indeed, $M \cup M'$ is just $M_\Xi$ with an extra 3-handle
to ``fill in the hole in $M'$'',
and the addition of a 3-handle does not affect the
skein module of a 3-manifold,
as graphs and isotopies of graphs can avoid the hole.
(Here the bonudary values in question are the empty configurations,
but this works in general, since $\ZCY(S^2) \cong \Vect$.)
\end{proof}

\begin{corollary}
Let $M$ be a 3-manifold,
with boundary components $\del M = \cup N_k$.
Let $\Xi_k$ be a full multicurve for $N_k$,
and let $P_\Xi = \prod \Xi_k$.
Then
\[
\ZCYsk(M; (E,P_\Xi)) \cong \kk
\]
In particular, $\ZCYsk(M) \cong \kk$ for a closed 3-manifold.
\label{c:skein-3mfld-1dim}
\end{corollary}

\begin{proof}
By \lemref{l:skein-add-handlebody},
this reduces to the case when $M$ is closed.

Present $M$ by Heegaard splitting
$M = M' \cup_N M''$ (see \ocite{saveliev})
and let $\Xi', \Xi''$ be the multicurves
that define the attaching of $M',M''$ respectively.
Let $\Xi$ be some multicurve such that
the 3-manifold defined by $\Xi$ and $\Xi''$
is the 3-sphere.
Then again by \lemref{l:skein-add-handlebody},
\[
\ZCYsk(M) \cong \ZCYsk(M''; (E,P_{\Xi'}))
\cong \ZCYsk(M''; (E,P_{\Xi}))
\cong \ZCYsk(S^3) \cong \kk
\]
where the middle equality holds by
\prpref{p:simple-object-multicurve}.
\end{proof}


\newpage
\section{Relation of $\ZCY$ to Other Invariants}
\label{s:signature}
\vspace{0.8in}

\subsection{Relation of $\ZCY$ to the Signature and Euler Characteristic of 4-Manifolds}
\par \noindent

In \ocite{CKY-eval}, the authors prove that, for $\cA = \Rep U_q sl_2$
(with $q$ a root of unity),

\begin{equation}
\ZCY(W) = \cD^{\chi(W)/2} \kappa^{\sigma(W)}
\label{e:classical}
\end{equation}
where $\chi, \sigma$ are the Euler characteristic
and signature of $W$, respectively,
and

\[
\kappa = \sqrt{p_+ / p_-}
\;\;,\;\;
p_\pm = \sum \theta_i^\pm d_i^2
=
\begin{tikzpicture}
\draw[regular] (0,0) circle (0.5cm);
\node[small_morphism] at (-0.5,0) {\smallerer $\theta^\pm$};
\end{tikzpicture}
\]
Note $\cD = p_+ p_-$
($= \eta^{-2}$ from \ocite{CKY-eval},\ocite{roberts-chainmail}).
In this section, we extended this result
to 4-manifolds with boundary and arbitrary modular categories $\cA$,
and then to 4-manifolds with corners.


The Crane-Yetter invariant associated to a 4-manifold with boundary
is not a scalar,
but a vector $\ZCY(W) \in \ZCY(\del W)$.
By \corref{c:skein-3mfld-1dim},
$\ZCY(\del W)$ is 1-dimensional,
so we need to choose the correct identification
$\ZCY(\del W) \cong \kk$ to reproduce \eqnref{e:classical}.
It is natural to guess that the identification
$\ZCY(\del W) \cong \kk$ that we are after should send this
element $\emptyskein{M} \mapsto 1$,
but this is not correct;
the correct normalization is given in
\defref{d:emptyskein-normalization} below.
In order to motivate that normalization,
we first consider the following example:

\begin{example}
\label{x:zcy-compute-2-handle}
Let $\cA$ be modular.
Let $W_L = \cH_0 \cup (\cH_2^{(1)} \sqcup \cdots \sqcup \cH_2^{(n)})$,
where the 2-handle $\cH_2^{(m)}$
is attached along the $m$-th component $L_m$ of a framed link
$L = \bigsqcup L_m \subset \del \cH_0$;
Then $W_L$ has a natural handle decomposition
as a cobordism $W_L : \emptyset \to M_L$.
Consider its dual handle decomposition as a cobordism
$W_L : \ov{M_L} \to \emptyset$
(\defref{d:dual-handle-decomp}).
Recall that the belt spheres of the dual handles
are the attaching spheres of the original handles,
so in particular the belt sphere of $\cH_2^{*(m)}$,
the dual 2-handle to $\cH_2^{(m)}$,
is $L_m$.
Then, using \defref{d:zcyksk-elementary},
these 2-handles define a cobordism
$\bigsqcup \cH_2^{*(m)} : \ov{M_L} \to \ov{\del \cH_0}$,
such that
\begin{equation}
\ZCYsk(\bigsqcup \cH_2^{(m)}) (\emptyskein{\ov{M_L}})
= \Omega L \in \ZCYsk(\ov{\del \cH_0})
\end{equation}
where $\Omega L$ means we color each component of $L$
by the regular color $\sum_i d_i \id_i$.
Thus, we have
\begin{equation}
\ZCYsk(W_L) (\emptyskein{\overline{M_L}})
= \ZCYsk(\cH_0^*) (\Omega L)
= \ZCYsk(\cH_4) (\Omega L)
= \ZRT(\Omega \ov{L})
\end{equation}
where $\ov{L}$ is the mirror image of $L$
and $\cH_0^*$ is the dual handle to $\cH_0$.
(The formula in \defref{d:zcyksk-elementary} for a 4-handle
states that $\ZCYsk(\cH_4)(\Gamma) = \ZRT(\Gamma)$,
but $\Gamma$ should be a graph in $S^3 = \ov{\del \cH_4}$,
hence since $L$ was given as a link in $\del \cH_0$,
we need to take the mirror image $\ov{L}$,
which is isotopic to the image of $L$ under an orientation-reversing
homeomorphism of $S^3$.)

This is essentially the Reshetikhin-Turaev invariant of
$M_L$ \ocite{rt-3mfld},
up to a normalizing factor of $\cD^{-1/2}$ (see \ocite{BK}):
\begin{equation}
\label{e:zrt-defn}
\ZRT(M_L) := \kappa^{-\sigma(L)} \cdot \cD^{(-|L|-1)/2}
	\cdot \ZRT(\Omega \ov{L})
\end{equation}
where $\sigma(L)$ is the signature of the linking matrix of $L$;
this invariant only depends on the 3-manifold $M_L$,
and not the link $L$.
(See \rmkref{r:rt-mirror} for clarification.)

Thus, since $\sigma(L) = \sigma(W_L)$
and $|L| + 1 = \chi(W_L) = \chi(\ov{W_L})$,
we see that
\[
\ZCYsk(W_L) (\frac{1}{\ZRT(M_L)} \cdot \emptyskein{\overline{M_L}})
= \kappa^{\sigma(W)} \cdot \cD^{\chi(W_L)/2}
\]


Note that under this normalization of $\ZRT$,
we have
\begin{align*}
\ZRT(S^3) &= \cD^{-1/2}
\\
\ZRT(M_1 \# M_2) &= \cD^{1/2} \ZRT(M_1) \ZRT(M_2)
\end{align*}

We will discuss Reshetikhin-Turaev invariants
in greater detail in \secref{s:rt-boundary-theory}.
(see also \rmkref{r:rt-notation}).
\end{example}

\begin{remark}
In \ocite{BK}*{Theorem 4.1.16},
the formula for $\ZRT(M_L)$
is the same as \eqnref{e:zrt-defn}
except it is $\Omega L$ instead of $\Omega \ov{L}$.
However, this apparent discrepancy is due to our conventions.
More specifically, in our graphical calculus, we treat morphisms
as going from top to bottom,
and our braiding and twist is the left-hand twist,
while \ocite{BK} uses the opposite conventions.
In particular, flipping a diagram through a horizontal plane
switches between our and their conventions,
hence evaluating $\ZRT(\Omega \ov{L})$
with our conventions
is the same as evaluating $\ZRT(\Omega L)$ with the conventions
of \ocite{BK}.
\label{r:rt-mirror}
\end{remark}

Motivated by the example above, we define the following:

\begin{definition}
Let $\cA$ be modular. For a closed 3-manifold $M$,
define the \emph{normalized empty skein} $\emptynorm{M}$ by
\[
\emptynorm{M} := \frac{1}{\ZRT(\overline{M})} \cdot \emptyskein{M}
\in \ZCYsk(M)
\]
Note the orientation reversal in $\ZRT(\overline{M})$.
\label{d:emptyskein-normalization}
\end{definition}

Note that $\emptynorm{M}$ is sensitive to the orientation of $M$,
as $\ZRT(M) \neq \ZRT(\overline{M})$ in general.

\begin{proposition}
Let $\cA$ be modular. For a 4-manifold $W$,
\[
\ZCYsk(W) (\emptynorm{\overline{\del W}})
= \kappa^{\sigma(W)} \cD^{\chi(W)/2}
\]
More generally,
let $W$ be a cobordism $W : M \to M'$
between closed 3-manifolds.
Then
\[
\ZCYsk(W) (\emptynorm{M}) = \kappa^{\sigma(W)} \cD^{\chi(W)/2} \cdot
\emptynorm{M'}
\]
\label{p:eval-zcy}
\end{proposition}

\begin{proof}
The more general statement follows directly from the
first statement.
Indeed, let $W'$ be a 4-manifold
with boundary $\del W' = \overline{M'}$.
Then, assuming the first statement has been proven,
\[
\ZCYsk(W\cup_{M'} W') (\emptynorm{\overline{\del W}})
= \kappa^{\sigma(W \cup_{M'} W')} \cD^{\chi(W \cup_{M'} W')/2}
\]
By Novikov additivity \ocite{novikov}*{7.1},
$\sigma(W \cup_{M'} W') = \sigma(W) + \sigma(W')$,
and by additivity of Euler characteristic
(together with $\chi(M^3)=0$ for any closed 3-manifold),
$\chi(W \cup_{M'} W') = \chi(W) + \chi(W')$.
Then, using the fact that $\ZCY(M^3)$ is 1-dimensional
for closed 3-manifolds (\corref{c:skein-3mfld-1dim}),
$\ZCYsk(W)(\emptynorm{M}) = y \cdot \emptynorm{M'}$
for some scalar $y \in \kk$.
The above observations imply that
$y = \kappa^{\sigma(W)} \cD^{\chi(W)/2}$.
Thus, we will focus on proving the first statement.

Let $W$ be presented by a handle decomposition,
arranged in increasing index,
and let $W_k$ be the union of $j$-handles,
where $j \leq k$
(so $W_4 = W, W_\inv = \emptyset$).
Denote $M_k = \overline{\del W_k}$,
and $W_k' = W_k \backslash \Int(W_{k-1})$,
so $W_k'$ is a cobordism
\[
W_k' : M_k \to M_{k-1}
\]
which is a composition of several elementary cornered cobordisms
of index $4-k$ (dual to the original $k$-handles).
Let $n_k$ be the number of original $k$-handles
(i.e. the number of $4-k$-handles that make up $W_k'$).

Observe that $M_3 = M_4 \sqcup \coprod_{n_4} S^3$,
so
\[
\ZRT(\overline{M_3}) = \ZRT(\overline{M_4}) \cdot \cD^{-n_4/2}
\]
By \defref{d:zcyksk-elementary}
(using $k=0$),
we have
\[
\ZCYsk(W_4') (\emptyskein{M_4}) = \cD^{n_4} \cdot \emptyskein{M_3}
\]
So
\[
\ZCYsk(W_4') (\emptynorm{M_4}) = \cD^{n_4/2} \emptynorm{M_3}
\]

Now consider $k=3$.
Again by \defref{d:zcyksk-elementary} (using $k=1$),
\[
\ZCYsk(W_3') (\emptyskein{M_3}) = \cD^{-n_3} \cdot \emptyskein{M_2}
\]

Suppose $n_3 = 1$, and the removal of that 3-handle
(i.e. addition of dual 1-handle)
connects two components of $M_3$,
so that if $M_3 = M_3' \sqcup M_3''$,
then $M_2 = M_3' \# M_3''$.
Then
\[
\ZRT(\overline{M_2}) = \cD^{1/2} \cdot \ZRT(\overline{M_3})
\]
Suppose the removal of the 3-handle does not connect two components,
but connects one component $M_3' \subseteq M_3$ to itself,
resulting in $M_2' \subseteq M_2$.
If $L$ is a framed link such that $M_3' = M_L$,
then $M_2' = M_{L \cup \bigcircle}$,
where $L \cup \bigcircle$ is $L$ with an extra unknotted 0-framed component
that is unlinked with $L$.
Then it is easy to see that we also have
$\ZRT(\overline{M_2}) = \cD^{1/2} \cdot \ZRT(\overline{M_3})$.
Thus, for multiple 3-handles,
\[
\ZRT(\overline{M_2}) = \cD^{n_3/2} \cdot \ZRT(\overline{M_3})
\]
and so
\[
\ZCYsk(W_3') (\emptynorm{M_3}) = \cD^{-n_3/2} \emptynorm{M_2}
\]

Finally we come to $k=2$. We will consider $W_2$ all at once.
By multiplicativity under disjoint unions,
we may assume that $W_2$ is connected.
Furthermore, we cancel 0- and 1-handles so that $n_0 = 1$.


The attaching maps of 1-handles may be visualized
as a pair of balls in $S^3$.
The attaching maps of 2-handles form a framed link $L$,
and links that go through the 1-handle
go into one ball and appear at the other one.

We use the ``dotted circle'' notation for 1-handles,
as developed in \ocite{akbulut}
(see also \ocite{gompf-stipsicz}*{Chapter 5.4}).
We may bring the pairs of balls together,
joining the arcs from $L$ appropriately,
and put a new unknotted circle around them;
we put a dot on these new circles
to distinguish them from $L$,
and let $L'$ be the unlink consisting of these dotted circles.

Thus, $\del W_1$ is depicted as $S^3 = \del W_0$
with $n_1$ dotted circles forming an unlink $L'$,
and the attaching maps for the 2-handles is simply
the framed link $L$ in the complement of $L'$.
By \defref{d:zcyksk-elementary},
\[
\ZCYsk(W_2')(\emptyskein{M_2}) = \Omega L \in \ZCYsk(M_1)
\]

Once again by \defref{d:zcyksk-elementary},
removing 1-handles (i.e. attaching dual 3-handles to $M_1$)
amounts to removing all strands
(after ``bunching'' them up with \lemref{l:summation})
with $i\neq \one$ that goes through the 1-handle.
Thus, by \lemref{l:killing},
this is equivalent to coloring the dotted circles by
$1/\cD \cdot$ regular color:
\[
\ZCYsk(W_2)(\emptyskein{M_2})
= \ZCYsk(W_0)
	(\cD^{-n_1} \cdot \Omega (L\cup L'))
= \cD^{-n_1} \cdot \ZRT(\Omega (\ov{L\cup L'}))
\]

Now $\ov{M_2} = \del W_2 \simeq M_{L \cup L'}$,
so
\[
\ZRT(\overline{M_2}) = \ZRT(M_{L \cup L'})
= \kappa^{-\sigma(L \cup L')} \cdot \cD^{(-|L'|-|L|-1)/2}
	\cdot \ZRT(\Omega(\ov{L \cup L'}))
\]
Since $|L| = n_2, |L'| = n_1$, we have
\[
\ZCYsk(W_2) (\emptynorm{M_2})
= \kappa^{\sigma(L\cup L')} \cdot \cD^{(n_2 - n_1 + 1)/2}
\]

Thus, we have
\[
\ZCYsk(W) (\emptynorm{\overline{\del W}})
= \kappa^{\sigma(L \cup L')}
\cdot \cD^{(n_4 - n_3 + n_2 - n_1 + 1)/2}
= \kappa^{\sigma(L \cup L')}
\cdot \cD^{\chi(W)/2}
\]

It remains to see that $\sigma(L \cup L') = \sigma(W)$.
Observe that $W_4' \cup W_3'$ is a disjoint union of
4-dimensional handlebodies ($\natural m B^3 \times S^1$),
so $\sigma(W_4' \cup W_3') = 0$.
and hence by Novikov additivity \ocite{novikov},
$\sigma(W) = \sigma(W_2)$.
Finally, replacing a dotted circle
by a 0-framed unknot can be achieved by surgery on the 4-manifold
(see \ocite{gompf-stipsicz}*{Chapter 5.4}),
and thus $W_2$ and $W_{L \cup L'}$ are cobordant,
and hence, by \ocite{thom-cobordism},
have the same signature
$\sigma(W_2) = \sigma(W_{L\cup L'}) = \sigma(L \cup L')$.
(It is also not hard to show directly
that $\sigma(L \cup L') = \sigma(W_2)$,
e.g. using handle cancellations on the homology level.)
\end{proof}

\begin{remark}
It seems that the key point where the proof requires $\cA$
to be modular is when we traded 1-handles for 2-handles.
It will be interesting to see what one might obtain
for a non-modular $\cA$.
\end{remark}

\begin{corollary}
Let $\cA$ be modular,
and let $M$ be a closed 3-manifold,
Then
\[
\evsk(\emptyskein{M},\emptyskein{\ov{M}}) = \ZTV(M)
\]
\end{corollary}

\begin{proof}
Since $\sigma(M \times I) = \chi(M \times I) = 0$,
we have $\ZCYsk(M \times I)(\emptynorm{M} \tnsr \emptynorm{\ov{M}}) = 1$,
so $\evsk(\emptyskein{M},\emptyskein{\ov{M}})
= \ZCYsk(M \times I)(\emptyskein{M} \tnsr \emptyskein{\ov{M}})
= \ZRT(M) \ZRT(\overline{M}) = \ZTV(M)$.
(The last equality is proven in \ocite{roberts-chainmail},
where he works with $\cA = \Rep U_q sl_2$ (with $q$ a root of unity),
but all the arguments work for general modular category).
\end{proof}

\subsection{4-Manifolds with Corners, Wall Non-Additivity}
\par \noindent

Observe that, for 4-manifolds $W,W'$,
thought of as cobordisms $W : M \to M', W' : M' \to M''$,
composition aligns with additivity of signature
(Novikov additivity \ocite{novikov})
and Euler characteristic (since $\chi(M^3)=0$ for any closed 3-manifold).

As discussed in \secref{s:sk-equiv},
4-manifolds with corners may be viewed as
cornered cobordism,
and we might hope that a similar additivity result would follow.
However, Wall \ocite{wall} proved that when two 4-manifolds are glued along
a 3-manifold with boundary,
the signatures do not add on the nose,
but an error term is introduced.

Another issue that arises in this situation
is that the skein module of a 3-manifold with boundary
is not necessarily 1-dimensional.
This will be remedied by considering
an appropriate boundary value,
in particular, as one may naturally guess,
the simple objects from \prpref{p:simple-object-multicurve}.

Let us recall the main definitions and results
of \ocite{wall}, and refer the reader to the paper for details.

\begin{definition}
Let $L_1,L_2,L_3$ be Lagrangian subspaces of a symplectic
vector space $(V,\omega)$ over $\RR$.
For $x_1,x_1' \in L_1 \cap (L_2 + L_3)$,
define the form
\[
\Psi(x_1,x_1') = \omega(x_1,x_2') = \omega(x_1,-x_3')
\]
where $x_2' \in L_2, x_3' \in L_3$ such that
$x_1' + x_2' + x_3' = 0$.
It is easy to check
\begin{itemize}
\item $\Psi$ is independent of the choice of $x_2',x_3'$;
\item $\Psi$ is symmetric;
\item $\Psi$ vanishes on $L_1 \cap L_2 + L_1 \cap L_3$.
\end{itemize}
Thus, $\Psi$ descends to a symmetric bilinear form on
$\ov{V} = (L_1 \cap (L_2 + L_3)) / (L_1 \cap L_2 + L_1 \cap L_3)$;
We denote the signature of $\Psi$ by
\[
\sigma(V;L_1,L_2,L_3)
\]
\label{d:lag-triple}
\end{definition}

Note that the signature of $\Psi$ as a symmetric bilinear form
on $L_1 \cap (L_2 + L_3)$ and $\ov{V}$ are the same,
so there is no ambiguity.
It is clear that
\[
\sigma(V;L_{\tau(1)},L_{\tau(2)},L_3{\tau(3)})
= \text{sgn}(\tau) \cdot \sigma(V;L_1,L_2,L_3)
\]
We also note that if any pair of $L$'s are the same,
say $L_1 = L_2$, then $\sigma(V;L_1,L_2,L_3) = 0$.

The main result of \ocite{wall} is as follows,
recast in our setting:

\begin{theorem}[\ocite{wall}*{Theorem.}]
Let $N$ be a closed surface, and let
$W : M_1 \to_N M_2$, $W' : M_2 \to_N M_3$ be cornered cobordism.
Let $N$ be given the outward orientation with respect to $M_3$
(equivalently $M_1,M_2$).
Let $V = H_1(N;\RR)$, with the intersection pairing as the
symplectic structure.
Let $L_j = \ker(i_* : V \to H_1(M_j;\RR))$
be the kernel of the map on homology induced by the inclusion
$i : N \to M_j$;
it is known that $L_j$ is a Lagrangian subspace.
Then
\[
\sigma(W \cup_{M_2} W')
= \sigma(W) + \sigma(W') - \sigma(V;L_1,L_2,L_3)
\]
\label{t:wall-main}
\end{theorem}

Note that if we change the orientations of $W$ and $W'$,
we need to either flip the orientations of $M_j$'s,
or reverse the direction of the cobordisms;
for the former, the orientation on $N$ is flipped,
while in the latter, $L_1$ and $L_3$ are swapped.
In both cases, the sign of $\sigma(V;L_1,L_2,L_3)$
is flipped,
which is consistent with the rest of the equation.

We want to incorporate the $\sigma(V;L_1,L_2,L_3)$ term
into our Crane-Yetter setting.
First, a generalization of \defref{d:emptyskein-normalization}
to 3-manifolds with boundary:

\begin{definition}
Let $M$ be a 3-manifold with boundary $\del M = N$.
Let $\Xi$ be a full multicurve on $N$
(see \prpref{p:simple-object-multicurve}.
Let $M_\Xi^*$ be the closed 3-manifold obtained by
gluing 2-handles to $M$ along $\Xi$,
and then adding 3-handles (see \lemref{l:skein-add-handlebody}).
Define
\[
\emptynorm{M,\Xi}
= \frac{1}{\ZRT(\overline{M_\Xi^*})} \cdot P_\Xi \circ \emptyskein{M}
\]
where $P_\Xi$ is the projection in $\End_\ZCY(N)(E)$
defined by $\Xi$ (see \defref{d:proj-multicurve}).
\label{d:emptyskein-normalization-corner}
\end{definition}

\begin{theorem}
Let $\cA$ be modular.
Let $W : M_1 \to_N M_2$ be a cornered cobordism,
$N$ outwardly oriented with respect to $M_1$ (and $M_2$),
and let $L_1,L_2 \subseteq V = H_1(N;\RR)$ be the kernels of
the inclusions of $N$ into $M_1,M_2$,
as in \thmref{t:wall-main}.
Let $[\Xi] \subseteq V$ be the subspace spanned by
$\{[\xi] \; | \; \xi \in \Xi\}$.
Then
\[
\ZCYsk(W) (\emptynorm{M_1,\Xi}) = \kappa^{\sigma(W) - \sigma(V;L_1,L_2,[\Xi])}
	\cD^{(\chi(W) - (1-g))/2} \cdot \emptynorm{M_2,\Xi}
\]
\label{t:cornered-cobord-xi}
\end{theorem}

\begin{proof}
Let $M_j' = (M_j)_\Xi^*
= M_j \cup H$ (see \defref{d:emptyskein-normalization-corner}
above).
Let $\id_{M_2'} : M_2' \to M_2'$ be the identity cobordism on $M_2'$,
and let $W' = W \cup_{M_2} \id_{M_2'}$.
By \lemref{l:skein-add-handlebody},
$P_\Xi \circ \emptyskein{M_j} \cup \emptyskein{H}
= \emptyskein{M_j'}$.

Suppose $\ZCY(W)(\emptynorm{M_1,\Xi}) = z \cdot \emptynorm{M_2,\Xi}$.
Then
\begin{align*}
\ZCY(W')(\emptynorm{M_1'})
&=
\ZCY(W')(\emptynorm{M_1,\Xi} \cup \emptyskein{H})
\\ &= 
\ZCY(\id_{M_2'})(\ZCY(W)(\emptynorm{M_1,\Xi}) \cup \emptyskein{H})
\\ &=
z \cdot \emptynorm{M_2,\Xi} \cup \emptyskein{H}
\\ &=
z \cdot \emptynorm{M_2'}
\end{align*}

But by \prpref{p:eval-zcy},
$\ZCY(W')(\emptynorm{M_1'}) =
\kappa^{\sigma(W')} \cD^{\chi(W')/2} \cdot \emptynorm{M_2'}$.
By \thmref{t:wall-main},
$\sigma(W') = \sigma(W) + 0 - \sigma(V;L_1,L_2,[\Xi])$
(since $\sigma(id_{M_2'}) = \sigma(M_2' \times [0,1]) = 0$.
It is also easy to see that $\chi(M_2) = 1 - g$,
so $\chi(W') = \chi(W) + 0 - (1 - g)$.
(since $\chi(id_{M_2'}) = \chi(M_2' \times [0,1]) = 0)$.
Thus we are done.
\end{proof}

Note that one may verify directly that the formula given in
\thmref{t:cornered-cobord-xi}
respects composition of cornered cobordisms.
Namely, for $W' : M_2 \to_N M_3$
(not the $W'$ from the proof),
we should have
\[
\sigma(W \cup W') - \sigma(V;L_1,L_3,[\Xi])
=
(\sigma(W) - \sigma(V;L_1,L_2,[\Xi]))
+
(\sigma(W') - \sigma(V;L_2,L_3,[\Xi]))
\]
which, by \thmref{t:wall-main},
can be rearranged to
\[
\sigma(V;L_1,L_2,L_3)
+ \sigma(V;L_1,L_3,[\Xi])
=
\sigma(V;L_1,L_2,[\Xi]))
+ \sigma(V;L_2,L_3,[\Xi]))
\]
This also follows from \thmref{t:wall-main}:
consider $W'' = M_3' \times [0,1]$,
where $M_3' = M_3 \cup H$ as in the proof of \thmref{t:cornered-cobord-xi}
\lemref{l:skein-add-handlebody}
then the left side of the equation is the error term from
gluing $W$ to $W'$ first followed by $W''$,
while the right side is from
gluing $W''$ to $W'$ first followed by $W$.
In a sense, this equation expresses an associativity constraint.

\subsection{Relation to Reshetikhin-Turaev invariants}
\label{s:rt-boundary-theory}
\par \noindent

It is widely expected that the Reshetikhin-Turaev TQFT
is the ``boundary theory'' of the Crane-Yetter TQFT;
some version of this has appeared in \ocite{barrett-obs}.
Indeed, the heavy reliance of our constructions
of the Crane-Yetter invariants on evaluations of ribbon graphs
in terms of morphisms in $\cA$ would heavily suggest it
(although perhaps it is a self-fulfilling prophecy).
In this section we construct such a ``boundary theory''.
Throughout this section, we assume $\cA$ is modular.

First we address the state space in the Reshetikhin-Turaev TQFT.
As noted in \xmpref{x:skein-pairing-handlebody},
for a handlebody $H = \natural g S^1 \times B^2$
and some boundary value $\VV$ on $\del H$,
\begin{equation}
\ZCYsk(H;\VV) \simeq \dirsum
	\eval{(\bigotimes_b V_{\vec{b}})X_{i_1}X_{i_1}^* \cdots X_{i_g}X_{i_g}^*}
\simeq \tau(\del H, \VV)
\end{equation}
where $\tau(\del H, \VV)$ is the state space associated to
$(\del H,\VV)$ under the Reshetikhin-Turaev TQFT
(see \ocite{BK}*{(4.4.4)}).
This holds for any 3-manifold, as is expected from the
``boundary theory'' hypothesis:

\begin{theorem}
Let $\cA$ be modular.
For a 3-manifold with boundary $M$, the space of skeins with
some boundary value $\VV$
is isomorphic to the state space
$\tau(\del M, \VV)$ associated to its boundary
under the Reshetikhin-Turaev TQFT:
\begin{equation}
\ZCYsk(M;\VV) \simeq \tau(\del M, \VV)
\end{equation}
In particular, $\ZCYsk(M;\EE)$ only depends on $N = \del M$.
\label{t:cy-rt-3mfld}
\end{theorem}

\begin{proof}
Let $W : M \to_N \sqcup H$ be some cornered cobordism of $W$
from $M$ to a disjoint union of handlebodies.
Let $\VV \simeq \dirsum (\EE,P_\Xi)$ be a decomposition
of $\VV$ into simples (see \prpref{p:simple-object-multicurve}).
By \thmref{t:cornered-cobord-xi},
for each $(\EE,P_\Xi)$,
$\ZCYsk(W;N)$ is an isomorphism 
\[
\ZCYsk(W;N) : \ZCYsk(M;(\EE,P_\Xi)) \simeq \ZCYsk(\sqcup H;(\EE;P_\Xi))
\]
and thus it is an isomorphism
\[
\ZCYsk(W;N) : \ZCYsk(M;\VV) \simeq \ZCYsk(\sqcup H;\VV)
\]
By the preceding discussion, we are done.
\end{proof}

\begin{definition}
Let $\cA$ be premodular.
We define the \emph{category of extended cobordisms},
denoted $\cobx$, as follows.
An object of $\cobx$ is a pair
$(M,\VV)$, where $M$ is a 3-manifold,
and $\VV \in \ZCYsk(\del M)$
is a boundary value.

A morphism from $(M,\VV)$ to $(M',\VV')$
is a triple $(W,M_0,\vphi_0)$,
where $W$ is a 4-manifold with
$\del W = \ov{M} \cup_{\del M} M_0 \cup_{\del M'} M'$,
and $\vphi_0 \in \ZCYsk(M_0;\VV \cup \VV')$.

Composition is given by simply gluing the manifolds appropriately,
and composing the skeins;
more precisely, for morphisms
\begin{align*}
(W,M_0,\vphi_0) &: (M,\VV) \to (M',\VV')
\\
(W',M_0',\vphi_0') &: (M',\VV') \to (M'',\VV'')
\end{align*}
their composition is given by
\[
(W',M_0',\vphi_0') \circ (W,M_0,\vphi_0)
:= (W' \cup_{M'} W, M_0' \cup_{\del M'} M_0, \vphi_0' \cup_{\del M'} \vphi_0)
\]

The identity morphism of $(M,\VV)$ is given by
$(M \times I, \del M \times I, \id_\VV)$.
\label{d:cobx}
\end{definition}

\begin{lemma}
With disjoint union as the monoidal structure on $\cobx$,
it is rigid.
The unit object is the empty manifold $(\emptyset,\EE)$.
For an object $(M,\VV)$,
the same triple $(M \times I, \del M \times I, \id_\VV)$
that defines the identify morphism also defines the
(co)evaluation morphisms
\begin{align*}
\ev_{(M,\VV)} = (M \times I, \del M \times I, \id_\VV) &:
	(M,\VV) \sqcup (\ov{M},\VV) \to (\emptyset,\EE)
\\
\coev_{(M,\VV)} = (M \times I, \del M \times I, \id_\VV) &:
	(\emptyset,\EE) \to (M,\VV) \sqcup (\ov{M},\VV)
\end{align*}
\label{l:cobx-duals}
\end{lemma}

\begin{proof}
Straightforward.
\end{proof}

\begin{proposition}
For an object $(M,\VV)$ of $\cobx$,
define
\[
\cFx((M,\VV)) := \ZCYsk(M;\VV)
\]
and for a morphism
$(W,M_0,\vphi_0) : (M,\VV) \to (M',\VV')$,
we define
\begin{align*}
\cFx((W,M_0,\vphi_0)) : \cFx((M,\VV)) &\to \cFx((M',\VV'))
\\
\vphi &\mapsto \ZCYsk(W;\del M') (\vphi_0 \cup_{\del M} \vphi)
\end{align*}
Then $\cFx$ is a monoidal functor.
\label{p:cobx-functor}
\end{proposition}

\begin{proof}
Straightforward.

We note that the image of the evaluation morphism
under $\cFx$ gives the skein pairing (\defref{d:pairing-skein}).
\end{proof}

Let us recall the Reshetikhin-Turaev invariants of
3-manifolds with links in more detail \ocite{rt-3mfld}.
(The constructions here are adapted from the exposition in
\ocite{BK}*{Chapter 4.4}.)
Consider a coupon $c$ of a ribbon graph $\Gamma$,
and let $r_1,\ldots,r_g$ be some ribbons that attach to $c$ twice
(not necessarily all such ribbons).
Let $H(c,r_1,\ldots,r_g)$ be a small neighborhood of
$c \cup r_1 \cup \cdots \cup r_g$;
it is a genus $g$ handlebody.

\begin{definition}
A \emph{special graph} $X$ is an (uncolored) ribbon graph in $S^3$
with a distinguished collection of annuli $L$,
and a distinguished collection
$\Psi = \{(c,r_1,\ldots,r_g),\ldots\}$
of coupons and ribbons as above.
We define
\[
M_X = M_L \backslash H_\Psi
\]
where $M_L = \del W_L$,
$W_L$ is the 4-manifold obtained by attaching 2-handles to $B^4$
along $L \subset S^3 = \del B^4$,
and $H_\Psi := \bigcup H(c,r_1,\ldots,r_g)$
is the union of handlebodies associated to the distinguished coupons and ribbons.
\label{d:special-graph}
\end{definition}

In the following, $\ZCYsk(H_\Psi;\VV)$
replaces $\tau(\del H_\Psi,\VV)$
in the usual definition (say in \ocite{BK}).

\begin{definition}
Let $X$ be a special graph, and consider the graph
$\Gamma = (X \backslash L) \backslash H_\Psi$,
interpreted as a graph in $M_X$.
Given a coloring $\Phi$ of $\Gamma$,
and a coloring $\vphi$ of the coupons of $\Psi$,
we consider the evaluation
$\ZRT((X \backslash L,\Phi \cup \vphi) \cup \Omega L)$,
that is, $X \backslash L$ (as a graph in $S^3$) is
colored by $\Phi$ and $\vphi$,
and each component of $L$ is colored by the regular coloring
(see \xmpref{x:zcy-compute-2-handle}).
Thus, if we hold $\Phi$ fixed, so that it defines a boundary value $\VV$
on $\del H_\Psi$,
then the evaluation defines a functional
\begin{equation}
\ZRT((X \backslash L,\Phi \cup -) \cup \Omega L) :
	\ZCYsk(H_\Psi; \VV) \to \kk
\end{equation}
(it is clear that the possible colorings $\vphi$ of the coupons
of $\Psi$ span $\ZCYsk(H_\Psi;\VV)$).
Then by using the non-degenerate skein pairing \defref{d:pairing-skein},
this defines a vector

\begin{equation}
\tau(M_X, (\Gamma,\Phi)) \in \ZCYsk(\ov{H_\Psi};\VV)
\end{equation}

(By \xmpref{x:skein-pairing-handlebody},
the skein pairing agrees (up to a factor) with the pairing defined in
\ocite{BK}*{(4.4.5)}.)

The \emph{Reshetikhin-Turaev of the
	colored ribbon graph $(\Gamma,\Phi)$ in $M_X$}
is defined to be
\begin{equation}
\ZRT((M_X, (\Gamma,\Phi)))
= \kappa^{-\sigma(L)} \cdot \cD^{(-|L|-1)/2} \cdot \tau(M_X, (\Gamma,\Phi))
\in \ZCYsk(\ov{H_\Psi};\VV)
\end{equation}

More generally,
given a 3-manifold $M$ with a colored ribbon graph $\Gamma'$
with boundary value $\VV$',
if we have homeomorphisms $M \simeq M_X$
and $H \simeq H_\Psi$ that glue to a homeomorphism
$M \cup_{\del M \simeq \ov{\del H}} H \simeq M_X \cup H_\Psi = M_L$,
and it sends $(\Gamma',\Phi')$ to $(\Gamma,\Phi)$, $\VV'$ to $\VV$,
then we define
\[
	\ZRT((M,(\Gamma',\Phi'))) = \ZRT((M_X,(\Gamma,\Phi)))
\]
under the identification
$\ZCYsk(\ov{H};\VV') \simeq \ZCYsk(\ov{H_\Psi};\VV)$.
\label{d:rt-inv-BK}
\end{definition}

The definition above is a generalization of
\eqnref{e:zrt-defn} to 3-manifolds with boundary;
it is shown to be independent of the choice of special graph $X$
(i.e. the choice of homeomorphism in the last paragraph)
in \thmref{t:rt-inv-equiv}.
Now we give a definition of the Reshetikhin-Turaev invariant
that is based on $\cobx$:

\begin{definition}
Let $X$ be a special graph,
with $\Gamma = (X \backslash L) \backslash H_\Psi \subset M_X$
and $\Phi$ a coloring of $\Gamma$ as in \defref{d:rt-inv-BK} above.
Let $N = \del H_\Psi$, and let $\VV$ be the boundary value on $N$
defined by $\Phi$.
Consider the 4-manifold $W_L$ obtained from attaching
2-handles along $L$, so that $\del W_L = M_L$;
treat $\ov{W_L}$ as a cornered cobordism
$\ov{W_L} : M_X \to_N \ov{H_\Psi}$.
Then we define
\[
\taucy(M_X, (\Gamma,\Phi)) = \ZCYsk(\ov{W_L};N)((\Gamma,\Phi))
\in \ZCYsk(\ov{H_\Psi};\VV)
\]
Equivalently, we can consider the morphism
$(\ov{W_L},M_X,(\Gamma,\Phi)) : (\emptyset,\EE) \to (\ov{H_\Psi};\VV)$,
and define
\[
\taucy(M_X, (\Gamma,\Phi)) =
	\cFx((\ov{W_L},M_X,(\Gamma,\Phi)))(\emptyskein{\emptyset})
	\in \ZCYsk(\ov{H_\Psi};\VV)
\]
We then define
\[
\ZRTCY(M_X, (\Gamma,\Phi)) = 
	\kappa^{-\sigma(\ov{W_L})} \cdot \cD^{-\chi(\ov{W_L})/2}
	\cdot \taucy(M_X, (\Gamma,\Phi))
\]
More generally, for an arbitrary 3-manifold $M$
with a colored graph
$(\Gamma,\Phi) \in \ZCYsk(M;\VV)$,
and a disjoint union of handlebodies $H$ with
$\del M = N = \ov{\del H}$,
we define
\[
\ZRTCY(M,(\Gamma,\Phi)) =
	\kappa^{-\sigma(W)} \cdot \cD^{-\chi(W)/2} \cdot
	\ZCYsk(W;N)((\Gamma,\Phi))
	\in \ZCYsk(\ov{H};\VV)
\]
where $W$ is some 4-manifold with boundary $\del W = \ov{M \cup_N H}$,
i.e. a cornered cobordism $W : M \to_N \ov{H}$.
\label{d:rt-inv-cobx}
\end{definition}

\begin{lemma}
The value of $\ZRTCY(M,(\Gamma,\Phi)) \in \ZCYsk(\ov{H};\VV)$
depends only on $M$ and $H$ (and the identification of their boundaries),
and is independent of $W$.
\label{l:zrtcy-indep}
\end{lemma}
\begin{proof}
Follows easily from \prpref{p:eval-zcy}; details are left to the reader.
\end{proof}

The following theorem shows that
the invariants $\ZRT$ and $\ZRTCY$ are equivalent
(we repeat some definitions for the reader's convenience):

\begin{theorem}
Let $M$ be a 3-manifold, and let $\Gamma$ be a colored ribbon graph
in $M$ with boundary value $\VV$ on $N := \del M$.
Let $H$ be a disjoint union of handlebodies
with an identification  $\del H \simeq \ov{N}$,
so that we have an identification
of the Reshetikhin-Turaev state space $\tau(N,\VV)$
with the skein space $\ZCYsk(\ov{H};\VV)$.
Then for any cornered cobordism
$W: M \to \ov{H}$,
\[
\ZRT(M,\Gamma) = \kappa^{-\sigma(W)} \cdot \cD^{-\chi(W)/2}
	\cdot \ZCY(W;N)(\Gamma)
\in \ZCYsk(\ov{H};\VV)
\]
\label{t:rt-inv-equiv}
\end{theorem}

\begin{proof}
Choose some special graph $X$ with identifications
$M \simeq M_X$, $H \simeq H_\Psi$,
as in \defref{d:rt-inv-cobx}.

The right-hand side is $\ZRTCY(M,\Gamma)$,
which by \lemref{l:zrtcy-indep} is independent of the choice of $W$.
Take $W = \ov{W_L}$.
Since $\chi(W) = |L|+1$ and $\sigma(W) = \sigma(L)$,
it suffices to show that from $\tau = \taucy$.

As we have seen in \xmpref{x:zcy-compute-2-handle},
the 2-handles of $W_L$ give rise to $\Omega L$.
Thus, by \lemref{l:skein-pairing-dual},
\begin{align*}
\evsk(\vphi,\taucy(M_X, (\Gamma,\Phi)))
&= \evsk(\vphi, \ZCYsk(W_L;N)((\Gamma,\Phi)))
\\
&= \ZCYsk(W_L)((X \backslash L,\Phi \cup \vphi))
\\
&= \ZRT((X \backslash L,\Phi \cup \vphi) \cup \Omega L)
\\
&= \evsk(\vphi, \tau(M_X, (\Gamma,\Phi)))
\end{align*}
\end{proof}

The gluing axiom follows easily from \defref{d:rt-inv-cobx}:
\begin{lemma}
Let $M$ be a 3-manifold with boundary $\del M = N_1 \sqcup N_2 \sqcup N'$,
where $N_1,N_2$ are connected closed surfaces,
and let $f : N_1 \simeq \ov{N_2}$ be a homeomorphism.
Let $H_1,H_2$ be handlebodies with
$\del H_1 = N_1, \del H_2 = N_2$,
and suppose $f$ extends to a homeomorphism
$f : H_1 \simeq \ov{H_2}$.
Let $\Gamma$ be a colored ribbon graph
(omitting the coloring) with boundary value
$\VV = \VV_1 \sqcup \VV_2 \sqcup \VV'$,
and suppose that $f_*(\VV_1) = \VV_2$.
Then we have a graph $\Gamma' = \Gamma/f$ in $M' = M/f$, and
\[
\evsk(\ZRTCY(M,\Gamma)) = \ZRTCY(M',\Gamma')
\in \ZCY(H';\VV')
\]
for any handlebodies $H'$ with $\del H' = N'$,
where $\evsk$ applies the skein pairing
\[
\ZCYsk(H_1;\VV_1) \tnsr \ZCYsk(H_2;\VV_2) \to \kk
\]
\label{l:zrtcy-gluing}
\end{lemma}
\begin{proof}
Applying the skein pairing amounts to composing with
$\ev_{(H_1,\VV_1)}$, so that,
if $W$ is a 4-manifold with $\del W = M \cup (H_1 \sqcup H_2 \sqcup H')$,
then $W \cup_{H_1,H_2} H_1 \times I$
has boundary $M'$.
\end{proof}

It is also helpful to connect the gluing result above
to \defref{d:rt-inv-BK}.
Let $X$ be a special graph,
and let $\psi_1,\psi_2 \in \Psi$ define handlebodies
$H_{\psi_1}, H_{\psi_2}$.
Let $\Phi$ be some coloring of
$\Gamma = (X \backslash L) \backslash H_\Psi$,
leaving a boundary value $\VV$ on $\del H_\Psi$.
Let $\Psi' = \Psi \backslash \{\psi_1,\psi_2\}$,
and let $\VV', \VV_1, \VV_2$ be the boundary value $\VV$
restricted to $\del H_{\Psi'},\del H_{\psi_1},H_{\psi_2}$,
respectively.
We have
\begin{equation}
\ZRT(M_X,(\Gamma,\Phi))
\in \ZCYsk(\ov{H_\Psi};\VV)
= \ZCYsk(\ov{H_{\psi_1}};\VV_1)
	\tnsr \ZCYsk(\ov{H_{\psi_1}};\VV_2)
	\tnsr \ZCYsk(\ov{H_{\Psi'}};\VV')
\end{equation}

Suppose that there is an identification
$f : \psi_1 \simeq \psi_2$,
extending to an identification
$f : H_{\psi_1} \simeq \ov{H_{\psi_2}}$
which also identifies the boundary values,
i.e. $f_*(\VV_1) = \VV_2$.
Then, as in the lemma above,
we may perform the skein pairing
$\ZCYsk(H_{\psi_1};\VV_1) \tnsr \ZCYsk(H_{\psi_2};\VV_2) \to \kk$
on $\ZRT(M_X,(\Gamma,\Phi))$,
or equivalently, compose $(W_L,M_X,(\Gamma,\Phi))$
with the evaluation morphism for $(H_{\psi_1},\VV_1)$,
to obtain
\begin{equation}
\evsk(\ZRT(M_X,(\Gamma,\Phi)))
= \ev_{(H_{\psi_1},\VV_1)} \circ (W_L,M_X,(\Gamma,\Phi))
	(\emptyskein{\emptyset})
\in \ZCYsk(H_{\Psi'};\VV')
\label{e:rt-inv-gluing}
\end{equation}

Again, the 4-manifold associated to
$\ev_{(H_{\psi_1},\VV_1)} \circ (W_L,M_X,(\Gamma,\Phi))$
i.e. $W' := W_L \cup_{H_{\psi_1},H_{\psi_2}} H_{\psi_1} \times I$,
has boundary $M' := \del W' = M_X / f$.
The graph $\Gamma \subset M_X$ glues up under $f$ to some colored graph
$\Gamma' = \Gamma / f$.

We want a surgery description for $M'$.
The 4-manifold $W'$ is built from $W_L$
by attaching 1- and 2-handles;
we can see this from the dual handle decomposition
to the description of the cobordism
$H\cup_N \ov{H} \to \emptyset$ in
\xmpref{x:skein-pairing-handlebody}:
attach a 1-handle to neighborhoods of $c_1$ and $c_2$
(the coupons in $\psi_1$ and $\psi_2$),
then corresponding ribbons would glue up into $g$ annuli,
so we attach 2-handles to these annuli.

More concretely,
start with the special graph $X$.
Bring $c_1$ and $c_2$ close together via an ambient isotopy.
Remove $c_1,c_2$ and join the ribbons together according
to the identification $f$,
then add a dotted circle around these ribbons.
Add the $g$ annuli formed from the ribbons of $\psi_1,\psi_2$
to $L$
(note that these ribbons are not part of $\Gamma$).
Finally, replace the dotted circle with a 0-framed annulus,
and add it to $L$
(see discussion in proof of \prpref{p:eval-zcy}).
Now we have a new special link $X'$,
with distinguished link $L'$ (having $|L| + g + 1$ components),
and distinguished collection $\Psi' = \Psi \backslash \{\psi_1,\psi_2\}$
of coupons and ribbons,
such that $M_{X'} = M'$.
(See \ocite{BK}*{Figure 4.11}).

Given a homeomorphism $f : N \simeq N'$
which sends some boundary value $\VV$ on $N$
to $\VV' = f_*(\VV)$,
the Reshetikhin-Turaev construction gives an isomorphism
$\tau(f) : \tau(N,\VV) \simeq \tau(N',\VV')$
that is defined up to some factors of $\kappa$.
The isomorphism $\tau(f)$ is unchanged under isotopy
(maintaining $f_*(\VV) = \VV'$),
so it gives projective representations of
certain mapping class groups of surfaces.

Let us construct $\tau(f)$ in terms of Crane-Yetter theory.
We assume for simplicity that $N,N'$ are connected,
the disconnected case follows naturally.

\begin{definition}
Let $f : N \simeq N'$ be a homeomorphism between connected surfaces,
and let $f$ sends some boundary value $\VV$ on $N$
to boundary value $\VV' = f_*(\VV)$ on $N'$.

Let $H,H'$ be handlebodies with $\del H = N, \del H' = N'$,
so that $\ZCYsk(H;\VV),\ZCYsk(H';\VV')$
are stand-ins for the Reshetikhin-Turaev state spaces
$\tau(N,\VV),\tau(N',\VV')$.
Let $H''$ be some handlebody with $\del H'' = N$,
and such that $f : N \simeq N'$ extends to a homeomorphism
$f : H'' \simeq H'$.
\footnote{One may take $H'' = H'$, but this extra step
can be helpful for conceptual clarity
and simplify computation.
Note also that such an extension of $f$, if it exists,
is unique up to isotopy.}
For $\vphi \in \ZCYsk(H;\VV)$,
we have $\ZRT(H,\vphi) \in \ZCYsk(H'';\VV)$.
Then we define
\begin{equation}
\tau(f)(\vphi) = f_*(\ZRT(H,\vphi)) \in \ZCYsk(H';\VV')
\end{equation}

\end{definition}

In particular, if $f : N \simeq N'$ extends to $f : H \simeq H'$,
then $\tau(f) = f_*$ simply acts on the skeins directly.

Note that in general, $\tau(g \circ f) \neq \tau(g) \circ \tau(f)$;
the discrepancy is a factor arising from the error terms
in gluing cornered cobordisms:

\begin{lemma}
Let $f: N_1 \simeq N_2, f' : N_2 \simeq N_3$ be homeomorphisms between
closed genus $g$ surfaces,
sending boundary value $\VV_1$ to $\VV_2 = f_*(\VV_1)$ and
$\VV_2$ to $\VV_3 = f_*'(\VV_2)$.
Let $H_j$ be a handlebody with boundary $N_j$,
so that $\tau(f) : \ZCYsk(H_1;\VV_1) \simeq \ZCYsk(H_2;\VV_2)$
and $\tau(f') : \ZCYsk(H_2;\VV_2) \simeq \ZCYsk(H_3;\VV_3)$.
Let $L_j \subset H_1(N_j;\RR)$ be the kernel of the inclusion
$N_j \subset H_j$ as in \thmref{t:wall-main};
we identify them as subspaces of $V := H_1(N_3;\RR)$
under $f_*,f_*'$. Then

\begin{equation}
\tau(f' \circ f) =
	\kappa^{\sigma(V;L,L_2,L_3)} \cdot \cD^{(1-g)/2} \cdot
	\tau(f') \circ \tau(f)
\end{equation}
\label{l:tau-composition}
\end{lemma}
\begin{proof}
Let $H,H'$ be the auxiliary handlebodies used to define
$\tau(f),\tau(f')$;
i.e. $\del H = N_1$, $\del H' = N_2$,
and $f,f'$ extend to $f : H \simeq H_2, f' : H' \simeq H_3$.
Let $W : H_1 \to_{N_1} H, W' : H_2 \to_{N_2} H'$
be cornered cobordisms.
They compose to give $W'' := W' \cup_f W : H_1 \to_{N_2} H'$.
Then
\begin{align*}
\tau(f) = f_* \circ
	\kappa^{-\sigma(W)} \cdot \cD^{-\chi(W)/2} \cdot \ZCYsk(W;N_1)
\\
\tau(f') = f_*' \circ
	\kappa^{-\sigma(W')} \cdot \cD^{-\chi(W')/2} \cdot \ZCYsk(W';N_2)
\\
\tau(f' \circ f) = f_*' \circ
	\kappa^{-\sigma(W'')} \cdot \cD^{-\chi(W'')/2} \cdot \ZCYsk(W'';N_2)
\end{align*}
Note that
$\ZCYsk(W'';N_2) = \ZCYsk(W';N_2) \circ f_* \circ \ZCYsk(W;N_1)$,
so
\[
\tau(f' \circ f) = \kappa^{\sigma(W) + \sigma(W') - \sigma(W'')}
	\cdot \cD^{(\chi(W) + \chi(W') - \chi(W''))/2}
	\cdot \tau(f') \circ \tau(f)
\]
and the result follows from \thmref{t:wall-main}
and a simple Euler characteristic computation.
\end{proof}

Let us work out the action of the mapping class group of a torus.
Let $N = S^1 \times S^1$, and $H = B^2 \times S^1$.
Let $\vec{m} = \del B^2 \times \{*\},\vec{l} = \{*\} \times S^1$
be (directed) meridian and longitude,
depicted as follows:
\begin{equation}
\label{e:torus-meridian-longitude}
\begin{tikzpicture}
\draw (-2.5,0)
	.. controls +(-90:0.6cm) and +(180:2cm) .. (0,-1.5)
	.. controls +(0:2cm) and +(-90:0.6cm) .. (2.5,0)
	.. controls +(90:0.6cm) and +(0:2cm) .. (0,1.5)
	.. controls +(180:2cm) and +(90:0.6cm) .. (-2.5,0);
\draw (-1.23,0)
	.. controls +(60:0.3cm) and +(180:0.6cm) .. (0,0.45)
	.. controls +(0:0.6cm) and +(120:0.3cm) .. (1.23,0);
\draw (-1.4,0.2)
	.. controls +(-60:0.3cm) and +(180:0.8cm) .. (0,-0.4)
	.. controls +(0:0.8cm) and +(-120:0.3cm) .. (1.4,0.2);
\node at (-2.3,-1.3) {$H$};
\draw[midarrow] (-0.5,-0.35)
	.. controls +(160:0.4cm) and +(160:0.4cm) .. (-0.7,-1.46);
\node at (-0.7,-0.8) {\smallerer $\vec{m}$};
\draw[midarrow] (-2.5,0)
	.. controls +(-90:0.3cm) and +(180:2cm) .. (0,-1.2)
	.. controls +(0:2cm) and +(-90:0.3cm) .. (2.5,0);
\node at (-0.1,-0.95) {\smallerer $\vec{l}$};
\end{tikzpicture}
\end{equation}
The outward orientation on $N$ is $(\vec{m},\vec{l})$.
A matrix $\begin{pmatrix} a & c \\ b & d \end{pmatrix}$
acts on $N$ by sending $\vec{m} \mapsto a \cdot \vec{m} + b \cdot \vec{l}$
and $\vec{l} \mapsto c \cdot \vec{m} + d \cdot \vec{l}$.
In particular, the $S$-matrix\footnotemark \;
$\begin{pmatrix} 0 & 1 \\ -1 & 0 \end{pmatrix}$
swaps the meridian and longitude,
$\vec{m} \mapsto \cev{l}, \vec{l} \mapsto \vec{m}$.
Note that with boundary values,
an integer matrix is not enough
for specifying an element of the mapping class group.

Let $\VV$ be a boundary value on $N$ with exactly one arc $b$,
with no projectors (i.e. $\VV \in \hatZCYsk(N)$),
and let $V_{\vec{b}} = Y$ so that an inwardly-directed boundary ribbon
in $H$ has label $Y$.
Then we have 
\begin{equation}
\label{e:H-sk-param}
\ZCYsk(H;\VV) \simeq \dirsum_j \Hom_\cA(Y, X_j X_j^*)
\end{equation}

Let $\wdtld{S}$ be a homeomorphism of $N$ preserving $\VV$
which descends to $S$ when we forget $\VV$,
and looks like this near $b$
(drawn from an outside perspective):
\begin{equation}
\label{e:S-tld}
\begin{tikzpicture}
\draw[line width=1.2pt] (-0.1,0) -- (0.1,0);
\draw[midarrow={0.9}] (0,0.6) -- (0,-0.6);
\node at (0.2,0) {\smallerer $b$};
\node at (0.3,-0.4) {\smallerer $\vec{m}$};
\end{tikzpicture}
\;\;\;\; \mapsto \;\;\;\;
\begin{tikzpicture}
\draw[line width=1.2pt] (-0.1,0) -- (0.1,0);
\draw[midarrow={0.8}] (0,0)
	.. controls +(-90:0.25cm) and +(0:0.4cm) .. (-0.6,0);
\draw (0,0)
	.. controls +(90:0.25cm) and +(180:0.4cm) .. (0.6,0);
\end{tikzpicture}
\end{equation}

Let us compute the map associated to $\wdtld{S}_*$.
Consider $\vphi \in \ZCYsk(H;\VV)$,
\begin{equation}
\vphi = 
\begin{tikzpicture}
\node[ellipse_morphism={0}] (a) at (0,1) {\;\;};
\draw[midarrow_rev={0.9}] (a)
	.. controls +(0:1cm) and +(90:0.6cm) .. (1.9,0)
	.. controls +(-90:0.6cm) and +(0:1cm) .. (0,-1)
	.. controls +(180:1cm) and +(-90:0.6cm) .. (-1.9,0)
	.. controls +(90:0.6cm) and +(180:1cm) .. (a);
\draw[midarrow_rev={0.7}] (a) -- (0,1.35);
\draw (-0.05,1.35) -- (0.05,1.35);
\node at (-0.25,1.25) {\smallerer $Y$};
\node at (-1.25,1) {\smallerer $j$};
\draw (-2.5,0)
	.. controls +(-90:0.6cm) and +(180:2cm) .. (0,-1.5)
	.. controls +(0:2cm) and +(-90:0.6cm) .. (2.5,0)
	.. controls +(90:0.6cm) and +(0:2cm) .. (0,1.5)
	.. controls +(180:2cm) and +(90:0.6cm) .. (-2.5,0);
\draw (-1.23,0)
	.. controls +(60:0.3cm) and +(180:0.6cm) .. (0,0.45)
	.. controls +(0:0.6cm) and +(120:0.3cm) .. (1.23,0);
\draw (-1.4,0.2)
	.. controls +(-60:0.3cm) and +(180:0.8cm) .. (0,-0.4)
	.. controls +(0:0.8cm) and +(-120:0.3cm) .. (1.4,0.2);
\node at (-2.3,-1.3) {$H$};
\end{tikzpicture}
\end{equation}

By adding a 2-handle to $H$ along its core circle
(with blackboard framing,
depicted in \eqnref{e:s-matrix-2handle}
as the gray loop),
we get a cornered cobordism
$\cH_2 : H \to_N H''$,
and we have

\begin{equation}
\label{e:s-matrix-2handle}
\ZRT(H,\vphi) = 
\cD^{-1/2} \cdot
\begin{tikzpicture}
\node[ellipse_morphism={0}] (a) at (0,1) {\;\;};
\draw[midarrow_rev={0.9}] (a)
	.. controls +(0:1cm) and +(90:0.6cm) .. (1.9,0)
	.. controls +(-90:0.6cm) and +(0:1cm) .. (0,-1)
	.. controls +(180:1cm) and +(-90:0.6cm) .. (-1.9,0)
	.. controls +(90:0.6cm) and +(180:1cm) .. (a);
\draw[midarrow_rev={0.7}] (a) -- (0,1.35);
\draw (-0.05,1.35) -- (0.05,1.35);
\node at (-0.25,1.25) {\smallerer $Y$};
\node at (-1.25,1) {\smallerer $j$};
\draw (-2.5,0)
	.. controls +(-90:0.6cm) and +(180:2cm) .. (0,-1.5)
	.. controls +(0:2cm) and +(-90:0.6cm) .. (2.5,0)
	.. controls +(90:0.6cm) and +(0:2cm) .. (0,1.5)
	.. controls +(180:2cm) and +(90:0.6cm) .. (-2.5,0);
\draw (-1.23,0)
	.. controls +(60:0.3cm) and +(180:0.6cm) .. (0,0.45)
	.. controls +(0:0.6cm) and +(120:0.3cm) .. (1.23,0);
\draw (-1.4,0.2)
	.. controls +(-60:0.3cm) and +(180:0.8cm) .. (0,-0.4)
	.. controls +(0:0.8cm) and +(-120:0.3cm) .. (1.4,0.2);
\draw[opacity=0.5,med-gray,line width=5pt] (0,0.7)
	.. controls +(0:0.8cm) and +(90:0.4cm) .. (1.5,0)
	.. controls +(-90:0.4cm) and +(0:0.8cm) .. (0,-0.7)
	.. controls +(180:0.8cm) and +(-90:0.4cm) .. (-1.5,0)
	.. controls +(90:0.4cm) and +(180:0.8cm) .. (0,0.7);
\draw[regular] (1,-0.5) circle (0.2cm);
\end{tikzpicture}
=
\cD^{-1/2} \cdot
\begin{tikzpicture}
\node[ellipse_morphism={0}] (a) at (-0.95,-0.2) {\;\;};
\draw[midarrow_rev={0.9}] (a)
	.. controls +(0:0.6cm) and +(0:0.6cm) .. (-0.95,-0.6)
	.. controls +(180:0.6cm) and +(180:0.6cm) .. (a);
\draw (a) -- +(90:0.35cm);
\draw (-1,0.15) -- (-0.9,0.15);
\draw[thin_overline={1.5},regular] (-0.8,-0.5)
	.. controls +(100:0.8cm) and +(180:0.4cm) .. (0,1.8)
	.. controls +(0:0.4cm) and +(90:0.8cm) .. (0.85,0)
	.. controls +(-90:0.8cm) and +(0:0.4cm) .. (0,-1.8)
	.. controls +(180:0.4cm) and +(-80:0.3cm) .. (-0.78,-0.63);
\draw (-1.5,0)
	.. controls +(-90:2cm) and +(180:0.6cm) .. (0,-2.5)
	.. controls +(0:0.6cm) and +(-90:2cm) .. (1.5,0)
	.. controls +(90:2cm) and +(0:0.6cm) .. (0,2.5)
	.. controls +(180:0.6cm) and +(90:2cm) .. (-1.5,0);
\draw (0,-1.23)
	.. controls +(150:0.3cm) and +(-90:0.6cm) .. (-0.45,0)
	.. controls +(90:0.6cm) and +(-150:0.3cm) .. (0,1.23);
\draw (-0.2,-1.4)
	.. controls +(30:0.3cm) and +(-90:0.8cm) .. (0.4,0)
	.. controls +(90:0.8cm) and +(-30:0.3cm) .. (-0.2,1.4);
\node at (1.7,-1.6) {$H''$};
\end{tikzpicture}
\end{equation}

Then applying $\wdtld{S}_*$,

\begin{equation}
\tau(\wdtld{S}) (\vphi) =
\cD^{-1/2} \cdot
\begin{tikzpicture}
\node[ellipse_morphism={0}] (a) at (0,1.1) {\;\;};
\draw (a) .. controls +(180:0.3cm) and +(180:0.3cm) .. (0,0.6);
\draw[regular,thin_overline={1.5}] (0,0.85)
	.. controls +(0:1cm) and +(90:0.6cm) .. (1.9,0)
	.. controls +(-90:0.6cm) and +(0:1cm) .. (0,-1)
	.. controls +(180:1cm) and +(-90:0.6cm) .. (-1.9,0)
	.. controls +(90:0.6cm) and +(180:1cm) .. (0,0.85);
\draw[thin_overline={1.5},midarrow_rev={0.9}] (a)
	.. controls +(0:0.3cm) and +(0:0.3cm) .. (0,0.6);
\draw (a) -- (0,1.35);
\draw (-0.05,1.35) -- (0.05,1.35);
\draw (-2.5,0)
	.. controls +(-90:0.6cm) and +(180:2cm) .. (0,-1.5)
	.. controls +(0:2cm) and +(-90:0.6cm) .. (2.5,0)
	.. controls +(90:0.6cm) and +(0:2cm) .. (0,1.5)
	.. controls +(180:2cm) and +(90:0.6cm) .. (-2.5,0);
\draw (-1.23,0)
	.. controls +(60:0.3cm) and +(180:0.6cm) .. (0,0.45)
	.. controls +(0:0.6cm) and +(120:0.3cm) .. (1.23,0);
\draw (-1.4,0.2)
	.. controls +(-60:0.3cm) and +(180:0.8cm) .. (0,-0.4)
	.. controls +(0:0.8cm) and +(-120:0.3cm) .. (1.4,0.2);
\node at (-2.3,-1.3) {$H$};
\end{tikzpicture}
\end{equation}

The $T$-matrix\footnotemark[\value{footnote}] \;
$T = \begin{pmatrix} 1 & -1 \\ 0 & 1 \end{pmatrix}$
acts by left-hand Dehn twist.
\footnotetext{Note that these are not the same as the $S$ and $T$-matrices
in \ocite{BK}, but are conjugate to them under $\text{diag}(-1,1)$.}
Choosing a good $\wdtld{T}$ that fixes $\VV$,
it is clear that
\begin{equation}
\tau(\wdtld{T})(\vphi) = 
\begin{tikzpicture}
\node[ellipse_morphism={0}] (a) at (0,1) {\;\;};
\node[small_morphism] (b) at (0,-1) {\smallerer $\theta$};
\draw (a)
	.. controls +(0:1cm) and +(90:0.6cm) .. (1.9,0)
	.. controls +(-90:0.6cm) and +(0:1cm) .. (b);
\draw[midarrow_rev={0.8}] (b)
	.. controls +(180:1cm) and +(-90:0.6cm) .. (-1.9,0)
	.. controls +(90:0.6cm) and +(180:1cm) .. (a);
\draw[midarrow_rev={0.7}] (a) -- (0,1.35);
\draw (-0.05,1.35) -- (0.05,1.35);
\node at (-0.25,1.25) {\smallerer $Y$};
\node at (-1.25,1) {\smallerer $j$};
\draw (-2.5,0)
	.. controls +(-90:0.6cm) and +(180:2cm) .. (0,-1.5)
	.. controls +(0:2cm) and +(-90:0.6cm) .. (2.5,0)
	.. controls +(90:0.6cm) and +(0:2cm) .. (0,1.5)
	.. controls +(180:2cm) and +(90:0.6cm) .. (-2.5,0);
\draw (-1.23,0)
	.. controls +(60:0.3cm) and +(180:0.6cm) .. (0,0.45)
	.. controls +(0:0.6cm) and +(120:0.3cm) .. (1.23,0);
\draw (-1.4,0.2)
	.. controls +(-60:0.3cm) and +(180:0.8cm) .. (0,-0.4)
	.. controls +(0:0.8cm) and +(-120:0.3cm) .. (1.4,0.2);
\node at (-2.3,-1.3) {$H$};
\end{tikzpicture}
\end{equation}

It is clear that these agree with \ocite{BK}*{Definition 3.1.15}.

\begin{proposition}[\ocite{BK}*{Theorem 3.1.16}]
Using the identification
$\ZCYsk(H;\VV) \simeq \dirsum_j \Hom_\cA(Y, X_j X_j^*)$
from \eqnref{e:H-sk-param},
$\tau(\wdtld{S})^4 = - \circ \theta_Y^\inv$,
$(\tau(\wdtld{S})\tau(\wdtld{T}))^3 = \kappa \cdot \tau(\wdtld{S})^2$.
\label{p:sl2z-repn}
\end{proposition}

\begin{proof}
We will use \lemref{l:tau-composition} repeatedly.
Let us compare $\tau(\wdtld{S}^2)$ against $\tau(\wdtld{S})^2$.
Since $\chi(H) = 1-g = 0$, the $\cD$ factor does not contribute
to the error factor (this is the case for all computations here).
The kernel space $L \subset H_1(N;\RR) = V$ for the inclusion $N\subset H$
is spanned by the meridian.
The homeomorphism $\wdtld{S}$ swaps meridian $m$ and longitude $l$,
so if $L' \subset H_1(N;\RR)$ is spanned by $l$,
then the exponent of $\kappa$ in the error factor is
$\sigma(V;L,L',L) = 0$.
Thus, $\tau(\wdtld{S}^2) = \tau(\wdtld{S})^2$,
and similarly, $\tau(\wdtld{S}^4) = \tau(\wdtld{S}^2)^2 = \tau(\wdtld{S})^4$.

Since $\wdtld{S}^4$ is isotopic to the identity,
it extends to $H$.
More explicitly, it is easy to see from \eqnref{e:S-tld}
that $\wdtld{S}^4$ is the identity everywhere expect near $b$,
where it is rotated $360^\circ$ counter-clockwise,
dragging its surroundings with it.
It extends into the interior,
introducing a right-hand twist to the ribbon attached to $b$.
Thus, on the level of the skein space,
this amounts to precomposing with $\theta_Y^\inv$
(which corresponds to a right-hand twist under our convention).

Now let us consider the second equation.
Since $\wdtld{T}$ preserves the meridian,
by similar considerations as above, we have
$\tau(\wdtld{S}\wdtld{T}) = \tau(\wdtld{S})\tau(\wdtld{T})$.

Write $f = \wdtld{S}\wdtld{T}$ for simplicity,
and let $L \subseteq H_1(N;\RR) = V$ be spanned by the meridian as before.
We have 
\begin{align*}
\tau(f^2) &= \kappa^{\sigma(V;f_*^2(L),f_*(L),L)} \cdot \tau(f)^2
\\
\tau(f^3) &= \kappa^{\sigma(V;f_*^3(L),f_*(L),L)} \cdot \tau(f^2)\tau(f)
\end{align*}
It is easy to check that $f_*(\vec{m}) = - \vec{l}$,
$f_*^2(\vec{m}) = - \vec{m} - \vec{l}$,
$f_*^3(\vec{m}) = - \vec{m}$.
So $\sigma(V;f_*^3(L),f_*(L),L) = \sigma(V;L,f_*(L),L) = 0$.

Recall that $\sigma(V;f_*^2(L),f_*(L),L)$
is the signature of a certain symmetric bilinear form $\Psi$
on
\[
\ov{V} =
	(f_*^2(L) \cap (f_*(L) + L))/(f_*^2(L) \cap f_*(L) + f_*^2(L) \cap L)
	= f_*^2(L)
\]
Consider $\vec{m} + \vec{l} \in f_*^2(L)$,
and $-\vec{l} \in f_*(L), \vec{m} \in L$;
they sum to 0, so
\[
\Psi(\vec{m} + \vec{l}, \vec{m} + \vec{l})
= \omega(\vec{m} + \vec{l}, -\vec{l})
= -1
\]
where $\omega$ is the intersection form on $V$.
Thus,
\[
\tau(f^3) = \tau(f^2) \tau(f)
= \kappa^\inv \cdot \tau(f)^3
\]
Since $(\wdtld{S}\wdtld{T})^3 = \wdtld{S}^2$,
\[
(\tau(\wdtld{S})\tau(\wdtld{T}))^3
= \tau(f)^3
= \kappa \cdot \tau(f^3)
= \kappa \cdot \tau(\wdtld{S}^2)
= \kappa \cdot \tau(\wdtld{S})^2
\]
\end{proof}

\begin{remark}
Comparison with notation in \ocite{rt-3mfld}
(symbols on left are from \ocite{rt-3mfld},
which we recast in terms of our notation):
\begin{itemize}
\item $d_i = \dim X_i / p_+$ (see page 558, around (3.1.6))
\item $v_i = \theta_i$
\item $F$ simply means evaluate colored ribbon graph as morphism
\item $S_{i,j} = tr(c_{X_j,X_i} c_{X_i,X_j})$ but here $c$ is right-hand braid
\item $C = p_- / p_+$.
\item $\sigma_-(L)$ is number of non-positive eigenvalues.
Then $\sigma_-(L) = (|L| - \sigma(L) + null(L)) / 2$,
where $null(L) = \dim H_1$ is the nullity of the intersection matrix.
\item $\{L\}$ is label all strands by $d_i \id_i$, but here $d_i$ is their
definition,
so its $\{L\} = p_+^{-|L|} \cdot \Omega L$,
$\Omega L$ is label with regular coloring (our definition)
\item Thus $F(M) = F(M;L) = C^{-\sigma_-(L)} \{L\}
= (p_+/p_-)^{(|L| - \sigma(L) + null(L))/2} p_+^{-|L|} \cdot \Omega L
= \kappa^{-\sigma(L) + null(L)} \cD^{-|L|/2} \cdot \Omega L$
\end{itemize}

We also note that the $\Omega$ coloring in \ocite{roberts-chainmail}
is actually $\eta \cdot \Omega$ in our notation,
where $\eta = \cD^{-1/2}$.
\label{r:rt-notation}
\end{remark}



\begin{bibdiv}
\begin{biblist}

\bib{akbulut}{article}{
   author={Akbulut, Selman},
   title={On $2$-dimensional homology classes of $4$-manifolds},
   journal={Math. Proc. Cambridge Philos. Soc.},
   volume={82},
   date={1977},
   number={1},
   pages={99--106},
   issn={0305-0041},
   review={\MR{433476}},
   doi={10.1017/S0305004100053718},
}
\bib{atiyah}{article}{
   author={Atiyah, Michael},
   title={Topological quantum field theories},
   journal={Inst. Hautes \'{E}tudes Sci. Publ. Math.},
   number={68},
   date={1988},
   pages={175--186 (1989)},
   issn={0073-8301},
   review={\MR{1001453}},
}
\bib{novikov}{article}{
   author={Atiyah, M. F.},
   author={Singer, I. M.},
   title={The index of elliptic operators. III},
   journal={Ann. of Math. (2)},
   volume={87},
   date={1968},
   pages={546--604},
   issn={0003-486X},
   review={\MR{236952}},
   doi={10.2307/1970717},
}
\bib{PL-az}{article}{
   author={Armstrong, M. A.},
   author={Zeeman, E. C.},
   title={Piecewise linear transversality},
   journal={Bull. Amer. Math. Soc.},
   volume={73},
   date={1967},
   pages={184--188},
   issn={0002-9904},
   review={\MR{206964}},
   doi={10.1090/S0002-9904-1967-11708-7},
}
%
%
%
%
\bib{BK}{book}{
   label={BakK2001},
   author={Bakalov, Bojko},
   author={Kirillov, Alexander, Jr.},
   title={Lectures on tensor categories and modular functors},
   series={University Lecture Series},
   volume={21},
   publisher={American Mathematical Society},
   place={Providence, RI},
   date={2001},
   pages={x+221},
   isbn={0-8218-2686-7},
}
\bib{balsam-kirillov}{article}{ 
  author={Balsam, Benjamin },
  author={Kirillov, Alexander, Jr},
  title={Turaev-Viro invariants as an extended TQFT},
  eprint={arXiv:1004.1533},
}
\bib{barrett-obs}{article}{
   author={Barrett, John W.},
   author={Faria Martins, Jo\~{a}o},
   author={Garc\'{\i}a-Islas, J. Manuel},
   title={Observables in the Turaev-Viro and Crane-Yetter models},
   journal={J. Math. Phys.},
   volume={48},
   date={2007},
   number={9},
   pages={093508, 18},
   issn={0022-2488},
   review={\MR{2355097}},
   doi={10.1063/1.2759440},
}
%
%
\bib{BHLZ}{article}{
   author={Beliakova, Anna},
   author={Habiro, Kazuo},
   author={Lauda, Aaron D.},
   author={\v{Z}ivkovi\'{c}, Marko},
   title={Trace decategorification of categorified quantum $\mathfrak{sl}_2$},
   journal={Math. Ann.},
   volume={367},
   date={2017},
   number={1-2},
   pages={397--440},
   issn={0025-5831},
   eprint={arXiv:1404.1806},
}
\bib{BG-framed}{article}{
   author={Bellingeri, Paolo},
   author={Gervais, Sylvain},
   title={Surface framed braids},
   journal={Geom. Dedicata},
   volume={159},
   date={2012},
   pages={51--69},
   issn={0046-5755},
   review={\MR{2944520}},
   doi={10.1007/s10711-011-9645-5},
}
%
\bib{bryant-PL}{article}{
   author={Bryant, John L.},
   title={Piecewise linear topology},
   conference={
      title={Handbook of geometric topology},
   },
   book={
      publisher={North-Holland, Amsterdam},
   },
   date={2002},
   pages={219--259},
   review={\MR{1886671}},
}
\bib{cairns}{article}{
   author={Cairns, Stewart S.},
   title={A simple triangulation method for smooth manifolds},
   journal={Bull. Amer. Math. Soc.},
   volume={67},
   date={1961},
   pages={389--390},
   issn={0002-9904},
   review={\MR{149491}},
   doi={10.1090/S0002-9904-1961-10631-9},
}
\bib{casali-relative-pachner}{article}{
   author={Casali, Maria Rita},
   title={A note about bistellar operations on PL-manifolds with boundary},
   journal={Geom. Dedicata},
   volume={56},
   date={1995},
   number={3},
   pages={257--262},
   issn={0046-5755},
   review={\MR{1340786}},
   doi={10.1007/BF01263566},
}
\bib{cooke}{article}{
    author={Cooke, Juliet},
    title={Excision of Skein Categories and Factorisation Homology},
    date={2019},
    eprint={arXiv:1910.02630},
}
\bib{CKY-eval}{article}{
   author={Crane, Louis},
   author={Kauffman, Louis H.},
   author={Yetter, David},
   title={Evaluating the Crane-Yetter invariant},
   conference={
      title={Quantum topology},
   },
   book={
      series={Ser. Knots Everything},
      volume={3},
      publisher={World Sci. Publ., River Edge, NJ},
   },
   date={1993},
   pages={131--138},
   review={\MR{1273570}},
}
\bib{CKY}{article}{
   author={Crane, Louis},
   author={Kauffman, Louis H.},
   author={Yetter, David N.},
   title={State-sum invariants of $4$-manifolds},
   journal={J. Knot Theory Ramifications},
   volume={6},
   date={1997},
   number={2},
   pages={177--234},
   issn={0218-2165},
   review={\MR{1452438}},
   doi={10.1142/S0218216597000145},
}
\bib{CY}{article}{
    author={Crane, Louis},
    author={Yetter, David},
    title={A categorical construction of $4$D topological quantum field
        theories},
    conference={
        title={Quantum topology},
    },
    book={
        series={Ser. Knots Everything},
        volume={3},
        publisher={World Sci. Publ., River Edge, NJ},
    },
    date={1993},
    pages={120--130},
}
\bib{DSSP}{article}{
    title={Dualizable tensor categories},
    author={Douglas, Christopher},
    author={Schommer-Pries, Christopher},
    author={Snyder, Noah},
    year={2013},
    eprint={ arXiv:1312.7188}
}
\bib{edwards-kirby}{article}{
    author={Edwards, Robert D.},
    author={Kirby, Robion C.},
    title={Deformations of spaces of imbeddings},
    journal={Ann. Math. (2)},
    volume={93},
    date={1971},
    pages={63--88},
    review={\MR{0283802}},
    doi={10.2307/1970753},
}
\bib{ENO2005}{article}{
   label={ENO2005},
   author={Etingof, Pavel},
   author={Nikshych, Dmitri},
   author={Ostrik, Viktor},
   title={On fusion categories},
   journal={Ann. of Math. (2)},
   volume={162},
   date={2005},
   number={2},
   pages={581--642},
   issn={0003-486X},
}
\bib{ENO10}{article}{
   author={Etingof, Pavel},
   author={Nikshych, Dmitri},
   author={Ostrik, Victor},
   title={Fusion categories and homotopy theory},
   note={With an appendix by Ehud Meir},
   journal={Quantum Topol.},
   volume={1},
   date={2010},
   number={3},
   pages={209--273},
   issn={1663-487X},
}
\bib{EGNO}{book}{
   author={Etingof, Pavel},
   author={Gelaki, Shlomo},
   author={Nikshych, Dmitri},
   author={Ostrik, Victor},
   title={Tensor categories},
   series={Mathematical Surveys and Monographs},
   volume={205},
   publisher={American Mathematical Society, Providence, RI},
   date={2015},
   pages={xvi+343},
   isbn={978-1-4704-2024-6},
}	
\bib{sphere-braid}{article}{
   author={Fadell, Edward},
   author={Van Buskirk, James},
   title={On the braid groups of $E^{2}$ and $S^{2}$},
   journal={Bull. Amer. Math. Soc.},
   volume={67},
   date={1961},
   pages={211--213},
   issn={0002-9904},
   review={\MR{125578}},
   doi={10.1090/S0002-9904-1961-10570-3},
}
\bib{farb-margalit}{book}{
   author={Farb, Benson},
   author={Margalit, Dan},
   title={A primer on mapping class groups},
   series={Princeton Mathematical Series},
   volume={49},
   publisher={Princeton University Press, Princeton, NJ},
   date={2012},
   pages={xiv+472},
   isbn={978-0-691-14794-9},
   review={\MR{2850125}},
}
\bib{freyd}{article}{
   author={Johnson-Freyd, Theo},
   title={Heisenberg-picture quantum field theory},
   date={2015},
   eprint={arXiv:1508.05908},
}
\bib{GNN}{article}{
    author={Gelaki, Shlomo},
    author={Naidu, Deepak},
    author={Nikshych, Dmitri},
    title={Centers of graded fusion categories},
    journal={Algebra Number Theory},
    volume={3},
    date={2009},
    number={8},
    pages={959--990},
    issn={1937-0652},
}
\bib{gompf-stipsicz}{book}{
   author={Gompf, Robert E.},
   author={Stipsicz, Andr\'{a}s I.},
   title={$4$-manifolds and Kirby calculus},
   series={Graduate Studies in Mathematics},
   volume={20},
   publisher={American Mathematical Society, Providence, RI},
   date={1999},
   pages={xvi+558},
   isbn={0-8218-0994-6},
   review={\MR{1707327}},
   doi={10.1090/gsm/020},
}
\bib{hirschmazur}{book}{
   author={Hirsch, Morris W.},
   author={Mazur, Barry},
   title={Smoothings of piecewise linear manifolds},
   series={Annals of Mathematics Studies, No. 80},
   publisher={Princeton University Press, Princeton, N. J.; University of
   Tokyo Press, Tokyo},
   date={1974},
   pages={ix+134},
   review={\MR{0415630}},
}
\bib{hudson}{book}{
   author={Hudson, J. F. P.},
   title={Piecewise linear topology},
   note={University of Chicago Lecture Notes prepared with the assistance of
   J. L. Shaneson and J. Lees},
   publisher={W. A. Benjamin, Inc., New York-Amsterdam},
   date={1969},
   pages={ix+282},
   review={\MR{0248844}},
}
\bib{jones}{article}{
   author={Jones, Vaughan F. R.},
   title={A polynomial invariant for knots via von Neumann algebras},
   journal={Bull. Amer. Math. Soc. (N.S.)},
   volume={12},
   date={1985},
   number={1},
   pages={103--111},
   issn={0273-0979},
   review={\MR{766964}},
   doi={10.1090/S0273-0979-1985-15304-2},
}
\bib{kauffman-skein}{article}{
   author={Kauffman, Louis H.},
   title={State models and the Jones polynomial},
   journal={Topology},
   volume={26},
   date={1987},
   number={3},
   pages={395--407},
   issn={0040-9383},
   review={\MR{899057}},
   doi={10.1016/0040-9383(87)90009-7},
}
\bib{kervairemilnor}{article}{
   author={Kervaire, Michel A.},
   author={Milnor, John W.},
   title={Groups of homotopy spheres. I},
   journal={Ann. of Math. (2)},
   volume={77},
   date={1963},
   pages={504--537},
   issn={0003-486X},
   review={\MR{148075}},
   doi={10.2307/1970128},
}
\bib{kirillov-PL}{article}{
   author={Kirillov, Alexander, Jr.},
   title={On piecewise linear cell decompositions},
   journal={Algebr. Geom. Topol.},
   volume={12},
   date={2012},
   number={1},
   pages={95--108},
   issn={1472-2747},
   review={\MR{2889547}},
   doi={10.2140/agt.2012.12.95},
}
\bib{kirillov-stringnet}{article}{ 
  author={Kirillov, Alexander, Jr},
  title={String-net model of Turaev-Viro invariants},
  eprint={arXiv:1106.6033},
}
\bib{KT}{article}{
  label={KT},
  title={Factorization Homology and 4D TQFT},
  author={Kirillov, Alexander},
	author={Tham, Ying Hong},
  year={2020},
  eprint={2002.08571},
  archivePrefix={arXiv},
  primaryClass={math.QA}
}
\bib{ko-smolinsky}{article}{
   author={Ko, Ki Hyoung},
   author={Smolinsky, Lawrence},
   title={The framed braid group and $3$-manifolds},
   journal={Proc. Amer. Math. Soc.},
   volume={115},
   date={1992},
   number={2},
   pages={541--551},
   issn={0002-9939},
   review={\MR{1126197}},
   doi={10.2307/2159278},
}
%
\bib{maclane}{book}{
   author={Mac Lane, Saunders},
   title={Categories for the working mathematician},
   series={Graduate Texts in Mathematics},
   volume={5},
   edition={2},
   publisher={Springer-Verlag, New York},
   date={1998},
   pages={xii+314},
   isbn={0-387-98403-8},
   review={\MR{1712872}},
}
\bib{milnor-cobord}{book}{
   author={Milnor, John},
   title={Lectures on the $h$-cobordism theorem},
   note={Notes by L. Siebenmann and J. Sondow},
   publisher={Princeton University Press, Princeton, N.J.},
   date={1965},
   pages={v+116},
   review={\MR{0190942}},
}
\bib{muger}{article}{
    author={M\"{u}ger, Michael},
    title={On the structure of modular categories},
    journal={Proc. London Math. Soc. (3)},
    volume={87},
    date={2003},
    number={2},
    pages={291--308},
    issn={0024-6115},
    review={\MR{1990929}},
    doi={10.1112/S0024611503014187},
}
\bib{Ocn}{article}{
  label={O},
  author={Ocneanu, Adrian},
  title={Chirality for operator algebras},
  conference={
    title={Subfactors},
     address={Kyuzeso},
     date={1993},
  },
  book={
     publisher={World Sci. Publ., River Edge, NJ},
  },
  date={1994},
  pages={39--63},
  review={\MR{1317353}},
}
\bib{ooguri}{article}{
   author={Ooguri, Hirosi},
   title={Topological lattice models in four dimensions},
   journal={Modern Phys. Lett. A},
   volume={7},
   date={1992},
   number={30},
   pages={2799--2810},
   issn={0217-7323},
   review={\MR{1184565}},
   doi={10.1142/S0217732392004171},
}
\bib{Pachner-kons}{article}{
   author={Pachner, U.},
   title={Konstruktionsmethoden und das kombinatorische Hom\"{o}omorphieproblem
   f\"{u}r Triangulationen kompakter semilinearer Mannigfaltigkeiten},
   language={German},
   journal={Abh. Math. Sem. Univ. Hamburg},
   volume={57},
   date={1987},
   pages={69--86},
   issn={0025-5858},
   review={\MR{927165}},
   doi={10.1007/BF02941601},
}
\bib{Pachner-PL}{article}{
   author={Pachner, Udo},
   title={P.L. homeomorphic manifolds are equivalent by elementary
   shellings},
   journal={European J. Combin.},
   volume={12},
   date={1991},
   number={2},
   pages={129--145},
   issn={0195-6698},
   review={\MR{1095161}},
   doi={10.1016/S0195-6698(13)80080-7},
}
\bib{penrose-string}{article}{
   author={Penrose, Roger},
   title={Applications of negative dimensional tensors},
   conference={
      title={Combinatorial Mathematics and its Applications (Proc. Conf.,
      Oxford, 1969)},
   },
   book={
      publisher={Academic Press, London},
   },
   date={1971},
   pages={221--244},
   review={\MR{0281657}},
}
\bib{rt-ribbon}{article}{
   author={Reshetikhin, N. Yu.},
   author={Turaev, V. G.},
   title={Ribbon graphs and their invariants derived from quantum groups},
   journal={Comm. Math. Phys.},
   volume={127},
   date={1990},
   number={1},
   pages={1--26},
   issn={0010-3616},
   review={\MR{1036112}},
}
\bib{rt-3mfld}{article}{
   author={Reshetikhin, N.},
   author={Turaev, V. G.},
   title={Invariants of $3$-manifolds via link polynomials and quantum
   groups},
   journal={Invent. Math.},
   volume={103},
   date={1991},
   number={3},
   pages={547--597},
   issn={0020-9910},
   review={\MR{1091619}},
   doi={10.1007/BF01239527},
}
\bib{roberts-chainmail}{article}{
   author={Roberts, Justin},
   title={Skein theory and Turaev-Viro invariants},
   journal={Topology},
   volume={34},
   date={1995},
   number={4},
   pages={771--787},
   issn={0040-9383},
   review={\MR{1362787}},
   doi={10.1016/0040-9383(94)00053-0},
}
\bib{PL-rourke-sanderson}{book}{
   author={Rourke, C. P.},
   author={Sanderson, B. J.},
   title={Introduction to piecewise-linear topology},
   series={Ergebnisse der Mathematik und ihrer Grenzgebiete, Band 69},
   publisher={Springer-Verlag, New York-Heidelberg},
   date={1972},
   pages={viii+123},
   review={\MR{0350744}},
}
\bib{saveliev}{book}{
   author={Saveliev, Nikolai},
   title={Lectures on the topology of 3-manifolds},
   series={De Gruyter Textbook},
   edition={Second revised edition},
   note={An introduction to the Casson invariant},
   publisher={Walter de Gruyter \& Co., Berlin},
   date={2012},
   pages={xii+207},
   isbn={978-3-11-025035-0},
   review={\MR{2893651}},
}
\bib{stasheff}{article}{
    author={Stasheff, James Dillon},
    title={Homotopy associativity of $H$-spaces. I, II},
    journal={Trans. Amer. Math. Soc. 108 (1963), 275-292; ibid.},
    volume={108},
    date={1963},
    pages={293--312},
    issn={0002-9947},
    review={\MR{0158400}},
    doi={10.1090/s0002-9947-1963-0158400-5},
}
\bib{tham_elliptic}{article}{
    author={Tham, Ying Hong},
    title={The Elliptic Drinfeld Center of a Premodular Category},
    year={2019},
    eprint={arxiv:1904.09511}
}
\bib{tham_reduced}{article}{
    author={Tham, Ying Hong},
    title={Reduced Tensor Product on the Drinfeld Center},
    year={2020},
    eprint={arxiv:2004.09611}
}
\bib{thom-cobordism}{article}{
   author={Thom, Ren\'{e}},
   title={Quelques propri\'{e}t\'{e}s globales des vari\'{e}t\'{e}s diff\'{e}rentiables},
   language={French},
   journal={Comment. Math. Helv.},
   volume={28},
   date={1954},
   pages={17--86},
   issn={0010-2571},
   review={\MR{61823}},
   doi={10.1007/BF02566923},
}
\bib{TV}{article}{
   author={Turaev, V. G.},
   author={Viro, O. Ya.},
   title={State sum invariants of $3$-manifolds and quantum $6j$-symbols},
   journal={Topology},
   volume={31},
   date={1992},
   number={4},
   pages={865--902},
   issn={0040-9383},
   review={\MR{1191386}},
   doi={10.1016/0040-9383(92)90015-A},
}
\bib{wall}{article}{
   author={Wall, C. T. C.},
   title={Non-additivity of the signature},
   journal={Invent. Math.},
   volume={7},
   date={1969},
   pages={269--274},
   issn={0020-9910},
   review={\MR{246311}},
   doi={10.1007/BF01404310},
}
\bib{wassermansymm}{article}{
  label={Was},
  author={Wasserman, Thomas A.},
  title={The symmetric tensor product on the Drinfeld centre of a symmetric
  fusion category},
  journal={J. Pure Appl. Algebra},
  volume={224},
  date={2020},
  number={8},
  pages={106348},
  issn={0022-4049},
  review={\MR{4074586}},
  doi={10.1016/j.jpaa.2020.106348},
}
\bib{wasserman2fold}{article}{
  label={Was2},
  author={Wasserman, Thomas A.},
  title={The Drinfeld centre of a symmetric fusion category is 2-fold
  monoidal},
  journal={Adv. Math.},
  volume={366},
  date={2020},
  pages={107090},
  issn={0001-8708},
  review={\MR{4072794}},
  doi={10.1016/j.aim.2020.107090},
}
\bib{witten}{article}{
   author={Witten, Edward},
   title={Topological quantum field theory},
   journal={Comm. Math. Phys.},
   volume={117},
   date={1988},
   number={3},
   pages={353--386},
   issn={0010-3616},
   review={\MR{953828}},
}
\bib{yetter-hopf}{article}{
   author={Yetter, David},
   title={Portrait of the handle as a Hopf algebra},
   conference={
      title={Geometry and physics},
      address={Aarhus},
      date={1995},
   },
   book={
      series={Lecture Notes in Pure and Appl. Math.},
      volume={184},
      publisher={Dekker, New York},
   },
   date={1997},
   pages={481--502},
   review={\MR{1423189}},
}
\end{biblist}
\end{bibdiv}
\end{document}